\documentclass[11pt, a4paper]{article}
\title{Long-time dynamics for the Kelvin-Helmholtz equations\\
close to circular vortex sheets}

\author{
   Federico Murgante\thanks{Università Statale di Milano, Italy. \href{mailto:federico.murgante@unimi.it}{\texttt{federico.murgante@unimi.it}}}
    \and   
  Emeric Roulley\thanks{Scuola Internazionale Superiore di Studi Avanzati (SISSA), Trieste, Italy. \href{mailto:eroulley@sissa.it}{\texttt{eroulley@sissa.it}}}
    \and
    Stefano Scrobogna\thanks{Università degli Studi di Trieste, Dipartimento di Matematica, Informatica e Geoscienze, Italy. \href{mailto:stefano.scrobogna@units.it}{\texttt{stefano.scrobogna@units.it}}}
}
\date{}

\usepackage[a4paper]{geometry}
\usepackage{tipa}
\usepackage{empheq}
\usepackage{times}
\usepackage{stmaryrd}
\usepackage{fourier}
\usepackage[T1]{fontenc}
\usepackage{amssymb, amsmath,  bm, mathrsfs}
\usepackage{amsfonts}
\usepackage{extarrows}
\usepackage[english]{babel}
\usepackage{amssymb,amsthm,math}
\usepackage{mathtools}
\DeclareMathAlphabet{\mathcal}{OMS}{cmsy}{m}{n}
\DeclareFontFamily{U}{mathc}{}
\DeclareFontShape{U}{mathc}{m}{it}
{<->x*[1.03] mathc10}{}
\DeclareMathAlphabet{\mathscr}{U}{mathc}{m}{it}

\DeclareMathAlphabet{\mathpzc}{OT1}{pzc}{m}{it}

\usepackage{hyperref}
\usepackage[capitalise]{cleveref}
\usepackage{cite}
\allowdisplaybreaks[1]
\usepackage{xfrac}
\usepackage[utf8]{inputenc}
\usepackage[T1]{fontenc}
\usepackage{enumerate}
\usepackage{accents}
\usepackage{bbm}
\usepackage{thmtools}
\usepackage{todonotes}

\usepackage{upgreek}
\usepackage{tikz}
\usetikzlibrary{arrows, automata, backgrounds, shadows, patterns, calc, hobby, plotmarks, shapes}
\usetikzlibrary{shapes.misc}

\tikzset{cross/.style={cross out, draw=black, minimum size=2*(#1-\pgflinewidth), inner sep=0pt, outer sep=0pt},
cross/.default={1pt}}
\allowdisplaybreaks

\makeatletter
\@ifpackageloaded{stix}{%
}{%
  \DeclareFontEncoding{LS2}{}{\noaccents@}
  \DeclareFontSubstitution{LS2}{stix}{m}{n}
  \DeclareSymbolFont{stix@largesymbols}{LS2}{stixex}{m}{n}
  \SetSymbolFont{stix@largesymbols}{bold}{LS2}{stixex}{b}{n}
  \DeclareMathDelimiter{\lBrace}{\mathopen} {stix@largesymbols}{"E8}%
                                            {stix@largesymbols}{"0E}
  \DeclareMathDelimiter{\rBrace}{\mathclose}{stix@largesymbols}{"E9}%
                                            {stix@largesymbols}{"0F}
}
\makeatother

\DeclareSymbolFontAlphabet{\amsmathbb}{AMSb}%

\usepackage{listings}
\usepackage{color}

\definecolor{dkgreen}{rgb}{0,0.6,0}
\definecolor{gray}{rgb}{0.5,0.5,0.5}
\definecolor{mauve}{rgb}{0.58,0,0.82}

\makeatletter
\def\maketag@@@#1{\hbox{\m@th\normalfont\normalsize#1}}
\makeatother

\newcommand{\mats}[1]{\pare{#1}^{2\times2}}
\newcommand{\kap}{\gamma}

\newcommand{\dd}{\textnormal{d}}

\DeclareMathOperator{\sgn}{sgn}

\newcommand{\pare}[1]{\left( #1 \right)}
\newcommand{\bra}[1]{\left[ #1 \right]}
\newcommand{\bbra}[1]{\left\llbracket #1 \right\rrbracket}

\newcommand{\angles}[1]{\left\langle #1 \right\rangle}

\newcommand{\av}[1]{\left| #1 \right|}
\newcommand{\pbra}[2]{\set{ #1 \big. , \  #2 }}

\newcommand{\xfrac}[2]{\left. #1 \middle/ #2 \right. }
\newcommand{\set}[1]{\left\{ #1 \right\}}
\newcommand{\system}[1]{\left\{ #1 \right.}

\newcommand{\ceil}[1]{\left\lceil #1 \right\rceil}
\DeclareMathOperator{\OpBWPlain}{Op^{BW}}
\newcommand{\OpBW}[1]{{\OpBWPlain\left(#1\right)}}

\DeclareMathOperator{\OpvecPlain}{Op^{BW}_{vec}}
\newcommand{\Opvec}[1]{\OpvecPlain\left(#1\right)}

\newcommand{\pv}{\textnormal{p.v.}}

\newcommand{\Id}{\textnormal{Id}}
 
\newcommand{\JC}{{\bm J}_\mathbb{C}}
\newcommand{\EC}{{\bm E}_\mathbb{C}}

\newcommand{\defeq}{\triangleq}
\newcommand{\eqdef}{\triangleq}
\newcommand{\Cast}[2]{C^{#1}_\ast \pare{I; \dot H^{#2}\pare{\mathbb{T};\C^2}}}
\newcommand{\BallR}[2]{B^{#1}_{#2, \R}\pare{I;\epsilon_0}}
\newcommand{\Ball}[2]{B^{#1}_{#2}\pare{I;\epsilon_0}}

\def\fint{\mathop{\,\rlap{--}\!\!\!\int}\nolimits}
\newcommand{\intT}{\int_{\mathbb{T}}}
\newcommand{\vect}[2]{  \begin{bmatrix} #1 \\ #2 \end{bmatrix}  }

\renewcommand{\Re}{\textnormal{Re}}
\newcommand{\X}{{\mathfrak X }}
\renewcommand{\Im}{\textnormal{Im}}

\newcommand{\RN}[1]{%
  \textup{\uppercase\expandafter{\romannumeral#1}}%
}

\newcommand{\KF}{K\cF}
\newcommand{\KR}{K\cR}
\newcommand{\KM}{K\mathcal{M}}

\newcommand{\sG}{\mathsf{G}}

\newcommand{\sH}{\mathsf{H}}
\newcommand{\sK}{\mathsf{K}}
\newcommand{\sJ}{\mathsf{J}}

\newcommand{\sM}{\mathsf{M}}

\newcommand{\sX}{\mathsf{X}}

\newcommand{\st}{\mathsf{t}}

\newcommand{\ii}{{\rm i}}

\newcommand{\ps}[2]{\pare{ \left. #1 \ \right| \ #2 }}
\newcommand{\psc}[2]{\left\langle  #1 \ \middle\vert \ #2 \right\rangle}

\newcommand{\ee}{\end{equation}}
\newcommand{\vr}{\varrho}
\renewcommand{\upgamma}{{\mathtt{b}}}
\newcommand{\ov}{\overline}
\def\wt{\widetilde}

\newcommand{\x}{\xi}
\newcommand{\pa}{\partial}
\newcommand{\mM}{\mathcal{M}}
\newcommand{\mR}{\mathcal{R}}
\newcommand{\Gt}[2]{{\tilde{\Gamma}^{#1}_{#2}}}

\newcommand{\Gr}[2]{{{\Gamma}^{#1}_{#2}[\epsilon_0]}}
\newcommand{\vOmega}{\bm{\omega}_{\kap,\upgamma}}
\newcommand{\Br}[2]{{B^{#1}_{{s_0}_{#2}}(I;\epsilon_0)}}
\newcommand{\Lcal}{{\mathcal L}}

\newcommand{\Rcal}{{\mathcal R}}
\newcommand{\sg}[3]{{{\Sigma\Gamma}^{#1}_{#2}[\epsilon_0,#3]}}
\newcommand{\sr}[3]{{{\Sigma\Rcal}^{#1}_{#2}[\epsilon_0 ,#3]}}

\theoremstyle{theorem}
\newtheorem{theorem}{Theorem}[section]
\newtheorem*{maintheorem}{Main Theorem}

\crefname{maintheorem}{Main Theorem}{Main Theorems}
\newtheorem*{theorem*}{Theorem}
\newtheorem{prop}[theorem]{Proposition}
\newtheorem{proposition}[theorem]{Proposition}
\newtheorem{lemma}[theorem]{Lemma}

\theoremstyle{definition}
\newtheorem{definition}[theorem]{Definition}

\newtheorem{rem}[theorem]{Remark}
\newtheorem{remark}[theorem]{Remark}
\newtheorem{notation}[theorem]{Notation}

\newtheorem{step}{Step}
\makeatletter
\@addtoreset{step}{subsection}
\makeatother

\makeatletter
\@addtoreset{substep}{step}
\makeatother
\newtheorem{proofpart}{Part}
\makeatletter
\@addtoreset{proofpart}{subsection}
\makeatother

\makeatletter
\@addtoreset{proofsubpart}{part}
\makeatother

\numberwithin{equation}{section}

\newcommand{\vOpbw}[1]{\Opvec{#1}}

\begin{document}

\maketitle
\noindent
\begin{abstract}

We consider the Kelvin-Helmholtz system describing the evolution of a vortex-sheet near the circular stationary solution. Answering previous numerical conjectures in the 90s physics literature, we prove an almost global existence result for small-amplitude solutions. We first establish the existence of a linear stability threshold for the Weber number, which represents the ratio between the square of the background velocity jump and the surface tension. Then, we prove that for almost all values of the Weber number below this threshold any small solution lives for almost all times, remaining close to the equilibrium. Our analysis reveals a remarkable stabilization phenomenon: the presence of both non-zero background  velocity jump and capillarity effects enables to prevent nonlinear instability phenomena, despite the inherently unstable nature of the classical Kelvin-Helmholtz problem. This long-time existence would not be achievable in a setting where capillarity alone provides linear stabilization, without the richer modulation induced by the velocity jump. Our proof exploits the Hamiltonian nature of the equations. Specifically, we employ Hamiltonian Birkhoff normal form techniques for quasi-linear systems together with a general approach for paralinearization of non-linear singular integral operators. This approach allows us to control resonances and quasi-resonances at arbitrary order, ensuring the desired long-time stability result.

\end{abstract}

{\small
{\it Keywords:} Kelvin-Helmholtz, vortex sheets, paradifferential calculus, 
Birkhoff normal form.}

	{\small\tableofcontents}

	\allowdisplaybreaks

\section{Presentation of the problem and main result}

The Kelvin-Helmholtz (KH) equations is a classic of fluid dynamics, modeling the intricate behavior of vortex sheets at the interface between fluids with different velocities. Since their introduction by Lord Kelvin and Hermann von Helmholtz in the nineteenth century \cite{Helmholtz1868,Helmholtz1889,Thomson1871,Thomson1880}, these equations have provided crucial insights into fundamental hydrodynamic phenomena, from the formation of ocean waves to atmospheric turbulence. The classic KH problem addresses the instability of a plane vortex sheet, where the jump in tangential velocity across an interface drives the system linearly, generating the well-known instability that defines Kelvin-Helmholtz phenomena. Of particular interest is the interplay between background  velocity jump ($\upgamma$) and capillarity effects ($\kap$), which together determine critical stability thresholds and equilibrium states. This work explores the mathematical structures emerging from these interactions, with special focus on the Weber number $\beta\triangleq\upgamma^2/\kap$, which naturally emerges in the equilibrium spectrum and governs the system's stabilization properties. \\

	We consider a planar Euler system for two irrotational fluids with same density (constant equal to $1$) separated by an interface $\Gamma(t)$ homeomorphic to a circle and parametrized by $\mathpzc{z}(t,\cdot):\mathbb{T}\rightarrow\mathbb{R}^2.$ This interface divides the plane into two open components $\Omega^{\pm}(t)$ with $\Omega^{-}(t)$ bounded and $\Omega^{+}(t)$ unbounded. Given  two functions $ f^{\pm}:\Omega^{\pm}\pare{t}\to \mathbb{R} $ we define
\begin{align*}
\bbra{ f^{\pm} } \defeq f^--f^+.
\end{align*} The evolutionary system is thus composed of the following equations 
	\begin{equation}\label{Euler}
		\begin{cases}
			u_t^{\pm}+u^{\pm}\cdot\nabla u^{\pm}+\nabla p^{\pm}=0, & \textnormal{in }\Omega^{\pm}(t),\\
			\big(\mathpzc{z}_t-u^{\pm}|_{\Gamma(t)}\big)\cdot \mathpzc{z}_{x}=0,& \textnormal{at }\Gamma(t),\\
			\left. \bbra{ p^\pm }\right|_{\Gamma(t)}=\kap \mathpzc{k}\pare{\mathpzc{z}}, & \textnormal{at }\Gamma(t),\\
			u^{+}(t,{\bm x}) \to 0, & \text{as } |{\bm x}|\to +\infty,\\
			\nabla\cdot u^{\pm}=0, & \textnormal{in }\Omega^{\pm}(t),\\
			\nabla^{\perp}\cdot u^{\pm}=0, & \textnormal{in }\Omega^{\pm}(t).
			
		\end{cases}
	\end{equation}
	In the above set of equations, the quantities $u^{\pm},p^{\pm}$ are respectively the velocity field and pressure inside the domain $\Omega^{\pm}.$ The parameter $\kap \geqslant 0$ is the surface tension coefficient and $\mathpzc{k}\pare{\mathpzc{z}}$ is the curvature defined by
	$$\mathpzc{k}\pare{\mathpzc{z}}\defeq
	-\frac{\mathpzc{z}_{x}^{\perp}\cdot \mathpzc{z}_{xx}}{|\mathpzc{z}_x|^3}\cdot$$
	The last equation in \eqref{Euler} implies that the vorticity distribution $\boldsymbol{\omega}$ is localized on the curve $\Gamma(t)$ at time $t,$ namely
	\begin{align}\label{eq:vorticty_Dirac}
		\boldsymbol{\omega}(t,{\bm x})=\omega(t,x)\delta\big({\bm x}-\mathpzc{z}(t,x)\big),
        &&
        \omega\defeq \bbra{u^\pm}\cdot\mathpzc{z}_x,
        &&
        {\bm x}\in\mathbb{R}^2,\quad x\in\mathbb{T}.
	\end{align}
In the case in which
\begin{align}
\mathpzc{z}\pare{t, x} = r(t,x) \  \vect{\cos \pare{x+\Omega t}}{\sin \pare{x+\Omega t}} , 
&&
\pare{t, x,\Omega}\in I\times \mathbb{T}\times \mathbb{R},&& r(t,x)\defeq\sqrt{ 1 +2\eta\pare{t, x}}, \label{z:R}
\end{align}
where $I$ is a given interval of time, the system \eqref{Euler} was recast, in \cite{MRS},   as the {\it Contour Dynamic Equation (CDE)}
\begin{equation}
\label{eq:KH2}
\system{
\begin{aligned}
& \eta_t = \Omega\eta_{x}-\frac{1}{2} \mathpzc{H}\pare{\eta} \omega,
\\
&\omega_t =
 \Omega\omega_{x}-\pare{ \frac{\omega}{2}   \mathpzc{D} _0\pare{\eta}\omega }_x - \kap \pare{\mathpzc{K}\pare{\eta}}_x,
\end{aligned}
} 
\end{equation}
with
\begin{equation}
\label{eq:integral_operators_etaomega}
\begin{aligned}
 & \mathpzc{H}\pare{\eta}\omega \triangleq \intT \frac{\eta_x \pare{x} \bra{1 - \sqrt{\frac{1+ 2 \eta\pare{y}}{1+2\eta\pare{x}}}\cos\pare{x-y}  } + \sqrt{1+2\eta\pare{x}}\sqrt{1+2\eta\pare{y}}\sin\pare{x-y} }{1 + \eta \pare{x} + \eta\pare{y} - \sqrt{1+2\eta\pare{x}}\sqrt{1+2\eta\pare{y}}\cos \pare{x-y} } \ \omega\pare{y}\dd y , 
 \\
 &\mathpzc{D}_0\pare{\eta} \omega \triangleq
 \intT \frac{1 - \sqrt{\frac{1+ 2 \eta\pare{y}}{1+2\eta\pare{x}}}\cos\pare{x-y}  }{1 + \eta \pare{x} + \eta\pare{y} - \sqrt{1+2\eta\pare{x}}\sqrt{1+2\eta\pare{y}}\cos \pare{x-y} } \ \omega\pare{y}\dd y, \\
& \mathpzc{K}\pare{\eta} \triangleq \frac{\eta_{xx}-(1+2\eta)-3\left(\frac{\eta_x}{\sqrt{1+2\eta}}\right)^2}{\left(1+2\eta+\left(\frac{\eta_x}{\sqrt{1+2\eta}}\right)^2\right)^{\frac{3}{2}}}\cdot
\end{aligned}
\end{equation}
Throughout the document, we use the notation
$$\int_{\mathbb{T}}f(x)\dd x\triangleq\frac{1}{2\pi} \pv \int_{-\pi}^{\pi}f(x)\dd x.$$
We refer the reader to \cite{MRS} for a derivation of \eqref{eq:KH2} from \eqref{Euler}.
We also warn the reader with the change of notation for the surface tension and mean vorticity with respect to \cite{MRS}. Also, here we write the system in a rotating frame with angular velocity $\Omega$ but one can easily follow the changes. An explicit computation shows that
\begin{align}\label{eq:stationary_VS}
\text{for any }  \upgamma\in \mathbb{R} , \quad\pare{\eta, \omega} = \pare{0, \upgamma}  
 \quad
 \text{is a solution of \eqref{eq:KH2}}. 
\end{align}
Let us define the background  velocity jump
\begin{equation*}
\upgamma \triangleq \int_{\mathbb{T}} \omega\pare{x}\dd x, 
\end{equation*}
which is time independent according to \eqref{eq:KH2}. We can define the invertible change of variables
\begin{align}\label{eq:omega_psi}
\psi_x \triangleq \omega - \upgamma, 
&&
\psi = \partial_x^{-1}  \pare{ \omega - \upgamma }.
\end{align}
The change of variables \eqref{eq:omega_psi} allows us to determine the evolution equation for $ \psi $ modulo a real, time-dependent constant, which is 
\begin{equation}\label{eq:psi_constant}
\psi_t=-\frac{\psi_x+\upgamma}{2}\mathpzc{D}_0\pare{\eta}\bra{\psi_x+\upgamma}-\kap\mathpzc{K}\pare{\eta} + \mathpzc{c} \pare{t}. 
\end{equation}
Since the system \eqref{eq:KH2} depends only on $\pare{\eta,\psi_x}$, the projection onto the zero-th mode in \eqref{eq:psi_constant}, which includes the constant  $ \mathpzc{c} \pare{t}$, does not influence  its dynamics. Therefore, we disregard   $\mathpzc{c} \pare{t}$ and we impose that $ \psi $ belongs to the homogeneous Sobolev space $ \dot{H}^s\pare{\mathbb{T};\mathbb{R}} \defeq \xfrac{H^s\pare{\mathbb{T};\mathbb{R}}}{\mathbb{R}} $, where $H^s\pare{\mathbb{T};\mathbb{R}}$ denotes the classical Sobolev space of periodic real-valued functions, see Section \ref{sec:preliminaries}. We can now rewrite the system \eqref{eq:KH2} in terms of the variables $ \pare{\eta, \psi} $ using \eqref{eq:psi_constant} and get
\begin{equation}\label{eq:KH3}
		\left\lbrace
		\begin{aligned}
			&\eta_t=\Omega\eta_x-\frac{1}{2} \mathpzc{H}\pare{\eta}\bra{\psi_x +\upgamma },\\
			&\psi_t=\Omega\psi_x-\pare{ \frac{\psi_x+\upgamma}{2}\mathpzc{D}_0\pare{\eta}\bra{\psi_x+\upgamma}}-\kap\mathpzc{K}\pare{\eta}.
		\end{aligned} 
		\right. 
	\end{equation}
     Notice that 
     \begin{equation}
     \label{integal rep H}
	\begin{aligned}
		\mathpzc{H}(\eta)[\omega](x)&=2\,\textnormal{BR}(\mathpzc{z})[\omega]\cdot \mathpzc{z}_{x}^{\perp}(x) \\
		&=2\int_{\mathbb{T}}\frac{(\mathpzc{z}(x)-\mathpzc{z}(y))^{\perp}\cdot \mathpzc{z}_{x}^{\perp}(x)}{|\mathpzc{z}(x)-\mathpzc{z}(y)|^2}\omega(y)\textnormal{d} y \\
		&=2\int_{\mathbb{T}}\frac{(\mathpzc{z}(x)-\mathpzc{z}(y))\cdot \mathpzc{z}_{x}(x)}{|\mathpzc{z}(x)-\mathpzc{z}(y)|^2}\omega(y)\textnormal{d} y \\
		&=\int_{\mathbb{T}}\partial_{x}\Big[\log\big(|\mathpzc{z}(x)-\mathpzc{z}(y)|^2\big)\Big]\omega(y)\textnormal{d} y.
	\end{aligned}
    \end{equation}
	Therefore, the first equation of \eqref{eq:KH3} preserves the average and the natural phase space for  $\eta $ is 
    $$ H^s_0\pare{\mathbb{T};\mathbb{R}}\defeq\set{\eta\in H^s(\T;\R)\quad\textnormal{s.t.}\quad \int_\T \eta(x)\, \di x =0}.$$

 In the absence of capillarity $(\gamma=0)$, it is now understood that the problem is ill-posed in Sobolev regularity \cite{CO1989,Lebeau2002,Wu2006}, but it admits weak solutions \cite{Delort1991}, while one has to require at least analytic regularity on the initial data in order to have a satisfactory local well-posedness theory \cite{SS1985,SSBF1981,HH2003}. It is also well-established \cite[Chap. 9.3]{MB2002} that the presence of a background velocity jump $(\upgamma)$ across the interface drives linear instability in the system, generating the classical Kelvin-Helmholtz phenomena. Conversely, surface tension $(\gamma)$ exerts a stabilizing effect at the linear level, enabling local-in-time solutions to the full nonlinear equations \cite{Ambrose2003,Ambrose2007,AM2007, CCS2008,Lannes2013_ARMA}.
When stabilization is induced solely by capillarity, the existence time is limited to $T\sim\epsilon^{-1}$, where $\epsilon$ represents the magnitude of the initial datum. Besides, in the physics literature \cite{HLS1991} the authors perform several numerical simulations for \eqref{eq:KH2} varying the Weber number $We$ which corresponds, up to a period factor of $2\pi$, with the  parameter $$\beta\triangleq\frac{\upgamma^2}{\kap}$$
that captures the interplay between the previous two mentioned opposite phenomena. In particular, despite the slightly different geometry considered, it is numerically conjectured in \cite{HLS1991} that the behavior of the system below a critical Weber number, which is very close to \eqref{constraint beta}, "{\it ...is
quite predictable by linear theory, even over long times...}", \cite[p. 1939]{HLS1991}. The aim of our work is to give a rigorous mathematical proof of this numerical conjecture. To do so, we exploit the Hamiltonian nature of the quasi-linear system \eqref{eq:KH3}. Conversely to the classical formulations using the Dirichlet-Neumann operator, here the system is quite explicit and related to singular integral operators, see \eqref{eq:integral_operators_etaomega}.
By introducing a novel framework that combines the Hamiltonian Birkhoff normal form procedure for quasi-linear Hamiltonian systems \cite{BMM2022} with a generalization of the paralinearization method for singular integral operators developed in \cite{BCGS2023}, we prove that the lifespan extends to  
$
T\sim \epsilon^{-\pare{N+1}}
$
for any $N\in\mathbb{N}$. This result enables us to establish stability way beyond the classical local lifespan results of \cite{Ambrose2003, Ambrose2007, AM2007, CCS2008, Lannes2013_ARMA}, leading to what we refer to as {\it almost global well-posedness}. The precise statement is given in the following theorem.

\begin{theorem}[\textbf{Almost global existence of nearly circular vortex-sheets}]\label{thm:main}
    Let
    \begin{equation}\label{constraint beta}
    0<\beta_1<\beta_2<4\pare{2+\sqrt{3}}.
    \end{equation} 
    There exists a zero measure set $\mathcal{B}\subset\bra{\beta_1,\beta_2}$ such that for any values $\kap\in(0,\infty)$ of the surface tension and $\upgamma\in\mathbb{R}$ of the background velocity jump with 
    $$\frac{\upgamma^2}{\kap}\in\bra{\beta_1,\beta_2}\setminus\mathcal{B},$$
    for any $N\in\mathbb{N},$ there exists $s_0>0$ such that for any $s\geqslant s_0$ there exist $\varepsilon_0,c,C>0,$ such that for any $0<\varepsilon<\varepsilon_0$ and any initial datum
    $$\pare{\eta_0,\psi_0}\in H_0^{s+\frac{1}{4}}\pare{\mathbb{T}; \R }\times\dot{H}^{s-\frac{1}{4}}\pare{\mathbb{T};\R},\qquad\textnormal{with}\qquad\norm{\eta_0}_{H_0^{s+\frac{1}{4}}\pare{\mathbb{T};\R}}+\norm{\psi_0}_{\dot{H}^{s-\frac{1}{4}}\pare{\mathbb{T};\R}}\leqslant\varepsilon,$$
    the system \eqref{eq:KH3} admits a unique classical solution
    $$\pare{\eta,\psi}\in C^{0}\pare{\bra{-T_{\varepsilon},T_{\varepsilon}}
    ; \ H_0^{s+\frac{1}{4}}\pare{\mathbb{T};\R}\times\dot{H}^{s-\frac{1}{4}}\pare{\mathbb{T};\R}},\qquad T_{\varepsilon}\geqslant c\varepsilon^{-\pare{N+1}},$$
    with initial datum $\pare{\eta_0,\psi_0}$ and size
    $$\sup_{t\in[-T_{\varepsilon},T_{\varepsilon}]}\pare{\norm{\eta\pare{t,\cdot}}_{H_0^{s+\frac{1}{4}}\pare{\mathbb{T};\R}}+\norm{\psi\pare{t,\cdot}}_{\dot{H}^{s-\frac{1}{4}}\pare{\mathbb{T};\R}}}\leqslant C\varepsilon.$$
\end{theorem}


\begin{remark}
Let us make the following remarks about the previous theorem.
\begin{enumerate}

 \item  \emph{The almost-global well-posedness result of \Cref{thm:main} \emph{cannot} be achieved in settings where capillarity serves only a stabilizing parameter}, i.e. when $\beta=0$. The Weber number $\beta$, which emerges naturally in the equilibrium spectrum $\set{\omega_{\kap,\upgamma}(j)}_{j \in \mathbb{Z}^*}$ (see \eqref{eqspec intro}), captures the interplay between capillarity stabilization and Kelvin-Helmholtz instability. Remarkably, the incorporation of the (physically relevant) interplay between background velocity jump and capillarity is the key factor that generates stability, in particular we can pass from a Sobolev ill-posed problem ($\gamma=0$) to a almost-global well-posed one when $\gamma$ is arbitrarily small and $\upgamma^2$ is comparable. The Weber number $\beta$ modulates the linear frequencies $\omega_{\kap,\upgamma}(j)$ in a non-trivial fashion and  enables us to exclude resonances 
$$
\omega_{\kap,\upgamma}(j_1)\pm\ldots \pm  \omega_{\kap,\upgamma}(j_N)\not=0
$$
taking $ \beta $ outside a suitable zero-measure resonant set $ \mathcal{B}$. The non-resonance condition is a fundamental requirement for implementing the Hamiltonian Birkhoff normal form procedure in PDEs \cite{BMM2022,BDGS2007,BG2006,Bam2003, BFLM2024, FG2024, BFM2024, IP2019, BFG2020, BC2024, BMP2019, FGI2023, BFG2021}. 

\item Referring again to the work \cite{HLS1991} the authors notice, at page 1939, that
\begin{quote}
    \footnotesize
    We had hoped to see some repartition of energy from the
$k=1$ mode to smaller scales over large times. However, for
$We=10.0$ only a very slow increase is observed, if any, of
the width of the active spatial spectrum. 
\end{quote}
This unexpected localization phenomenon is consistent with our analysis. Indeed we prove that, after a suitable change of variables, each Fourier mode $u_j$ can exchange energy only with $u_{-j}$ for very long times, as we shall explain later. This is a consequence of the non-resonance conditions and the Hamiltonian structure of the equation which gives the conservation of super-actions (see \eqref{sa_intro}).

    \item 
    We identify a threshold for the  linear stability in \eqref{constraint beta}. 
    
    This constraint is imposed to ensure that the spectrum of the linearized equation at the trivial state $\pare{\eta,\psi}=\pare{0,0}$ is purely imaginary (see Section \ref{sec linop}), a necessary condition for implementing the Hamiltonian Birkhoff normal form argument. While this condition can be relaxed by restricting the phase space to $\mathbf{m}$-fold solutions for sufficiently large $\mathbf{m}$, such a restriction significantly narrows the class of admissible solutions. For more details, we refer the reader to Section \ref{sec linop} and \cite{MRS}. Although one could adapt the following analysis by verifying the $\mathbf{m}$-fold preserving properties of the transformations along the scheme, following a similar approach to \cite{HHR23}.

    \item The parameter $\Omega$ in \eqref{eq:KH2} represents the speed of rotation of the reference frame. Since the Kelvin-Helmholtz problem is invariant under rotations, $\Omega$ can be chosen arbitrarily without altering the shape of the solutions. We choose $\Omega$ as in \eqref{eq:Omega_def}. This choice simplifies some computations and does not affect generality.

    \item Global-in-time solutions with specific structures can be constructed, as demonstrated in \cite{MRS},  where we identified families of globally defined, uniformly rotating solutions. We also refer to \cite{WC2000,MRS,GPY2022,PS2020,CQZ2023a,CQZ2023b,MS2024,MNS2025, BJL2024,BLS2025} for the construction of families of steady solutions in slightly different settings. However, it remains unclear whether the almost global existence solutions stated in Theorem \ref{thm:main} are actually global in time. The quasi-linear structure of the Kelvin-Helmholtz system, combined with the absence of dispersion due to the periodic boundary conditions, makes the global existence of the Cauchy problem currently out of reach.
\end{enumerate}
    
\end{remark}

\paragraph{Ideas of the proof}
While the Dirichlet-Neumann operator approach pioneered by Zakharov, Craig, and Sulem in \cite{Zakharov1968,CS1993} has become standard for Water-Waves problems-and also employed to the two phase setting \cite{Lannes2013_ARMA}- we employ an alternative formulation. Following previous works such as \cite{CCG2010,CCFGG2013}, we utilize the Birkhoff-Roth integral operator formulation, which exploits the Dirac-$\delta$ structure of the vorticity  \eqref{eq:vorticty_Dirac} in conjunction with the Biot-Savart law to express the KH equations as a CDE. A crucial aspect of our approach is establishing that the KH equations thus derived possess a \emph{Hamiltonian} structure (\Cref{sec HAM}), which is the following
$$\vect{\eta_t}{\psi_t}=\bm{J}\nabla H(\eta,\psi),\qquad\bm{J}=\begin{bmatrix}
    0 & -1\\
    1 & 0
\end{bmatrix},$$
where the Hamiltonian is related to the pseudo kinetic energy $\mathpzc{E}_{\upgamma}$, the length $\mathpzc{L}$ of the free boundary and the angular momentum $\mathpzc{M}$ through
$$
H(\eta,\psi)\triangleq
\mathpzc{E}_{\upgamma}(\eta,\psi)+\kap\mathpzc{L}(\eta)+\Omega\mathpzc{M}(\eta,\psi).
$$
In different geometrical contexts the Hamiltonian structure of the KH system was already presented by Benjamin-Bridges \cite{BB1997_1,BB1997_2}. Once this Hamiltonian formulation is established, we are methodologically committed to working with the specific equations that arise from it, as any deviation would compromise the Hamiltonian property, which is essential for our analysis.
This Hamiltonian structure is a fundamental requirement for obtaining the almost global well-posedness result \Cref{thm:main} using the Hamiltonian Birkhoff normal form of \cite{BMM2022}. Since we are looking for a stability result near the trivial state $\pare{\eta,\psi}=(0,0),$ a quantity of interest is the linearization at this stationary solution. The linearized KH system there writes
$$\vect{\eta_t}{\psi_t} = {\bm L}_{\kap, \upgamma} \pare{D}\vect{\eta}{\psi} , \qquad
{\bm L}_{\kap, \upgamma}\pare{\xi}\defeq
\bra{
\begin{array}{cc}
0 & -\frac{\av{\xi}}{2}
\\[2mm]
{ \kap \av{\xi}^2 - \frac{\upgamma^2}{2} \av{\xi} - \pare{\kap-\upgamma^2}  } & 0
\end{array}
}.$$
The associated spectrum is given by $\lambda_{\gamma,\upgamma}^{\pm}(\xi)=\pm\ii\omega_{\gamma,\upgamma}\pare{\xi},$ with
\begin{align}\label{eqspec intro}
\omega_{\kap,\upgamma}\pare{\xi}=\sqrt{\frac{\kap\av{\xi}}{2}}\sqrt{\av{\xi}^2-\frac{\beta}{2}\av{\xi}+\beta-1},\qquad
\beta\triangleq\frac{\upgamma^2}{\kap}\cdot
\end{align}
Here we see appearing the parameter $\beta$ that modulates the equilibrium frequencies. This parameter is homogeneous to a wave number (inverse of a length). The modulation is fundamental for avoiding resonances later in the Hamiltonian Birkhoff normal form. Let us mention that the modulation of the linear frequencies by an external or geometrical parameter has been used to avoid resonances and construct quasi-periodic solutions for fluid models, see \cite{BBHM2018,BFM2021,BFM2021_2,BM2020,BHM2022,HHM2021,HHR23,HR25,HR26,R23}. Observe that $\omega_{\kap,\upgamma}(\xi)$ is real for any $|\xi|\geqslant1$ provided that 
$$0<\beta<4\pare{2+\sqrt{3}}.$$
There emerges our linear stability threshold.This means that one gets linear stability for typically small oscillations at small scales where the stabilizing effects of the surface tension are dominant.  
Also notice that the asymptotic of the linear frequencies is superlinear, namely as $\av{\xi}\to \infty$
$$\omega_{\kap,\upgamma}(\xi)\sim\sqrt{\frac{\kap}{2}}\av{\xi}^{\frac{3}{2}}.$$
Our purpose is to prove a nonlinear stability result near the circular interface corresponding to $\pare{\eta,\psi}=(0,0).$ To do so, we are able to obtain a suitable energy estimate of the form, for any $N\in\mathbb{N},$
\begin{align}\label{energy estim intro}
\norm{\pare{\eta,\psi}\pare{t}}_s^2 \leqslant C\pare{s}  \pare{\norm{\pare{\eta,\psi}\pare{0}}_s^2 
+ \int_0^t\norm{\pare{\eta,\psi}\pare{\tau}}_{s_0}^{N+1}\norm{\pare{\eta,\psi}\pare{\tau}}_s^2 \dd \tau }, 
&&
\forall 0 < t < T,
\end{align}
where we used the notation
$$\norm{\pare{\eta,\psi}}_{s}\triangleq\norm{\eta}_{H_0^{s+\frac{1}{4}}\pare{\mathbb{T};\R}}+\norm{\psi}_{\dot{H}^{s-\frac{1}{4}}\pare{\mathbb{T};\R}}.$$
Then, a bootstrap argument allows us to get an existence time of the form $T\geqslant c\varepsilon^{-N-1}$ where $\varepsilon$ is the size of the initial datum. Notice that the above estimate is highly non-trivial for two reasons: 
\begin{enumerate}
    \item The right-hand side contains the same number of derivatives as the left-hand side, which is particularly delicate given the quasi-linear nature of the equations.
    \item \label{homo:item} The integral term exhibits high homogeneity, a non-trivial property considering the quadratic nonlinearity of the equations.
\end{enumerate}
To address the first obstacle, we perform a paralinearization of the Kelvin-Helmholtz system \eqref{eq:KH3} together with a paradifferential reduction procedure to remove the space dependance in the positive order part, which allows to remain at the same level of regularity. The paralinearization result is given in \Cref{prop:paralinearizationKH} and writes as follows
\begin{equation}\label{eq:KH4_intro}
    \vect{\eta_t}{\psi_t} =\OpBW{{\bm Q}_{\kap, \upgamma}\pare{\eta,\psi; x, \xi} + {\bm B}_\upgamma \pare{\eta,\psi; x}\av{\xi} - \ii  V_\upgamma \pare{\eta, \psi;x}\Id_{\R^2}\ \xi + {\bm A}_{\bra{0}} \pare{\eta,\psi; x, \xi}   } \vect{\eta}{\psi}
 + {\bm R}\pare{\eta, \psi} \vect{\eta}{\psi},
 \end{equation}
where $\OpBW{\cdot}$ denotes the Bony-Weyl quantization in \eqref{BonyWeyl}. Each term of positive order has an explicit formula with
\begin{equation}
    \label{eq:terminacci}
\begin{aligned}
    {\bm Q}_{\kap, \upgamma}\pare{\eta,\psi; x, \xi}
&\triangleq
\bra{
\begin{array}{cc}
0 & -\frac{\av{\xi}}{2} \\
\kap\pare{1+\mathtt{f}\pare{\eta;x}} \pare{\av{\xi}^2-1} -\pare{ \frac{\upgamma^2}{2} + w_\upgamma\pare{\eta, \psi;x} }\av{\xi}  + \frac{\upgamma^2}{\pare{1+2\eta}} & 0
\end{array}
},\\
\mathtt{f}\pare{\eta;x}&\triangleq \pare{\frac{1+2\eta}{\pare{1+2\eta}^{2}+\eta_x^2}}^{\frac{3}{2}}-1\qquad w_\upgamma\pare{\eta, \psi;x}\triangleq\frac{1}{2}\pare{\pare{\pare{\psi_x+\upgamma}\frac{1+2\eta}{\pare{1+2\eta}^2+\eta_x^2}}^2-\upgamma^2},\\
{\bm B}_\upgamma\pare{\eta,\psi; x}
&\triangleq
\frac{1}{2}
\bra{
\begin{array}{cc}
B_\upgamma\pare{\eta, \psi; x} & 0 \\
B_\upgamma^2\pare{\eta, \psi; x}    & -B_\upgamma\pare{\eta, \psi; x} 
\end{array}
},\qquad B_{\upgamma}\pare{\eta,\psi;x}\triangleq\pare{\psi_x+\upgamma}\frac{2\eta_x}{\pare{1+2\eta}^2+\eta_x^2},\\
V_\upgamma\pare{\eta, \psi;x}&\defeq\frac{1}{2}\mathpzc{D}_0\pare{\eta}\bra{\upgamma + \psi_x} -\frac{\upgamma}{2}\cdot
\end{aligned}
\end{equation}
The matrix operator ${\bm A}_{[0]}$ is of order zero and ${\bm R}$ is a regularizing matrix operator up to a sufficiently large order. We believe that this paralinearization result is itself of interest for maybe future purposes. In our equations, the Dirichlet-Neumann operator does not explicitly appear, contrary to what occurs in the one-phase flat Water-Waves problem, see \cite{CS1993} and \cite[Chap. 1]{Lannes2013}. Instead, the equation derived in \eqref{eq:KH_Hamiltonian} features nonlinearities of convolution type with nonlinear singular convolution kernels. This fundamental difference in structure necessitates the development of a specialized paralinearization technique adapted to these convolution operators—a technique we develop in this work (\Cref{sec:paralinearization}) building on the previous work of the last author and collaborators in \cite{BCGS2023} for the less challenging case of $ \alpha$-SQG patches. Let us now expose our method to paralinearize the KH system. The key insight of our approach is that terms that resist standard paralinearization techniques (i.e. paraproducts, Bony paralinearization formula and composition of paradifferential operators) share a common structure of convolution type in the form
\begin{equation*}
H \pare{\eta} g \pare{x}\defeq
\pv \int_{-\pi}^{\pi} K \pare{\eta;x, z} \frac{g\pare{x-z}}{z} \dd z ,
\end{equation*}
where the function $K \pare{\eta;x, z}$ exhibits regularity at $z=0$ comparable to $\eta$. By Taylor expanding the function $z\mapsto K \pare{\eta;x, z}$ at $z=0$ and applying paraproduct expansions, we derive
\begin{subequations}
\begin{align}
\label{eq:Hideas1}
H \pare{\eta} g \pare{x} = & \  \sum_{\mathsf{j}=0}^\mathsf{J} \OpBW{K_\mathsf{j}\pare{\eta;x}} \ \pv \int_{-\pi}^{\pi} z^{\mathsf{j}-1} g\pare{x-z} \dd z
\\
\label{eq:Hideas2}
& \ + \int_{-\pi}^{\pi} \OpBW{R\pare{\eta;x, z}} g\pare{x-z}\dd z + \text{ l.o.t.},
\end{align}
\end{subequations}
with $J\in\mathbb{N}$, for any $j\in\{0,...,J\},$ $K_j$ being a $z$-independant function of $x$ and $R$ being a remainder satisfying $\av{R\pare{\eta:x, z}}=\mathcal{O}\pare{z^{\mathsf{J}+1}}$. This transformation reduces our analysis to two specific categories of terms, namely
\begin{itemize}
    \item \emph{Terms in the right-hand side of \eqref{eq:Hideas1}:} Classical theory (see \cite[p. 355]{Stein1993}) establishes that for any $\mathsf{j}\in\{0,...,\mathsf{J}\},$ 
    $$\pv \int_{-\pi}^{\pi} z^{\mathsf{j}-1} g\pare{x-z} \dd z = m_\mathsf{j}\pare{D} g,$$
    where $m_\mathsf{j}$ represents a Fourier multiplier of order $j$. Thus, through composition theorems for paradifferential operators, we obtain
\begin{equation*}
\OpBW{K_\mathsf{j}\pare{\eta;x}} \ \pv \int_{-\pi}^{\pi} z^{\mathsf{j}-1} g\pare{x-z} \dd z = \OpBW{K_\mathsf{j}\pare{\eta;x}m_\mathsf{j}\pare{\xi}} g + \text{ bounded terms}.
\end{equation*}
\item \emph{Terms in \eqref{eq:Hideas2}:} We leverage the decay properties of $R$ as $z\to 0$ to establish that these terms constitute paradifferential operators of order $-\pare{\mathsf{J}+1}$ modulo smoothing operators, as detailed in \cref{prop:reminders_integral_operator}.
\end{itemize}
The methodology outlined above is elaborated in detail in \Cref{sec:paralinearization} and formalized in \Cref{prop:paralinearizationKH}, representing one of the manuscript's principal contributions. Our work demonstrates that effective paralinearization is possible even when using the Birkhoff-Roth formulation rather than the Dirichlet-Neumann operator approach. The recovered paradifferential structure in \eqref{eq:KH4_intro} exhibits similarities with pure-capillarity one-phase Water-Waves equations, which allows us to derive several established results concerning vortex sheets, including the necessity of capillarity for system stabilization (cf. \cite{Ambrose2003}). Furthermore, we identify purely nonlinear, unstable terms characteristic of the KH equations  (cf. the term $w_\upgamma$ in \eqref{eq:terminacci}, which is non-nil even when $\upgamma =0$), highlighting the enhanced instability of KH compared to Water-Waves systems. For further analysis of these distinctive unstable terms, we direct the interested reader to \Cref{rem:unstable_terms_KH}.\\
Once the paralinearization obtained, our goal is to remove the $x$-dependence of the positive order terms in order to run an energy estimate in the same Sobolev space, namely without loss of derivatives. This is done through a classical paradifferential reduction procedure requiring to reformulate the problem with complex variables. Defining the complex coordinates
\begin{align*}
 U= \vect{u}{\bar u },&&  u \defeq \mathsf{m}_{\kap, \upgamma}(D)^{-1} \eta+ \ii \mathsf{m}_{\kap, \upgamma}(D) \psi, &&\mathsf{m}_{\kap, \upgamma}\pare{{\xi}}\defeq
\sqrt{\frac{{\av{\xi}}}{2\omega_{\kap, \upgamma}\pare{{\xi}}}}, 
 \end{align*}
the paralinearized KH system is equivalent to the complex Hamiltonian system 
\begin{equation}\label{complex_intro}
U_t=   {\bm J}_\mathbb{C} \  \OpBW{{\bm A}_{\frac{3}{2}}\pare{U;x} \  \omega_{\kap, \upgamma}\pare{ \xi }  + {\bm A}_{1}\pare{U;x, \xi} + {\bm A}_{\frac{1}{2}}\pare{U;x } \  \av{\xi}^{\frac{1}{2}} + {\bm A}_{\bra{0}} \pare{U;x, \xi}} U +{\bm R}\pare{U} U. 
\end{equation}
In the above system, each ${\bm A}_{m}$ corresponds a matrix of $x$-dependent symbols of order $m$, while $\bm R $ is a matrix of regularizing operators up to any fixed order. Then we follow the Hamiltonian method developed by the first author and collaborators \cite{BMM2022} (see also \cite{BD2018,BMM2,FI, BCGS2023, MMS24, MM2024} for non Hamiltonian approach) which consists to perform a series of transformations:
\begin{enumerate}[\bf i]
    \item we first perform an \textit{Alinhac Good Unknown} transformation—a nilpotent matrix-valued paradifferential change of variable introduced in \cite{Lannes2005,AM2009}. This transformation eliminates the unbounded terms ${\bm B}_\upgamma$ in \eqref{eq:terminacci}, which constitute the only unbounded contributions in the one-phase gravity water wave system, thus proving essential for developing local well-posedness theory in the pure gravity setting;

    \item we then \textit{diagonalize} and \textit{reduce to constant coefficients} the resulting system at arbitrary order, modulo smoothing operators (whose regularizing effects depend on the initial data's regularity). This technique, initially developed in \cite{BD2018}, has become standard for implementing normal-form techniques in quasi-linear systems \cite{MMS24,BFP2018,BMM2022}. The method involves conjugation with flows generated by paradifferential operators, where generators are selected based on desired cancellations. We reduce the equation to a diagonal, paradifferential constant-coefficient form by iterative application on the degrees of the paradifferential operators.

\end{enumerate} 
At the end of such procedure we are able to define a transformed, equivalent (in Sobolev) variable
\begin{align*}
W \defeq \bm{B}(U)U, & & W = \vect{w}{\bar{w}} \,
\end{align*}
that satisfies a \emph{constant-coefficient}, \emph{scalar} equation given by
\begin{equation}\label{eq:constcoeff_intro} W_t = \Opvec{\ii \bigg( \big(1+\mathpzc{v}(U;t)\big) \omega_{\kap, \upgamma}(\xi) + \mathpzc{V}_\upgamma(U;t)\xi + \mathpzc{b}_{\frac{1}{2}}(U;t) |\xi|^{\frac{1}{2}} + \mathpzc{b}_{0}(U;t,\xi) \bigg) } W + \pmb{\mathsf{R}}(U;t)W,
\end{equation}
up to a smoothing remainder $\pmb{\mathsf{R}}$ (see \Cref{prop:constantcoeff} and \eqref{def Opvec} for the definition of $\Opvec{\cdot}$). The transformed equation \eqref{eq:constcoeff_intro} has the crucial property that its para-differential part is in constant-coefficient form. To overcome the second obstacle \Cref{homo:item}, we aim to implement a Hamiltonian Birkhoff normal form up to homogeneity degree $N$. However, unlike the original complex system \eqref{complex_intro}, \eqref{eq:constcoeff_intro} no longer possesses the fundamental Hamiltonian structure. This structure is essential to ensure that certain \emph{non-trivial} resonant terms do not contribute to energy estimates, as explained below. Recovering this structure is the purpose of the Darboux symplectic corrector as designed in \cite{BMM2022}. To understand why this correction is needed, we first examine the role of non-resonance conditions. The non-resonance conditions in \Cref{sec:nonreso} ensure the exclusion of resonances, meaning that  
\begin{equation*}  
\sigma_1\omega_{\kap,\upgamma}(j_1) + \dots + \sigma_N\omega_{\kap,\upgamma}(j_N) \neq 0  
\end{equation*}  
unless the indices $(\sigma_1, \dots, \sigma_N) \in \{\pm\}^N$ and $(j_1, \dots, j_N) \in \mathbb{Z}^N$ are \emph{super-action preserving} (see \Cref{def:SAPindex}). This can happen when $N = 2p$ is even, with  
\[
\sigma_1 = \dots = \sigma_p = +, \qquad \sigma_{p+1} = \dots = \sigma_{2p} = -,
\]  
and either 
\begin{align*}
\text{(i)}\ j_\ell = j_{p+\ell} \qquad  \text{or } \qquad  \text{(ii)} \ j_\ell = -j_{p+\ell}.
\end{align*}
Case (i) corresponds to \emph{trivial resonances}. The associated monomials in the vector fields of \eqref{eq:constcoeff_intro} take the form  
\[
|u_{j_1}|^2 \dots |u_{j_{p-1}}|^2 u_{j_p} e^{\ii j_p x}.
\]  
Proving that these terms do not contribute to Sobolev energy estimates is typically straightforward, as it suffices to show that their coefficients are purely imaginary. In contrast, case (ii) involves monomials of the form  
\[
 u_{j_1} \overline{u_{-j_1}} \dots u_{j_{p-1}} \overline{u_{-j_{p-1}}} u_{j_p} e^{-\ii j_p x}.
\]  
These terms couple different Fourier modes, making it more challenging to show that they do not affect energy estimates. However, if the vector field possesses the strong algebraic property of being Hamiltonian, these monomials—called \emph{super-action preserving}—automatically admit infinitely many conservation laws, known as \emph{super-actions}. Specifically, for any $ n \in \mathbb{N} $, the quantities  
\be \label{sa_intro}
J_n(u) \triangleq |u_n|^2 + |u_{-n}|^2
\ee
are conserved. As a consequence, a Hamiltonian, super-action preserving vector field remains transparent to any Sobolev energy estimate.  
However, the system \eqref{eq:constcoeff_intro} for $W$
lacks this Hamiltonian structure, preventing direct application of these conservation laws. To overcome this issue, we introduce the Darboux symplectic correction, as detailed in \Cref{darboux:0}, restoring the necessary structure and allowing us to exploit these properties effectively. Since the map $\bB(U)$ satisfies the hypotheses of \cite[Theorem 7.1]{BMM2022}, we apply it to obtain a new constant-coefficient equation:  

\begin{equation} \label{teo62_intro}
\begin{aligned}
\pa_tZ_0 &=  \ii \vOmega(D)Z_0 +\Opvec{\ii (\mathpzc{d}_{\frac32})_{\leqslant N}(Z_0;\xi)+\ii (\mathpzc{d}_{\frac32})_{>N}(U;t,\xi) }Z_0  + \pmb{R}_{\leqslant N}(Z_0)Z_0+ \pmb{R}_{>N}(U;t)U.
\end{aligned}
\end{equation}  
This equation is Hamiltonian up to homogeneity $N$, meaning that  

\[
\Opvec{\ii (\mathpzc{d}_{\frac32})_{\leqslant N}(Z_0;\xi) }Z_0  + \pmb{R}_{\leqslant N}(Z_0)Z_0= \bm J_\C \nabla H_{\leqslant N}
\]  
for some real Hamiltonian function $ H_{\leqslant N} $. At this point, we begin the algorithmic procedure of reducing the degrees of homogeneity (see \Cref{birkfinalone}).  
The final outcome is a super-action preserving Hamiltonian equation of the form  
\begin{equation} \label{final:eq_intro}
\begin{aligned}
\pa_tZ &= {\ii \vOmega(D)} Z +  \bm J_\C \nabla H^{(\tSAP)}_{\frac32}(Z)+\bm J_\C \nabla H^{(\tSAP)}_{-\vr}(Z)+\Opvec{ \ii (\mathpzc{d}_{\frac32})_{>N}(U;t,\xi)}Z+ \pmb{R}_{>N}(U;t)U.
\end{aligned}
\end{equation}  
Since the super-action preserving Hamiltonian terms ${\ii \vOmega(D)} Z$, $\bm J_\C \nabla H^{(\tSAP)}_{\frac32}(Z)$, and $\bm J_\C \nabla H^{(\tSAP)}_{-\vr}(Z)$ do not contribute to the energy estimate, we obtain, for small solutions $\| Z\| \sim \| U\|_s \lesssim \varepsilon$, the energy bound  
\[
\frac{\di}{\di t} \| Z\|_s^2 \lesssim \varepsilon^{N+3},
\]  
which allows us to prove \Cref{thm:main}.

\subparagraph{Structure of the manuscript} The Hamiltonian formulation can be found in \Cref{sec HAM} while the non-resonance conditions are proved in \Cref{sec:nonreso}.
The paralinearization of the system \eqref{eq:KH3} is  carried out in Section \ref{sec:paralinearization} and the final result is stated in  \Cref{prop:paralinearizationKH}. The next step is to reformulate the results using the complex notation, see Section \ref{sec:complex}. Then, we implement a reducibility procedure to get rid of the space dependence of the positive order part. This part is now rather classical and the corresponding final result is given in Proposition \ref{prop:constantcoeff}. With this in hand, one can perform the Hamiltonian Birkhoff normal form. It is done in Section \ref{sec:BNF} and requires non-resonance conditions for frequency vectors composed with the equilibrium spectrum. Such conditions are checked in Section \ref{sec:nonreso} and are the reasons for the introduction of the zero measure set $\mathcal{B}$.\\

\textbf{Aknowledgments :} FM is supported by the ERC STARTING GRANT 2021 ”Hamiltonian Dynamics, Normal Forms and Water Waves” (HamDyWater-Wavesa), Project Number: 101039762. ER is supported by PRIN 2020 ”Hamiltonian and Dispersive PDEs” project number: 2020XB3EFL. SS is supported by PRIN 2022 ”Turbulent effects vs Stability in Equations from Oceanography” (TESEO), project number: 2022HSSYPN.

\section{Hamiltonian structure and non-resonance conditions}
Here we highlight the Hamiltonian nature of the system \eqref{eq:KH3}. Then, we study the associated linearization at the trivial solution $\pare{\eta,\psi}=\pare{0,0}$ and discuss the non-resonance property of the corresponding eigenvalues. This latter fact is crucial for implementing the Birkhoff normal form in Section \ref{riduzione_e_stima}.
\subsection{Derivation of the Hamiltonian formulation}\label{sec HAM}
Let us now exhibit the Hamiltonian nature of the Kelvin-Helmholtz system \eqref{eq:KH3}.
\begin{prop}\label{prop Ham KH}
	The system \eqref{eq:KH3} is Hamiltonian. More precisely, let us consider
	\begin{equation}\label{def Hamiltonian VS}
		H(\eta,\psi)\triangleq
\mathpzc{E}_{\upgamma}(\eta,\psi)+\kap\mathpzc{L}(\eta)+\Omega\mathpzc{M}(\eta,\psi),
	\end{equation}
	where $\mathpzc{E}_{\upgamma}(\eta,\psi),$ $\mathpzc{L}(\eta)$ and $\mathpzc{M}(\eta,\psi)$ are the pseudo
    kinetic energy, the length of the free boundary and the angular momentum, respectively defined by 
	\begin{align}\label{cinetica}
		\mathpzc{E}_{\upgamma}(\eta,\psi)&\triangleq
-\frac{1}{4}\int_{\mathbb{T}}\int_{\mathbb{T}}\big(\psi_{x}(x)+\upgamma\big)\big(\psi_{x}(y)+\upgamma\big)\log\big(|\mathpzc{z}(x)-\mathpzc{z}(y)|^2\big)\dd y \dd x,\\
		\mathpzc{L}(\eta)&\triangleq
\int_{\mathbb{T}}|\mathpzc{z}_x(x)|\dd x,\\
\mathpzc{M}(\eta,\psi)&\triangleq\int_{\mathbb{T}}\psi_x(x)\eta(x)\dd x.
	\end{align}
	Then, the equations \eqref{eq:KH3} are equivalent to
	\begin{align}\label{eq:KH_Hamiltonian}
	\vect{\eta_t}{\psi_t} =\vect{-\nabla_\psi H(\eta,\psi) }{\nabla_\eta H(\eta,\psi)}= {\bm J} \  \nabla H \pare{\eta, \psi}, 
\qquad
	{\bm J}\defeq
	\begin{bmatrix}
	0&-1\\1&0
	\end{bmatrix}.
	\end{align}
    In addition, the Hamiltonian $H$ is resversible and invariant under translations, namely defining the transformations
    $$\mathtt{S}\begin{bmatrix}
        \eta\\
        \psi
    \end{bmatrix}(x)=\begin{bmatrix}
        \eta\\
        -\psi
    \end{bmatrix}(-x),\qquad\st_\varsigma\begin{bmatrix}
        \eta\\
        \psi
    \end{bmatrix}(x)=\begin{bmatrix}
        \eta\\
        \psi
    \end{bmatrix}(x+\varsigma),$$
    we have
    \begin{equation}\label{rev+invtrans}
        H\circ\mathtt{S}=H=H\circ\st_\varsigma,\qquad\forall \varsigma\in\mathbb{T}.
    \end{equation}
\end{prop}
\begin{proof}
	$\blacktriangleright$ \textit{The pseudo kinetic part:}\\
We compute the variation of $\mathpzc{E}_{\upgamma}(\eta,\psi)$ with respect to $\psi$. Using integration by parts and \eqref{integal rep H}, we get 
\begin{align}
    \dd_{\psi}\mathpzc{E}_{\upgamma}(\eta,\psi)[\hat{\psi}]&=\frac{1}{2}\int_{\mathbb{T}}\left(\int_{\mathbb{T}}\big(\psi_{y}(y)+\upgamma\big)\pa_x\log\big(|\mathpzc{z}(x)-\mathpzc{z}(y)|^2\big)\dd y \right)\hat \psi(x)\dd x \notag\\
    &= \frac{1}{2}\int_{\mathbb{T}}\mathpzc{H}(\eta)[\partial_{x}\psi+\upgamma](x)\hat{\psi}(x)\dd x\label{nuovona}
\end{align}
	From \eqref{nuovona}, we deduce
	\begin{equation}\label{grad E psi}
		\nabla_{\psi}\mathpzc{E}_{\upgamma}(\eta,\psi)=\frac{1}{2}\mathpzc{H}(\eta)[\partial_{x}\psi]+\frac{\upgamma}{2}\mathpzc{H}(\eta)[1].
	\end{equation}
	Now we turn to the differentiation of $\mathpzc{E}_{\upgamma}(\eta,\psi)$ with respect to $\eta$. Recall the notation in \eqref{z:R}.
	Differentiating \eqref{cinetica} with respect to $\eta$, we infer
	$$\dd_{\eta}\mathpzc{E}_{\upgamma}(\eta,\psi)[\hat{\eta}]=-\frac{1}{2}\int_{\mathbb{T}}\int_{\mathbb{T}}\big(\psi_{x}(x)+\upgamma\big)\big(\psi_{x}(y)+\upgamma\big)\frac{\hat{\eta}(x)+\hat{\eta}(y)-\big(\hat{\eta}(x)\frac{r(y)}{r(x)}+\hat{\eta}(y)\frac{r(x)}{r(y)}\big)\cos(x-y)}{|\mathpzc{z}(x)-\mathpzc{z}(y)|^2}\dd x \dd y .$$
	By symmetry, we can reduce this expression to
	\begin{align*}
		\dd_{\eta}\mathpzc{E}_{\upgamma}(\eta,\psi)[\hat{\eta}]&=-\int_{\mathbb{T}}\big(\psi_{x}(x)+\upgamma\big)\left(\int_{\mathbb{T}}\big(\psi_{x}(y)+\upgamma\big)\frac{1-\frac{r(y)}{r(x)}\cos(x-y)}{|\mathpzc{z}(x)-\mathpzc{z}(y)|^2}\hat{\eta}(x) \dd y \right)\dd x\\
		&=-\int_{\mathbb{T}}\frac{1}{2}\big(\psi_{x}(x)+\upgamma\big)\mathpzc{D}_0(\eta)[\partial_{x}\psi+\upgamma](x)\hat{\eta}(x)\dd x,
	\end{align*}
	which implies in turn
	\begin{equation}\label{grad E eta}
		\nabla_{\eta}\mathpzc{E}_{\upgamma}(\eta,\psi)=-\frac{\psi_{x}+\upgamma}{2}\mathpzc{D}_0(\eta)[\partial_{x}\psi+\upgamma].
	\end{equation}
	$\blacktriangleright$ \textit{The length part:}\\
	Recall that $\mathpzc{z}(x)=\big(1+h(x)\big)e^{\ii x}$ with $h=\sqrt{1+2\eta}-1.$ Therefore,
	$$\mathpzc{z}_{x}(x)=\Big(h_{x}(x)+\ii\big(1+h(x)\big)\Big)e^{\ii x},\qquad|\mathpzc{z}_{x}(x)|^2=h_{x}^{2}(x)+\big(1+h(x)\big)^2.$$
	Thus, 
	$$\mathpzc{L}(\eta)=\int_{\mathbb{T}}\sqrt{h_{x}^{2}(x)+\big(1+h(x)\big)^2}\dd x\triangleq \hspace{3pt}\widetilde{\hspace{-3pt} \mathpzc{L}}(h).$$
	Differentiating, we obtain
	$$d_{h}\hspace{3pt}\widetilde{\hspace{-3pt} \mathpzc{L}}(h)[\hat{h}]=\int_{\mathbb{T}}\frac{h_{x}(x)\hat{h}_{x}(x)+\big(1+h(x)\big)\hat{h}(x)}{\sqrt{h_{x}^{2}(x)+\big(1+h(x)\big)^2}}\dd x.$$
	Integrating by parts, we get
	\begin{align*}
		d_{h}\hspace{3pt}\widetilde{\hspace{-3pt} \mathpzc{L}}(h)[\hat{h}]&=\int_{\mathbb{T}}\left[\frac{\big(1+h(x)\big)}{\sqrt{h_{x}^{2}(x)+\big(1+h(x)\big)^2}}-\partial_{x}\left(\frac{h_{x}(x)}{\sqrt{h_{x}^{2}(x)+\big(1+h(x)\big)^2}}\right)\right]\hat{h}(x)\dd x\\
		&=-\int_{\mathbb{T}}\big(1+h(x)\big)\frac{\big(1+h(x)\big)\big(h_{xx}(x)-1-h(x)\big)-2h_{x}^2(x)}{\left(h_{x}^{2}(x)+\big(1+h(x)\big)^2\right)^{\frac{3}{2}}}\hat{h}(x)\dd x.
	\end{align*}
It is easy to see that
$$\frac{(1+h)(h_{xx}-1-h)-2h_{x}^2}{\big(h_{x}^{2}+(1+h)^2\big)^{\frac{3}{2}}}=\frac{\eta_{xx}-(1+2\eta)-3\left(\frac{\eta_x}{\sqrt{1+2\eta}}\right)^2}{\left(1+2\eta+\left(\frac{\eta_x}{\sqrt{1+2\eta}}\right)^2\right)^{\frac{3}{2}}}=\mathpzc{K}(\eta).$$
 As a consequence,
	$$\nabla_{h}\hspace{3pt}\widetilde{\hspace{-3pt} \mathpzc{L}}(h)=-(1+h)\mathpzc{K}(\eta).$$
	Applying the chain rule, we deduce that
	\begin{equation}\label{grad A}
		\nabla_{\eta}\mathpzc{L}(\eta)=\nabla_{h}\hspace{3pt}\widetilde{\hspace{-3pt} \mathpzc{L}}(h)\cdot\nabla_{\eta}h=-(1+h)\mathpzc{K}(\eta)\cdot\frac{1}{1+h}=-\mathpzc{K}(\eta).
	\end{equation}
    $\blacktriangleright$ \textit{The momentum part:}\\
	One readily has
	$$d_{\eta}\mathpzc{M}(\eta,\psi)[\hat{\eta}]=\int_{\mathbb{T}}\psi_{x}(x)\hat{\eta}(x)dx$$
	and, via integration by parts
	$$d_{\psi}\mathpzc{M}(\eta,\psi)[\hat{\psi}]=\int_{\mathbb{T}}\hat{\psi}_{x}(x)\eta(x)dx=-\int_{\mathbb{T}}\eta_{x}(x)\hat{\psi}(x)\dd x.$$
	Hence,
	\begin{equation}\label{grad M}
		\nabla_{\eta}\mathpzc{M}(\eta,\psi)=\psi_{x},\qquad\nabla_{\psi}\mathpzc{M}(\eta,\psi)=-\eta_{x}.
	\end{equation}
Gathering \eqref{def Hamiltonian VS}, \eqref{grad E psi} and \eqref{grad M}, we get
\begin{align}\label{grad H psi}
	\nabla_{\psi}H(\eta,\psi)&=\nabla_{\psi}\mathpzc{E}_{\upgamma}(\eta,\psi)+\Omega\nabla_{\psi}\mathpzc{M}(\eta,\psi)\nonumber\\
	&=-\Omega\eta_{x}+\frac{1}{2}\mathpzc{H}(\eta)[\partial_x\psi+\upgamma].
\end{align}
Putting together \eqref{def Hamiltonian VS}, \eqref{grad E eta}, \eqref{grad A} and \eqref{grad M} yields
\begin{align}\label{grad H eta}
	\nabla_{\eta}H(\eta,\psi)&=\nabla_{\eta}\mathpzc{E}_{\upgamma}(\eta,\psi)+\kap\nabla_{\eta}\mathpzc{A}(\eta)+\Omega\nabla_{\eta}\mathpzc{M}(\eta,\psi)\nonumber\\
	&=\Omega\psi_{x}-\frac{\psi_{x}+\upgamma}{2}\mathpzc{D}_0(\eta)[\partial_{x}\psi+\upgamma]-\kap\mathpzc{K}(\eta).
\end{align}
Comparing \eqref{eq:KH3} with \eqref{grad H psi} and \eqref{grad H eta} concludes the desired result.\\
$\blacktriangleright$ \textit{Invariances :} The properties \eqref{rev+invtrans} are easily obtained by changes of variables $x\mapsto-x$ and $x\mapsto x+\varsigma$. This ends the proof of Proposition \ref{prop Ham KH}.
\end{proof}
\subsection{Analysis of the linearization of \eqref{eq:KH3}}\label{sec linop}
\begin{definition}\label{def:Fourier_mult}
Let $ m\in \R $, we define the space of Fourier multipliers of order $ m $, $ \tilde{\Gamma}^m_0 $,  as the space of smooth functions from $ \R \setminus\{0\}$ to $ \C $ of the form $ \xi\mapsto a\pare{\xi} $ such that
\begin{align*}
\av{\partial_\xi^\alpha a \pare{\xi} }\leqslant C_\beta \angles{\xi}^{m-\alpha},  \  
&& \forall \alpha \in \N, \av{\xi}\geqslant 1/2. 
\end{align*}
\end{definition}
\noindent Following \cite{MRS}, the linearization of \eqref{eq:KH3} around $ \pare{\eta, \psi}=\pare{0, 0} $ is given by
\begin{equation}\label{eq:linearized_system}
\system{
\begin{aligned}
& \eta_t = \pare{\Omega- \frac{\upgamma}{2}} \eta_x - \frac{\av{D}}{2} \psi,\\
& \psi_t = \pare{ \kap \av{D}^2 - \frac{\upgamma^2}{2} \av{D} - \pare{\kap-\upgamma^2}  }\ \eta
+\pare{\Omega-\frac{\upgamma}{2}} \psi_x.
\end{aligned}
}
\end{equation}
Namely, we can write \eqref{eq:linearized_system} as 
\begin{align}\label{eq:linearized_system_matrix}
\vect{\eta_t}{\psi_t} = {\bm L}_{\kap, \upgamma} \pare{D}\vect{\eta}{\psi} , 
&&
{\bm L}_{\kap, \upgamma}\pare{\xi}\defeq
\bra{
\begin{array}{cc}
\ii \pare{\Omega- \frac{\upgamma}{2}}\xi & -\frac{\av{\xi}}{2}
\\[2mm]
{ \kap \av{\xi}^2 - \frac{\upgamma^2}{2} \av{\xi} - \pare{\kap-\upgamma^2}  } & \ii \pare{\Omega- \frac{\upgamma}{2}} \xi 
\end{array}
}.
\end{align}
The eigenvalues of $ {\bm L}_{\kap, \upgamma}\pare{\xi} $ are given by
\begin{equation}\label{eq:lambdas}
 \lambda^\pm_{\kap, \upgamma}\pare{\xi}\triangleq
\ii \pare{\Omega- \frac{\upgamma}{2}} \xi \pm  \, \sqrt{-\frac{\av{\xi}}{2} \pare{ \kap \av{\xi}^2 - \frac{\upgamma^2}{2} \av{\xi} - \pare{\kap-\upgamma^2}  } }
\ \in \tilde{\Gamma}^{3/2}_0 \  
. 
\end{equation}
We want the eigenvalues in \eqref{eq:lambdas} to be purely imaginary, this happens if and only if
\begin{equation}\label{eq:positivity_sqrt}
 \kap \av{\xi}^2 - \frac{\upgamma^2}{2} \av{\xi} - \pare{\kap-\upgamma^2} > 0, 
\end{equation}
The condition \eqref{eq:positivity_sqrt} is satisfied for any $|\xi|\geqslant 1$ if
\begin{equation*}
    \beta\defeq\frac{\upgamma^2}{\kap} \in [0,\beta_{+}),\qquad\beta_{+}\triangleq4(2+\sqrt{3}) \approx 14,928... 
\end{equation*}
Indeed, for $|\xi|\in [ 1,2] $, one has 
$$
 \kap \av{\xi}^2 - \frac{\upgamma^2}{2} \av{\xi} - \pare{\kap-\upgamma^2} \geqslant \min\left\{ \frac{\upgamma^2}{2}, 2\kap  \right\} >0
$$
 for any $ \beta>0$. While, for $ |\xi|>2 $, it reduces to 
\be \label{nuova:cond}
\beta < \min_{|\xi|>2 } \left(\frac{ |\xi|^2-1}{\frac{|\xi|}{2}-1}\right)=4(2+\sqrt{3}). 
\ee
\begin{remark}
We consider the restriction in \eqref{eq:positivity_sqrt} for $ |\xi|\geqslant 1 $. If we further restrict to $ \xi \in \Z^*$, we can improve $ \beta_+ $ from $ 4(2+\sqrt{3}) $ to $ \beta_+ = 15 $. However, in \Cref{sec:nonreso}, we extend the function $ \lambda^\pm_{\kap, \upgamma}(\xi) $ to apply the Delort-Szeftel \Cref{Del-S}. We believe that the argument in \Cref{prop dioph} could be modified to maintain the slightly less restrictive condition $ \beta_+ = 15 $. Nevertheless, to preserve the simplicity of our approach, we do not pursue this further analysis.

\end{remark}
 Then we obtain that
\begin{align}
\label{eq:disp_relation}
\lambda^\pm_{\kap, \upgamma}\pare{\xi}=
\ii\pare{\Omega- \frac{\upgamma}{2}} \xi \pm \ii \ \omega_{\kap, \upgamma}\pare{{\xi}}, 
&&
\omega_{\kap, \upgamma}\pare{{\xi}}\triangleq
\sqrt{\frac{\av{\xi}}{2} \pare{ \kap\pare{  \av{\xi}^2 -1 } - \upgamma^2 \pare{ \frac{\av{\xi}}{2} -1 } } } \in \tilde{\Gamma}^{3/2}_0 \ . 
\end{align}
Notice that we can expand $ \omega_{\kap, \upgamma}\pare{{\xi}} $ and obtain that
\begin{align}\label{eq:expansion_omega}
\omega_{\kap, \upgamma}\pare{{\xi}} 
=
\sqrt{\frac{\kap}{2}} \ \av{\xi}^{\frac{3}{2}} - \frac{1}{\sqrt{2\kap}} \pare{\frac{\upgamma}{2}}^2 \ \av{\xi}^{\frac{1}{2}} 
+\omega _{\kap, \upgamma;-\frac{1}{2}}\pare{{\xi}}, 
&&
\omega _{\kap, \upgamma;-\frac{1}{2}}\pare{{\xi}}\in\tilde{\Gamma}^{-\frac{1}{2}}_0. 
\end{align}
In the sequel, we make the following natural choice for $\Omega$
\begin{equation}\label{eq:Omega_def}
\Omega\triangleq\frac{\upgamma}{2}\cdot
\end{equation}

\begin{rem}
Already at linear level standard computations show that in order to close energy estimates for the system \eqref{eq:linearized_system} we need a discrepancy in regularity between $ \eta $ and $ \psi $, namely we can close the energy estimates on \eqref{eq:linearized_system} in the case in which $ \eta\in H^{s + \frac{1}{4}}_0\pare{\mathbb{T};\mathbb{R}} $ and $ \psi\in \dot{H}^{s-\frac{1}{4}}\pare{\mathbb{T};\mathbb{R}} $, such regularity gap shall persist at nonlinear level as well. This is not unexpected, and the same behavior is present for the one-phase water waves problem, both at linear and nonlinear level, cf. \cite{ABZ2014}.
\end{rem}

\subsection{Non-resonance conditions}\label{sec:nonreso}
In this subsection, we study the non-resonances between the frequencies. This is needed in the application of the normal form algorithm performed in Section \ref{sec:BNF}. Specifically, we prove that there are no resonances between linear frequencies, except for the \emph{super-action-preserving} ones, as defined below.
\begin{definition}[$\tSAP$ multi-index]\label{def:SAPindex}
A multi-index 
$(\alpha, \beta) \in \N^{\Z^*} \times \N^{\Z^*} $ 
is {\em super-action preserving } if 
\begin{equation} \label{sap}
\alpha_n + \alpha_{-n} = \beta_n + \beta_{-n} \ , \qquad \forall n \in \N \, .
\end{equation}
\end{definition}
A super-action preserving multi-index $(\alpha,\beta) $  satisfies $ |\alpha| = |\beta|$ where
$ |\alpha | \triangleq  \sum_{j \in \Z^*} \alpha_j  $.
If a multi-index $(\alpha, \beta) \in  \N^{\Z^*} \times \N^{\Z^*} $ is not super-action preserving, then the set 
\be\label{ennonegotico}
\mathfrak N(\alpha, \beta) \triangleq 	\Big\{ n \in \N \quad\textnormal{s.t.}\quad 
\alpha_{n} + \alpha_{-n} - \beta_n  -\beta_{-n} \neq 0 \Big\}  
\ee
is not empty 
and, since 
$$\mathfrak N(\alpha, \beta) \subset\Big\{ n \in \N \quad\textnormal{s.t.}\quad 
\alpha_{n} + \alpha_{-n} +\beta_n  +\beta_{-n} \neq 0 \Big\},$$ 
its cardinality satisfies
\be\label{cardNab}
 | \mathfrak N(\alpha, \beta) | \leqslant |\alpha + \beta| = |\alpha | + | \beta | \, . 
\ee

The main result of the present section is the following:
\begin{proposition}\label{nres}
    Let $M\in\mathbb{N}^*$ and $0<\beta_1<\beta_2<4\pare{2+\sqrt{3}}.$ Then, there exist $\tau,\delta>0$  and a zero measure set $\mathcal{B}\subset\bra{\beta_1,\beta_2}$  such that for any $\beta\in\bra{\beta_1,\beta_2}\setminus\mathcal{B}$ the following holds: there is $ \nu>0$ such that for any multi-index $\pare{\alpha,\alpha'}\in(\mathbb{N}^*)^{\mathbb{Z}^*}\times(\mathbb{N}^*)^{\mathbb{Z}^*}$ of length $\av{\alpha+\alpha'}\leqslant M,$ which is not super-action preserving (in the sense of Definition \ref{def:SAPindex}), one has
    $$\av{\vec{\omega}_{\kap,\upgamma}\cdot\pare{\alpha-\alpha'}}\geqslant\frac{\nu}{\pare{\displaystyle\max_{j\in\textnormal{supp}\pare{\alpha\cup\alpha'}}|j|}^{\tau}},$$
    where
    $$\vec{\omega}_{\kap,\upgamma}\triangleq\pare{\omega_{\kap,\upgamma}\pare{{j}}}_{j\in\mathbb{Z}^*}.$$
\end{proposition}

The rest of \Cref{sec:nonreso} is dedicated to the proof of \Cref{nres}. This latter is a consequence of the Delort-Szeftel Theorem that we recall here for the convenience of the reader. For its proof, we refer to \cite[Theorem 5.1]{DS}.

\begin{theorem}\label{Del-S}
    Let $d\in\mathbb{N}^*,$ $r_0>0$ and $\beta_1,\beta_2\in\mathbb{R}.$ We denote $B_{r_0}\pare{\R^d}\subset\mathbb{R}^d$ the ball centered at the origin and of radius $r_0.$ Consider $f:B_{r_0}\pare{\R^d}\times\bra{\beta_1,\beta_2}\to\mathbb{R}$ a continuous sub-analytic function and $\rho:B_{r_0}\pare{\R^d}\to\mathbb{R}$ a non-zero real-analytic function. We assume the following facts.
    \begin{enumerate}
        \item The function $f$ is real-analytic on $\{x\in B_{r_0}\pare{\R^d}\quad\textnormal{s.t.}\quad\rho(x)\neq0\}\times\bra{\beta_1,\beta_2}.$
        \item For any $\bar x\in B(0,r_0)$ with $\rho(\bar x)\neq0,$ the equation $f(\bar x,\beta)=0$ admit finitely many solutions in $\bra{\beta_1,\beta_2}.$ 
    \end{enumerate}
    Then, there exist $N_0\in\mathbb{N}$ and $\alpha_0,\delta,C>0$ such that for any $\alpha\in(0,\alpha_0]$, any integer $N\geqslant N_0$ and any $x\in B(0,r_0)$ with $\rho(x)\neq0$, we have
$$\av{\left\lbrace\beta\in\bra{\beta_1,\beta_2}\quad\textnormal{s.t.}\quad\av{f(x,y)}\leqslant\alpha\av{\rho(x)}^{N}\right\rbrace}\leqslant C\alpha^{\delta}\av{\rho(x)}^{N\delta}.$$
\end{theorem}

First observe that we can write
\begin{equation}\label{intro beta}
    \omega_{\kap, \upgamma}\pare{ \xi }=\sqrt{\frac{\kap\av{\xi}}{2}}\sqrt{\av{\xi}^2-\frac{\beta}{2}\av{\xi}+\beta-1},\qquad\beta\triangleq\frac{\upgamma^2}{\kap}\cdot
\end{equation}
Notice that $\omega_{\kap,\upgamma}(2)=\sqrt{3\kap}$ is independent of $\beta$ and that 
$$
\omega_{\kap,\upgamma}(5)- \sqrt{5} \omega_{\kap,\upgamma}(3)\equiv 0.
$$
The above relation shows that there no resonance between the modes $3$ and $5$. However, the couple $(3,5)$ appears to be singular in the analysis of the non-degeneracy for the application of the Delort-Szeftel Theorem. That's why we need to treat it separately.

\subsubsection*{Application of \Cref{Del-S} }

The application of \Cref{Del-S} allows us to control quasi-resonances at arbitrary order via polynomial bounds, thus inducing a finite, but recoverable, loss of derivatives. The result we obtain is the following one:

\begin{proposition}\label{prop dioph}
    Fix $0<\beta_1<\beta_2<4\pare{2+\sqrt{3}}$, $\mathtt{A}\in\mathbb{N}$ and $\mathtt{M}\in\mathbb{N}^*$. Then, there exist $\nu_0,\tau,\delta>0,$ depending on $\mathtt{A}$ and $\mathtt{M},$ such that for any $\tilde \nu\in(0,\nu_0),$ there exists a set $\mathcal{B}_{\tilde \nu}\subset\bra{\beta_1,\beta_2}$ of measure $O(\tilde \nu^{\delta})$ such that for any $\beta\in\bra{\beta_1,\beta_2}\setminus\mathcal{B}_{\tilde{\nu}},$ the following holds: denote $n_0=2$, $n_1=3$ and for any distinct integers $n_2,..., n_{\mathtt{A}}\in\N\setminus\{2,3,5\}$ and any $\vec{c}=\pare{c_0,c_1,...,c_{\mathtt{A}}}\in\mathbb{R}^{\mathtt{A}+1}\setminus\{0\}$ with $\max_{a=0,...,\mathtt{A}}\av{c_a}\leqslant\mathtt{M},$ we have
    \begin{equation}\label{dioph1}
    \av{\sum_{a=0}^{\mathtt{A}}c_a\omega_{\kap,\upgamma}(n_a)}\geqslant\widetilde{\nu}\pare{\sum_{a=0}^{\mathtt{A}}n_a}^{-\tau}.
    \end{equation}
\end{proposition}
\begin{proof}
We denote $\vec{n}=\pare{n_0,n_1,...,n_{\mathtt{A}}}\in\pare{\mathbb{N}^*}^{\mathtt{A}+1}$ with $n_2,...,n_{\mathtt{A}}\not\in\{2,3,5\}$ and distinct, we introduce the notations
\begin{equation}\label{x(n)t(n)}
x_0\pare{\vec{n}}\triangleq\pare{\sum_{a=0}^{\mathtt{A}}n_a}^{-1},\qquad x_a\pare{\vec{n}}\triangleq x_0\pare{\vec{n}}\sqrt{n_a-1}.
\end{equation}
We note that 
\be \label{lim:xa}
0\leqslant|x_a\pare{\vec{n}}|\leqslant 1, \quad \text{for any } a=0,\dots, \tA.
\ee
The condition \eqref{dioph1} is equivalent to
\begin{equation}\label{dioph2}
    \av{f_{\vec c}(x,\beta)}\geqslant \frac{\widetilde{\nu}}{\sqrt{\kap}} x_0^{\tau+3}, \quad x=\pare{x_0(\vec n ),x_1(\vec n),\dots, x_\tA(\vec n)}
    \end{equation}
    where
\begin{equation}\label{effeccid}
f_{\vec c}(x,\beta)\triangleq\sum_{a=1}^{\mathtt{A}}r_a\lambda(x_a,x_0,\beta)+c_0\sqrt{3}x_0^3,\qquad x=(x_0,x_1,\dots, x_\tA) \in B_1\pare{ \R^{\tA+1}},
\end{equation}
    with
    $$r_a\triangleq c_a\sqrt{x_a^2+x_0^2},\qquad\lambda(y,x_0,\beta)\triangleq\sqrt{y^4+\pare{2-\frac{\beta}{2}}x_0^2y^2+\frac{\beta}{2}x_0^4}.$$
The function $f_{\vec c }\colon B_1\pare{ \R^{\tA+1}}\times \bra{\beta_1,\beta_2}\to \R $ is continuous and sub-analytic. Remark that for any $l\in\mathbb{N}^*,$
\begin{equation}\label{def:mu}
    \partial_{\beta}^{l}\lambda(y,x_0,\beta)=\binom{\frac{1}{2}}{l}\mu^{l}(y,x_0,\beta)\lambda(y,x_0,\beta),\qquad\mu(y,x_0,\beta)\triangleq\frac{x_0^2\pare{x_0^2-y^2}}{x_0^2\pare{x_0^2-y^2}\beta+2y^2\pare{y^2+2x_0^2}}\cdot
\end{equation}
Note that, since $0<{\beta}<\beta_2< 4(2+\sqrt{3})$,
the denominator in \eqref{def:mu} satisfies 
\be \label{denomu}
x_0^2\pare{x_0^2-y^2}\beta+2y^2\pare{y^2+2x_0^2}>0\quad \text{for any } y\in [0,1], \ x_0\in (0,1].
\ee
We introduce the polynomial function $\rho:[-1,1]^{\mathtt{A}+1}\to\mathbb{R}$
$$\rho(x_0,x_1,...,x_{\mathtt{A}})\triangleq x_0\prod_{a=1}^{\mathtt{A}}\pare{x_a^2-x_0^2}\prod_{1\leqslant a<b\leqslant\mathtt{A}}\pare{\pare{x_0^2-x_a^2}x_b^2\pare{x_b^2+2x_0^2}-\pare{x_0^2-x_b^2}x_a^2\pare{x_a^2+2x_0^2}}.
$$
The condition \eqref{denomu} implies that the function $f_{\vec c }$ is real-analytic on $\{\rho\not=0\}\times \bra{\beta_1,\beta_2}$.
We fix now $ \bar x=\pare{\bar x_0, \dots \bar x_\tA}$ such that $ \rho( \bar x ) \not =0 $. Then $ \bar x_0\not= 0$ and,  in view also  of \eqref{denomu},  one has 
\be
\mu\pare{\bar x_a,\bar x_0,{\beta}}\not=\mu\pare{\bar x_b,\bar x_0,{\beta}}, \quad \text{for any }\beta \in \pare{0,4(2+\sqrt{3})}.
\ee
The following Lemma ensures that the function $ f_{\vec c}(x,\beta)$ fullfills the second assumption of Delort-Szeftel Theorem \ref{Del-S}.  
\begin{lemma}\label{nondeg}
The solutions of the equation 
$$
f_{\vec c}\pare{\bar x, \beta}=0,
$$
if any, are finitely many.
\end{lemma}
\begin{proof}
Since the function $\beta \to f_{\vec c}\pare{\bar x, \beta} $  is analytic for $\beta\in [\beta_1,\beta_2]$, we are only left to prove that it is not identically zero. To do so we fix $\bar \beta\in (\beta_1,\beta_2)$ and we note that
$$\forall l=1,...,\mathtt{A},\quad\partial_{\beta}^lf\pare{\bar x,\bar \beta}=0\qquad\Leftrightarrow\qquad A\pare{\bar x}\vec{r}=0,$$
where $\vec{r}\triangleq\pare{r_1,...,r_{\mathtt{A}}}$ and 
$$A\pare{\bar x}\triangleq\begin{pmatrix}
    \mu\pare{\bar x_1,\bar x_0,\bar \beta}\lambda\pare{\bar x_1,\bar x_0,\bar \beta} & ... & \mu\pare{\bar x_{\mathtt{A}},\bar x_0,\bar \beta}\lambda\pare{\bar x_{\mathtt{A}},\bar x_0,\bar \beta}\\
    \mu^2\pare{\bar x_1,\bar x_0,\bar \beta}\lambda\pare{\bar x_1,\bar x_0,\bar \beta} & ... & \mu^2\pare{\bar x_{\mathtt{A}},\bar x_0,\bar \beta}\lambda\pare{\bar x_{\mathtt{A}},\bar x_0,\bar \beta}\\
    \vdots & & \vdots\\ 
    \mu^{\mathtt{A}}\pare{\bar x_1,\bar x_0,\bar \beta}\lambda\pare{\bar x_1,\bar x_0,\bar \beta} & ... & \mu^{\mathtt{A}}\pare{\bar x_{\mathtt{A}},\bar x_0,\bar \beta}\lambda\pare{\bar x_{\mathtt{A}},\bar x_0,\bar \beta}
\end{pmatrix}.$$
 Since $\vec c \not=0$ and $\rho(\bar x)\not=0$ we have also  $ \bar x_0 \not =0$ and  $\vec{r}\neq0.$ Then $\beta\mapsto f(\bar x,\beta)\equiv0$ implies that $\det\pare{A(\bar x)}=0.$ Besides, by $\mathtt{A}$-linearity of the determinant and recognizing a Vandermonde determinant, we find
$$\det\pare{A\pare{\bar x}}=\prod_{a=1}^{\mathtt{A}}\mu\pare{\bar x_a,\bar x_0,\bar\beta}\lambda\pare{\bar x_a,\bar x_0,\bar\beta}\prod_{1\leqslant a<b\leqslant\mathtt{A}}\pare{\mu\pare{\bar x_a,\bar x_0,\bar \beta}-\mu\pare{\bar x_b,\bar x_0,\bar\beta}}.$$
By construction, since $\bar x\in\{\rho\neq0\}$, one has that $\det\pare{A\pare{\bar x}}\neq0$ at the given $\overline{\beta}.$ So $\beta\mapsto f(\bar x,\beta)$ cannot be identically zero proving Lemma \ref{nondeg}. 
\end{proof}

We thus conclude that 
there are $N_0 \in \N^*$, $\alpha_0, \delta, C  >0$, such that for any $\alpha \in (0, \alpha_0]$, any $ N \in \N^*$, $N \geqslant N_0$, any $x\in X$ with $\rho(x) \neq 0$,
\begin{equation}
\label{meas2gh}
\av{ \left\lbrace \beta \in \bra{\beta_1,\beta_2} \quad\textnormal{s.t.}\quad\av{f_{\vec c} \, (x, \beta)}\leqslant \alpha \av{\rho(x)}^N 
\right\rbrace}\leqslant C \alpha^\delta \, \av{\rho(x)}^{N\delta} \, . 
\end{equation}
To conclude the proof we need the following:

\begin{lemma}
   Let $\tau_1\triangleq 2\mathtt{A}+1+8\binom{\mathtt{A}}{2}.$ Then 
\begin{equation}\label{rho.estgh}
    \pare{\sum_{a=1}^{\mathtt{A}}n_a}^{-\tau_1}\lesssim\av{\rho\pare{x(\vec{n})}}\lesssim\pare{\sum_{a=1}^{\mathtt{A}}n_a}^{-1}.
\end{equation} 
\end{lemma}
\begin{proof}
By definition \eqref{x(n)t(n)}, we have 
\begin{align*}
&\pare{(x_0^2-x_a^2)x_b^2(x_b^2+2x_0^2)-(x_0^2 -x_b^2 )x_a^2 (x_a^2 +2x_0^2 )}_{|x_0=x_0\pare{\vec{n}},\, x_a=x_a\pare{\vec{n}},\, x_b=x_b\pare{\vec{n}}}\\
&=x_0^8\pare{\vec{n}}\left[\pare{2-n_a}\pare{n_b^2-1}-\pare{2-n_b}\pare{n_a^2-1}\right].
\end{align*}
The equation
$$\pare{2-n_a}\pare{n_b^2-1}=\pare{2-n_b}\pare{n_a^2-1}$$
is equivalent to
\begin{equation}\label{equ arit}
\pare{2-n_a}\pare{n_b-n_a}\pare{n_b-\frac{2n_a-1}{n_a-2}}=0.
\end{equation}
Since $n_a\neq2$ and $n_b\neq n_a,$ we must have $n_b=\frac{2n_a-1}{n_a-2}\cdot$
But 
$$2n_a-1=2(n_a-2)+3.$$
Hence
$$2n_a-1\equiv 3 \quad \text{mod.} \ [n_a-2].$$
Therefore
$$n_a-2|2n_a-1\quad\Leftrightarrow\quad n_a\in\{3,5\}.$$
We conclude that the only non-trivial couples of integers solutions to \eqref{equ arit} are
$$(3,5)\qquad\textnormal{and}\qquad(5,3).$$
However, we have excluded the possibility $n_a=5$ or $n_b=5.$ This implies that the equation \eqref{equ arit} is not solved with our choice of $\vec{n}$. Thus,
\be \label{lb1}
\av{\pare{\pare{x_0^2-x_a^2}x_b^2\pare{x_b^2+2x_0^2}-\pare{x_0^2 -x_b^2}x_a^2 \pare{x_a^2 +2x_0^2 }}_{|x_0=x_0\pare{\vec{n}},\, x_a=x_a\pare{\vec{n}},\, x_b=x_b\pare{\vec{n}}}}\geqslant x_0^8(\vec{n}).\ee
Moreover, since $n_a\not=2$ for $a=1,\ldots, \mathtt{A}$, one has 
\be\label{lb2}
\av{x_a^2\pare{\vec{n}}-x_0^2\pare{\vec{n}}}\geqslant x_0^2\pare{\vec{n}}.
\ee
 Gathering \Cref{lb1,lb2} we obtain the lower bound in \eqref{rho.estgh}. To prove the upper bound it is sufficient to use the bounds in \eqref{lim:xa}.
\end{proof}

Consider the set 
\begin{equation}\label{defKhf}
\begin{aligned}
    \cB(\alpha, N)&\triangleq\bigcup_{
              \substack{ \vec{n}=(n_2, \ldots, n_\tA) \in \N^{\tA},\, 1 \leqslant n_1 < \ldots < n_{\tA}  \\ 
                        \vec c \in (\Z^*)^{\tA}, \, |\vec c|_{\infty} \leqslant \mathtt{M} }}\cB_{\vec c,\vec n}(\alpha, N) \subset \bra{\beta_1,\beta_2}, \\
                        \cB_{\vec c,\vec n}(\alpha, N)&\triangleq\left\lbrace\beta \in \bra{\beta_1,\beta_2} \quad\textnormal{s.t.}\quad \av{f_{\vec c}\pare{x(\vec n), \beta}}\leqslant\alpha \av{\rho\pare{x(\vec n)}}^N 
\right\rbrace.
\end{aligned}
 \end{equation}
We fix now $\underline N\in \N $ such that $\bar N \delta> 1 $ and $ \cB_\alpha \triangleq \cB(\alpha, \underline N )$. Then we get

 \begin{equation} \label{measK1gh}
 \big| \cB_\alpha\big|  \leqslant
C(\tA, \mathtt{M}) \alpha^\delta \sum_{n_2, \ldots, n_\tA \in \N } \left(\sum_{a=0}^\tA n_a\right)^{-N\delta}  \leqslant C'(\tA, \mathtt{M}) \alpha^\delta 
 \end{equation}
In conclusion, 
for any $\beta \in \bra{\beta_1,\beta_2} \setminus \cB_ \alpha$, for any  
$ \vec n = (n_0,n_1, \ldots, n_\tA) $ as in the hypothesis,
 any 
$  \vec c \in (\Z^*)^{\tA} $ with $ |\vec c \, |_{\infty} \leqslant \tM $,   
 it results,
 by \eqref{defKhf} and \eqref{rho.estgh}, that
\be\label{f1gh}
\av{{f_{\vec c} \, \pare{x(\vec n),  \beta}}} 
 >  \alpha \av{\rho\pare{x(\vec n)}}^{\underline{N}} \geqslant 
c(\tA) \alpha\left(\sum_{a=1}^\tA n_a\right)^{-\tau_1 \underline{N} } .
\ee
Recalling the definition of 
$ f_{\vec c} $ in \eqref{effeccid} and $ x_0(\vec n) $ in \eqref{x(n)t(n)}, 
the lower bound \eqref{f1gh} 
implies \eqref{prop dioph} with $ \tau \triangleq\tau_1 \underline{N} - 3 $ (cfr. \eqref{dioph2}) 
and re-denoting $ \tilde \nu \triangleq\alpha c(\tA)$.
\end{proof}

\subsubsection*{Proof of \Cref{nres}}
   Fix a multi-index $\pare{\alpha,\alpha'}\in(\mathbb{N}^*)^{\mathbb{Z}^*}\times(\mathbb{N}^*)^{\mathbb{Z}^*}$ of length $\av{\alpha+\alpha'}\leqslant M.$ We denote $\mathfrak{N}\pare{\alpha, \alpha'}$ as in \eqref{ennonegotico}. 
   Since the couple $(\alpha,\alpha')$ is not super-action preserving, then \begin{equation}\label{notempty}
       \mathfrak{N}(\alpha,\alpha')\neq\varnothing.
   \end{equation}
   Then, we can write
    \begin{align*}
        \vec{\omega}_{\kap, \upgamma}\cdot\pare{\alpha-\alpha'}&=\sum_{j\in\mathbb{Z}^*}\omega_{\kap,\upgamma}\pare{\av{j}} \pare{\alpha_j-\alpha_j'}=\sum_{n\in\mathbb{N}^*}\omega_{\kap,\upgamma}(n)\pare{\alpha_{n}+\alpha_{-n}-\alpha_{n}'-\alpha_{-n}'}\\
&=\sum_{n\in\mathfrak{N}\pare{\alpha,\alpha'}}\omega_{\kap,\upgamma}(n)\pare{\alpha_{n}+\alpha_{-n}-\alpha_{n}'-\alpha_{-n}'}.
    \end{align*}
    We use the notation 
    $$\mathtt{A}\triangleq\av{\mathfrak{N}\pare{\alpha,\alpha'}\setminus\{2,3,5\}}+1.$$ 
     We also denote
    $$
     \mathfrak{N}(\alpha,\alpha')\setminus\{2,3,5\}=\{n_2,...n_{\mathtt{A}}\}.
    $$
    Therefore, 
    \begin{align*}
        \vec{\omega}_{\kap,\upgamma}\cdot\pare{\alpha-\alpha'}=\sum_{a=2}^{\mathtt{A}}c_a\omega_{\kap,\upgamma}(n_a)+c_0\omega_{\kap,\upgamma}(2)+c_{1}\omega_{\kap,\upgamma}(3),
    \end{align*}
    with
    $$c_0\triangleq\alpha_{2}+\alpha_{-2}-\alpha_2'-\alpha_{-2}',\qquad c_{1}\triangleq\pare{\alpha_{3}+\alpha_{-3}-\alpha_{3}'-\alpha_{-3}'}+\sqrt{5}\pare{\alpha_{5}+\alpha_{-5}-\alpha_{5}'-\alpha_{-5}'}$$
    and for any $a=2,...,\mathtt{A},$
    $$c_a\triangleq\alpha_{n_a}+\alpha_{-n_a}-\alpha_{n_a}'-\alpha_{-n_a}'.$$
    Observe that for any $a=0,...,\mathtt{A},$
    $\av{c_a}\leqslant4M$, where $M$ is defined in \cref{nres}. Let us assume the absurd hypothesis $\vec{c}\triangleq(c_0,c_{1},c_2,...,c_{\mathtt{A}})=0.$ By definition, under such absurd hypothesis, we obtain that $\mathfrak{N}(\alpha,\alpha')\subset\{3,5\}.$ But $c_{1}\in\mathbb{Z}[\sqrt{5}]$ so $c_{1}=0$ implies also that $3,5\not\in\mathfrak{N}(\alpha,\alpha').$ Hence, $\mathfrak{N}(\alpha,\alpha')=\varnothing$ which is a contradiction with \eqref{notempty}, thus we can safely assume $\vec c \neq 0$. We consider the set
    $$\mathcal{B}\triangleq\bigcap_{\tilde{\nu}\in(0,\nu_0)}\mathcal{B}_{\tilde{\nu}},$$
    where $\nu_0$ and $\mathcal{B}_{\tilde{\nu}}$ are introduced in Proposition \ref{prop dioph}.
    Hence, for any $\beta \in \bra{\beta_1,\beta_2}\setminus\cB$ there is $\tilde{\nu}\in(0,\nu_0)$ such that $\beta \in \bra{\beta_1,\beta_2}\setminus \cB_{\tilde{\nu}}$ and,  applying  Proposition \ref{prop dioph} we get the desired result with $\nu\triangleq\frac{\tilde{\nu}}{M^{\tau}}$.\qed

\section{Functional setting}\label{sec:preliminaries}


Throughout the document, we shall use the notations 
$$\mathbb{N}\triangleq\{0,1,2,...\},\qquad\mathbb{N}^*\triangleq\mathbb{N}\setminus\{0\},\qquad\mathbb{Z}\triangleq\mathbb{N}\cup(-\mathbb{N}),\qquad\Z^*\triangleq \Z\setminus\{0\}.$$
Along the paper we deal with real parameters 
\begin{equation}\label{parametri_belletti}
s\geqslant s_0  \gg K \gg \varrho \gg N  \geqslant0,
\end{equation}
where $ N \in \N$. 
The values of $ s, s_0, K $ and $ \vr $ may vary from line to line while still being true the relation \eqref{parametri_belletti}. 
We expand a  $2\pi$-periodic function $u\in L^2 (\T;\C)$ in Fourier series as
\begin{align*}
u(x)= \sum_{j \in \mathbb{Z}} \hat{u}\pare{j} e^{\im j x},
&&
 \hat{u}\pare{j} \defeq {\cal F}_{x \to j}\pare{j} \defeq u_j \triangleq\int_{\mathbb{T}}u(x) e^{- \im j x }\,\di x . 
\end{align*}
The function $u$ is real-valued if and only if  $ \overline{u_j}  = u_{-j}  $, for any $ j \in \Z $. 
\noindent
For any $ s\in\mathbb{R},$ we define the Sobolev space $  H^{s} \defeq H^{s}(\T;\C) $ 
with norm
\begin{align*}
\norm{u}_{s} \defeq \norm{u}_{H^{s}} = \pare{
\sum_{ j \in \mathbb{Z} } \angles{j}^{2s}  \av{\hat u\pare{j}}^2 
} ^{\frac12} , 
&&
 \angles{j} \defeq \max \set{1, \av{j}} . 
\end{align*}
We define $ \Pi_0 u \defeq  u_0 $  the average of $ u $ and 
\begin{equation*}
\Pi_0^\bot  \defeq  \Id - \Pi_0  . 
\end{equation*}
We define $ H^s_0 $ the subspace of zero average functions of $ H^s $ for which
 we also denote $ \norm{u}_s = \norm{u}_{H^s}  = \norm{u}_{H^s_0}  $, and with $ \dot{H}^s \defeq \xfrac{H^s}{\mathbb{C}} $ endowed with the norm $ \norm{u}_{\dot{H}^s}\defeq\norm{\Pi_0^\bot u}_s $. 
We define, on   $\dot L^2(\T;\C)\triangleq\dot H^0(\T;\C)$,
the complex scalar product $\psc{\cdot}{\cdot}_\C$ and the real symmetric bilinear form $\psc{\cdot}{\cdot}_\R$ as follows:  for any $ u, v \in \dot  L^2 (\T;\C)$,
\begin{align}\label{scalar products}
\psc{u}{v} _{ \C } \defeq \int_\T  \Pi_0^\bot u(x)\, \overline{\Pi_0^\bot {  v(x)}} \, \di x, 
&&
\psc{u}{v} _{ \R } \defeq \int_\T \Pi_0^\bot u(x)\,  \Pi_0^\bot{  v(x)} \, \di x\,.
\end{align}
Moreover we define the real subspace of $\dot H^s(\T;\C^2)$
$$
\dot H^s_\R(\T;\C^2)\triangleq\left\{ U= \vect{ u^+}{ u^-}\in \dot H^s(\T;\C^2)\quad\textnormal{s.t.}\quad u^-= \overline{u^+}\right\}.
$$
We also denote 
$$ \dot H^{\infty}(\mathbb{T};\mathbb{C}^2)
\defeq \bigcap_{s \in \R} \dot H^{s}(\mathbb{T};\mathbb{C}^2),\qquad\dot H_{\mathbb{R}}^{\infty}(\mathbb{T};\mathbb{C}^2)
\defeq \bigcap_{s \in \R} \dot H_{\mathbb{R}}^{s}(\mathbb{T};\mathbb{C}^2).$$
Given an interval $ I\subset \R$ symmetric with respect to $ t = 0 $
and $s\in \R$, we define the space
$$
C_*^K \pare{ I;\dot{H}^s\pare{ \mathbb{T};\mathbb{C}^2 } } \defeq
\bigcap_{k=0}^K C^k \pare{ I; \dot{H}^{s- \frac{3}{2} k}\pare{ \mathbb{T};\mathbb{C}^2 }  } ,
$$
endowed with the norm
\begin{equation}\label{def norm Ks}
\sup_{t\in I} \norm{ U (t, \cdot)}_{K,s} \qquad
{\rm where} \qquad
\norm{ U(t, \cdot)}_{K,s}\defeq \sum_{k=0}^K \norm{ \partial_t^k U(t, \cdot)} _{\dot H ^{s- \frac32 k}} .
\end{equation}
We also consider its subspace 
$$\label{eq:CKastreali}
C_{*\R}^K\pare{I,\dot{H}^s\pare{\mathbb{T},\mathbb{C}^2}}\triangleq\left\{ U\in C_*^K\pare{I,\dot{H}^s\pare{\mathbb{T},\mathbb{C}^2}} \quad\textnormal{s.t.}\quad 
U=\begin{pmatrix} u \\ \bar{u}\end{pmatrix}\right\}.$$  
Given $r>0$ we set $B^K_s(I;r)$ the ball of radius $r$ in $C_*^K\pare{I,\dot{H}^s\pare{\mathbb{T},\mathbb{C}^2}}$ and by 
 $B^K_{s,\R}(I;r)$ the ball of radius $r$ in $C_{*\R}^K\pare{I,\dot{H}^s\pare{\mathbb{T},\mathbb{C}^2}}$.\\

A vector field $ X(u) $ is {\it translation invariant} if
$$ X \circ \st_\varsigma = \st_\varsigma \circ X,\qquad\forall\varsigma \in \R , $$ 
where the translation operator $\st_{\varsigma}$ is defined by
 \begin{equation*}
\st_\varsigma \colon u(x) \mapsto u(x + \varsigma).
\end{equation*}
 Given a linear operator  $ R(u) [ \cdot ]$ acting on $L^2_0(\T;\C)\triangleq H^0_0(\T;\C)$
we associate the linear  operator  defined by the relation
$ \ov{R(u)} v \defeq \ov{R(u) \ov{v} } $ for any  $ v \in L^2_0(\T;\C).$
An operator $R(u)$ is {\em real } if $R(u) = \ov{R(u)} $ for any $ u $ real. 

\subsection{Paradifferential calculus}\label{sec:para}

We introduce the para-differential calculus developed in \cite{BD2018,BMM2022}.

\paragraph{Classes of symbols.}
Roughly speaking the class $\wt{\Gamma}_p^m$ contains symbols of order $m$ and homogeneity $p$ in $u$, whereas the class $\Gamma_{K,K',p}^m$ contains non-homogeneous symbols of order $m$ that vanish at degree at least $p$ in $u$ and that are $(K-K')$-times differentiable in $t$. We can think the parameter $K'$ like the number of time derivatives of $u$ that are contained in the symbols.

\begin{definition}[Symbols]\label{definsymb}
Let $m\in \R$, $p,N\in \N $,
$ K, K' \in \N$ with $ K' \leqslant K  $, and $ \epsilon_0>0$.
\begin{enumerate}[i)]

		
	\item {\bf $p$-Homogeneous symbols.} We denote by $\widetilde{\Gamma}^m_p$ the space of  $p$-linear symmetric maps from $\left( \dot{H}^{\infty}\left(\mathbb{T};\mathbb{\C}^2\right)\right)^p$ to
        $ C^\infty(\mathbb{T}\times \R; \C) $ ,
		$ (x, \xi) \mapsto a_p(U_1, \dots, U_p ;x,\xi)$ whose associated polynomial has the form
		\be \label{expr HOM SYMB}
		a_p(U;x,\xi)\triangleq a_p(U, \dots, U; x, \xi)\triangleq\sum_{\substack{\vec{\jmath}\in \Z^p\\ \vec{\sigma}\in \{ \pm\}^p}} a_{\vec{\jmath}}^{\vec{\sigma}}(\xi) u_{\vec{\jmath}}^{\vec{\sigma}} \, e^{\ii (\vec{\sigma}\cdot \vec{\jmath}) x},
		\ee
		where $a_{\vec{\jmath}}^{\vec{\sigma}}(\xi)$ are complex valued Fourier multipliers,  satisfying
        $$
        a_{\vec{\jmath}}^{\vec{\sigma}}(\xi)\triangleq a_{j_1, \dots, j_p}^{\sigma_1, \dots, \sigma_p}(\xi) = a_{j_{\pi(1)}, \dots, j_{\pi(p)}}^{\sigma_{\pi(1)}, \dots, \sigma_{\pi(p)}}(\xi) \quad \text{ for any } \pi \text{ permutation of } \{1, \dots, p\}\,,
        $$
        and
        for some $	\mu\geqslant 0$, 
		\be \label{estim symbol}
		| \partial_\xi^\beta a_{\vec{\jmath}}^{\vec{\sigma}}(\xi) | \leqslant C_\beta \la \vec{\jmath}\ra^{\mu}  \langle \xi\rangle^{m-\beta},  \quad \forall \vec{\jmath}\in \Z^p,\, \vec{\sigma}\in\{\pm\}^p, \, \beta\in \N.
		\ee
        We have used the following notations for given $ \vec{\jmath}= (j_1,\dots,j_p) \in \Z^p$ and $\vec{\sigma}=(\sigma_1,\dots,\sigma_p)\in\{\pm\}^{p},$
        $$\langle\vec{\jmath}\rangle\triangleq\max(1,|\vec{\jmath}|),\qquad|\vec{\jmath}|\defeq \max(|j_1|, \ldots, |j_p| ),\qquad u_{\vec{\jmath}}^{\vec{\sigma}} \defeq u_{j_1}^{\sigma_1} \dots u_{j_p}^{\sigma_p},$$
        where for a fixed $u\in\mathbb{C},$ we denote
$$u^{+}\triangleq u\qquad\textnormal{and}\qquad u^{-}\triangleq\overline{u}.$$
		We  denote by $\widetilde{\Gamma}^m_0 $ the space of constant coefficients symbols $ \xi \mapsto a(\xi)$ which satisfy \eqref{estim symbol} with $\mu= 0 $.
\item  {\bf Non-homogeneous symbols. }  We denote by $\Gamma_{K,K',p}^m[\epsilon_0]$ the space of functions  $ a(U;t, x,\xi) $,
defined for $ U \in B^{K'}_{s_0}\pare{I;\epsilon_0}$ for some $s_0$ large enough, with complex values, such that for any $0\leqslant k\leqslant K-K'$, any $s\geqslant s_0$, there is  $0<\epsilon_0(s)<\epsilon_0$ such that for any $ \beta\in \N$ the following holds. There is $C\triangleq  C_{s,\beta}>0$ such that  for any $ U\in  B^{K'}_{s_0}\pare{I;\epsilon_0(s)} \cap C_{*}^{k+K'}\pare{I; \dot H^{s}\pare{\mathbb{T};\mathbb{C}}}$ and $\alpha \in \N$, with $\alpha \leqslant s-s_0$ one has the estimate
\begin{equation}\label{estim nonhom symb}
\av{ \partial_t^k\partial_x^\alpha \partial_\xi^\beta a\pare{ U;t, x,\xi }}  \leqslant C \langle \xi \rangle^{m-\beta} \| U \|_{k+K',s_0}^{p-1}\|U\|_{k+K',s} \, .
\end{equation}
If $ p = 0 $ the right hand side has to be replaced by $ C \langle \xi \rangle^{m-\beta} $.

We say that a non-homogeneous symbol  $a(U;x,\xi) $ is \emph{real} if it is real valued for any
$ U \in B^{K'}_{s_0,\R}\pare{I;\epsilon_0} $.

\item
{\bf Symbols.} We denote by $\Sigma \Gamma_{K,K',p}^m[\epsilon_0,N]$ the space of 
symbols 
\be\label{expansion symbols}
a(U;t, x,\xi)= \sum_{q=p}^{N} a_q\pare{ U;x,\xi } + a_{>N}(U;t, x,\xi) 
\ee
where $a_q $, $q=p,\dots, N$
are homogeneous symbols in $ \wt{\Gamma}_q^m $ and  
$a_{>N} $ is 
a non-homogeneous symbol in $ \Gamma_{K,K',N+1}^m $. 

We say that a symbol  $a(U;t, x,\xi) $ is \emph{real} if it is real valued for any
$ U \in B^{K'}_{s_0,\R}(I;\epsilon_0)$.\\
 We shall also denote with $ \Sigma_p^N \Gamma^m_q$ the subspace of $\Sigma \Gamma_{K,K',p}^m[\epsilon_0,N]$ made of pluri-homogeneous symbols, namely symbols which expand as in \eqref{expansion symbols} with $a_{>N}\equiv 0$.
\end{enumerate}
\end{definition}
\begin{remark}\label{remsymb} Let us make the following remarks.
\begin{enumerate}[\textbullet]
    \item Given a $p$-homogeneous symbol $a_p \in \widetilde{\Gamma}^{m}_p$, we shall often, with a slight abuse of notation, identify the associated $p$-homogeneous polynomial with the $p$-linear symbol itself
$$
a_p(U_1, \ldots, U_p; x, \xi)\quad \leftrightsquigarrow \quad a_p(U; x, \xi)\triangleq a_p(U, \ldots, U; x, \xi).
$$
 This identification is harmless due to the standard correspondence between $p$-homogeneous polynomials and symmetric $p$-linear maps.
 \item  If $ a ( U; \cdot  )$ is a homogeneous
symbol in $ \widetilde \Gamma_p^m $ then it  belongs to  
the class of non-homogeneous symbols 
$\Gamma^m_{K,0,p} [\epsilon_0] $, for any $ \epsilon_0 > 0 $.
\item The classical properties expected for a symbol hold: 
 if $a $ is a symbol in $ \Sigma \Gamma^m_{K,K',p}[\epsilon_0,N] $
then $ \partial_x a  $ is in $ \Sigma \Gamma^{m}_{K,K',p}[\epsilon_0,N]   $ and
$ \partial_\xi a $ belongs to $  \Sigma \Gamma^{m-1}_{K,K',p}[\epsilon_0,N]   $.
If in addition $ b $ is a symbol in $ \Sigma \Gamma^{m'}_{K,K',p'}[\epsilon_0,N]  $ then
their product 
$a b $ is a symbol in $ \Sigma \Gamma^{m+m'}_{K,K',p+p'}[\epsilon_0,N]  $.
\end{enumerate}
\end{remark}


We also define classes of functions in analogy with our classes of symbols.

\begin{definition}[Functions] \label{defin funct}
Let $p, N \in \N $,
 $K,K'\in \N$ with $K'\leqslant K$, $\epsilon_0>0$.
We denote by $\tilde{\mathcal{F}}_{p}$, resp. $\mathcal{F}_{K,K',p}[\epsilon_0]$,
 $\Sigma\mathcal{F}_{K,K',p}[\epsilon_0,N]$,
the subspace of $\tilde{\Gamma}^{0}_{p}$, resp. $\Gamma^0_{K,K',p}[\epsilon_0]$,
resp. $\Sigma\Gamma^{0}_{K,K',p}[\epsilon_0,N]$,
made of those symbols which are independent of $\xi $.
We write $\tilde{\mathcal{F}}^{\R}_{p}$,   resp. $\mathcal{F}_{K,K',p}^{\R}[\epsilon_0]$,
$\Sigma\mathcal{F}_{K,K',p}^{\R}[\epsilon_0,N]$,  to denote functions in $\tilde{\mathcal{F}}_{p}$,
resp. $\mathcal{F}_{K,K',p}[\epsilon_0]$,    $\Sigma\mathcal{F}_{K,K',p}[\epsilon_0,N]$,
which are real valued for any $ u \in B^{K'}_{s_0}\pare{I;\epsilon_0}$.
\end{definition}

\smallskip

\noindent{\bf Paradifferential quantization.}
Given $p\in \N,$ we consider   functions
  $\chi_{p}\in C^{\infty}(\R^{p}\times \R;\R)$ and $\chi\in C^{\infty}(\R\times\R;\R)$, 
  even with respect to each of their arguments, satisfying, for $0<\delta_0\leqslant \tfrac{1}{10}$,
  \begin{equation}
  \label{cut off defin}
      \begin{aligned}
&{\rm{supp}}\, \chi_{p} \subset\{(\xi',\xi)\in\R^{p}\times\R\quad\textnormal{s.t.}\quad |\xi'|\leqslant\delta_0 \langle\xi\rangle\} ,\qquad \chi_p (\xi',\xi)\equiv 1\,\,\, \rm{ for } \,\,\, |\xi'|\leqslant \tfrac{1}{2}\delta_0 \langle\xi\rangle ,
\\
&\rm{supp}\, \chi \subset\{(\xi',\xi)\in\R\times\R\quad\textnormal{s.t.}\quad |\xi'|\leqslant\delta_0 \langle\xi\rangle\} ,\qquad \quad
 \chi(\xi',\xi) \equiv 1\,\,\, \rm{ for } \,\,\, |\xi'|\leqslant \tfrac{1}{2}\delta_0   \langle\xi\rangle . 
\end{aligned}
  \end{equation}
For $p=0,$ we set $\chi_0\equiv1$. 
We assume moreover that 
\begin{align*}
|\partial_{\xi}^{\ell}\partial_{\xi'}^{\beta}\chi_p(\xi',\xi)|&\leqslant C_{\ell,\beta}\langle\xi\rangle^{-\ell-|\beta|},\quad\forall \ell\in \N, \,\beta\in\N^{p},\\
|\partial_{\xi}^{\ell}\partial_{\xi'}^{\beta}\chi(\xi',\xi)|&\leqslant C_{\ell,\beta}\langle\xi\rangle^{-\ell-\beta}, \quad\forall \ell, \,\beta\in\N .
\end{align*} 
If $ a (x, \xi) $ is a smooth symbol 
we define its Weyl quantization  as the operator
acting on a
$ 2 \pi $-periodic function
$u$ 
 as
$$
{\rm Op}^{W}(a)u(x)= \sum_{k\in \Z}
\Big(\sum_{j\in\Z}\hat{a}\big(k-j, \tfrac{k+j}{2}\big) u_j \Big){e^{\im k x}},
$$
where $\hat{a}(k,\xi)$ is the $k^{th}-$Fourier coefficient of the $2\pi-$periodic function $x\mapsto a(x,\xi)$.

\begin{definition}{\bf (Bony-Weyl quantization)}\label{defin quant}
If $a(U;x,\xi)$ is a symbol in $\widetilde{\Gamma}^{m}_{p}$, 
respectively in  $\Gamma^{m}_{K,K',p}[\epsilon_0]$,
we set
\begin{equation*}
\begin{aligned}
    a_{\chi_{p}}(U;x,\xi)&\triangleq \!\!\!\!\! \sum_{\substack{\vec{\jmath}\in (\Z^*)^p\\ \vec{\sigma}\in \{ \pm\}^p}}\chi_p (\vec{\jmath},\xi) a_{\vec{\jmath}}^{\vec{\sigma}}(\xi) u_{\vec{\jmath}}^{\vec{\sigma}} e^{\ii (\vec{\sigma}\cdot \vec{\jmath}) x},
\\
 a_{\chi}(U;x,\xi)&\triangleq\sum_{j\in \Z^*} 
\chi (j,\xi )\hat{a}(U;j,\xi)e^{\im j x} , 
\end{aligned}
\end{equation*}
where in the last equality $  \hat a(U; j,\xi) $ stands for $j^{th}$ Fourier coefficient of $a(U;x, \xi)$ with respect to the $ x $ variable, and 
we define the \emph{Bony-Weyl} quantization of $ a(U; \cdot)$ as 
 \begin{align}\label{BonyWeyl}
 &\OpBW{a({U};\cdot)}v(x)\triangleq 
 {\rm Op}^{W} (a_{\chi_{p}}({U};\cdot))v(x)=  \!\!\!\!\!\!\!\!\!\sum_{\substack{(\vec{\jmath},j,k)\in (\Z^*)^{p+2}\\ \vec{\sigma}\in \{ \pm\}^p\\ \vec{\sigma}\cdot \vec{\jmath}+j=k}}\chi_p 
\left(\vec{\jmath},\frac{j+k}{2}\right)
 a_{\vec{\jmath}}^{\vec{\sigma}}\left(\frac{j+k}{2}\right) 
 u_{\vec{\jmath}}^{\vec{\sigma}} v_j
 {e^{\ii k x}} \, ,\\
 &\OpBW{a(U;\cdot)}v(x)\triangleq {\rm Op}^{W} (a_{\chi}(U;\cdot))v(x)=
 \!\!\!\!\!\! \sum_{(j,k)\in (\Z^*)^2} \!\!\!  \chi \left(k-j,\frac{j+k}{2} \right)
\hat{a}\left(U; k-j, \frac{k+j}{2}\right)v_j {e^{\ii k x}} \, .\label{BonyWeylnon}
 \end{align}
\end{definition}
\noindent Note that if 
$ \chi \Big( k-j, \tfrac{ k + j }{2}\Big) \neq 0 $
then $ |k-j| \leqslant \delta_0 \langle \frac{j + k}{2} \rangle  $
and therefore, for $ \delta_0 \in (0,1)$, 
\begin{equation*}
\frac{1-\delta_0}{1+\delta_0} |k| \leqslant
|j| \leqslant \frac{1+\delta_0}{1-\delta_0}|k| \, , \quad \forall j, k \in \Z\, . 
\end{equation*}
This relation shows that the action of a paradifferential  operator does not spread much the Fourier support of functions.\\
If $a$ is a symbol in  $\Sigma\Gamma^{m}_{K,K',p}[\epsilon_0,N]$,
we define its \emph{Bony-Weyl} quantization
$$
\OpBW{a(U;\cdot)} =\sum_{q=p}^{N}
\OpBW{a_q(U;\cdot)} + \OpBW{a_{>N}(U; \cdot) } .
$$
We define as well
\begin{align}\label{def Opvec}
\Opvec{a\pare{U;x,\xi}}\defeq
\OpBW{
\begin{bmatrix}
a\pare{U;x, \xi} & 0 \\ 0 & \overline{a^\vee\pare{U;x, \xi}}
\end{bmatrix}
} , 
&&
a^\vee\pare{U;x, \xi}\defeq a\pare{U;x, -\xi} . 
\end{align}

\begin{rem}
\begin{itemize}
\item
 The operator
$ \OpBW{a} $
maps functions with zero average in functions with zero average, 
and $ \Pi_0^\bot \OpBW{a} = \OpBW{a}\Pi_0^\bot $. 

\item
If $ a$ is a homogeneous  symbol, 
the two definitions  of quantization in \eqref{BonyWeyl}-\eqref{BonyWeylnon} differ 
by a  smoothing operator according to
Definition \ref{defin smoothingop} below, see \cite[page 50]{BD2018}.

\item    
Definition \ref{defin quant}
is  independent of the cut-off functions $\chi _{p}$, $\chi$,
up to smoothing operators (Definition \ref{defin smoothingop}).

\item
The action of
$ \OpBW{a} $ on  the spaces $ \dot H^s $ only depends
on the values of the symbol $  a(u;t, x,\xi)$
for $|\xi|\geqslant1$.
Therefore, we may identify two symbols $ a(u;t, x,\xi)$ and
$ b(u;t, x,\xi)$ if they agree for $|\xi| \geqslant1/2$.
In particular, whenever we encounter a symbol that is not smooth at $\xi=0 $,
such as, for example, $a = g(x)|\x|^{m}$ for $m\in \R^*$, or $ \sgn{\xi} $,
we will consider its smoothed out version
$\pare{1-\chi(\xi)}a\pare{x, \xi}$, where $\chi$ is defined in \eqref{cut off defin}. Similarly for $p$-homogeneous symbols.

\end{itemize}
\end{rem}

\begin{rem}
Given  a paradifferential  operator
$ A = \OpBW{a(x,\xi)} $ it results
\begin{equation*}
\ov{ A} = \OpBW{\overline{a(x, - \xi)}}, \qquad
A^\intercal = \OpBW{a(x, - \xi)}, \qquad
A^*= \OpBW{\overline{a(x,  \xi)}},
\end{equation*}
where $ A^\intercal $  is the transposed  operator with respect to the real scalar product
$ \psc{\cdot}{\cdot}_\R $ in \eqref{scalar products}, and
$ A^* $ denotes the adjoint  operator  with respect to the complex
scalar product $\psc{\cdot}{\cdot}_\C $ on $  \dot{L}^2$ in \eqref{scalar products}. It results $ A^* = \ov{A}^\intercal $.

\begin{itemize}
\item A paradifferential operator $A= \OpBW{a(x,\xi)} $ is {\it real} (i.e. $A = \ov{A} $) if
\begin{equation}\label{real cond} 
\ov{a(x,\xi)}= a(x,-\xi)   \, . 
\end{equation}

\item It is {\it symmetric} (i.e. $A = A^\intercal $)
 if $$a(x,\xi) = a(x,-\xi). $$
\end{itemize}
\end{rem}

We now provide the action of a paradifferential operator on Sobolev spaces, cf.  \cite[Prop. 3.8]{BD2018}.

\begin{lemma}[Action of a paradifferential operator]
  \label{lemactionpara}
    Let $ m \in \R $.
  
  \begin{enumerate}[i)]
  \item \label{item:OpBWmaps1}
If $ p \in \N $, there is $ s_0 > 0 $ such that for any symbol $ a $ in $\Gt{m}{p}$,
there is  a constant $ C > 0 $, depending only on $s$ and on \eqref{estim symbol}
with $ \upgamma = \beta = 0 $,
such that, for any $ (U_1, \ldots, U_p ) $, for $ p \geqslant1  $,
\begin{equation*}
  \norm{\OpBW{a(U_1, \ldots, U_p;\cdot)} u_{p+1}}_{\dot H^{s-m}}\leqslant C
\norm{ U_1 }_{\dot H^{s_0}}  \cdots \norm{ U_p }_{\dot H^{s_0}}
\norm{u_{p+1}}_{\dot H^{s}}  .
\end{equation*}
If  $ p = 0 $
the above bound holds replacing the right hand side with
$ C \norm{u_{p+1}}_{\dot H^{s}} $.
\item \label{item:OpBWmaps2}
Let $\epsilon_0>0$, $p\in \N $, $K'\leqslant K\in \N $, $a$ in $\Gr{m}{K,K', p}$.
There is $ s_0 > 0 $,
and  a constant $ C $, depending only on $s$, $\epsilon_0 $, and on \eqref{estim nonhom symb} with $ 0 \leqslant \alpha  \leqslant 2, \beta = 0$,
such that, for any $ t $ in $ I $, any $ 0\leqslant k\leqslant K-K'$, any $ U $ in $\Br{K}{}$,
    \begin{align}\label{para:dt}
  \norm{\OpBW{ \partial_t^ka(U;t,\cdot) } }_{\Lcal \pare{ \dot H^{s},\dot H^{s-m} } }\leqslant C 
 \| U(t, \cdot)\|_{k+K',s_0}^p \, ,
\end{align}
so that
$\norm{ \OpBW{a(U;t,\cdot)} v(t) } _{K-K', s-m} \leqslant C  \| U(t, \cdot)\|_{K,s_0}^p
\| v(t) \|_{K- K', s} $.

  \end{enumerate}
\end{lemma}

\paragraph{Classes of $m$-Operators and smoothing Operators.}
Given integers $\pare{ n_1,\ldots,n_{p+1} } \in (\N^*)^{p+1}$ we denote by $\max_{2}\set{n_1 ,\ldots, n_{p+1}}$
the second largest among  $ n_1,\ldots, n_{p+1}$.
We now define  $ m $-operators which  include the class of paradifferential operators of order $ m $
(see Remark \ref{rem smoothingop}) and allow to define the smoothing remainders when
the order $ m $ is negative (see Definition \ref{defin smoothingop}).
The class $\tilde{\mathcal{M}}^{m}_{p}$ denotes multilinear
 operators that lose $m$ derivatives
 and are $p$-homogeneous in $u $,
while the class $\mathcal{M}_{K,K',p}^{m}$ contains non-homogeneous
operators  which lose $m$ derivatives,
vanish at degree at least $ p $ in $ u $, satisfy tame estimates
 and are $(K-K')$-times differentiable in $ t $. 
The  constant $ \mu $ in \eqref{eq:bound_fourier_representation_m_operators} takes into account possible loss of derivatives in the ``low" frequencies. 

The following definition is taken from  \cite[Def. 2.5]{BMM2022} (see also its Fourier characterization in \cite[Lemma 2.9]{BMM2022}.

\begin{definition}[ Classes of $m$-operators]\label{defin mop}
Let  $ m \in \R $,  $p,N\in \N$,
$K,K'\in\N$ with $K'\leqslant K$, and $ \epsilon_0 > 0 $.

\begin{enumerate}[i)]

\item \label{item:maps1}
{\bf $p$-homogeneous $m$-operators.}
We denote by $\tilde{\mathcal{M}}^{m}_{p}$
 the space of $ (p+1)$-linear symmetric translation invariant operators from 
$ \pare{\dot H^\infty \pare{\mathbb{T};\mathbb{C}^2}}^{p} \times \dot H^\infty \pare{\mathbb{T};\mathbb{C}} $ to 
$ \dot H^\infty \pare{\mathbb{T};\mathbb{C}} $, whose associated polynomial has the form
\begin{equation} \label{map:Fourier}
M(U)v \defeq M\pare{U, \ldots, U} v = 
 \sum_{\substack{  \vec\sigma\in \{\pm\}^{p} \\  k -j = \vec \sigma \cdot \vec \jmath} }
\!\!\!\!\!\!
  M_{\vec \jmath, j,k}^{\vec \sigma} \, u_{\vec \jmath}^{\vec \sigma}\,  v_{j}  \, {e^{\ii k x} }
 \end{equation}
with coefficients $   M_{\vec \jmath, j,k}^{\vec \sigma} $ symmetric in $ (j_1,\sigma_1), \ldots, (j_{p},\sigma_p) $, 
 satisfying the following:  there are $\mu \geqslant0$, $C>0$ such that, for any 
 $ \vec \jmath=( j_1, \ldots, j_{p}) \in (\Z^*)^p$, $ j, k \in  \Z^* $, it results 
 \begin{equation}\label{eq:bound_fourier_representation_m_operators}
\av{ M_{\vec \jmath_{p}, j,k}^{\vec \sigma_{p}} } \leqslant C \   {\rm max} _2 \set{\av{j_1}, \ldots,  \av{j_{p}}, \av{ j} } ^\mu \   \max\set{ \av{j_1}, \ldots , \av{j_{p}},\av{ j} }^m , 
\end{equation}

If $ p=0 $ the right hand side of \eqref{map:Fourier} must be substituted with $ \sum _{j\in \bZ} M_j v_j e^{\ii j x} $ with $ \av{M_j}\leqslant C \av{ j}^m $. 
 \item
 {\bf Non-homogeneous $m$-operators.}
  We denote by  $\mathcal{M}^{m}_{K,K',p}[\epsilon_0]$
  the space of operators $(U,t,v)\mapsto M(U;t) v $ defined on { $B^{K'}_{s_0}\pare{I;\epsilon_0}$} for some $ s_0 >0  $,
  which are linear in the variable $ v $ and such that the following holds true.
  For any $s\geqslant s_0$ there are $C>0$ and
  $\epsilon_0(s)\in]0,\epsilon_0[$ such that for any
  { $U \in B^{K'}_{s_0}\pare{I;\epsilon_0} \cap C^K_{*}\pare{ I,\dot H^{s}(\T;\C) }$,
  any $ v \in C^{K-K'}_{*} \pare{ I,\dot H^{s}(\T;\C) } $}, any $0\leqslant k\leqslant K-K'$, $t\in I$, we have that { 
\begin{equation}
\label{stima:emmeop}
\norm{{\partial_t^k\left(M(U;t)v\right)}}_{s- \frac32 k-m}
 \leqslant C   \sum_{k'+k''=k}  \pare{ \|{v}\|_{k'',s}\|{U}\|_{k'+K',{s_0}}^{p}
 +\|{v}\|_{k'',{s_0}}\|U\|_{k'+K',{s_0}}^{p-1}\| U \|_{k'+K',s} }  .
\end{equation}}
In case $ p = 0$ we require  the estimate
$ \|{\partial_t^k\left(M(U;t) v \right)}\|_{s- \frac32k-m}
 \leqslant C  \| v \|_{k,s}$.  
  We say that a non-homogeneous $ m $-operator  $M\pare{U;t} $ is \emph{real} if it is real valued for any
$ u \in B^{K'}_{s_0}\pare{I;\epsilon_0} $.

 \item
 {\bf $m$-Operators.}
We denote by $\Sigma\mathcal{M}^{m}_{K,K',p}[\epsilon_0,N]$
the space of operators 
\begin{equation}
\label{maps}
M(U;t)v =\sum_{q=p}^{N}M_{q}(U)v+M_{>N}(U;t)v ,
\end{equation}
where $M_{q} $ are homogeneous $m$-operators in $ \tilde{\mathcal{M}}^{m}_{q}$, $q=p,\ldots, N$  and
$M_{>N}$  is a non--homogeneous $m$-operator
in $\mathcal{M}^{m}_{K,K',N+1}[\epsilon_0]$.  We say that a  $ m $-operator  $M\pare{u;t} $ is \emph{real} if it is real valued for any
$ u \in  B^{K'}_{s_0}\pare{I;\epsilon_0} $.
We shall also denote with  $ \Sigma_p^N \wt{\cM}^m_q$ the subspace of \\
$\Sigma \cM_{K,K',p}^m[\epsilon_0,N]$ made of pluri-homogeneous $m$-operators, namely symbols which expand as in \eqref{expansion symbols} with $M_{>N}\equiv 0$.
\end{enumerate}
\end{definition}
\begin{rem}
 By  \cite[Lemma 2.8]{BMM2022},
if $ M( U_1,\dots , U_p)$ is a $p$--homogeneous
 $m$-operator in $  \widetilde \mM_p ^m$ then
  $ M(U) = M(U, \ldots, U) $ is a  non-homogeneous $ m$-operator in
 $  {\cal M}^m_{K,0,p}[\epsilon_0] $ for any $\epsilon_0>0$ and $K\in \N$.
 We shall say that $ M(u) $ is in $  \tilde \mM_p ^m $. 
  \end{rem}

\begin{notation}
   \begin{itemize}
       \item  If $M(U_1, \ldots, U_p)$ is a $p$-homogeneous $m$-operator, we shall often denote by $M(U)$ the associated $p$-homogeneous polynomial, as in \eqref{map:Fourier}, and write $M(U) \in \widetilde{\mathcal{M}}_p^m$. Conversely, a $p$-homogeneous polynomial can be represented by a $(p+1)$-linear form of the type $ M(U_1, \ldots, U_p) U_{p+1} $, which may not be symmetric in the first $p$ variables. If this form satisfies the symmetric estimate \eqref{eq:bound_fourier_representation_m_operators}, then it corresponds to an $ m $-operator in $ \widetilde{\mathcal{M}}_p^m $ obtained by symmetrizing the internal variables. In the sequel, we adopt this identification without further comment.
\item given an operator  $M(U;t)$ in $ \Sigma \cM_{K,K',p}^m[r, N]$   of the form \eqref{maps} 
we denote by  
\be \label{projectlessN}
\cP_{\leqslant N}[ M(U;t)] \triangleq  \sum_{q=p}^{N} M_q(U) \, ,
 \quad \text{resp.} \qquad 
 \cP_{q}[ M(U;t)] \triangleq M_q(U) \, , 
 \ee
the projections on the pluri-homogeneous, resp. homogeneous, operators in 
 $\Sigma_p^N \widetilde \cM^m_q $ , resp. in $\wt\cM_q^m$. Given an integer $ p\leqslant p'\leqslant N$ we also denote 
$$
 \cP_{\geqslant p'}[ M(U;t)] \triangleq  \sum_{q=p'}^{N} M_q(U), \qquad \cP_{\leqslant  p'}[ M(U;t)] \triangleq  \sum_{q=p}^{p'} M_q(U) \, .
$$
 The same notation will be also used to denote 
 pluri-homogeneous/homogeneous components of symbols. 
   \end{itemize}
 \end{notation}

If $m  \leqslant 0 $ the  operators in $ \Sigma \mM^{m}_{K,K',p}[\epsilon_0,N]$ are referred to as smoothing operators.

 \begin{definition}[Smoothing operators] \label{defin smoothingop}
Let $ \vr\geqslant0$. A $ (-\vr)$-operator $R(U)$ belonging to $ \Sigma \mM^{-\vr}_{K,K',p}[\epsilon_0,N]$ 
is called  a smoothing operator. 
We also denote
\begin{align*}
 \wt{\mathcal{R}}^{-\vr}_{p}\defeq \wt{\mathcal{M}}^{-\vr}_{p} \, ,
&& 
 \mathcal{R}^{-\vr}_{K,K',p}[\epsilon_0]\defeq\mathcal{M}^{-\vr}_{K,K',p}[\epsilon_0] \, , 
 && 
  \Sigma\mathcal{R}^{-\vr}_{K,K',p}[\epsilon_0,N]\defeq\Sigma\mathcal{M}^{-\vr}_{K,K',p}[\epsilon_0,N] \,, && \Sigma_p^N \wt{\mR}^{-\vr}_q\defeq\Sigma_p^N \wt{\mM}^{-\vr}_q .
\end{align*}
\end{definition}
\begin{rem} \label{rem smoothingop}
$ \bullet $
Lemma \ref{lemactionpara} implies that,  if  $a(U; t, \cdot)$ is a symbol  in $\Sigma\Gamma^{m}_{K,K',p}\bra{\epsilon_0 , N}$, 
 $m\in\mathbb{R} $,  then the associated paradifferential operator 
 $ \OpBW{a(U; t, \cdot)} $ defines a $ m $-operator in
$\Sigma  \cM^{m}_{K,K',p}\bra{\epsilon_0 , N}$.

$ \bullet $
The composition of smoothing operators $ R_1 \in \sr{-\vr}{K,K',p_1}{N}$
 and  $ R_2  \in \sr{-\vr}{K,K',p_2}{N} $ is a smoothing operator $ R_1 R_2 $  in 
$ \sr{-\vr}{K,K',p_1+p_2}{N} $.  This is a particular case of Proposition \ref{prop compo mop}-($i)$ below.
\end{rem}
\begin{definition}[{\bf Homogeneous vector fields}]
	Let $m \in \R$ and $ p, N \in \N$. We denote by $\wt \X_{p+1}^m$ the space of $ (p+1)$-homogeneous vector fields of the form $ X(U)=M(U)U$ where $M(U)$ is a matrix of $p$-homogeneous $m$-operators in $ \mats{\wt \cM_{p}^m}$. 
In particular, one has the Fourier expansion 
$$X(U)=\begin{bmatrix}
    X(U)^+\\
    X(U)^-
\end{bmatrix},\qquad X(U)^{\sigma}\triangleq\sum_{ \substack{ (\vec{\jmath}, k, \vec{\sigma}, - \sigma) \in \fT_{p+2}
}}X_{\vec{\jmath},k}^{\vec{\sigma},\sigma} u_{\vec{\jmath}}^{\vec{\sigma}}e^{\ii\sigma k},$$
where the  Fourier restriction $\mathfrak{T}_{p+2}$ is  the set of momentum preserving indices, defined, for a given $q\in\N^*$, as 
\be \label{mompresind}
\mathfrak{T}_q\triangleq \left\{  \pare{ \vec{\jmath} , \vec \sigma} \in \Z^q \times \{\pm \}^q \quad\textnormal{s.t.}\quad  \vec    \sigma \cdot\vec{\jmath}   = 0  \right\}.
\ee
    	We denote 
	$ \Sigma_{p+1}^{N+1}\wt\X^{m}_q$ the class of pluri-homogeneous vector fields. 
	The vector fields in  $\wt \X_{p+1}^{- \varrho}$, $ \varrho \geqslant 0 $, are called smoothing. 
\end{definition}

\paragraph{Symbolic calculus.}
Let
$ \s(D_{x},D_{\x},D_{y},D_{\eta}) \defeq D_{\x}D_{y}-D_{x}D_{\eta}  $
where $D_{x}\defeq\frac{1}{\ii}\pa_{x}$ and $D_{\x},D_{y},D_{\eta}$ are similarly defined.
The following is   Definition 3.11 in \cite{BD2018}.

\begin{definition}[Asymptotic expansion of composition symbol]
Let $ p $, $ p' $ in $\N $, $ K, K' \in \N$ with $K'\leqslant K$,  $ \vr  \geqslant0 $, $m,m'\in \R$, $\epsilon_0>0$.
Consider symbols $a \in \Sigma\Gamma_{K,K',p}^{m}[\epsilon_0,N]$ and $b\in \Sigma \Gamma^{m'}_{K,K',p'}[\epsilon_0,N]$. For $U$ in $B_{\s}^{K}(I;\epsilon_0)$
we define, for $\vr< \s- s_0$, the symbol
\begin{equation*}
(a\#_{\vr} b)\pare{ U;t, x,\x } \defeq\sum_{k=0}^{\vr}\frac{1}{k!}
\left(
\frac{\ii}{2}\s\pare{ D_{x},D_{\x},D_{y},D_{\eta} } \right)^{k}
\Big[a(U;t, x,\x)b(U;t,y,\eta)\Big]_{|_{\substack{x=y, \x=\eta}}}
\end{equation*}
modulo symbols in $ \Sigma \Gamma^{m+m'-\vr}_{K,K',p+p'}[\epsilon_0,N] $.
\end{definition}

The symbol $ a\#_{\vr} b $ belongs  to $\Sigma\Gamma^{m+m'}_{K,K',p+p'}[\epsilon_0,N]$.
Moreover
\begin{equation*}
a\#_{\vr}b = a b + \frac{1}{2 \ii }\{a,b\} 
\end{equation*} 
up to a symbol in $\Sigma\Gamma^{m+m'-2}_{K,K',p+p'}[\epsilon_0,N]$,
where
$$
\{a,b\}  \defeq  \pa_{\xi}a \ \pa_{x}b -\pa_{x}a \ \pa_{\xi}b
$$
denotes the Poisson bracket.
\smallskip
The following result is proved in Proposition $3.12$ in \cite{BD2018}.

\begin{proposition}[Composition of Bony-Weyl operators] \label{prop compBW}
Let $p,q,N, K, K'  \in \N$ with $ K' \leqslant K $,  $\vr \geqslant0 $, $m,m'\in \R$, $\epsilon_0>0$.
Consider  symbols
$a\in \Sigma {\Gamma}^{m}_{K,K',p}[\epsilon_0,N] $ and $b\in \Sigma {\Gamma}^{m'}_{K,K',q}[\epsilon_0, N]$.
Then
\begin{equation*}
\OpBW{a(U;t, x,\x)}\circ\OpBW{b(U;t, x,\x)} - \OpBW{(a\#_{\vr} b)(U;t, x,\x)}
\end{equation*}
is a smoothing operator in $ \Sigma {\mathcal{R}}^{-\vr+m+m'}_{K,K',p+q}[\epsilon_0,N]$. 
\end{proposition}

  We have the following result, see e.g. Lemma 7.2 in \cite{BD2018}.
\begin{lemma}[Bony paraproduct decomposition]
\label{lemma Bonyprod}
Let  $u_1, u_2$ be functions in $H^\s(\T;\C)$   with
$\s >\frac12$. Then
\begin{equation*}
u_1 u_2  =  \OpBW{u_1 }u_2 + \OpBW{u_2 }u_1 +
R_1(u_1)u_2 +  R_2(u_2)u_1
\end{equation*}
where for $ \mathsf{j} =1, 2 $, $R_\mathsf{j}$ is a homogeneous smoothing operator in $ \wt \mR^{-\vr}_{1}$ for any $ \vr \geqslant0$. 
\end{lemma}
We now state other composition results for  $m$-operators which follow as in 
\cite[Proposition 2.15]{BMM2022}.
\begin{proposition}[Compositions of $m$-operators] \label{prop compo mop}
Let $p, p', N, K, K' \in \N$ with $K'\leqslant K$ and $\epsilon_0>0$. Let $m,m' \in \R$.
Then
\begin{enumerate}
\item
 If  $M(U;t) $ is  in
$ \Sigma\mathcal{M}^{m}_{K,K',p}[\epsilon_0,N]$ and $M'(U;t) $ is  in
$ \Sigma\mathcal{M}^{m'}_{K,K',p'}[\epsilon_0,N] $ then the composition
$ M(u;t)\circ M'(U;t) $
is  in $\Sigma\mathcal{M}^{m+\max(m',0)}_{K,K',p+p'}[\epsilon_0,N]$.

\item \label{item:MM_ext}
If  $M (U) $ is a homogeneous $ m$-operator in $  \wt{\mathcal{M}}_{p}^{m}$
and $M^{(\ell)}(U;t)$, $\ell=1,\dots,p+1$, are matrices of  $ m_\ell $-operators  in
$ \Sigma \mM^{m_\ell}_{K,K',q_\ell}[\epsilon_0,N]$ with  $m_\ell \in \R$,
$q_\ell\in \N$,
then
$$
M \pare{ M^{(1)}(U;t)u, \ldots,  M^{(p)}(U;t)U }M^{(p+1)}(U;t)
$$
belongs to $ \Sigma\mM_{K,K',p+\bar q}^{m+ \bar m}[\epsilon_0,N]$ with $ \bar m\defeq\sum_{\ell=1}^{p+1} \max(m_\ell,0)$ and $\bar q\defeq \sum_{\ell=1}^{p+1}q_\ell$.

\item
Let $a$ be a symbol in $\sg{m}{K,K',p}{N}$ with $m\geqslant0$ and $R$ a smoothing operator in $\sr{-\vr}{K,K',p'}{N}$. Then 
\begin{align*}
\OpBW{a(U; t, \cdot)} \circ R(U;t) \, , 
\ \  
 R(U;t) \circ \OpBW{a(U; t, \cdot)} \,  \ \in  \sr{-\vr+m}{K,K',p+p'}{N}.
\end{align*}
\item If $a_p$ is in $ \wt{\Gamma}_p^m$ and $ M(U)\in \Sigma \cM_{K,K',p'}^{m'}[r,N]$ then $a_p(M(U), U, \ldots, U;x,\xi)\in\Sigma \Gamma^m_{K,K',p+p'}[r,N]$ and 
$$
\OpBW{a(W, U, \ldots, U;x,\xi)}_{|W= M(U)} = \OpBW{ a_p(M(U), U, \ldots, U;x,\xi)}+ R(U)
$$
where $R(U) \in  \Sigma \cR_{K,K',p+p'}^{-\vr}[r,N]$, for any $\vr>0$. In particular if $ a \in  \Sigma \Gamma_{K,K',p}^{m}[r,N]$ then 
\begin{align*}
&\text{ If} \quad \pa_t U = M_0(U)U, \quad M_0(U)\in \Sigma \cM_{K,K',0}^{m'}[r,N],\\ 
& \text{then} \quad  \pa_t \OpBW{a(U; x, \xi)}= \OpBW{a_{[1]}(U;x,\xi)}+ R(U),
\end{align*}
where $a_{[1]}\in \Sigma \Gamma_{K,K'+1,p}^{m}[r,N]$ and $R(U)\in  \Sigma \cR_{K,K',p'}^{-\vr}[r,N]$
\end{enumerate}
\end{proposition}

\begin{notation}
 In the sequel if $ K'= 0 $ we  denote a 
symbol $ a (U; t, x, \xi) $ in $\Gamma_{K,0,p}^m[\epsilon_0]$  simply as
$ a (U; x, \xi)  $,   
and a smoothing operator in $ R (U;t) $ 
in 
$  \Sigma\mathcal{R}^{-\vr}_{K,0,p}[\epsilon_0,N] $ simply as $ R(U)  $, 
without writing the $ t $-dependence.
\end{notation}

We finally provide the Bony paralinearization formula
of the composition operator whose proof is a combination of \cite[Lemma 3.19]{BD2018}. 

\begin{lemma}[Bony Paralinearization formula]
  \label{lemma paralin}
Let $F$ be a smooth $\C$-valued function defined on a neighborhood of zero in $\C $, vanishing at zero at order $q\in
\N $. Then there are $ s_0,\epsilon_0 >  0 $
 such that if $u\in B_{H^{s_0}(\T;\R)}(\varepsilon_0),$ then
\begin{align*}
    F(u) =  \  \OpBW{ F'\pare{u} } u + R\pare{u}u,
\end{align*}
where   $ R\pare{u} $ is a smoothing operator
in \, $\Sigma \cR^{-\vr}_{K,0,q'}\bra{\epsilon_0, N}$, $q' \defeq \max(q-1,1)$, for any $\vr\geqslant0$. 
\end{lemma}

\subsection{Spectrally localized maps}

In this section we introduce the class of spectrally localized map, needed to apply the Darboux symplectic correction (see \Cref{darboux:0}). This class was introduced first in \cite[Def. 2.15]{BMM2022}

\begin{definition}[Spectrally localized maps]\label{defin specloc}
Let $m \in \mathbb{R}, p, N \in \mathbb{N}, K, K' \in \mathbb{N} $ with $K' \leqslant K$ and $r > 0$.
\begin{enumerate}[i)]
\item \textbf{Spectrally localized $p$-homogeneous maps.} We denote by $ \widetilde{\mathcal{S}}^m_p $ the subspace of $m$-operators $S(U)$ in $ \wt{\mathcal{M}}^m_p $ whose coefficients $S_{\vec \jmath, j,k}^{\vec \sigma} $ (see \eqref{map:Fourier}) satisfying the following spectral condition: there are $ \delta > 0, C > 1 $ such that 
$$S_{\vec \jmath, j,k}^{\vec \sigma}\not=0 \quad \implies | \vec \jmath | \leqslant \delta | j |, \quad C^{-1} |k| \leqslant |j| \leqslant C |k|.$$

We denote $ \widetilde{\mathcal{S}} \triangleq \bigcup_{p} \widetilde{\mathcal{S}}^m_p $ and by $ \Sigma^N_p \widetilde{\mathcal{S}}^m_q $ the class of pluri-homogeneous spectrally localized maps of the form $ \sum_{q=p}^N S_q $ with $ S_q \in \widetilde{\mathcal{S}}^m_q $ and $ \Sigma_p \widetilde{\mathcal{S}}^m_q \triangleq \bigcup_{N \in \mathbb{N}} \Sigma^N_p \widetilde{\mathcal{S}}^m_q $. For $ p \geqslant N + 1 $ we mean that the sum is empty.

\item \textbf{Non-homogeneous spectrally localized maps.} We denote $ \mathcal{S}^{m}_{K, K', p}[\epsilon_0] $ the space of maps $ (U, t, V) \mapsto S(U; t)V $ defined on $ B^{K'}_{K}(I; r) \times I \times C^{0}(I, H^{s_0}(T, \mathbb{C})) $ for some $ s_0 > 0 $, which are linear in the variable $ V $ and such that the following holds true. For any $ s \in \mathbb{R} $ there are $ C > 0 $ and $ r(s) \in [0, \epsilon_0] $ such that for any $ U \in B^{K'}_{K}(I; r(s)) \cap C^{\ast}(I, H^{s}(T; \mathbb{C}^2)) $, any $ V \in C^{\ast}_{K - K'}(I, H^{s}(T, \mathbb{C})), $ any $ 0 \leqslant k \leqslant K - K', t \in I $, we have that

\begin{align*}
\| \partial^{k}_{t} (S(U; t)V) (t, \cdot) \|_{H^{s - \frac{3}{2}k - m}} \leqslant& C \sum_{k' + k'' = k} \| U \|^{p}_{k', s_0} \| V \|_{H^{k'', s}},\quad&& \text{if} \   p\geqslant 1,\\
\| \partial^{k}_{t} (S(U; t)V) \|_{H^{s - \frac{3}{2}k - m}} \leqslant & C \| V \|_{k, s},\quad && \text{if} \ p=0.
\end{align*}

We denote $ S^{m}_{K, K', N}[\epsilon_0] \triangleq \bigcup_{p} \mathcal{S}^{m}_{K, K', p}[\epsilon_0] $.

\item \textbf{Spectrally localized Maps.} We denote by $ \Sigma\cS^m_{K, K', p}[r, N] $, the space of maps $ (U, t, V) \mapsto S(U; t)V $ of the form
$$
S(U; t)V = \sum_{q = p}^{N} S_q (U)V + S_{> N}(U; t)V,
$$
where $ S_q $ are spectrally localized homogeneous maps in $ \widetilde{\mathcal{S}}^m_q, q = p, \ldots, N $ and $ S_{> N} $ is a non-homogeneous spectrally localized map in $ \mathcal{S}^{m}_{K, K', N + 1}[\epsilon_0] $. We denote by $ \pare{\Sigma^m_{K, K', p}[r, N]}^{2\times 2} $ the space of $ 2 \times 2 $ matrices whose entries are spectrally localized maps in $ \Sigma^{m}_{K, K', p}[r, N] $. We will use also the notation $ \Sigma^m_{K, K', p}[r, N] \triangleq \bigcup_{l \geqslant0} \Sigma^m_{K, K', p}[r, N + l] $.

\end{enumerate}
\end{definition}
\subsection{$ z$-dependent paradifferential calculus}  \label{sec:paradiff}

 The following  ``Kernel-functions", 
that depend on the "convolutive $ 2\pi $-periodic variable" $ z $,   
have to be considered as Taylor remainders of functions $ K\pare{u; x, z} $ at $ z=0 $
which are smooth  in $ u $ and 
which have finite regularity in $ x $ and $ z $. 
 A Kernel function 
is  a $ z $-dependent family of functions 
 (cfr. Definition \ref{defin funct}) with coefficients of size proportional to 
$ |z|^n_{\mathbb{T}}  $. For $ n > - 1 $ such singularity is integrable in $ z $.

\begin{definition}[Kernel functions]\label{defin kerfunct}
Let $n\in \R$,  $p,N\in \N$,  
$ K \in \N$, and $ \epsilon_0>0 $.
\begin{enumerate}[i)]
\item $p$-{\bf homogeneous Kernel-functions.} If $ p \in \N$
we denote  $\wt{\KF}^n_p $ the space of 
$ z $-dependent, $ p $-homogeneous maps 
from $ \dot H^\infty \pare{\mathbb{T};\mathbb{C}} $ to the space of 
$ x $-translation invariant real functions  
$ \kappa (u;x,z) $ of class $ \cC^\infty $ 
in $ (x, z ) \in \mathbb{T}^2 $ with Fourier expansion 
\begin{equation*}
\kappa (u;x,z)=  
\sum_{j_1, \ldots, j_p \in \Z^*}  \! \kappa_{j_1, \ldots, j_p} (z) 
u_{j_1} \cdots u_{j_p}
 e^{\im (j_1+ \cdots + j_p) x} \, , \quad z \in \T\setminus \{ 0 \}  \, , 
\end{equation*} 
with coefficients $ \kappa_{j_1, \ldots, j_p} (z) $ of class  $ \cC^\infty\pare{\mathbb{T}; \C }  $,  
symmetric in $ (j_1, \ldots, j_p ) $, 
satisfying the reality condition 
$ \overline{\kappa_{j_1, \ldots, j_p}}\pare{z} =
\kappa_{-j_1, \ldots, -j_p} \pare{z} $ 
and the 
following: for any $ l \in \N $, there exist $ \mu > 0 $ and 
a constant $ C > 0 $ such that
\begin{equation}\label{estim homkernel}
\av{\partial_{z}^{l} \kappa_{j_1, \ldots, j_p} (z) }
\leqslant C \av{\vec{\jmath}_p}^{\mu}\  \av{z}^{n-l}_{\mathbb{T}}  \, , \quad
\forall  \vec \jmath _p = \pare{ j_1, \ldots, j_p } \in (\Z^*)^p   \, .
\end{equation}
For $ p = 0 $ we denote by $\wt{\KF}^n _0 $ the space of 
maps $ z \mapsto \kappa (z) $ which satisfy 
$ \av{\partial_{z}^{l} \kappa (z) } \leqslant C \,  \av{z}^{n-l}_{\mathbb{T}} $.
\item  {\bf Non-homogeneous Kernel-functions.}   We denote by $\KF _{K,0,p}^{n}[\epsilon_0]$ the space of $ z $-dependent, 
 real functions  $ \kappa (u;x,z) $,
defined for $ u \in B_{s_0}^0 (I;\epsilon_0)  $ for some $s_0$ large enough, such that for any $0\leqslant k\leqslant K $ and $ l\leqslant  \max\set{0, \ceil{1+ n }\big. } $, any $s\geqslant s_0$, there are $C>0$, $0<\epsilon_0(s)<\epsilon_0$ and for any $ u \in B_{s_0}^K\pare{ I;\epsilon_0(s) }\cap C_{*}^{k}\pare{ I, \dot H^{s} \pare{ \mathbb{T};\mathbb{C} } }$ and any $\upgamma \in \N$, with $\alpha\leqslant s-s_0,  $ one has the estimate
\begin{equation}\label{estim HOMkernel2}
\av{ \partial_t^k\partial_x^\alpha\partial_z ^l \kappa \pare{ u;x, z } }  \leqslant C  \| u \|_{k,s_0}^{p-1}\|u\|_{k,s} \ \av{z} ^{n-l}_{\mathbb{T}} \, , \qquad z\in\mathbb{T}\setminus \{ 0 \}   \, .
\end{equation}
If $ p = 0 $ the right hand side in \eqref{estim HOMkernel2} 
has to be replaced by $   \av{z} ^{n-l}_{\mathbb{T}}  $.

\item
{\bf Kernel-functions.} We denote by $\Sigma \KF _{K,0,p}^{n}[\epsilon_0,N]$ the space of real functions  
of the form 
\begin{equation*}
\kappa (u;x,z)= \sum_{q=p}^{N} 
\kappa _q\pare{ u;x,z } + \kappa _{>N}(u;x,z ) , 
\end{equation*}
where $\kappa_q \pare{ u;x,z }  $, $q=p,\dots, N$ 
are homogeneous Kernel functions in $ \wt{\KF}_q^n $,  and 
$\kappa _{>N}(u;x,z )    $ 
is a non-homogeneous Kernel function in $ \KF _{K,0,N+1}^{n} [\epsilon_0] $.

A Kernel function  $\kappa(u;x,z) $ is \emph{real} if it is real valued for any
$ u \in  B_{s_0,\R}^0 (I;\epsilon_0) $.
\end{enumerate}
\end{definition}

 We list some properties of the Kernel functions. 
In view of the second point of Remark \ref{remsymb}, 
 a homogeneous Kernel function $ \kappa ( u; x,z) $ in $ \wt{\KF}^n_p $ 
defines  a non-homogeneous  Kernel function  in $ \KF _{K,0,p}^{n}[\epsilon_0]  $
for any $ \epsilon_0 >0 $. 

\begin{rem}\label{rem:intkerfunct} Let us make the following remarks.
\begin{enumerate}[\textbullet]
    \item Let $ \kappa \pare{u;x, z} $ be 
a Kernel function in $ \Sigma \cF^{n}_{K, 0, p}\bra{\epsilon_0, N} $ with $ n\geqslant0 $, 
which admits a continuous  extension in $ z=0 $. Then its trace 
 $  \kappa\pare{u;x, 0} $ at $ z = 0 $ 
 is a function in $ \Sigma \cF^{\mathbb{R}}_{K, 0, p}\bra{\epsilon_0, N} $.
 \item If $ \kappa (u; x, z)$ is a homogeneous   Kernel function  $ \wt{\KF}^n_p $, 
the two definitions  of quantization in \eqref{BonyWeyl} differ 
by a  Kernel smoothing operator in $ \wt{\KR}^{-\vr,n}_p $, for any $ \vr > 0 $, 
according to
Definition \ref{defin kernel smoothing} below.
\item (Sum and product of Kernel functions)  
If $ \kappa_1 (u;x,z) $ is a  Kernel function in 
$ \Sigma \KF ^{n_1}_{K, 0, p_1}\bra{\epsilon_0, N} $ 
and $ \kappa_2 (u;x,z) $  in
$ \Sigma \KF ^{n_2}_{K, 0, p_2}\bra{\epsilon_0, N} $, 
then the sum 
$ (\kappa_1 + \kappa_2 )(u;x,z)$ is a  Kernel function in $ \Sigma \KF^{\min\set{n_1, n_2}}_{K, 0, \min\set{p_1, p_2}}\bra{\epsilon_0, N} $ and the product
$ ( \kappa_1  \kappa_2 )(u;x,z)$ is a  Kernel function in 
$ \Sigma \KF^{n_1+ n_2}_{K, 0, p_1 + p_2}\bra{\epsilon_0, N} $.
\item (Integral of Kernel functions)
Let $ \kappa \pare{u; x, z} $ be a Kernel function  in $ \Sigma \KF^{n}_{K, 0, p}\bra{\epsilon_0, N} $ with $  n>-1 $. Then 
$ \fint \kappa\pare{u; x, z}\dd z $ is a function in 
$ \Sigma \cF^{\mathbb{R}}_{K, 0, p}\bra{\epsilon_0, N}$. 
This follows directly integrating  \eqref{estim homkernel} and \eqref{estim HOMkernel2} in $ z $.
\end{enumerate}
\end{rem}
The $ m $-Kernel-operators  defined below are a $ z $-dependent family of $  m $-operators 
 (cfr. Definition \ref{defin mop}) with coefficients of size proportional to 
$ |z|^n_{\mathbb{T}}  $. 
For $ n > - 1 $ such singularity is integrable in $ z $. 
A family of $ z$-dependent paraproduct operators associated to 
Kernel functions defines a  $0$-Kernel operator, see Remark
\ref{remark kerfunct}. {The kind of operators will appear only in case  $ m < 0$, as smoothing operators 
in the  composition of Bony-Weyl quantizations of 
Kernel-functions (see Definition \ref{defin kernel smoothing}).}

\begin{definition}
Let  $ m,n \in \R $,  $p,N\in \N$,
$K\in\N$ with $ \epsilon_0 > 0 $.
\begin{enumerate}[i)]
\item
{\bf $p$-homogeneous $m$-Kernel-operator.}
We denote by $\widetilde{\KM}^{m,n}_{p}$ 
 the space  of 
$ z $-dependent, $ x $-translation invariant  homogeneous $ m $-operators 
according to  \Cref{defin mop}, \cref{item:maps1},  in which the constant 
$ C $  is substituted with $ C\av{z}^n_{\mathbb{T}} $, equivalently
\begin{equation}\label{expansion homokerop}
M( u ;z)v \pare{x} = 
\sum_{\substack{(\vec \jmath_{p},j,k) \in \Z ^{p+2} \\  
j_1 + \ldots + j_p + j = k } } M_{\vec \jmath_{p}, j,k}\pare{z} u_{j_1} \ldots u_{j_p} v_{j}  
e^{\ii  k x}  \, ,  \qquad 
  z \in \mathbb{T}\setminus \{ 0 \}  \, , 
\end{equation} 
 with coefficients satisfying
\be \label{estim fourier kerop}
\big| M_{\vec \jmath _p, j, k}\pare{z} \big| \leqslant C \   {\rm max} _2 \set{\av{j_1}, \ldots,  \av{j_{p}}, \av{j}} ^\mu \ \max\set{\av{j_1}, \ldots , \av{j_{p}}, \av{j}} ^m \av{z}^n_{\mathbb{T}} \, . 
\ee
If $ p=0 $ the right hand side  of \eqref{expansion homokerop} is replaced by
 $ \sum _{j\in \bZ} M_j \pare{z} v_j e^{\ii j x} $ with $ \big| M_j \pare{z} \big| 
 \leqslant C \av{ j}^m \av{z}^n_{\mathbb{T}} $. 
 \item
 {\bf Non-homogeneous $m$-Kernel-operator.}
  We denote by  $\KM ^{m,n}_{K,0,p}[\epsilon_0]$
  the space of $ z $-dependent, 
  non-homogeneous operators $ M(u;z) v  $ defined for any 
$ z  \in\mathbb{T}\setminus \{ 0 \}   $,  such that for any $ 0\leqslant k\leqslant K $
\begin{equation}\label{estim dtkerop}
\norm{{\partial_t^k\left(M(u;z)v\right)}} _{s- \alpha k-m} \\
 \leqslant C  \av{z}^n_{\mathbb{T}} \   \sum_{k'+k''=k}  \pare{ \|{v}\|_{k'',s}\|{u}\|_{k',{s_0}}^{p}
 +\|{v}\|_{k'', {s_0}}\|u\|_{k', {s_0}}^{p-1}\| u \|_{k', s} } \, .
\end{equation}
 \item
 {\bf $m$-Kernel-Operator.}
We denote by $\Sigma\KM^{m,n}_{K,0,p}[\epsilon_0,N]$
the space of operators 
of the form 
\begin{equation}\label{exp kerop}
M(u;z)v =\sum_{q=p}^{N}M_{q}(u)v+M_{>N}(u;z)v 
\end{equation}
where   $M_{q} $  are homogeneous $m$-Kernel 
 operators  in $ \widetilde{\KM}^{m,n}_{q}$, $q=p,\ldots, N$  and $M_{>N}$ is a non--homogeneous $m$-Kernel-operator  in 
$\mathcal{M}^{m,n}_{K,0,N+1}[\epsilon_0]$. 
We denote 
by $ \Sigma_p^N \wt {\cM}^{m}_q $ the space
pluri-homogeneous $m$-Kernel operators of the form \eqref{exp kerop} with $ M_{>N} = 0 $.
\end{enumerate}
\end{definition}

\begin{rem}\label{remark kerfunct}
Given  a Kernel function $ \kappa\pare{u;x, z} $ in $ 
\Sigma \KF^{n}_{K,0,p}\bra{\epsilon_0, N} $ then $ \OpBW{\kappa\pare{u;x, z}} $ is $ 0$-
Kernel operator in $ \Sigma \KM^{0, n}_{K, 0, p}\bra{\epsilon_0, N} $.
\end{rem}

\begin{definition}[Kernel-smoothing operators]\label{defin kernel smoothing}
Given 
$ \vr > 0 $ we define the homogeneous and non-homogeneous Kernel-smoothing operators as
\begin{align*}
\wt \KR ^{-\vr,n}_p \defeq \wt \KM ^{-\vr,n}_p , 
&&
\KR^{-\vr,n}_{K, 0, p}\bra{\epsilon_0} \defeq \KM^{-\vr,n}_{K, 0, p}\bra{\epsilon_0}, 
&&
\Sigma \KR^{-\vr,n}_{K, 0, p}\bra{\epsilon_0, N}
\defeq \Sigma \KM^{-\vr,n}_{K, 0, p}\bra{\epsilon_0, N}. 
\end{align*}
\end{definition}

In view of \cite[Lemma 2.8]{BMM2022}, 
if $ M\pare{u, \ldots , u;z} $ is a homogeneous 
$ m$-Kernel operator in $  \wt\KM^{m,n}_p $ then $  M\pare{u, \ldots , u; z}  $
defines a non-homogeneous  $ m$-Kernel operator in $ \KM^{m,n}_{K, 0, p}\bra{\epsilon_0} $ for any $ \epsilon_0 > 0 $ and $ K\in \mathbb{N} $.\\

The classes of paraproducts associated to 
Kernel functions and $ m $-Kernel-operators are closed w.r.t.  
compositions as we list below, cf. \cite{BCGS2023}.

\begin{prop}[Composition of $ z $-dependent operators] \label{prop compo zdepop} Let $ m,n,m',n'\in\mathbb{R} $, and integers 
$ K, p, p', N\in\mathbb{N}$ with $ p, p'\leqslant N $.
\begin{enumerate}
\item \label{item:OpOp_ext_z} Let $ \kappa \pare{u;x,z} \in \Sigma \KF ^{n}_{K, 0, p}\bra{\epsilon_0, N} $ and $  \kappa' \pare{u;x,z} \in \Sigma \KF ^{n'}_{K, 0, p'}\bra{\epsilon_0, N}  $ 
be Kernel functions. Then 
\begin{equation*}
\OpBW{\kappa \pare{u;x,z}}\circ \OpBW{\kappa'\pare{u;x,z}} = \OpBW{\kappa \ \kappa' \pare{u;x,z}} + R\pare{u;z},
\end{equation*}
where $ R\pare{u;z} $ is a Kernel-smoothing operator in $ \Sigma\KR^{-\vr, n+n'}_{K, 0, p+p'}\bra{\epsilon_0, N} $  for any $ \vr \geqslant0 $; 

\item  \label{item:MM_ext_z} Let $ M \pare{u;  z} $ be a $ m$-Kernel operator in 
$ \Sigma \KM^{m, n}_{K, 0 ,p}\bra{\epsilon_0, N} $ and 
$ M' \pare{u;  z} $ be an $ m'$-operator belonging to 
$ \Sigma \KM^{m', n'}_{K, 0,p'}\bra{\epsilon_0, N} $. Then $ M\pare{u;  z} \circ M' \pare{u;  z} $ belongs to 
$ \Sigma \KM^{m+\max\pare{m',0}, n+n'}_{K, 0 , p+p'}\bra{\epsilon_0, N} $; 

\item  \label{item:OpR_ext_z} Let $ \kappa \pare{u;x,z} $ be a Kernel function in 
$ \Sigma\KF^n_{K, 0, p}\bra{\epsilon_0, N} $ and $ R\pare{u;z} $ 
be a Kernel smoothing operator in $ \Sigma\KR^{-\vr, n'}_{K, 0, p'}\bra{\epsilon_0, N} $ then $ \OpBW{\kappa\pare{u; x, z}}\circ R\pare{u;z} $ and $  R\pare{u;z} \circ \OpBW{\kappa\pare{u; x, z}} $   are a Kernel smoothing operator  in $ \Sigma \KR^{-\vr, n+n'}_{K, 0, p+p'}\bra{\epsilon_0, N} $; 

\item  \label{item:MM_int_z} 
Let $ M \pare{u;z} $ be an homogeneous $ m$-Kernel operator 
in $ \wt \KM^{m,n}_1 $, and $ M' \pare{u;z} $ in $ \Sigma \KM^{0,0}_{K,0,0}\bra{\epsilon_0, N} $ then  $ M\pare{M'\pare{u;z}u;z} \in \Sigma \KM^{m,n}_{K,0,1}\bra{\epsilon_0, N} $.
\end{enumerate}
\end{prop}


Finally  integrating \eqref{estim fourier kerop} and \eqref{estim dtkerop} in $ z $ we deduce the following lemma.

\begin{lemma}[Integrals of Kernel smoothing operators]\label{lemma actzsmoo}
 Let $ R\pare{u;z} $ be a Kernel smoothing  operator belonging to $ \Sigma \cR^{-\vr, n}_{K,0, p}\bra{\epsilon_0, N} $ 
 with $ n >  - 1 $.
  Then 
\begin{equation*}
\int_{\mathbb{T}} R\pare{u;z} g\pare{x-z} \dd z = R_1\pare{u}g , \qquad 
\int_{\mathbb{T}} R\pare{u;z}  \dd z = R_2\pare{u} , 
\end{equation*}
where $ R_1 \pare{u} $, $ R_2 \pare{u} $ are smoothing operators 
 in $ \Sigma \cR^{-\vr}_{K,0,p}\bra{\epsilon_0, N} $. 
\end{lemma}

The following proposition will be crucial  in Section \ref{sec:paralinearization}, and is proved in \cite{BCGS2023}.

\begin{prop}
\label{prop:reminders_integral_operator}
Let  $ n > -1 $ and $ \kappa \pare{u; x , z} $ be a Kernel-function  in 
$ \Sigma\KF^{n}_{K,0,p}\bra{\epsilon_0 , N} $. 
Let  us define the operator, for any $ g \in \dot H^s (\mathbb{T};\mathbb{R}) $, $ s \in \mathbb{R}  $,   
\begin{equation*}
\pare{\cT_\kappa g} 
\pare{x} \defeq \int_{\mathbb{T}} \OpBW{\kappa \pare{u; \bullet , z}} g\pare{x-z} \dd z   \, .
\end{equation*}
Then there exists
\begin{itemize}
\item   a symbol $ a \pare{u; x ,  \xi}  $ 
in $ \Sigma\Gamma^{-\pare{1+n}}_{K,0, p}\bra{\epsilon_0 , N} $ satisfying \eqref{real cond};

\item a pluri-homogeneous smoothing operator $R\pare{u}  $ 
in $ \Sigma_p^N \wt {\cR}^{-\vr}_q $ 
for any $ \vr > 0 $;
\end{itemize}
 such that $ \cT_\kappa g = \OpBW{ a \pare{u; x,  \xi}} g + R\pare{u} g $.
\end{prop}

\section{Paralinearization of the Kelvin-Helmholtz system}\label{sec:paralinearization}

\begin{notation}
    In the present section we use the following notation
    \begin{equation*}
        B^K_{s, \R}\pare{I;r}\defeq B_{C^K_\ast\pare{I;\dot H ^s\pare{\T;\R}}}\pare{0, r}\ .
    \end{equation*}
    We warn the reader that such notation is conflictive with the notation introduced at page \pageref{eq:CKastreali}, but we think that in the restricted context of the paralinearization procedure outlined here there is no risk of confusion and it helps to streamline the mathematical statements that we present.
\end{notation}

In the present section we paralinearize the system \cref{eq:KH3} in the $(\eta,\psi)$-variables. The result we obtain is the following one.

\begin{theorem}\label{prop:paralinearizationKH}
Let $ N\in \mathbb{N}$, $\kap \geqslant 0$, $ \upgamma\in \mathbb{R} $ and $ \varrho \geqslant 0 $, for any $ K\in \mathbb{N}$ there exists $ s_0 >0 $ and $ \epsilon_0 > 0 $ such that if $\eta, \psi \in B^K_{s_0, \mathbb{R}}\pare{I;\epsilon_0} $ is a solution of \cref{eq:KH3} then $ \pare{\eta, \psi} $ solves the paradifferential equation
\begin{multline}\label{eq:KH4}
\vect{\eta_t}{\psi_t} =\OpBW{{\bm Q}_{\kap, \upgamma}\pare{\eta,\psi; x, \xi} + {\bm B}_\upgamma \pare{\eta,\psi; x}\av{\xi} - \ii  V_\upgamma \pare{\eta, \psi;x}\Id_{\R^2}\ \xi + {\bm A}_{\bra{0}} \pare{\eta,\psi; x, \xi}   } \vect{\eta}{\psi}
 + {\bm R}\pare{\eta, \psi} \vect{\eta}{\psi} , 
\end{multline}
where
\begin{itemize}

\item  
The matrix of symbols ${\bm Q}_{\kap, \upgamma}$ satisfy \eqref{real cond} and is given by 
\begin{equation}\label{quone}
{\bm Q}_{\kap, \upgamma}\pare{\eta,\psi; x, \xi}
\triangleq
\bra{
\begin{array}{cc}
0 & -\frac{\av{\xi}}{2} \\
\kap\pare{1+\mathtt{f}\pare{\eta;x}} \pare{\av{\xi}^2-1} -\pare{ \frac{\upgamma^2}{2} + w_\upgamma\pare{\eta, \psi;x} }\av{\xi}  + \frac{\upgamma^2}{\pare{1+2\eta}} & 0
\end{array}
}\in\pare{  \Sigma \Gamma^2_{K,0,0}\bra{\epsilon_0, N} }^{2\times 2}, 
\end{equation}
with 
\begin{equation}\label{eq:wgamma}
    \begin{aligned}
    \mathtt{f}\pare{\eta;x}&\triangleq \pare{\frac{1+2\eta}{\pare{1+2\eta}^{2}+\eta_x^2}}^{\frac{3}{2}}-1\in \Sigma \cF^{\mathbb{R}}_{K,0,1}\bra{\epsilon_0, N} ,\\
    w_\upgamma\pare{\eta, \psi;x}&\triangleq \frac{1}{2}\pare{W_\upgamma^2\pare{\eta, \psi;x}-\upgamma^2}\in \Sigma \cF^{\mathbb{R}}_{K,0,1}\bra{\epsilon_0, N},\\
    W_\upgamma\pare{\eta, \psi;x}&\triangleq\pare{\psi_x+\upgamma}\frac{1+2\eta}{\pare{1+2\eta}^2+\eta_x^2}\in  \Sigma \cF^{\mathbb{R}}_{K,0,0}\bra{\epsilon_0, N}.
\end{aligned}
\end{equation}
 In particular, $ W_\upgamma\pare{0, 0;x}\equiv \upgamma$; 

 \item 
\begin{equation}
\label{eq:pzcBgamma}
{\bm B}_\upgamma\pare{\eta,\psi; x}
\triangleq
\frac{1}{2}
\bra{
\begin{array}{cc}
B_\upgamma\pare{\eta, \psi; x} & 0 \\
B_\upgamma^2\pare{\eta, \psi; x}    & -B_\upgamma\pare{\eta, \psi; x} 
\end{array}
}\in \pare{ \Sigma\cF^{\mathbb{R}}_{K,0,1}\bra{\epsilon_0, N} }^{2\times 2} , 
\end{equation}
where 
$$B_\upgamma  \pare{\eta, \psi;x}\defeq\pare{\psi_x+\upgamma}\frac{\sJ^0\pare{\eta;x}}{1+2\eta}\in  \Sigma \cF^{\mathbb{R}}_{K,0,1}\bra{\epsilon_0, N},\qquad\sJ^0\pare{\eta;x}\triangleq\frac{2\eta_x\pare{1+2\eta}}{\pare{1+2\eta}^2+\eta_x^2}\in \Sigma\cF^{\mathbb{R}}_{K, 0, 1}\bra{\epsilon_0, N}.$$
In particular  $  {\bm B}_{\upgamma}$ satisfy \eqref{real cond}; 

\item $ V_\upgamma\pare{\eta, \psi;x}\in \Sigma\cF^{\mathbb{R}}_{K, 0, 1}\bra{\epsilon_0, N} $ and is explicitly defined as
\begin{equation}\label{eq:Vgamma}
V_\upgamma\pare{\eta, \psi;x}\defeq\frac{1}{2}\mathpzc{D}_0\pare{\eta}\bra{\upgamma + \psi_x} -\frac{\upgamma}{2}\ , 
\end{equation}
where $ \mathpzc{D}_0\pare{\eta} $ is defined in  \eqref{eq:integral_operators_etaomega}; 

\item $  {\bm A}_{\bra{0}}\pare{\eta,\psi; x, \xi} \in\pare{  \Sigma \Gamma^0_{K,0,1}\bra{\epsilon_0, N} }^{2\times 2} $ satisfies \eqref{real cond} and is explicitly defined as
\begin{equation*}
{\bm A}_{\bra{0}} \pare{\eta,\psi; x, \xi}
\defeq
\bra{
\begin{array}{cc}
A_{\bra{0}}^\eta\pare{\eta,\psi; x, \xi} & A_{\bra{-1}}^\eta\pare{\eta,\psi; x, \xi} \\
A_{\bra{0}}^\psi \pare{\eta,\psi; x, \xi} &  A_{\bra{-1}}^\psi\pare{\eta,\psi; x, \xi}
\end{array}
}, 
\end{equation*}
with $  A_{\bra{m}}^\mathsf{u} \pare{\eta,\psi; x, \xi}\in \Sigma \Gamma^m _{K,0,1}\bra{\epsilon_0, N} $ for $ \mathsf{u}=\eta, \psi $ and $ m\in \mathbb{R} $;

\item $ {\bm R}\pare{\eta, \psi} \in \pare{ \Sigma \cR^{-\vr}_{K,0,1}\bra{\epsilon_0, N}  }^{2\times 2}$ and is real-valued. 

\end{itemize}
\end{theorem}

\begin{rem}\label{rem:unstable_terms_KH}
Notice that the quasilinear contribution $ -\pare{ \frac{\upgamma^2}{2} + w_\upgamma\pare{\eta, \psi;x} }\av{\xi} =\frac{W_\upgamma^2 \pare{\eta, \psi;x}}{2} \ \av{\xi} $ is not nil when $ \upgamma =0  $, as it is evident from \cref{eq:W,eq:W0}. This term, in particular, is an unstable contribution that is not present in the one-phase version of the present system, cf. \cite{BD2018,BMM2022,ABZ2014}. 
\end{rem}

\begin{notation}\label{notation:notation_paralinearization}
Along this section, for $ \eta \in  \BallR{K}{s_0} $ and  $ W\defeq\vect{w_1}{w_2}\in \pare{ \BallR{K}{s_0} }^2 $, we use the notation
\begin{enumerate}
\item For any $ x, z\in \mathbb{T} $ (cf. \eqref{z:R})
\begin{align}\label{eq:auxiliary}
r=r\pare{x} = r\pare{\eta;x}= \sqrt{1+2\eta\pare{x}} \in \Sigma \cF^{\mathbb{R}}_{K, 0, 1}\bra{\epsilon_0, N}, 
&&
\delta_z\eta \defeq \eta\pare{x}-\eta\pare{x-z}\in \widetilde{\KF}_1^{\ 1} \ ;
\end{align}


\item $ V\pare{W ;x} $ is a generic element in $ \Sigma\cF^{\mathbb{R}} _{K, 0, 1}\bra{\epsilon_0, N} $ (cf. \Cref{defin funct}) and $ V^n\pare{W;x, z} $ is a generic element in $ \Sigma\cF^n _{K, 0, 1}\bra{\epsilon_0, N}$, $ n > -1 $ (cf. \Cref{defin kerfunct});

\item
$ A_{\bra{m}} \pare{W;x, \xi} $ is a generic element in $ \Sigma\Gamma^m _{K, 0, 1}\bra{\epsilon_0, N} $ for $ m\in \mathbb{R} $ that satisfies \eqref{real cond} and $ {\bm A}_{\bra{m}} \pare{W;x, \xi} $ is a generic element in $ \pare{ \Sigma\Gamma^m _{K, 0, 1}\bra{\epsilon_0, N} }^{2\times 2} $ for $ m\in \mathbb{R} $ whose entries satisfy \eqref{real cond};

\item
$ R\pare{W} $ is a generic element in $ \Sigma\cR^{-\vr} _{K, 0, 1}\bra{\epsilon_0, N} $ which is real-valued and $ R\pare{W; z} $ is a generic element in $ \Sigma\cR^{-\vr, n} _{K, 0, 1}\bra{\epsilon_0, N}$, $ n > -1 $ which is real-valued. Similarly $ {\bm R}\pare{W} $ is a generic element in $ \pare{ \Sigma\cR^{-\vr} _{K, 0, 1}\bra{\epsilon_0, N} }^{2\times 2} $ which is real-valued and $ {\bm R}\pare{W; z} $ is a generic element in $ \Sigma\cR^{-\vr, n} _{K, 0, 1}\bra{\epsilon_0, N}$, $ n > -1 $ which is real-valued. 
\end{enumerate}
\end{notation}

\begin{rem}
Accordingly to the notation introduced in \Cref{notation:notation_paralinearization} we write
\begin{equation*}
{\bm R}\pare{W}W\cdot\vect{1}{1} = R_1\pare{W} w_1 + R_2 \pare{W}w_2, 
\end{equation*}
where $ R_1 $ and $ R_2 $ are elements of the space $ \Sigma\cR^{-\vr}_{K, 0, 1}\bra{\epsilon_0, N} $. The same holds when we write $ {\bm R}\pare{W;z}W\cdot\vect{1}{1} $. 
\end{rem}

Notice that from \Cref{eq:integral_operators_etaomega} we derive the relation
\begin{equation}\label{eq:HH0}
\mathpzc{H}\pare{\eta}\omega = \eta_x \ \mathpzc{D}_0\pare{\eta} \omega + \mathpzc{H}_0 \pare{\eta}\omega, 
\end{equation}
where
\begin{equation}\label{eq:H0}
\mathpzc{H}_0 \pare{\eta}\omega \triangleq
 \intT \frac{ \sqrt{1+2\eta\pare{x}}\sqrt{1+2\eta\pare{y}}\sin\pare{x-y} }{1 + \eta \pare{x} + \eta\pare{y} - \sqrt{1+2\eta\pare{x}}\sqrt{1+2\eta\pare{y}}\cos \pare{x-y} } \ \omega\pare{y}\dd y.
\end{equation}

\subsection{Paralinearization of $ \mathpzc{H}_0\pare{\eta} $}\label{para:H0}

The present section is dedicated to paralinearize the nonlocal operator $\mathpzc{H}_0(\eta)$ given in \eqref{eq:H0}.

\begin{proposition}\label{lem:H0_paralinearization}
Let $ N\in \mathbb{N}$ and $ \varrho \geqslant 0 $, for any $ K\in \mathbb{N} $ there exist $ s_0 > 0 $ and $ \epsilon_0 > 0 $ such that if 
$ \eta, g \in \BallR{K}{s_0}$, 
the nonlinear operator $ \mathpzc{H}_0\pare{\eta} $ in \eqref{eq:H0} admits the following paralinearization

\begin{enumerate}

\item \label{item:H0_paralinearization_1}
\begin{equation}\label{eq:H0_paralinearization}
\begin{aligned}
    \mathpzc{H}_0\pare{\eta}g&= \OpBW{-\ii  \pare{ 1+\sK^0\pare{\eta; x} }\sgn \xi + A_{\bra{-2}}\pare{\eta;x, \xi}} g
\\
&\quad+
\OpBW{\frac{g}{1+2\eta}\sK^{\prime 0} \pare{\eta;x} \av{\xi} +  A_{\bra{0}}\pare{\eta, g;x, \xi} } \eta
+
 {\bm R}\pare{\eta, g}\vect{\eta}{g}\cdot\vect{1}{1},
\end{aligned}
\end{equation}
where
\begin{equation}\label{eq:sK0}
    \begin{aligned}
\sK^0\pare{\eta;x}&\defeq-\frac{2\eta_x^2}{\pare{1+2\eta}^2+\eta_x^2}\in \cF^{\mathbb{R}}_{K, 0, 1}\bra{\epsilon_0, N},\\ 
 	\sK^{\prime 0}\pare{\eta;x}&\defeq -\frac{4\eta_x\pare{1+2\eta}^3}{\pare{\pare{1+2\eta}^2+\eta_x^2}^2}\in \cF^{\mathbb{R}}_{K, 0, 1}\bra{\epsilon_0, N}, \end{aligned}
\end{equation}
$A_{\bra{m}}\pare{\eta, g;x, \xi}\in \Sigma\Gamma^{m} _{K, 0, 1}\bra{\epsilon_0, N} $ satisfies \eqref{real cond} and ${\bm R}\pare{\eta, g} \in \pare{ \Sigma\cR^{-\vr} _{K, 0, 1}\bra{\epsilon_0, N} }^{2\times 2} $ is real-valued.

\item  \label{item:H0_paralinearization_2}
\begin{equation}
\label{eq:H0[1]_paralinearization}
\mathpzc{H}_0\pare{\eta}\bra{1}
=
\OpBW{\frac{\sK^{\prime 0}\pare{\eta;x}}{1+2\eta} \ \av{\xi} + A_{\bra{0}}\pare{\eta;x, \xi}} \eta + R\pare{\eta}\eta, 
\end{equation}
where
\begin{itemize}

\item $ \sK^{\prime 0}\pare{\eta ; x}\in \cF^{\mathbb{R}}_{K, 0, 1}\bra{\epsilon_0, N}  $ and is explicitly defined in \eqref{eq:sK0};


\item
$ A_{\bra{m}}\pare{\eta;x, \xi}\in \Sigma\Gamma^{m} _{K, 0, 1}\bra{\epsilon_0, N} $ satisfying \eqref{real cond};

\item
$ R\pare{\eta} \in \Sigma\cR^{-\vr} _{K, 0, 1}\bra{\epsilon_0, N} $ and is real-valued. 
\end{itemize}
\end{enumerate}

\end{proposition}

\begin{proof} 
\begin{proofpart}[Proof of \Cref{item:H0_paralinearization_1}]
With the notation introduced in \eqref{eq:auxiliary} we can rewrite \eqref{eq:H0} as
\begin{equation}\label{eq:H0_1}
\mathpzc{H}_0\pare{\eta}g =
\intT \sG_z \pare{\frac{\delta_z\eta}{r^2}} g \pare{x-z}\dd z
=
\mathpzc{H} g +  \intT \pare{ \sG_z \pare{\frac{\delta_z\eta}{r^2}} -  \sG_z\pare{0} } g \pare{x-z}\dd z , 
\end{equation}
where
\begin{equation*}
\sG_z\pare{\sX}\defeq \frac{\sqrt{1-2\sX}\sin z}{{1-\sX-\sqrt{1-2\sX}\cos z}}\cdot
\end{equation*}
Let us now define the desingularization of $ \sG_z $
\begin{equation}\label{eq:Kz}
\sK_z \pare{\sX}\defeq \sG_z\pare{\sX \ 2\sin\pare{z/2}} \ 2\tan\pare{z/2}
=
\frac{\sqrt{1-4\sX \ \sin\pare{z/2} }\pare{2\sin\pare{z/2}}^2}{{1-2\sX \ \sin\pare{z/2}-\sqrt{1-4\sX \ \sin\pare{z/2}}\cos z}} , 
\end{equation}
so that
\begin{equation}\label{eq:H0_2}
\pare{ \sG_z \pare{\frac{\delta_z\eta}{r^2}} -  \sG_z\pare{0} } g \pare{x-z}
=
\pare{ \sK_z \pare{\frac{\Delta_z\eta}{r^2}} -  2 } \frac{g \pare{x-z}}{2\tan\pare{z/2}}\cdot
\end{equation}
Notice that from \eqref{eq:Kz} we derive 
\begin{align}
\label{eq:K'z}
\sK'_z\pare{\sX}
=
&  
 -  \frac{\sX \pare{2\sin\pare{z/2}}^4}{{\left( 1 - 2 \, \sX \sin\left(z/2\right) - \sqrt{1- 4 \, \sX \sin\left(z/2\right)} \cos\left(z\right) \right)^2} \sqrt{1- 4 \, \sX \sin\left(z/2\right)}}
  \cdot
\end{align}
We need the following technical result whose proof is postponed at page \pageref{sec:characterization_Kernels}:

\begin{lemma}\label{lem:characterization_Kernels}
Let  $\sK_z (\sX) $, be as in \cref{eq:Kz}. Then 
\begin{align}\label{eq:Kernel_in_kernelfunction}
\sK_z \pare{\frac{\Delta_z \eta}{1+2\eta}} - 2
\in 
 \Sigma\KF ^{0}_{K, 0, 1}\bra{\epsilon_0, N} , 
 &&
 \sK_z' \pare{\frac{\Delta_z \eta}{1+2\eta}}  
\in 
 \Sigma\KF ^{0}_{K, 0, 1}\bra{\epsilon_0, N}, 
\end{align}
are Kernel functions, which admit the expansions 
\begin{equation}
  \label{eq:Taylor_expansion_kernels0}
  \begin{aligned}
  \sK_z \pare{\frac{\Delta_z \eta}{1+2\eta}} - 2
  = & \  \sK^0 \pare{\eta;x}
  +  \sK^1 \pare{\eta;x} \ 2\tan\pare{z/2}
  + V^2 \pare{\eta;x,z} 
 , \\
 \sK'_z \pare{\frac{\Delta_z \eta}{1+2\eta}} = & \ \sK^{\prime 0} \pare{\eta;x} 
 +  \sK^{\prime 1} \pare{\eta;x} \ \sin z
  + V^2\pare{\eta;x,z},  
  \end{aligned}
 \end{equation}
 where $\sK^0\pare{\eta;x}$ and $\sK^{\prime 0}\pare{\eta;x}$ are functions in $ \cF^{\mathbb{R}}_{K, 0, 1}\bra{\epsilon_0, N}$ having the expressions \eqref{eq:sK0}.Then,
 $\sK^1 \pare{\eta;x}$ and $ \sK^{\prime 1} \pare{\eta;x}$ are functions in  $ \Sigma \cF^{\mathbb{R}}_{K, 0, 1}\bra{\epsilon_0, N}$ while $ V^2\pare{\eta;x,z} $ are Kernel-functions in $  \Sigma\KF ^{2}_{K, 0, 1}\bra{\epsilon_0, N}$ as per \Cref{notation:notation_paralinearization}. 

\end{lemma}

\medskip

Bony paraproduct decomposition (cf. \Cref{lemma Bonyprod}) give us that
\begin{equation}\label{eq:H0_3}
    \begin{aligned}
 \pare{ \sK_z \pare{\frac{\Delta_z\eta}{r^2}} -  2} g\pare{x-z}
& =
 \OpBW{\sK_z \pare{\frac{\Delta_z\eta}{r^2}} -  2} g\pare{x-z}
 +
 \OpBW{g\pare{x-z}}  \bra{\sK_z \pare{\frac{\Delta_z\eta}{r^2}} -  2}
 \\
&\quad +
 R_1\pare{\sK_z \pare{\frac{\Delta_z\eta}{r^2}} -  2} g\pare{x-z}
 +R_2\pare{g\pare{x-z}}  \bra{\sK_z \pare{\frac{\Delta_z\eta}{r^2}} -  2}. 
\end{aligned}
\end{equation}
In view of \eqref{eq:Kernel_in_kernelfunction} we can apply \Cref{prop compo zdepop}, \cref{item:MM_int_z,item:MM_ext_z} and obtain that
\begin{align}\label{eq:H0_3.1}
R_1\pare{\sK_z \pare{\frac{\Delta_z\eta}{r^2}} -  2} g\pare{x-z}
=
R\pare{\eta ;z}g, 
&&
R_2\pare{g\pare{x-z}}  \bra{\sK_z \pare{\frac{\Delta_z\eta}{r^2}}-  2}
=
R\pare{\eta ;z}g, 
\end{align}
for suitable $ R\pare{\bullet;z}\in \Sigma \KR^{-\vr, 0}_{K, 0, 1}\bra{\epsilon_0, N} $. 
We use now Bony paralinearization formula of Lemma \ref{lemma paralin} and Bony paraproduct decomposition and obtain that
\begin{equation}\label{eq:H0_4}
    \begin{aligned}
&\sK_z \pare{\frac{\Delta_z\eta}{r^2}}-2\\
&=
\OpBW{\sK_z'  \pare{\frac{\Delta_z\eta}{r^2}}}
\bra{
\OpBW{r^{-2}}\Delta_z\eta + \OpBW{\Delta_z \eta}\bra{r^{-2}-1}+R_1\pare{r^{-2}-1}\Delta_z\eta + R_2 \pare{\Delta_z \eta}\bra{r^{-2}-1}
}\\
&\quad+R\pare{\frac{\Delta_z\eta}{r^2}}\frac{\Delta_z\eta}{r^2}\cdot
\end{aligned}
\end{equation}
Then, we apply again \Cref{lemma paralin},  \Cref{prop compo zdepop}, \cref{item:MM_int_z,item:MM_ext_z} in order to get that
\begin{align}\label{eq:H0_5}
\OpBW{\sK_z'  \pare{\frac{\Delta_z\eta}{r^2}}}
\pare{ 
R_1\pare{r^{-2}-1}\Delta_z\eta + R_2 \pare{\Delta_z \eta}\bra{r^{-2}-1} }
=
R\pare{\eta; z}\eta,\\
R\pare{\frac{\Delta_z\eta}{r^2}}\frac{\Delta_z\eta}{r^2}=R\pare{\eta; z}\eta.
\end{align}
Besides, combining \Cref{prop compo zdepop}, \cref{item:OpOp_ext_z} and \cref{lemma paralin} we obtain that
\begin{equation}\label{eq:H0_6}
\begin{aligned}
    &\OpBW{\sK_z'  \pare{\frac{\Delta_z\eta}{r^2}}}\pare{ 
\OpBW{r^{-2}}\Delta_z\eta + \OpBW{\Delta_z \eta}\bra{r^{-2}-1} }
\\
&=
\OpBW{r^{-2}\sK_z'  \pare{\frac{\Delta_z\eta}{r^2}}} \Delta_z\eta
+
\OpBW{\tilde V^0 \pare{\eta;x,z}} \eta
+
R\pare{\eta, g ;z}\eta, 
\end{aligned}
\end{equation}
where
$$\tilde V^0(\eta;x,z)\triangleq r^{-2}\big(K_z'(X)X\big)|_{X=\frac{\Delta_z\eta}{r^2} }.$$
We plug \cref{eq:H0_6,eq:H0_5} in \cref{eq:H0_4} and the resulting equation and \cref{eq:H0_3.1} in \cref{eq:H0_3} and obtain, after using
$$\int_{\mathbb{T}}\frac{1}{\tan(z/2)}\dd z=0$$
and applying \Cref{prop compo zdepop}, \cref{item:OpOp_ext_z}, that
\begin{equation}\label{eq:H0_7}
\begin{aligned}
\pare{ \sK_z \pare{\frac{\Delta_z\eta}{r^2}} -  2} g\pare{x-z}
 &=
 \OpBW{\sK_z \pare{\frac{\Delta_z\eta}{r^2}} -  2} g\pare{x-z}
 +
 \OpBW{\frac{g\pare{x-z}}{r^2}\sK_z'  \pare{\frac{\Delta_z\eta}{r^2}}} \Delta_z\eta
 \\
 &\quad+
\OpBW{V^0\pare{\eta, g ;x,z}} \eta
+
{\bm R}\pare{\eta, g;z}\vect{\eta}{g}\cdot\vect{1}{1}, 
\end{aligned} 
\end{equation}
where
$$V^0(\eta,g;x,z)\triangleq\tilde V^0(\eta;x,z)g(x-z).$$
We Taylor expand in $z$
$$V^0(\eta,g;x,z)=\overline{V}_0(\eta,g;x)+\overline{V}_1(\eta,g;x,z),\qquad\overline{V}_1\in \Sigma\KF ^{1}_{K, 0, 1}\bra{\epsilon_0, N}.$$
We thus plug \eqref{eq:H0_7} in \eqref{eq:H0_2} and insert the resulting equation in \eqref{eq:H0_1} and, after application of \Cref{rem:intkerfunct,lemma actzsmoo} we obtain that
\begin{equation}\label{eq:H0_8}
\begin{aligned}
    \mathpzc{H}_0\pare{\eta}g &= \mathpzc{H}g
+\intT \OpBW{\sK_z \pare{\frac{\Delta_z\eta}{r^2}} -  2} \frac{g\pare{x-z}}{2\tan\pare{z/2}}\dd z
\\
&\quad+\intT  \OpBW{\frac{g\pare{x-z}}{r^2}\sK_z'  \pare{\frac{\Delta_z\eta}{r^2}}} \frac{\delta_z\eta}{4\sin\pare{z/2} \ \tan\pare{z/2}}\dd z
\\
&\quad+
\OpBW{V\pare{\eta, g;x}} \eta
+
{\bm R}\pare{\eta, g;z}\vect{\eta}{g}\cdot\vect{1}{1},
\end{aligned}
\end{equation}
where
$$V(\eta,g;x)\triangleq\int_{\mathbb{T}}\frac{\overline{V}_1(\eta,g;x,z)}{2\tan(z/2)}\dd z.$$
We use now the identity $ \frac{1}{4\sin\pare{z/2} \ \tan\pare{z/2}} =\frac{1}{4\sin^2\pare{z/2}} - \frac{1}{8\cos^2 \pare{z/4}} $ and \Cref{prop:reminders_integral_operator,rem:intkerfunct} in order to transform
\begin{multline}\label{eq:H0_9}
\intT  \OpBW{\frac{g\pare{x-z}}{r^2}\sK_z'  \pare{\frac{\Delta_z\eta}{r^2}}} \frac{\delta_z\eta}{4\sin\pare{z/2} \ \tan\pare{z/2}}\dd z
\\
=
\intT  \OpBW{\frac{g\pare{x-z}}{r^2}\sK_z'  \pare{\frac{\Delta_z\eta}{r^2}}} \frac{\delta_z\eta}{4\sin^2\pare{z/2} }\dd z
+
\OpBW{V\pare{\eta, g;x} + A_{\bra{-2}}\pare{\eta, g;x, \xi}} \eta + R\pare{\eta, g}\eta , 
\end{multline}
so that plugging \eqref{eq:H0_9} in \eqref{eq:H0_8} we obtain that
\begin{equation}\label{eq:H0_10}
\begin{aligned}
    \mathpzc{H}_0\pare{\eta}g &= \mathpzc{H}g
+\intT \OpBW{\sK_z \pare{\frac{\Delta_z\eta}{r^2}} -  2 } \frac{g\pare{x-z}}{2\tan\pare{z/2}}\dd z
\\
&\quad+\intT  \OpBW{\frac{g\pare{x-z}}{r^2}\sK_z'  \pare{\frac{\Delta_z\eta}{r^2}}} \frac{\delta_z\eta}{4\sin^2\pare{z/2}}\dd z
\\
&\quad+
\OpBW{V\pare{\eta, g;x} + A_{\bra{-2}}\pare{\eta, g;x, \xi}} \eta
+
{\bm R}\pare{\eta, g;z}\vect{\eta}{g}\cdot\vect{1}{1}.
\end{aligned}
\end{equation}
Now, we can use the Taylor-like expansions \eqref{eq:Taylor_expansion_kernels0}, \Cref{prop:reminders_integral_operator}  and the fact that $ g $ has zero average and obtain that
\begin{equation}\label{eq:H0_11}
\intT \OpBW{\sK_z \pare{\frac{\Delta_z\eta}{r^2}} -  2} \frac{g\pare{x-z}}{2\tan\pare{z/2}}\dd z= 
\OpBW{\sK^0\pare{\eta; x}} \mathpzc{H}g + \OpBW{A_{\bra{-2}}\pare{\eta;x, \xi}} g + R\pare{\eta, g}g.
\end{equation}
Next, denoting
\begin{equation*}
\tilde{\sK}\pare{\eta, g;x, z}\defeq  \frac{g\pare{x-z}}{1+2\eta}\sK_z'  \pare{\frac{\Delta_z\eta}{1+2\eta}}
-
 \frac{g}{1+2\eta}\sK^{\prime 0}  \pare{\eta;x}
 \in 
 \Sigma\KF^1_{K, 0, 1}\bra{\epsilon_0, N}, 
\end{equation*}
we Taylor-expand $ z\mapsto \tilde{\sK}\pare{\eta, g;x, z} $ obtaining
\begin{equation*}
\tilde{\sK}\pare{\eta, g;x, z}
=
\tilde{\sK}^1\pare{\eta, g;x}\sin z + V^2\pare{\eta, g;x,z}, 
\end{equation*}
so that applying \cref{rem:intkerfunct,prop:reminders_integral_operator}  we obtain that
\begin{equation}\label{eq:H0_12}
\begin{aligned}
    &\intT  \OpBW{\frac{g\pare{x-z}}{r^2}\sK_z'  \pare{\frac{\Delta_z\eta}{r^2}}} \frac{\delta_z\eta}{4\sin^2\pare{z/2}}\dd z
\\
&=
\OpBW{\frac{g}{1+2\eta}\sK^{\prime 0} \pare{\eta;x} } \av{D}\eta
+
\OpBW{\tilde{\sK}^1\pare{\eta, g;x}}\mathpzc{H} \eta + \OpBW{V\pare{\eta, g; x} + A_{\bra{-2}}\pare{\eta, g;x, \xi}}\eta + R\pare{\eta, g}\eta.
\end{aligned} 
\end{equation}
We plug \cref{eq:H0_12,eq:H0_11} in \cref{eq:H0_10} and use \Cref{prop compBW} and obtain the desired quasi-linear expansion

\begin{equation}\label{eq:H0_13}
\begin{aligned}
    \mathpzc{H}_0\pare{\eta}g &= \OpBW{1+\sK^0\pare{\eta; x}} \mathpzc{H}g
+
\OpBW{\frac{g}{1+2\eta}\sK^{\prime 0} \pare{\eta;x} } \av{D}\eta
\\
&\quad+
\OpBW{ A_{\bra{0}} \pare{\eta, g;x, \xi}} \eta + \OpBW{ A_{\bra{-2}}\pare{\eta;x, \xi} } g
+
{\bm R}\pare{\eta, g;z}\vect{\eta}{g}\cdot\vect{1}{1} ,
\end{aligned}
\end{equation}
thus \eqref{eq:H0_paralinearization} can be derive applying \Cref{prop compBW} to \eqref{eq:H0_13} combined with $ \mathpzc{H}=\OpBW{-\ii \sgn{\xi}} $. 
\end{proofpart}

\begin{proofpart}[Proof of  \Cref{item:H0_paralinearization_2}]
From \cref{eq:H0_1,eq:H0_2} we have that
\begin{equation}\label{eq:H0gamma_-1}
\mathpzc{H}_0\pare{\eta}\bra{1} = \intT  \pare{ \sK_z \pare{\frac{\Delta_z\eta}{r^2}} -  2 } \frac{\dd z}{2\tan\pare{z/2}}\cdot 
\end{equation}
Thus, we apply \cref{lemma paralin} and obtain that
\begin{equation}\label{eq:H0gamma_0}
\sK_z \pare{\frac{\Delta_z\eta}{r^2}} -  2
=
\OpBW{\sK_z' \pare{\frac{\Delta_z\eta}{r^2}}}\bra{\frac{\Delta_z\eta}{r^2}}
+ R \pare{\frac{\Delta_z\eta}{r^2}}\bra{\frac{\Delta_z\eta}{r^2}}. 
\end{equation}
Notice that from \Cref{prop compo zdepop}, \cref{item:MM_int_z} we have that $ R \pare{\frac{\Delta_z\eta}{r^2}} \in \Sigma\KR^{-\vr, 0}_{K, 0, 1}\bra{\epsilon_0, N} $ and by Taylor expansion we have that
\begin{equation*}
\tilde{R}\pare{\eta;x, z} \defeq 
R \pare{\frac{\Delta_z\eta}{r^2}} -R \pare{\frac{\eta_x}{r^2}}  \in \Sigma\KR^{-\vr, 1 }_{K, 0, 1}\bra{\epsilon_0, N} , 
\end{equation*}
and, since $  \frac{\Delta_z}{r^2} \in \Sigma\KM^{1, 0}_{K, 0, 1}\bra{\epsilon_0, N} $ applying \Cref{prop compo zdepop}, \cref{item:MM_ext_z} we obtain that
\begin{equation}\label{eq:H0gamma_1}
\tilde{R}\pare{\eta;x, z} \circ \frac{\Delta_z}{r^2}\in \Sigma\KR^{-\vr+1, 1 }_{K, 0, 1}\bra{\epsilon_0, N} , 
\end{equation}
so that, thanks to \eqref{eq:H0gamma_1}, we have that
$$\intT  R \pare{\frac{\Delta_z\eta}{r^2}}\bra{\frac{\Delta_z\eta}{r^2}} \frac{\dd z}{2\tan\pare{z/2}}
=
R\pare{\frac{\eta_x}{r^2}}\bra{\frac{1}{r^2} \  \intT \frac{\delta_z\eta}{4\sin\pare{z/2}\tan\pare{z/2}}\dd z}
+
\intT R\pare{\eta;x, z} \dd z \ \eta.$$
Next we use the fact that 
$$\intT \frac{\delta_z\eta}{4\sin\pare{z/2}\tan\pare{z/2}}\dd z =m_1\pare{D}\eta,\qquad \textnormal{with}\qquad m_1\pare{\xi}\in \wt{\Gamma}^1_0,$$ the fact that $ r^{-2} m_1\pare{D} \in \Sigma\KM^{1, 0}_{K,0,1}\bra{\epsilon_0, N} $, \Cref{prop compo zdepop}, \cref{item:MM_ext_z} and \cref{lemma actzsmoo} to obtain that
\begin{equation}\label{eq:H0gamma_2}
\intT  R \pare{\frac{\Delta_z\eta}{r^2}}\bra{\frac{\Delta_z\eta}{r^2}} \frac{\dd z}{2\tan\pare{z/2}}
=
R\pare{\eta}\eta. 
\end{equation}
Next, computations similar to the ones performed in \cref{eq:H0_4,eq:H0_6} allow us to deduce that
\begin{equation}\label{eq:H0gamma_3}
\begin{aligned}
    \intT  \OpBW{\sK_z' \pare{\frac{\Delta_z\eta}{r^2}}}\bra{\frac{\Delta_z\eta}{r^2}} \frac{\dd z}{2\tan\pare{z/2}}
&=
\intT  \OpBW{r^{-2} \sK_z' \pare{\frac{\Delta_z\eta}{r^2}}}\bra{\frac{\delta_z\eta}{4\sin\pare{z/2}\tan\pare{z/2}}} \dd z
\\
&\quad+\OpBW{V\pare{\eta;x}}\eta + R\pare{\eta}\eta . 
\end{aligned}
\end{equation}
Taylor-expanding as in \eqref{eq:H0_12} and using \Cref{prop compBW} we obtain that
\begin{equation}\label{eq:H0gamma_4}
\begin{aligned}
    &\intT  \OpBW{r^{-2} \sK_z' \pare{\frac{\Delta_z\eta}{r^2}}}\bra{\frac{\delta_z\eta}{4\sin\pare{z/2}\tan\pare{z/2}}} \dd z
\\
&=
\OpBW{r^{-2}\sK^{\prime 0}\pare{\eta;x}}\av{D}\eta + \OpBW{A_{\bra{0}}\pare{\eta;x,\xi}}\eta + R\pare{\eta}\eta
\\
&= 
\OpBW{r^{-2}\sK^{\prime 0}\pare{\eta;x}\av{\xi} + A_{\bra{0}}\pare{\eta;x,\xi}}\eta + R\pare{\eta}\eta.
\end{aligned}
\end{equation}
Combining \Cref{eq:H0gamma_0,eq:H0gamma_2,eq:H0gamma_3,eq:H0gamma_4} 
we obtain \eqref{eq:H0[1]_paralinearization}. 

\end{proofpart}

\end{proof}

\paragraph{Proof of \Cref{lem:characterization_Kernels}.}\label{sec:characterization_Kernels}

Let us prove at first \eqref{eq:Kernel_in_kernelfunction}. We prove the result for $ \sK_z $ only being the procedure for $ \sK_z' $ the same. Notice that from \cref{eq:Kz} we immediately have that the map $ \pare{z, \sX}\mapsto \sK_z\pare{\sX} $ is analytic in $ \mathbb{T}\times \pare{-\frac{1}{4}, \frac{1}{4}} $.
Let us now define
\begin{align*}
\overline{\sK}_z \pare{\mathsf{x}, \mathsf{y}} \defeq \sK_z\pare{\frac{\mathsf{y}}{1+2\mathsf{x}}}, 
\end{align*}
which is again analytic for $\av{ \mathsf{x}}, \av{\mathsf{y}} < \varepsilon_1 $ small. By analyticity we have that
\begin{align*}
\overline{\sK}_z \pare{\mathsf{x}, \mathsf{y}} = \sum _{p_1, p_2 =0}^\infty \underbrace{\frac{\partial_\mathsf{x}^{p_2}\partial_\mathsf{y}^{p_1} \overline{\sK}_z \pare{0,0}}{p_1!p_2!}}_{\eqdef \mathsf{k}_{p_1, p_2}\pare{z}} \mathsf{x}^{p_2} \mathsf{y}^{p_1}, 
\end{align*}
with
\begin{align*}
\av{\partial_\mathsf{x}^{p_2}\partial_\mathsf{y}^{p_1} \overline{\sK}_z \pare{0,0}}\leqslant C p_1!p_2! \  \varepsilon_1 ^{-\pare{p_1 + p_2}}. 
\end{align*}
From \cref{eq:Kz} it is immediate that
\begin{equation*}
\sK_z\pare{0}= 2, 
\end{equation*}
so that
\begin{align*}
\sK_z\pare{\frac{\Delta_z \eta }{1+2\eta}} = \overline{\sK}_z \pare{\eta, \Delta_z \eta } = & \  2+ \sum_{p\geqslant1} 
   \underbrace{\sum_{\substack{p_1 \geqslant1 \\ p_1 + p_2 = p }} \mathsf{k}_{ p_1, p_2}\pare{z} \ \eta^{p_2} \pare{\Delta_z \eta}^{p_1} }_{ {=: 
   \wt \sK^{p} \pare{f;x, z} } }\\
  = & \   2+  \sum_{p = 1}^N \wt \sK^{p} \pare{\eta;x, z} + \underbrace{\sum_{p > N} \wt \sK^{p} \pare{\eta;x, z}  }_{ =:  \sK^{> N} \pare{\eta;x, z}} \, .
\end{align*}
We claim that  for any $  p\in\mathbb{N} $ and  $  \ell =0,\ldots , 7 $,
\begin{align}\label{eq:claim_Khomog}
 \partial_z^\ell \wt \sK^{p} \pare{\eta; x, z}  \ & \in \wt \KF^0_p ,
 \\
 \label{eq:tail_homogeneous_deco_kernels}
   \partial_z^\ell  \sK^{ > N} \pare{\eta;x, z}   & \in \KF^0_{K, 0, N +1}\bra{\epsilon_0} . 
\end{align}
The proof of \cref{eq:claim_Khomog,eq:tail_homogeneous_deco_kernels} is the same as the one performed in order to prove   \cite[Eqs. (4.34) and (4.35)]{BCGS2023} and is thus omitted. The proof for \eqref{eq:Kernel_in_kernelfunction} for $ \sK'_z\pare{\frac{\Delta_z\eta}{1+2\eta}} $ is the same as the one performed for  $ \sK_z\pare{\frac{\Delta_z\eta}{1+2\eta}} $ with the sole difference that $  \sK'_z\pare{0}=0  $. \\
The proof of \eqref{eq:Taylor_expansion_kernels0} follows the lines of the proof of   \cite[Eq. (4.23)]{BCGS2023}, and is generic enough to be applied in the present case. It remains to prove \cref{eq:sK0}, we can  Taylor-expand in $ z=0 $ and obtain that
\begin{equation*}
\begin{aligned}
\sK_z\pare{\sX} = & \ \sK_0\pare{\sX} + \left. \partial_z \sK_z\pare{\sX}\right|_{z=0} z + R_1 \pare{\sK;\sX}\pare{z} ,
\\
\sK'_z\pare{\sX} = & \ \sK'_0\pare{\sX} + \left. \partial_z \sK'_z\pare{\sX}\right|_{z=0} z + R_1 \pare{\sK' ;\sX}\pare{z} .
\end{aligned}
\end{equation*}
Explicit computations show that
\begin{equation*}
\begin{aligned}
\sK_0\pare{\sX} \defeq & \  2 - \frac{2\sX^2}{1+\sX^2}, 
&
 \left. \partial_z \sK_z\pare{\sX}\right|_{z=0} \defeq & -\frac{4\sX^3}{\pare{1+\sX^2}^2},
 \\
 \sK_0'\pare{\sX} \defeq & \  -\frac{4\sX}{\pare{1+\sX^2}^2},
 &
 \left. \partial_z \sK'_z\pare{\sX}\right|_{z=0} \defeq & \ 
 -\frac{4\sX^2\pare{3-\sX^2}}{\pare{1+\sX^2}^3}\cdot
\end{aligned}
\end{equation*}
Hence setting $ \sX \defeq\frac{\Delta_z\eta}{1+2\eta} $ and Taylor expanding we obtain that
\begin{equation*}
\begin{aligned}
\sK_0 \pare{\frac{\Delta_z\eta}{1+2\eta}} = & \ 2 \underbrace{- 2\frac{\eta_x^2}{\pare{1+2\eta}^2 + \eta_x^2}}_{\eqdef \sK^0\pare{\eta; x}} + 4\frac{\pare{1+2\eta}^2\eta_x\eta_{xx}}{\pare{\pare{1+2\eta}^2 + \eta_x^2}^2} z + V^2\pare{\eta;x,z}
,
\\
\sK'_0 \pare{\frac{\Delta_z\eta}{1+2\eta}}
 = & \
 \underbrace{-\frac{4\pare{1+2\eta}\eta_x}{\pare{\pare{1+2\eta}^2 + \eta_x^2}^2}}_{\eqdef \sK^{\prime 0}\pare{\eta;x}}
 +
 2\frac{\pare{\pare{1+2\eta}^2 -3\eta_x^2}\pare{1+2\eta}^3\eta_{xx}}{\pare{\pare{1+2\eta}^2 + \eta_x^2}^3}
 z + V^2\pare{\eta;x,z}
,
\\
 \left. \partial_z \sK_z\pare{ \frac{\Delta_x \eta}{1+2\eta} }\right|_{z=0} \defeq & -\frac{4\pare{1+2\eta}\eta_x^3}{\pare{\pare{1+2\eta}^2+\eta_x^2}^2}
 +
 R_0 \pare{\left. \partial_z \sK_z\pare{ \frac{\Delta_x \eta}{1+2\eta} }\right|_{z=0} ;x}\pare{z}
 , 
 \\
 \left. \partial_z \sK'_z\pare{ \frac{\Delta_x \eta}{1+2\eta} }\right|_{z=0} \defeq & \ 
 -\frac{4\pare{1+2\eta}^2\eta_x^2\pare{3\pare{1+2\eta}^2-\eta_x^2}}{\pare{\pare{1+2\eta}^2+\eta_x^2}^3}+
 R_0 \pare{\left. \partial_z \sK'_z\pare{ \frac{\Delta_x \eta}{1+2\eta} }\right|_{z=0} ;x}\pare{z} 
 , 
\end{aligned}
\end{equation*}
so that defining
\begin{align*}
\sK^1\pare{\eta;x}\defeq & \ \frac{4\pare{1+2\eta}^2\eta_x\eta_{xx}}{\pare{\pare{1+2\eta}^2 + \eta_x^2}^2} -\frac{4\pare{1+2\eta}\eta_x^3}{\pare{\pare{1+2\eta}^2+\eta_x^2}^2}=\frac{4\pare{1+2\eta}\eta_x\pare{\pare{1+2\eta}\eta_{xx}-\eta_x^2}}{\pare{\pare{1+2\eta}^2+\eta_x^2}^2},
\\
\sK^{\prime 1}\pare{\eta;x}\defeq & \ \frac{2\pare{\pare{1+2\eta}^2 -3\eta_x^2}\pare{1+2\eta}^3\eta_{xx}-4\pare{1+2\eta}^2\eta_x^2\pare{3\pare{1+2\eta}^2-\eta_x^2}}{\pare{\pare{1+2\eta}^2 + \eta_x^2}^3},
\end{align*}
which are analytic applications w.r.t. $\eta$ small in $W^{2, \infty}$, proving \eqref{eq:Taylor_expansion_kernels0} and \eqref{eq:sK0}.  \qed

\subsection{Paralinearization of $ \mathpzc{D}_0\pare{\eta} $ }\label{para:D0}

Here we perform the paralinearization of the nonlinear operator $\mathpzc{D}_0(\eta)$ introduced in \eqref{eq:integral_operators_etaomega}.
\begin{proposition}\label{lem:D0_paralinearization}
Let $ N\in \mathbb{N}$ and $ \varrho \geqslant 0 $, for any $ K\in \mathbb{N}$ there exists $ s_0 > 0 $ and $ \epsilon_0 > 0 $ such that if 
$ \eta, g \in \BallR{K}{s_0}$  
and  $ \mathpzc{D}_0\pare{\eta} $ be as in \eqref{eq:integral_operators_etaomega}, we have that

\begin{enumerate}

\item  \label{item:D0_paralinearization_1}
\begin{equation}\label{eq:D0_paralinearization}
\begin{aligned}
    \mathpzc{D}_0\pare{\eta}g &=  \OpBW{ -\ii \ \frac{\sJ^0\pare{\eta;x}}{r^2} \sgn{\xi} + A_{\bra{-2}}\pare{\eta;x, \xi}}g 
\\
&\quad+
\OpBW{\frac{g}{r^4}\sJ^{\prime 0}\pare{\eta; x} \av{\xi}  + A_{\bra{0}}\pare{\eta, g;x, \xi}} \eta
+ {\bm R}\pare{\eta, g}\vect{\eta}{g}\cdot\vect{1}{1} ,
\end{aligned}
\end{equation}
where
\begin{align}
 \sJ^0\pare{\eta;x}&\defeq  \frac{2\eta_x\pare{1+2\eta}}{\pare{1+2\eta}^2+\eta_x^2}\in \cF^{\mathbb{R}}_{K, 0, 1}\bra{\epsilon_0, N} ,
   \label{eq:sJ0}\\
  \sJ^{\prime 0}\pare{\eta;x}&\defeq\frac{2\pare{1+2\eta}^2\pare{\pare{1+2\eta}^2-\eta_x^2}}{\pare{\pare{1+2\eta}^2+\eta_x^2}^2}\in \cF^{\mathbb{R}}_{K, 0, 0}\bra{\epsilon_0, N},\label{eq:sJprime0}
 \end{align}
$ A_{\bra{m}}\pare{\eta, g;x, \xi}\in \Sigma\Gamma^{m} _{K, 0, 1}\bra{\epsilon_0, N} $ and
$ {\bm R}\pare{\eta, g} \in \pare{ \Sigma\cR^{-\vr} _{K, 0, 1}\bra{\epsilon_0, N} }^{2\times 2} $.

\item \label{item:D0_paralinearization_2}
\begin{equation}\label{eq:D0_paralinearization_2}
\mathpzc{D}_0\pare{\eta}\bra{1} = 1 + \OpBW{ -\frac{2}{\pare{1+2\eta}^2} +  \frac{\sJ^{\prime 0}\pare{\eta;x}}{r^4} \ \av{\xi} + A_{\bra{0}}\pare{\eta;x,\xi}}\eta + R\pare{\eta}\eta \ ,
\end{equation}
where
\begin{itemize}

\item $ \sJ^{\prime 0}\pare{\eta ; x} $  is explicitly defined in \eqref{eq:sJprime0}, and is such that $ \sJ^{\prime 0}\pare{0;x} \equiv 2 $;


\item
$ A_{\bra{0}}\pare{\eta ;x, \xi}\in \Sigma\Gamma^{0} _{K, 0, 1}\bra{\epsilon_0, N} $;

\item
$ R\pare{\eta} \in \Sigma\cR^{-\vr} _{K, 0, 1}\bra{\epsilon_0, N} $. 
\end{itemize}

\end{enumerate}

\end{proposition}

\begin{proof}
\begin{proofpart}[Proof of \Cref{lem:D0_paralinearization}, \cref{item:D0_paralinearization_1}]
Let us recall that $ \mathpzc{D}_0\pare{\eta} $ is defined in \eqref{eq:integral_operators_etaomega}. 
We proceed analogously as in the previous section. We have, using the auxiliary functions defined in \eqref{eq:auxiliary} and the fact that $ g $ is of zero average,  that
\begin{equation}\label{eq:D0_1}
\mathpzc{D}_0\pare{\eta}g =\intT \frac{1}{r^2}\pare{ \sH_z \pare{\frac{\delta_z\eta}{r^2}} - \mathsf{H}_z\pare{0} } \frac{g\pare{x-z}}{2\sin\pare{z/2}}\dd z\,, 
\end{equation}
where
\begin{equation*}
\sH_z\pare{\sX}\defeq  \frac{1-\sqrt{1-2\sX}\cos z}{1-\sX - \sqrt{1-2\sX}\cos z}  \ 2\sin\pare{z/2}. 
\end{equation*}
We define
\begin{equation}\label{eq:Jz}
\sJ_z\pare{\sX} \defeq \sH_z\pare{\sX \ 2\sin\pare{z/2}}
. 
\end{equation}

We have the following technical result:

\begin{lemma}\label{lem:characterization_Kernels_J}
Let  $\sJ_z (\sX) $, be as in \cref{eq:Jz}. Then 
\begin{align}\label{eq:Kernel_in_kernelfunction_J}
\sJ_z \pare{\frac{\Delta_z \eta}{1+2\eta}}_{\geqslant1}
\defeq
\sJ_z \pare{\frac{\Delta_z \eta}{1+2\eta}} 
-
2\sin\pare{z/2}
\in 
 \Sigma\KF ^{0}_{K, 0, 1}\bra{\epsilon_0, N} , 
 &&
 \sJ_z' \pare{\frac{\Delta_z \eta}{1+2\eta}}
\in 
 \Sigma\KF ^{0}_{K, 0, 0}\bra{\epsilon_0, N}, 
\end{align}
are Kernel functions, which admit the expansion 
\begin{equation}
  \label{eq:Taylor_expansion_kernels0_J}
  \begin{aligned}
  \sJ_z \pare{\frac{\Delta_z \eta}{1+2\eta}} 
  -
2\sin\pare{z/2}
  = & \  \sJ^0 \pare{\eta;x}
  +\sJ^1\pare{\eta;x}\ 2\sin \pare{ z/2 }
  + V^2 \pare{\eta;x,z}, \\
 \sJ'_z \pare{\frac{\Delta_z \eta}{1+2\eta}} = & \ \sJ^{\prime 0} \pare{\eta;x} 
 +\sJ^{\prime 1 }\pare{\eta;x}\sin z
 + V^2
   \pare{\eta;x,z},  
  \end{aligned}
 \end{equation}
 where $\sJ^{0}\pare{\eta;x},\sJ^{\prime 0}\pare{\eta;x}$ are defined in \eqref{eq:sJ0}-\eqref{eq:sJprime0} and $ \sJ^1\pare{\eta;x}, \sJ^{\prime 1}\pare{\eta;x} \in \Sigma\cF^{\mathbb{R}}_{K, 0, 1}\bra{\epsilon_0, N}.$
\end{lemma}

The proof of \Cref{lem:characterization_Kernels_J} follows exactly the same lines of the proof of \Cref{lem:characterization_Kernels} (cf. page \pageref{sec:characterization_Kernels}) and is thus omitted for the sake of brevity. \\

We apply \cref{lemma Bonyprod} and obtain that
\begin{equation}\label{eq:D0_2}
\begin{aligned}
    \frac{1}{r^2}\sJ_z \pare{\frac{\Delta_z\eta}{r^2}}_{\geqslant1} g\pare{x-z}
&=
\OpBW{\frac{1}{r^2}\sJ_z \pare{\frac{\Delta_z\eta}{r^2}}}_{\geqslant1} g\pare{x-z}
+
\OpBW{g\pare{x-z}}\bra{\frac{1}{r^2}\sJ_z \pare{\frac{\Delta_z \eta}{1+2\eta}}_{\geqslant1}}
\\
&\quad+
R_1 \pare{\frac{1}{r^2}\sJ_z \pare{\frac{\Delta_z \eta}{1+2\eta}}_{\geqslant1}} g\pare{x-z}
+
R_2 \pare{g\pare{x-z}}\bra{\frac{1}{r^2}\sJ_z \pare{\frac{\Delta_z \eta}{1+2\eta}}_{\geqslant1}}.
\end{aligned}
\end{equation}
In view of \eqref{eq:Kernel_in_kernelfunction_J} we can apply \Cref{prop compo zdepop}, \cref{item:MM_int_z,item:MM_ext_z} and obtain that
\begin{align}\label{eq:D0_3}
R_1 \pare{\frac{1}{r^2}\sJ_z \pare{\frac{\Delta_z \eta}{1+2\eta}}_{\geqslant1}} g\pare{x-z}
=
R\pare{\eta, g ;z}g, 
&&
R_2 \pare{g\pare{x-z}}\bra{\frac{1}{r^2}\sJ_z \pare{\frac{\Delta_z \eta}{1+2\eta}}_{\geqslant1}}
=
R\pare{\eta, g ;z}\eta.
\end{align}
Then, we apply \cref{lemma paralin,lemma Bonyprod,prop compBW} in a similar fashion to the procedure detailed in the previous section and obtain that
\begin{equation}\label{eq:D0_4}
\OpBW{g\pare{x-z}}\bra{\frac{1}{r^2}\sJ_z \pare{\frac{\Delta_z \eta}{1+2\eta}}_{\geqslant1}}
=
\OpBW{\frac{g\pare{x-z}}{r^4}\sJ'_z\pare{\frac{\Delta_z\eta}{r^2}}} \Delta_z\eta
+
\OpBW{V^0\pare{\eta, g;x, z}} \eta + R\pare{\eta, g;z}\eta. 
\end{equation}
Next, we use the expansions of \Cref{lem:characterization_Kernels_J} and obtain that
\begin{equation}
\label{eq:D0_5}
\begin{aligned}
\OpBW{\frac{1}{r^2}\sJ_z \pare{\frac{\Delta_z \eta}{1+2\eta}}_{\geqslant1}} g\pare{x-z}
= & \ 
\OpBW{\frac{\sJ^0\pare{\eta;x}}{r^2} + V^1\pare{\eta, g;x, z}}g\pare{x-z} ,
\\
\OpBW{\frac{g\pare{x-z}}{r^4}\sJ'_z\pare{\frac{\Delta_z\eta}{r^2}}} \Delta_z\eta
=
& \
\OpBW{\frac{g}{r^4}\sJ^{\prime 0}\pare{\eta; x}}\Delta_z\eta + \OpBW{V\pare{\eta, g;x} + V^1\pare{\eta, g;x, z}}\delta_z\eta,
\\
\OpBW{V^0\pare{\eta, g;x, z}} \eta 
= & \ 
\OpBW{V\pare{\eta, g;x}+   V^1\pare{\eta, g;x, z}}\eta.
\end{aligned}
\end{equation}
We plug the results in \cref{eq:D0_5,eq:D0_4,eq:D0_3,eq:D0_2,eq:Jz} into \cref{eq:D0_1} and obtain, after an application of \Cref{prop:reminders_integral_operator}, that
\begin{equation*}
\begin{aligned}
    \mathpzc{D}_0\pare{\eta}g &= \OpBW{\frac{\sJ^0\pare{\eta;x}}{r^2}} \intT \frac{g\pare{x-z}}{2\sin\pare{z/2}}\dd z
+
\OpBW{\frac{g}{r^4}\sJ^{\prime 0}\pare{\eta; x}}\intT \frac{\delta_z\eta}{4\sin^2\pare{z/2}}\dd z
\\
&\quad+\OpBW{ A_{\bra{0}} \pare{\eta, g;x, \xi}} \eta + \OpBW{ A_{\bra{-2}}\pare{\eta;x, \xi} } g
+
{\bm R}\pare{\eta, g;z}\vect{\eta}{g}\cdot\vect{1}{1} , 
\end{aligned}
\end{equation*}
from which we conclude after standard symbolic manipulations the identity in \eqref{eq:D0_paralinearization}. 
\end{proofpart}

\begin{proofpart}[Proof of \Cref{lem:D0_paralinearization}, \cref{item:D0_paralinearization_2}]
 Arguing similarly as it was done in order to deduce \eqref{eq:D0_1} and using  \cref{eq:Jz,eq:Taylor_expansion_kernels0_J} we have that
\begin{equation}\label{eq:D0gamma_1}
\mathpzc{D}_0\pare{\eta}\bra{1} = r^{-2} +\intT \frac{1}{r^2}\sJ_z\pare{\frac{\Delta_z\eta}{r^2}}_{\geqslant1}\frac{\dd z}{2\sin\pare{z/2}}\cdot
\end{equation}
We apply \cref{lemma paralin} and obtain that
\begin{equation}\label{eq:D0gamma_2}
\mathpzc{D}_0\pare{\eta}\bra{1} = r^{-2}
+
\intT \frac{1}{r^2}
\set{
\OpBW{\sJ'_z \pare{\frac{\Delta_z\eta}{r^2}}}\bra{\frac{\Delta_z\eta}{r^2}}
+
R\pare{\frac{\Delta_z\eta}{r^2}}\bra{\frac{\Delta_z\eta}{r^2}}
} \frac{\dd z}{2\sin\pare{z/2}}\cdot
\end{equation}
Thus, computations similar to the ones that lead to \eqref{eq:H0gamma_2} give us that
\begin{equation} \label{eq:D0gamma_3}
\intT \frac{1}{r^2}
R\pare{\frac{\Delta_z\eta}{r^2}}\bra{\frac{\Delta_z\eta}{r^2}}
 \frac{\dd z}{2\sin\pare{z/2}} = R \pare{\eta}\eta.
\end{equation}
Next, similar computations to the ones that lead to \eqref{eq:H0gamma_3} give us that
\begin{equation}\label{eq:D0gamma_4}
\begin{aligned}
    \intT  \frac{1}{r^2}\OpBW{\sJ_z' \pare{\frac{\Delta_z\eta}{r^2}}}\bra{\frac{\Delta_z\eta}{r^2}} \frac{\dd z}{2\sin\pare{z/2}}
&=
\intT  \OpBW{r^{-4} \sJ_z' \pare{\frac{\Delta_z\eta}{r^2}}}\bra{\frac{\delta_z\eta}{\pare{ 2\sin\pare{z/2} }^2 }} \dd z
\\
&\quad+\OpBW{V\pare{\eta;x}}\eta + R\pare{\eta}\eta ,
\end{aligned} 
\end{equation}
so that arguing as in order to deduce \eqref{eq:H0gamma_4}
\begin{equation}\label{eq:D0gamma_5}
\begin{aligned}
    &\intT  \OpBW{r^{-4} \sJ_z' \pare{\frac{\Delta_z\eta}{r^2}}}\bra{\frac{\delta_z\eta}{\pare{ 2\sin\pare{z/2} }^2 }} \dd z
\\
&=
\OpBW{r^{-4}\sJ^{\prime 0}\pare{\eta;x}}\av{D}\eta + \OpBW{A_{\bra{0}}\pare{\eta;x,\xi}}\eta + R\pare{\eta}\eta
\\
&= 
\OpBW{r^{-4}\sJ^{\prime 0}\pare{\eta;x}\av{\xi} + A_{\bra{0}}\pare{\eta;x,\xi}}\eta + R\pare{\eta}\eta.
\end{aligned}
\end{equation}Thus, combining \cref{eq:D0gamma_1,eq:D0gamma_2,eq:D0gamma_3,eq:D0gamma_4,eq:D0gamma_5}  and the paralinearization $ r^{-2}=1+\OpBW{-\frac{2}{\pare{1+2\eta}^2}}\eta +R\pare{\eta}\eta $ derived using \Cref{lemma paralin} we obtain \cref{eq:D0_paralinearization_2}. 
\end{proofpart}

\end{proof}

\subsection{Paralinearization of $ \mathpzc{H}\pare{\eta} $}

In view of \eqref{eq:HH0}, putting together the paralinearizations of \Cref{para:H0,para:D0}, we prove the following.

\begin{lemma}\label{lem:paralin_pzcH}
Let $ N\in \mathbb{N}$, $ \upgamma\in \mathbb{R} $ and $ \varrho \geqslant 0 $, for any $ K\in \mathbb{N} $ there exists $ s_0 > 0 $ and $ \epsilon_0 > 0 $ such that if 
$ \eta, g \in \BallR{K}{s_0}$, 
let $ \mathpzc{H}\pare{\eta} $ be as in \eqref{eq:HH0} then we have that
\begin{enumerate}
\item
\begin{multline}\label{eq:Hetaomega_2}
\mathpzc{H}\pare{\eta}g =
\OpBW{-\ii \ \sgn{\xi} + A_{\bra{-2}}\pare{\eta;x, \xi} }g
\\
 + \OpBW{  -\frac{g}{1+2\eta} \ \sJ^0\pare{\eta;x} \ \av{\xi}
+  2 V_0\pare{\eta, g;x} \ii \xi
+ A_{\bra{0}}\pare{\eta,g;x, \xi}
}\eta
+ {\bm R}\pare{\eta, g}\vect{\eta}{g}\cdot\vect{1}{1},
\end{multline}
where
\begin{itemize}
\item $\sJ^0$$\pare{\eta ; x}\in \cF^{\mathbb{R}}_{K, 0, 1}\bra{\epsilon_0, N}$ and is explicitly defined in \eqref{eq:sJ0};

\item $ V_0 \pare{\eta, \partial_x^{-1} g;x}\defeq \frac{1}{2}\mathpzc{D}_0\pare{\eta}g \ \in \Sigma \cF^{\mathbb{R}}_{K, 0, 1}\bra{\epsilon_0, N} $, cf. \eqref{eq:integral_operators_etaomega};



\item
$ A_{\bra{m}}\pare{\eta, g;x, \xi}\in \Sigma\Gamma^{m} _{K, 0, 1}\bra{\epsilon_0, N} $;

\item
$ R\pare{\eta, g} \in \Sigma\cR^{-\vr} _{K, 0, 1}\bra{\epsilon_0, N} $. 
\end{itemize}

\item
\begin{equation}\label{eq:Hetaomega_2.1}
\mathpzc{H}\pare{\eta}\bra{1} =  \OpBW{  -\frac{\sJ^0\pare{\eta;x}}{1+2\eta}  \ \av{\xi}
+\ii \ \tilde V\pare{\eta ;x} \xi
+ A_{\bra{0}}\pare{\eta;x, \xi}
}\eta
+ R\pare{\eta}\bra{\eta},
\end{equation}
where
\begin{itemize}
\item $ \sJ^0\pare{\eta ; x}\in \cF^{\mathbb{R}}_{K, 0, 1}\bra{\epsilon_0, N} $ and is explicitly defined in \eqref{eq:sJ0};

\item $ \tilde V \pare{\eta;x}\defeq \mathpzc{D}_0\pare{\eta}\bra{1} \ \in \Sigma \cF^{\mathbb{R}}_{K, 0, 0}\bra{\epsilon_0, N} $ with $ \tilde{V}\pare{0;x} \equiv 1 $;



\item
$ A_{\bra{0}}\pare{\eta ;x, \xi}\in \Sigma\Gamma^{0} _{K, 0, 1}\bra{\epsilon_0, N} $;

\item
$ R\pare{\eta} \in \Sigma\cR^{-\vr} _{K, 0, 1}\bra{\epsilon_0, N} $. 
\end{itemize}

\end{enumerate}

\end{lemma}

\begin{proof}
We use now the expression in \cref{eq:HH0} and \Cref{lem:D0_paralinearization,lem:H0_paralinearization} as well as \Cref{prop compBW} and obtain that
\begin{multline}\label{eq:Hetaomega_1}
\mathpzc{H}\pare{\eta}g =  \OpBW{-\ii \ \pare{1+\sK^0\pare{\eta;x} + \frac{\eta_x}{r^2} \ \sJ^0\pare{\eta;x}}\sgn \xi + A_{\bra{-2}}\pare{\eta;x, \xi} }g
\\
+
\OpBW{  \frac{g}{r^2}\pare{\sK^{\prime 0}\pare{\eta;x}   + \frac{\eta_x}{r^2} \ \sJ^{\prime 0}\pare{\eta;x}}\av{\xi}
+\ii \ V_0\pare{\eta, \partial_x^{-1} g;x}  \xi 
+ A_{\bra{0}}\pare{\eta, g;x, \xi}
}\eta
 + {\bm R}\pare{\eta, g}\vect{\eta}{g}\cdot\vect{1}{1}. 
\end{multline}
Notice that in order to deduce \eqref{eq:Hetaomega_1} we used the fact that $ \mathpzc{D}_0\pare{\eta}g \in \Sigma\cF^{\mathbb{R}}_{K, 0, 1}\bra{\epsilon_0, N} $, which stems immediately from the paralinearization provided in \Cref{lem:D0_paralinearization}, \cref{item:D0_paralinearization_1}, next we relabeled $ \frac{1}{2} \mathpzc{D}_0\pare{\eta}g = V_0\pare{\eta, \partial_x^{-1} g;x} $. 
Next we use \Cref{eq:sK0,eq:sJ0,eq:sJprime0} and obtain that
\begin{equation*}
\begin{aligned}
\sK^0\pare{\eta;x} + \frac{\eta_x}{r^2} \ \sJ^0\pare{\eta;x} = & \ 0, 
\\
\sK^{\prime 0}\pare{\eta;x}  + \frac{\eta_x}{r^2} \ \sJ^{\prime 0}\pare{\eta;x}
= & \ -\sJ^0\pare{\eta;x}. 
\end{aligned}
\end{equation*}
Thus, transforming \eqref{eq:Hetaomega_1} into \eqref{eq:Hetaomega_2}. The proof of \eqref{eq:Hetaomega_2.1} is almost identical to the proof of \eqref{eq:Hetaomega_2} and is hence omitted. 
\end{proof}

\subsection{Proof of \Cref{prop:paralinearizationKH}}

We can finally paralinearize \cref{eq:KH3}.
%
%
We use \Cref{eq:Hetaomega_2,eq:Hetaomega_2.1}, we set $ g = \partial_x \psi $ and \Cref{prop compBW} in order to obtain, from \eqref{eq:KH3}, that
\begin{equation}\label{eq:paralinearization_eq_eta}
\begin{aligned}
    \eta_t &= \OpBW{ -\frac{\av{\xi}}{2} +  A_{\bra{-2}}\pare{\eta;x, \xi} }\psi
\\
&\quad+\OpBW{\frac{1}{2}B_\upgamma \pare{\eta, \psi;x}\av{\xi} - \ii \ V_\upgamma \pare{\eta, \psi;x}\xi + A_{\bra{0}}\pare{\eta, \psi;x, \xi}}\eta
+  R\pare{\eta, \psi} \bra{\eta+\psi},  
\end{aligned}
\end{equation}
where
\begin{align}\label{eq:B}
B_\upgamma  \pare{\eta, \psi;x}\defeq
B_0  \pare{\eta, \psi;x} + \upgamma \tilde{B} \pare{\eta;x}, 
&&
B_0  \pare{\eta, \psi;x} \defeq
 \frac{ \psi_x}{1+2\eta} \sJ^0\pare{\eta;x}, 
&&
\tilde{B} \pare{\eta;x}
\defeq
\frac{\sJ^0\pare{\eta;x}}{1+2\eta}  
\end{align}
and 
$$
V_\upgamma\pare{\eta, \psi;x}\defeq\frac{1}{2}\mathpzc{D}_0\pare{\eta}\bra{\upgamma + \psi_x}-\frac{\upgamma}{2}.
$$ 
Let us now paralinearize the term
\begin{equation*}
-\frac{\upgamma + \psi_x}{2} \mathpzc{D}_0\pare{\eta}\bra{\upgamma + \psi_x}. 
\end{equation*}
We apply \Cref{lemma Bonyprod} and obtain that
\begin{equation}\label{eq:PT_psi_-1}
\begin{aligned}
    -\frac{ \psi_x}{2} \mathpzc{D}_0\pare{\eta}\bra{ \upgamma + \partial_x \psi} 
&=
 \OpBW{-\frac{\psi_x}{2}} \bra{\mathpzc{D}_0\pare{\eta}\bra{ \upgamma + \partial_x \psi}}
-
\OpBW{\frac{\mathpzc{D}_0\pare{\eta} \bra{ \upgamma + \partial_x \psi}}{2}} \bra{\psi_x}
\\
&\quad+ R_1 \pare{-\frac{\psi_x}{2}} \bra{\mathpzc{D}_0\pare{\eta} \bra{ \upgamma + \partial_x \psi}}
+
R_2 \pare{\mathpzc{D}_0\pare{\eta} \bra{ \upgamma + \partial_x \psi}} \bra{-\frac{\psi_x}{2}}. 
\end{aligned}
\end{equation}
Recall  that $ \frac{\mathpzc{D}_0\pare{\eta} \bra{ \upgamma + \partial_x \psi}}{2}  = V_\upgamma \pare{\eta, \psi;x} $, so that we can apply \Cref{prop compo mop,prop compBW} and obtain that 
\begin{equation}
\label{eq:PT_psi_0}
\begin{gathered}
\OpBW{ - \frac{\mathpzc{D}_0\pare{\eta} \bra{ \upgamma + \partial_x \psi}}{2}} \bra{\psi_x}
=  \
\OpBW{- \ii \ V_\upgamma \pare{\eta, \psi;x}\xi} \psi + R\pare{\eta, \psi}\psi, 
\\
 R_1 \pare{-\frac{\psi_x}{2}} \bra{\mathpzc{D}_0\pare{\eta} \bra{ \upgamma + \partial_x \psi}}
+
R_2 \pare{\mathpzc{D}_0\pare{\eta} \bra{ \upgamma + \partial_x \psi}} \bra{-\frac{\psi_x}{2}}
=  \
{\bm R}\pare{\eta, \psi}\vect{\eta}{\psi}\cdot\vect{1}{1}. 
\end{gathered}
\end{equation}
Next we apply the paralinearization stated in \Cref{lem:D0_paralinearization} and the composition theorems in \Cref{prop compo mop,prop compBW} and obtain that
\begin{multline}\label{eq:PT_psi_1}
 \OpBW{-\frac{\psi_x}{2}} \bra{\mathpzc{D}_0\pare{\eta} \bra{\upgamma + \partial_x \psi}}
 =
 \OpBW{-\frac{1}{2}\frac{\psi_x}{1+2\eta} \ \sJ^0\pare{\eta;x}\av{\xi} + A_{\bra{-2}}\pare{\eta;x, \xi}}\psi
 \\
 +\OpBW{-\frac{1}{2}\frac{\psi_x\pare{\upgamma + \psi_x}}{1+2\eta}\sJ^{\prime 0}\pare{\eta; x}\av{\xi} + A_{\bra{0}}\pare{\eta, \psi;x, \xi}}\eta
 +
 {\bm R}\pare{\eta, \psi}\vect{\eta}{\psi}\cdot\vect{1}{1}, 
\end{multline}
so, using \eqref{eq:B}, we obtain that
\begin{multline}\label{eq:PT_psi_1.1}
 \OpBW{-\frac{\psi_x}{2}} \bra{\mathpzc{D}_0\pare{\eta} \bra{\upgamma + \partial_x \psi}}
 =
 \OpBW{-\frac{1}{2}B_0\pare{\eta, \psi;x}\av{\xi} + A_{\bra{-2}}\pare{\eta;x, \xi}}\psi
 \\
 +\OpBW{-\frac{1}{2}\frac{\psi_x\pare{\upgamma + \psi_x}}{1+2\eta}\sJ^{\prime 0}\pare{\eta; x}\av{\xi} + A_{\bra{0}}\pare{\eta, \psi;x, \xi}}\eta
 +
 {\bm R}\pare{\eta, \psi}\vect{\eta}{\psi}\cdot\vect{1}{1}. 
\end{multline}
We invoke now \Cref{lem:D0_paralinearization} and obtain, using the notation introduced in \eqref{eq:B}, that
\begin{multline}\label{eq:PT_psi_1.2}
-\frac{\upgamma}{2}\mathpzc{D}_0\pare{\eta}\bra{\upgamma + \psi_x}
=
\OpBW{-\frac{\upgamma}{2} \tilde{B}\pare{\eta;x} \av{\xi} + A_{\bra{-2}}\pare{\eta;x, \xi}}\psi
\\
-\frac{\upgamma^2}{2}
+\OpBW{ \frac{\upgamma^2}{\pare{1+2\eta}} -\frac{\upgamma}{2}\frac{\upgamma + \psi_x}{\pare{1+2\eta}^2} \sJ^{\prime 0}\pare{\eta;x} \av{\xi} + A_{\bra{0}}\pare{\eta, \psi;x, \xi}} \eta
+{\bm R}\pare{\eta, \psi}\vect{\eta}{\psi}\cdot\vect{1}{1}. 
\end{multline}
Notice that using \cref{eq:PT_psi_0,eq:PT_psi_-1} we obtain that
\begin{equation}
\label{eq:PT_psi_1.3}
\begin{aligned}
    -\frac{\upgamma + \psi_x}{2} \mathpzc{D}_0\pare{\eta}\bra{\upgamma + \psi_x}
&=
-\frac{\upgamma}{2}\mathpzc{D}_0\pare{\eta}\bra{\upgamma + \psi_x}
+
\OpBW{-\frac{\psi_x}{2}} \bra{\mathpzc{D}_0\pare{\eta} \bra{\upgamma + \partial_x \psi}}
\\
&\quad+ \OpBW{- \ii \ V_\upgamma \pare{\eta, \psi;x}\xi} \psi + {\bm R}\pare{\eta, \psi}\vect{\eta}{\psi}\cdot\vect{1}{1}.
\end{aligned}
\end{equation}
From \Cref{eq:sJ0,eq:sJprime0} we obtain the relation
\begin{equation}\label{eq:sJ0-sJ0prime}
\sJ^{\prime 0}\pare{\eta;x}
{=\frac{2\pare{1+2\eta}^4}{\pare{\pare{1+2\eta}^2+\eta_x^2}^2} - \pare{\sJ^0\pare{\eta;x}}^2}
\end{equation}
so that using \cref{eq:B,eq:sJ0-sJ0prime,eq:PT_psi_1.1,eq:PT_psi_1.2}  and denoting with
\begin{align}\label{eq:W}
W_\upgamma\pare{\eta, \psi;x}
\defeq
W_0\pare{\eta, \psi;x} + \upgamma\tilde{W}\pare{\eta;x}\in \Sigma \cF^{\mathbb{R}}_{K, 0, 0}\bra{\epsilon_0, N}, 
\end{align}
where
\begin{align}\label{eq:W0}
W_0\pare{\eta, \psi;x}
\defeq\frac{\psi_x\pare{1+2\eta}}{\pare{1+2\eta}^2+\eta_x^2}\in \Sigma \cF^{\mathbb{R}}_{K, 0, 1}\bra{\epsilon_0, N}, 
&&
\tilde W\pare{\eta;x}
\defeq\frac{1+2\eta}{\pare{1+2\eta}^2+\eta_x^2}\in \Sigma \cF^{\mathbb{R}}_{K, 0, 0}\bra{\epsilon_0, N},
\end{align}
 we obtain that
\begin{multline}\label{eq:PT_psi_2}
-\frac{\upgamma}{2}\mathpzc{D}_0\pare{\eta}\bra{\upgamma + \psi_x}
+
\OpBW{-\frac{\psi_x}{2}} \bra{\mathpzc{D}_0\pare{\eta} \bra{\upgamma + \partial_x \psi}}
 =
 -\frac{\upgamma^2}{2}+
  \OpBW{-\frac{1}{2}B_\upgamma \pare{\eta, \psi;x}\av{\xi} + A_{\bra{-2}}\pare{\eta;x, \xi}}\psi
 \\
  +\OpBW{\frac{\upgamma^2}{\pare{1+2\eta}} -\frac{1}{2}\pare{W_\upgamma^2\pare{\eta, \psi;x} - B_\upgamma ^2\pare{\eta, \psi;x}}\av{\xi} + A_{\bra{0}}\pare{\eta, \psi;x, \xi}}\eta
 +
 {\bm R}\pare{\eta, \psi}\vect{\eta}{\psi}\cdot\vect{1}{1}.
\end{multline}
We thus insert \eqref{eq:PT_psi_2} in \eqref{eq:PT_psi_1.3} and obtain that
\begin{equation}\label{eq:paralin_quasilin_psi}
\begin{aligned}
-\frac{\upgamma + \psi_x}{2} \mathpzc{D}_0\pare{\eta}\bra{\upgamma + \psi_x}
=
& \ 
-\frac{\upgamma^2}{2}+
\OpBW{-\frac{1}{2}B_\upgamma\pare{\eta, \psi;x}\av{\xi} - \ii \ V_\upgamma\pare{\eta, \psi;x}\xi + A_{\bra{-2}}\pare{\eta;x, \xi}}\psi
 \\
 & \  +\OpBW{\frac{\upgamma^2}{\pare{1+2\eta}} -\frac{1}{2}\pare{W_\upgamma^2\pare{\eta, \psi;x}  - B_\upgamma^2\pare{\eta, \psi;x}}\av{\xi} + A_{\bra{0}}\pare{\eta, \psi;x, \xi}}\eta
\\
& \  +
 {\bm R}\pare{\eta, \psi}\vect{\eta}{\psi}\cdot\vect{1}{1}.
\end{aligned}
\end{equation}
An iterated application of \Cref{lemma paralin} give us that
\begin{equation}\label{eq:paralin_curvature}
    \begin{aligned}
-\kap \mathpzc{K}\pare{\eta}& = \OpBW{\kap\pare{1+\mathtt{f}\pare{\eta;x}}\pare{ \av{\xi}^2-1 } }\eta + R\pare{\eta}\eta,\\
\mathtt{f}\pare{\eta;x}&\triangleq \pare{\frac{1+2\eta}{\pare{1+2\eta}^{2}+\eta_x^2}}^{\frac{3}{2}}-1\in \Sigma\cF^{\mathbb{R}}_{K,0,1}\bra{\epsilon_0, N}.
\end{aligned}
\end{equation}
We can thus now plug \cref{eq:paralin_curvature,eq:paralin_quasilin_psi} in the second equation of \eqref{eq:KH3} (recall that $\Omega$ is set in \eqref{eq:Omega_def} in order to cancel the 0-homogeneous components of the transport) and obtain that
\begin{equation}\label{eq:paralin_eq_psi}
\begin{aligned}
    \psi_t &= 
\OpBW{-\frac{1}{2}B_\upgamma\pare{\eta, \psi;x}\av{\xi} - \ii \ V_\upgamma\pare{\eta, \psi;x}\xi + A_{\bra{-2}}\pare{\eta;x, \xi}}\psi
\\
&\quad+\OpBW{\kap\pare{1+\mathtt{f}\pare{\eta;x}}\pare{ \av{\xi}^2-1 } + \frac{\upgamma^2}{\pare{1+2\eta}} -\frac{1}{2}\pare{W_\upgamma^2\pare{\eta, \psi;x}  - B_\upgamma^2\pare{\eta, \psi;x}}\av{\xi} + A_{\bra{0}}\pare{\eta, \psi;x, \xi}}\eta
\\
  &\quad+
 {\bm R}\pare{\eta, \psi}\vect{\eta}{\psi}\cdot\vect{1}{1}. 
\end{aligned}
\end{equation}
Finally, we combine \cref{eq:paralin_eq_psi,eq:paralinearization_eq_eta} to obtain \eqref{eq:KH4} after a renaming, if needed.

\section{Complex Hamiltonian formulation of the Kelvin-Helmholtz equations}\label{sec:complex}
We begin with the real Hamiltonian system \eqref{eq:KH_Hamiltonian}, and introduce the following associated complex variables
\begin{align}\label{change complex}
    \vect{\eta}{\psi}\defeq \mathpzc{C}\vect{v}{\bar v} \ , 
    &&
    \mathpzc{C} \defeq \frac{1}{\sqrt{2}}
    \begin{bmatrix}
        1&1\\-\ii & \ii
    \end{bmatrix} \ , 
    &&
    \mathpzc{C}^{-1} = \frac{1}{\sqrt{2}}
    \begin{bmatrix}
        1&\ii\\1 & -\ii
    \end{bmatrix}. 
\end{align}
Under this change of variables \eqref{eq:KH_Hamiltonian} is equivalent to
\begin{equation}\label{eq change complex}
    \vect{v_t}{\bar v_t} = \mathpzc{C}^{-1} {\bm J} \nabla_{\pare{\eta, \psi}} H\pare{\mathpzc{C}\vect{v}{\bar v}}.
\end{equation}
Defining $H_\mathbb{C}\defeq H\circ\mathpzc{C}$ and noting that
\begin{equation*}
    \vect{\partial_\eta }{\partial_\psi} = \frac{1}{\sqrt{2}}
    \begin{bmatrix}
        1&1\\\ii & -\ii
    \end{bmatrix}
    \vect{\partial_v}{\partial_{\bar v}}, 
\end{equation*}
we obtain that \eqref{eq change complex} can be written as an Hamiltonian system in the complex  coordinates \eqref{change complex} as
\begin{align}\label{new complexHAM}
   V_t = \JC \nabla_{U} H_\mathbb{C} \pare{V}, \qquad
    V\defeq  \vect{v}{\bar v}, 
\end{align}
where the {\it complex Poisson tensor} (cf. \eqref{eq:KH_Hamiltonian}) is defined as
\begin{equation}\label{def bfJC}
{\bm J}_\mathbb{C} \defeq -\ii {\bm J} 
=
\begin{bmatrix}
    0 & \ii \\-\ii & 0
\end{bmatrix}
. 
\end{equation}
Next we define the Fourier symbol
\begin{equation}
\begin{aligned}
&\mathsf{m}_{\kap, \upgamma}\pare{{\xi}}\defeq
\sqrt{\frac{\av{\xi}}{2\omega_{\kap, \upgamma}\pare{ \xi }}}
=
\sqrt[4]{\frac{\av{\xi}}{2\pare{ \kap\pare{\av{\xi}^2 -1} - \upgamma^2\pare{\frac{{|\xi|}}{2}-1} }}}
\quad \in \tilde{\Gamma}^{-\frac{1}{4}}_0,\\
&\mathsf{m}_{\kap, \upgamma}\pare{{\xi}}^{-1}= \sqrt[4]{\frac{2\pare{ \kap\pare{{\xi}^2 -1} - \upgamma^2\pare{\frac{{|\xi|}}{2}-1} }}{{|\xi|}}}\quad \in \tilde{\Gamma}^{\frac{1}{4}}_0.
\end{aligned}
\label{eq:msigmagamma}
\end{equation}
Then, we consider the symplectic matrix
\begin{equation}\label{def bmS}
    {\bm S}_{\kap,\upgamma}(D)\triangleq{
\begin{bmatrix}
\mathsf{m}_{\kap, \upgamma}\pare{ {D}} &0
\\
0 & \mathsf{m}_{\kap, \upgamma}\pare{ {D}}^{-1}
\end{bmatrix}
}\qquad\textnormal{and}\qquad{\bm\sM} _{\kap, \upgamma}\pare{ {D}}\triangleq{\bm S}_{\kap,\upgamma}(D)\circ\mathcal{C}.
\end{equation}
Notice that
\begin{equation}\label{eq:sM}
\begin{aligned}
{\bm\sM} _{\kap, \upgamma}\pare{ {\xi}} \defeq &\  \frac{1}{\sqrt{2}}
\bra{
\begin{array}{cc}
\mathsf{m}_{\kap, \upgamma}\pare{ {\xi}} & \mathsf{m}_{\kap, \upgamma}\pare{ {\xi}} \\
- \ii \mathsf{m}_{\kap, \upgamma}\pare{ {\xi}}^{-1} &  \ii \mathsf{m}_{\kap, \upgamma}\pare{ {\xi}}^{-1}
\end{array}
}
& \in \pare{ \tilde{\Gamma}^{1/4}_0 }^{2\times 2},
\\
{\bm\sM}_{\kap, \upgamma}^{-1} \pare{ {\xi}}
\defeq & \ 
\frac{1}{\sqrt{2}}
\bra{
\begin{array}{cc}
\mathsf{m}_{\kap, \upgamma}\pare{ {\xi}}^{-1} & \ii \ \mathsf{m}_{\kap, \upgamma}\pare{ {\xi}}
\\
\mathsf{m}_{\kap, \upgamma}\pare{ {\xi}}^{-1} & -\ii \ \mathsf{m}_{\kap, \upgamma}\pare{ {\xi}}
\end{array}
}
& \in \pare{ \tilde{\Gamma}^{1/4}_0 }^{2\times 2}.
\end{aligned}
\end{equation}
We define the {\it complex coordinates} $ U $ as
\begin{align}\label{eq:complex_var}
U\defeq\vect{u}{\bar{u}} \defeq  {\bm\sM}_{\kap, \upgamma}^{-1} \pare{ {D}}\vect{\eta}{\psi}=\mathcal{C}^{-1}\circ{\bm S}_{\kap,\upgamma}^{-1}(D)\vect{\eta}{\psi} . 
\end{align}

\begin{notation}\label{notation:A_refugium_peccatorum}
Let $ K\in \mathbb{N}$, $ \epsilon_0 > 0 $, $ m\in \mathbb{R} $, $ N\in \mathbb{N}$ and $ 0\leqslant K'\leqslant K $. {We work on a time interval $I=[0,T]$ for some fixed $T>0$ to be determined. This latter is a priori implicit but will correspond to $c\varepsilon^{-(N+1)}.$} 
Moreover, from now on we denote with
\begin{itemize}
\item $ A_{\bra{m;K'}}\pare{U; x, \xi} $  any generic element in the space $ \Sigma\Gamma^{m}_{K, K', 1 }\bra{\epsilon_0, N} $, while $  {\bm A}_{\bra{m;K'}}\pare{U; x, \xi}  $ is a generic element of $ \pare{\Sigma\Gamma^{m}_{K, K', 1 }\bra{\epsilon_0, N}}^{2\times 2} $ such that $ \JC \OpBW{{\bm A}_{\bra{m;K'}}\pare{U;x, \xi}} $ is lineary Hamiltonian up to homogeneity $ N $ (cf. \Cref{def:LinHamup}), whose explicit expression may vary from line to line. We use as well the simplified notation $ A_{\bra{m;0}}\pare{U; x, \xi} \defeq A_{\bra{m}}\pare{U; x, \xi} $ and $  {\bm A}_{\bra{m;0}}\pare{U; x, \xi} \defeq  {\bm A}_{\bra{m}}\pare{U; x, \xi} $; 

\item $ R_{\bra{K'}}\pare{U} $ and $  {\bm R}_{\bra{K'}}\pare{U} $ any generic element in the space $ \Sigma\cR^{-\vr}_{K, K', 1 }\bra{\epsilon_0, N} $ and $ \pare{\Sigma\cR^{-\vr}_{K, K', 1 }\bra{\epsilon_0, N}}^{2\times 2} $ respectively, whose explicit expression may vary from line to line. We denote as well $ R_{\bra{0}}\pare{U} \defeq R\pare{U} $ and $ {\bm R}_{\bra{0}}\pare{U}\defeq {\bm R}\pare{U}.$
\end{itemize}
\end{notation}

We prove the following:

\begin{prop}[Kelvin-Helmholtz equations in complex Hamiltonian coordinates] \label{prop:KH_complex}
Let $ N\in \mathbb{N} $, $\kap > 0$, $ \upgamma\in \mathbb{R} $ and $ \varrho \geqslant 0 $, for any $ K\in \mathbb{N}$ there exists $ s_0 >0 $ and $ \epsilon_0 > 0 $ such that if $\eta, \psi \in B^K_{s_0, \mathbb{R}}\pare{I;\epsilon_0} $ is a solution of \cref{eq:KH3}, then $ U $ defined in \eqref{eq:complex_var} solves the complex Hamiltonian system 
\begin{align}\label{eq:KH6}
U_t
 =   {\bm J}_\mathbb{C} \  \OpBW{{\bm A}_{\frac{3}{2}}\pare{U;x} \  \omega_{\kap, \upgamma}\pare{ \xi }  + {\bm A}_{1}\pare{U;x, \xi} + {\bm A}_{\frac{1}{2}}\pare{U;x } \  \av{\xi}^{\frac{1}{2}} + {\bm A}_{\bra{0}} \pare{U;x, \xi}} U +{\bm R}\pare{U} U ,
\end{align}
where
\begin{itemize}
\item  $ {\bm J}_\mathbb{C} $ is defined in \cref{def bfJC};
\item $ {\bm A}_{\frac{3}{2}}\pare{U;x} \in \pare{\Sigma\cF^{\mathbb{R}}_{K, 0, 0}\bra{\epsilon_0, N}}^{2\times 2} $ is defined as
\begin{equation}\label{eq:A3/2}
{\bm A}_{\frac{3}{2}}\pare{U, x}\defeq \begin{bmatrix}
0&1\\1&0
\end{bmatrix} + \frac{\mathtt{f}\pare{ \pare{{\bm\sM} _{\kap, \upgamma}\pare{D}  U}_1 ;x}}{2} \begin{bmatrix}
1 & 1 \\
1 & 1
\end{bmatrix} ,
\end{equation}
where $ \mathtt{f}\pare{\eta; x} $ is defined in \cref{eq:paralin_curvature}; 

\item
$ {\bm A}_{1}\pare{U;x, \xi} \in \pare{\Sigma\Gamma^1 _{K, 0, 1}\bra{\epsilon_0, N}}^{2\times 2} $ is defined as
\begin{equation*}
{\bm A}_{1}\pare{U;x,\xi}
\defeq   - V_{\upgamma}\pare{{\bm\sM} _{\kap, \upgamma}\pare{D}  U;x}\xi
\begin{bmatrix}
0&-1\\ 1&0
\end{bmatrix}
+
\frac{B_\upgamma\pare{{\bm\sM} _{\kap, \upgamma}\pare{\av{D}}  U;x}}{2} \ \av{\xi}
\begin{bmatrix}
\ii &0\\0& -\ii 
\end{bmatrix}, 
\end{equation*}
where $ V_{\upgamma}\pare{\eta, \psi;x}\in \Sigma \cF^{\mathbb{R}}_{K, 0, 1}\bra{\epsilon_0, N}  $ and $ B_\upgamma\pare{\eta, \psi; \xi} \in \Sigma \cF^{\mathbb{R}}_{K, 0, 1}\bra{\epsilon_0, N}  $ are functions in  defined in \cref{eq:Vgamma,eq:B};

\item $ {\bm A}_{\frac{1}{2}}\pare{U;x} \in \pare{\Sigma\cF^{\mathbb{R}}_{K, 0, 1}\bra{\epsilon_0, N}}^{2\times 2} $ is defined as
\begin{equation}\label{eq:A1/2}
{\bm A}_{\frac{1}{2}} \pare{U, x} 
\defeq
A_{\frac{1}{2}}\pare{U;x}  \ \begin{bmatrix}
1 & 1 \\
1 & 1
\end{bmatrix}
\defeq 
\frac{1}{2}  \frac{1}{\sqrt{2\kap}} \pare{ \frac{  B^2_\upgamma \pare{\sM U ; x}}{2} +  \frac{\upgamma^2}{2}  \mathtt{f}\pare{\eta;x}  - w_\upgamma\pare{\sM U ; x } }  \ \begin{bmatrix}
1 & 1 \\
1 & 1
\end{bmatrix} ,
\end{equation}
where $ w_\upgamma\pare{\eta, \psi; x} $ is defined in \cref{eq:wgamma};

\item $ {\bm A}_{\bra{0}}\pare{U;x, \xi} $ and $ {\bm R}\pare{U} $ are as in \Cref{notation:notation_paralinearization}. 

\end{itemize}
\end{prop}

\begin{proof}

\begin{step}[Diagonalization of the linear part of \eqref{eq:KH4}]
We can diagonalize the matrix defined in \eqref{eq:linearized_system_matrix} as
\begin{equation}
\label{eq:diagonal_matrix}
{\bm L}_{\kap, \upgamma}\pare{\xi} = {\bm \sM}_{\kap, \upgamma}\pare{\xi} {\bm D}_{\kap, \upgamma}\pare{\xi} {\bm \sM}_{\kap, \upgamma}^{-1}\pare{\xi},\qquad{\bm D}_{\kap, \upgamma}\pare{\xi}
\defeq
\ii\omega_{\kap, \upgamma}\pare{ \xi }
\begin{bmatrix}
1 & 0 \\
0 & -1
\end{bmatrix}\in \pare{\tilde{\Gamma}^{\frac{3}{2}}_0}^{2\times 2} 
, 
\end{equation}
where $ \omega _{\kap, \upgamma}\pare{\xi}\in \tilde{\Gamma}^{\frac{3}{2}}_0  $ are defined in \cref{eq:disp_relation}. 
Defining $ U $ as in \eqref{eq:complex_var} 
where $ \vect{\eta}{\psi} $ solves \eqref{eq:linearized_system}, 
the vector-valued complex function $ U $ solves the diagonal linear system
\begin{equation*}
U_t = {\bm D}_{\kap, \upgamma}\pare{D} U . 
\end{equation*}
In particular thanks to the relation \eqref{eq:complex_var}  combined with \eqref{eq:msigmagamma}, which implies that $ \mathsf{m}_{\kap, \upgamma}\pare{\av{D}}\bra{1}=0 $, 
 assure us that if $ \pare{\eta, \psi}\in H^{s+\frac{1}{4}}_0\pare{\mathbb{T};\mathbb{R}} \times \dot{H}^{s-\frac{1}{4}}\pare{\mathbb{T};\mathbb{R}} $ then $ u\in H^s_0\pare{\mathbb{T};\mathbb{C}}\simeq \dot{H}^s\pare{\mathbb{T};\mathbb{C}} $ for $ s>0 $. \\
 
 Next, expanding \eqref{eq:msigmagamma} we have that
 \begin{equation}
 \label{eq:m_expansion_xi}
 \mathsf{m}_{\kap, \upgamma}\pare{\xi} = 
 \frac{1}{\pare{2\kap\av{\xi}}^{\frac{1}{4}}} +\pare{\frac{\upgamma}{2}}^2 \frac{1}{\pare{2\kap\av{\xi}}^{\frac{5}{4}}}   
 +
  \mathsf{m}_{\kap, \upgamma;  -\frac{9}{4}}\pare{\av{\xi}} , 
 \end{equation}
 with $ \mathsf{m}_{\kap, \upgamma;  -\frac{9}{4}}\pare{\av{\xi}}\in \tilde{\Gamma}^{-\frac{9}{4}}_0 $.

\begin{notation}
From now on we use the abbreviated notation $ \sM \defeq {\bm \sM}_{\kap, \upgamma}\pare{\av{D}} $ defined in \eqref{eq:sM} and $ \mathsf{m}\defeq \mathsf{m}_{\kap, \upgamma}\pare{\av{D}} $ defined in \eqref{eq:complex_var} for the sake of brevity.
\end{notation}

\end{step}

\begin{step}[Reformulation of \cref{eq:KH4} in complex coordiantes]

Recall the matrix of Fourier symbols ${\bm L}_{\kap, \upgamma} $ in \eqref{eq:linearized_system} and the matrix of symbols ${\bm Q}_{\kap, \upgamma}$ in \eqref{quone}; we define  
\begin{align}\label{eq:pzcQgeq1}
{\bm Q} _{\kap, \upgamma; \geqslant 1}\pare{\eta, \psi;x, \xi}\defeq 
{\bm Q}_{\kap, \upgamma}\pare{\eta, \psi;x, \xi} - {\bm L}_{\kap, \upgamma}\pare{\xi}= \begin{bmatrix}
0&0\\
\mathsf{q}\pare{\eta,\psi;x,\xi} & 0
\end{bmatrix} && \in  \Sigma\Gamma^{2}_{K, 0, 1}\bra{\epsilon_0, N}, 
\end{align} 
where, from \cref{eq:sq,quone}, $\mathsf{q}\pare{U;x, \xi}$ is the real-valued symbol 
\begin{equation}\label{eq:q_expansion_xi}
\begin{aligned}
\mathsf{q}\pare{\eta,\psi;x, \xi}= & \  \kap\ \mathtt{f}\pare{\eta;x}  \pare{\av{\xi}^2-1} -  w_\upgamma\pare{\eta, \psi;x}  \av{\xi}  + \upgamma^2\pare{ \frac{1}{\pare{1+2\eta}}-1 }
\\
= & \ \kap\ \mathtt{f}\pare{\eta;x}\av{\xi}^2 -  w_\upgamma\pare{\eta, \psi;x}  \av{\xi}  + \mathsf{q}_0\pare{\eta,\psi;x, \xi},
\end{aligned} 
\end{equation}
with
\begin{equation*}
\mathsf{q}_0\pare{\eta,\psi;x, \xi} \defeq - \kap\ \mathtt{f}\pare{\eta;x}  + \upgamma^2\pare{ \frac{1}{1+2\eta}-1 } \in \Sigma \Gamma^0_{K, 0, 1}\bra{\epsilon_0, N}.  
\end{equation*}
Then, by \eqref{eq:diagonal_matrix}, \cref{eq:KH4} becomes
\begin{equation}\label{eq:KH5}
\begin{aligned}
    &\vect{\eta}{\psi}_t -\sM {\bm D}_{\kap, \upgamma}\pare{D}\sM^{-1} \vect{\eta}{\psi} 
\\
&=
\OpBW{{\bm Q}_{\ \kap, \upgamma; \geqslant1}\pare{\eta,\psi; x, \xi} + {\bm B}_\upgamma \pare{\eta,\psi; x}\av{\xi} - \ii V_{\upgamma} \pare{\eta, \psi;x}\ \Id_{\mathbb{R}^2}\ \xi + {\bm A}_{\bra{0}}\pare{\eta,\psi; x, \xi}   } \vect{\eta}{\psi}
\\
&\quad + {\bm R}\pare{\eta, \psi} \vect{\eta}{\psi} . 
\end{aligned}
\end{equation}
Since $ U=\sM^{-1}\vect{\eta}{\psi} $ we can derive the evolution equation for $ U $ multiplying \ \eqref{eq:KH5} from the left for $ \sM^{-1} $ obtaining
\begin{equation}\label{eq:complex_conjugation_0}
\begin{aligned}
  &U_t - {\bm D}_{\kap, \upgamma}\pare{D}  U 
\\
&=
\sM^{-1} \OpBW{{\bm Q}_{\ \kap, \upgamma; \geqslant 1}\pare{\sM U; x, \xi} + {\bm B}_\upgamma \pare{\sM U; x}\av{\xi} - \ii V_{\upgamma} \pare{\sM U ;x}\ \Id_{\mathbb{R}^2}\ \xi + {\bm A}_{\bra{0}}\pare{\sM U; x, \xi}   } \sM U
\\
 &\quad+\sM^{-1}  {\bm R}\pare{\sM U} \sM U.  
\end{aligned}
\end{equation}
Recall that $ \sM^{-1} $ is explicitly defined in \eqref{eq:sM}.
We have the following  relation 
\begin{multline}\label{eq:conj_complex_Matrix}
\mathsf{M}^{-1} 
\begin{bmatrix}
A_1 & A_2 \\
A_3 & A_4
\end{bmatrix}
\mathsf{M}=
\\
\frac{1}{2}
\begin{bmatrix}
\mathsf{m}^{-1} A_1 \mathsf{m} + \mathsf{m} A_4 \mathsf{m}^{-1} + \ii \mathsf{m} A_3 \mathsf{m} - \ii \mathsf{m}^{-1} A_2 \mathsf{m}^{-1} & \mathsf{m}^{-1} A_1 \mathsf{m} - \mathsf{m} A_4 \mathsf{m}^{-1} + \ii \mathsf{m} A_3 \mathsf{m} + \ii \mathsf{m}^{-1} A_2 \mathsf{m}^{-1} \\
\mathsf{m}^{-1} A_1 \mathsf{m} - \mathsf{m} A_4 \mathsf{m}^{-1} - \ii \mathsf{m} A_3 \mathsf{m} - \ii \mathsf{m}^{-1} A_2 \mathsf{m}^{-1} & \mathsf{m}^{-1} A_1 \mathsf{m} + \mathsf{m} A_4 \mathsf{m}^{-1} - \ii \mathsf{m} A_3 \mathsf{m} + \ii \mathsf{m}^{-1} A_2 \mathsf{m}^{-1}
\end{bmatrix}.
\end{multline}
Notice now that given
\begin{align*}
a \pare{U;x,\xi} \in \Sigma\Gamma^{m_a}_{K, 0, 1}\bra{\epsilon_0, N}, 
&&
\alpha\pare{\xi}\in \tilde{\Gamma}^{m_\alpha}_0, 
&&
\beta\pare{\xi} \in \tilde{\Gamma}^{m_\beta}_0, 
\end{align*}
applying \Cref{prop compBW} we have that
\begin{equation}\label{eq:composizione_triplice}
\begin{aligned}
    \alpha\pare{D}\OpBW{a\pare{U;x,\xi}}\beta\pare{D} 
 &= \OpBW{a\pare{U;x, \xi}\alpha\pare{\xi}\beta\pare{\xi} + \frac{1}{2\ii} \ \pare{a\pare{U;x, \xi}}_x \pare{\alpha_\xi\pare{\xi} \beta\pare{\xi} - \alpha\pare{\xi}\beta_\xi\pare{\xi}}}
 \\
& \quad+\OpBW{\Sigma\Gamma^{m_a+m_\alpha +m_\beta-2}_{K, 0, 1}\bra{\epsilon_0, N}} + \Sigma\cR^{-\vr}_{K, 0, 1}\bra{\epsilon_0, N}.
\end{aligned}
\end{equation}
Then, applying \eqref{eq:conj_complex_Matrix}, we have that
\begin{equation}\label{eq:sq}
\sM^{-1} \OpBW{{\bm Q}_{\ \kap, \upgamma; \geqslant1}\pare{\sM U; x, \xi}}\sM = \frac{\ii}{2} \  \mathsf{m} \ \OpBW{\mathsf{q}\pare{U;x,\xi}} \mathsf{m} \ 
\begin{bmatrix}
1 & 1 \\
- 1 & - 1
\end{bmatrix}
,
\end{equation}
where $\mathsf{q}\pare{U;x,\xi}$ is the real valued symbol in \eqref{eq:q_expansion_xi} evaluated at $(\eta,\psi)=\sM U $. 
We apply \Cref{eq:composizione_triplice} and obtain that
\begin{equation}
\label{eq:sq1}
\mathsf{m} \ \OpBW{\mathsf{q}\pare{U;x,\xi}} \mathsf{m} =\OpBW{\mathsf{q}\pare{U;x, \xi} \mathsf{m}_{\kap, \upgamma}^2\pare{\xi} + A_{\bra{-\frac{1}{2}}}\pare{U;x, \xi}} + R\pare{U}. 
\end{equation}
Thus, combining \cref{eq:sq,eq:sq1}, we obtain
\begin{equation} \label{eq:complex_conjugation_1}
\begin{aligned}
   &\sM^{-1} \OpBW{{\bm Q}_{\ \kap, \upgamma; \geqslant1}\pare{\sM U; x, \xi}}\sM 
\\
&= \frac{\ii}{2} \  \pare{ \OpBW{\mathsf{q}\pare{U;x, \xi} \mathsf{m}_{\kap, \upgamma}^2\pare{\xi} + A_{\bra{-\frac{1}{2}}}\pare{U;x, \xi}} + R\pare{U} } \begin{bmatrix}
1 & 1 \\
- 1 & - 1
\end{bmatrix}
. 
\end{aligned}
\end{equation}
Next we apply \cref{eq:conj_complex_Matrix,eq:composizione_triplice} to the matrix-valued symbol $ {\bm B}_\upgamma \pare{\sM U ; x}\av{\xi} $ defined in \eqref{eq:pzcBgamma} and obtain that
\begin{equation}\label{eq:complex_conjugation_2}
\begin{aligned}
    &\sM^{-1} \OpBW{{\bm B}_\upgamma \pare{\sM U ; x}\av{\xi}}  \sM \\
    &=
\frac{1}{2}\OpBW{B_\upgamma \pare{\sM U ; x}\av{\xi}}
\begin{bmatrix}
0 & 1 \\ 1 & 0
\end{bmatrix}
+
 \frac{\ii}{4} \OpBW{B^2_\upgamma \pare{\sM U ; x}\av{\xi} \mathsf{m}_{\kap, \upgamma}^2 \pare{\xi} } 
\begin{bmatrix}
1 & 1 \\
-1 & -1
\end{bmatrix}
\\
&\quad+
\OpBW{{\bm A}_{\bra{0}}\pare{U;x, \xi}} + {\bm R}\pare{U}
.
\end{aligned}
\end{equation} 
 The same procedure give us the conjugations
\begin{equation}\label{eq:complex_conjugation_3}
\begin{aligned}
\sM^{-1} \OpBW{\ii V_{\upgamma} \pare{\sM U;x}\xi}\Id_{\mathbb{R}^2} \sM
= & \ 
\OpBW{\ii V_{\upgamma} \pare{\sM U;x}\xi}\Id_{\mathbb{C}^2} + 
\OpBW{{\bm A}_{\bra{0}}\pare{U;x, \xi}} + {\bm R}\pare{U}, 
\\
\sM^{-1}\OpBW{{\bm A}_{\bra{0}}\pare{\sM U ; x, \xi}} \sM = & \ \OpBW{{\bm A}_{\bra{0}}\pare{U;x, \xi}} + {\bm R}\pare{U},
\end{aligned}
\end{equation}
while \Cref{prop compo mop} implies that
\begin{equation}\label{eq:complex_conjugation_4}
\sM^{-1}{\bm R}\pare{\sM U }\sM \leadsto {\bm R}\pare{U}. 
\end{equation}
We plug \Cref{eq:complex_conjugation_1,eq:complex_conjugation_2,eq:complex_conjugation_3,eq:complex_conjugation_4}
in \Cref{eq:complex_conjugation_0} and obtain that
\begin{equation}\label{eq:complexNH}
U_t-{\bm D}_{\kap, \upgamma}\pare{D} U = \OpBW{\sum_{\mathsf{j}=0}^2 {\bm A}_{\frac{3-\mathsf{j}}{2}}^{\textnormal{NH}}\pare{U;x,\xi} + {\bm A}_{\bra{0}}\pare{U;x, \xi} } U +{\bm R}\pare{U}U, 
\end{equation}
with
\begin{equation*}
{\bm A}_{\frac{3-\mathsf{j}}{2}}^{\textnormal{NH}}\pare{U;x,\xi}\in \pare{ \Sigma\Gamma^{\frac{3-\mathsf{j}}{2}}_{K, 0, 1}\bra{\epsilon_0, N} }^{2\times 2}, 
\end{equation*}
and are explicitly defined as
\begin{equation}\label{eq:ANH3/2-j}
\begin{aligned}
{\bm A}_{\frac{3}{2}}^{\textnormal{NH}}\pare{U;x,\xi}
\defeq & \ 
\frac{\ii}{2} \  \mathsf{q}\pare{U;x, \xi} \mathsf{m}_{\kap, \upgamma}^2\pare{\xi} 
\begin{bmatrix}
1 & 1 \\
- 1 & - 1
\end{bmatrix},
\\
{\bm A}_{1}^{\textnormal{NH}}\pare{U;x,\xi}
\defeq & - \ii V_{\upgamma}\pare{\sM U;x}\xi
\begin{bmatrix}
1&0\\0&1
\end{bmatrix}
+
\frac{B_\upgamma\pare{\sM U;x}}{2} \ \av{\xi}
\begin{bmatrix}
0&1\\1&0
\end{bmatrix},
\\
{\bm A}_{\frac{1}{2}}^{\textnormal{NH}}\pare{U;x,\xi}
\defeq & \frac{\ii}{4} \  B^2_\upgamma \pare{\sM U ; x}\av{\xi} \mathsf{m}_{\kap, \upgamma}^2 \pare{\xi} 
\begin{bmatrix}
1 & 1 \\
- 1 & - 1
\end{bmatrix}.
\end{aligned}
\end{equation}
We now expand the symbols in \eqref{eq:ANH3/2-j} in decreasing para-differential orders using \cref{eq:q_expansion_xi,eq:m_expansion_xi}, thus obtaining that
\begin{equation}
\label{eq:expansion_homogeneous_products}
\begin{aligned}
\mathsf{q}\pare{U;x, \xi} \mathsf{m}_{\kap, \upgamma}^2\pare{\xi}  = & \ \sqrt{\frac{\kap}{2}} \ \mathtt{f} \pare{\eta;x} \av{\xi}^{\frac{3}{2}} + \frac{1}{\sqrt{2\kap}}\pare{\pare{\frac{\upgamma}{2}}^2 \mathtt{f}\pare{\eta;x} - w_\upgamma\pare{\sM U ; x }}\av{\xi}^{\frac{1}{2}} + A_{\bra{-\frac{1}{2}}}\pare{U;x, \xi},
\\
 B^2_\upgamma \pare{\sM U ; x}\av{\xi} \mathsf{m}_{\kap, \upgamma}^2\pare{\xi}  = & \ \frac{1}{\sqrt{2\kap}} \  B^2_\upgamma \pare{\sM U ; x}\av{\xi}^{\frac{1}{2}} + A_{\bra{-\frac{1}{2}}}\pare{U;x, \xi}.
\end{aligned}
\end{equation}
Thus, inserting \eqref{eq:expansion_homogeneous_products} in \eqref{eq:ANH3/2-j}, we obtain that
\begin{equation*}
\sum_{\mathsf{j}=0}^2 {\bm A}_{\frac{3-\mathsf{j}}{2}}^{\textnormal{NH}}\pare{U;x,\xi}
=
\sum_{\mathsf{j}=0}^2 {\bm A}_{\frac{3-\mathsf{j}}{2}}^{\textnormal{H}}\pare{U;x,\xi}
+{\bm A}_{\bra{-\frac{1}{2}}} \pare{U;x, \xi} , 
\end{equation*}
where
\begin{equation*}
\begin{aligned}
{\bm A}_{\frac{3}{2}}^{\textnormal{H}}\pare{U;x,\xi}
\defeq & \ 
\frac{\ii}{2} \  \sqrt{\frac{\kap}{2}} \ \mathtt{f} \pare{\eta;x} \av{\xi}^{\frac{3}{2}} 
\begin{bmatrix}
1 & 1 \\
- 1 & - 1
\end{bmatrix},
\\
{\bm A}_{1}^{\textnormal{H}}\pare{U;x,\xi}
\defeq & \ {\bm A}_{1}^{\textnormal{NH}}\pare{U;x,\xi},
\\
{\bm A}_{\frac{1}{2}}^{\textnormal{H}}\pare{U;x,\xi}
\defeq & \ 
\frac{\ii}{2}  \frac{1}{\sqrt{2\kap}} \pare{ \frac{  B^2_\upgamma \pare{\sM U ; x}}{2} + \pare{\frac{\upgamma}{2}}^2 \mathtt{f}\pare{\eta;x} - w_\upgamma\pare{\sM U ; x } } \av{\xi}^{\frac{1}{2}} \ \begin{bmatrix}
1 & 1 \\
- 1 & - 1
\end{bmatrix}.
\end{aligned} 
\end{equation*}
Therefore, \eqref{eq:complexNH} becomes
\begin{equation}\label{eq:complexH}
U_t-{\bm D}_{\kap, \upgamma}\pare{D} U = \OpBW{\sum_{\mathsf{j}=0}^2 {\bm A}_{\frac{3-\mathsf{j}}{2}}^{\textnormal{H}}\pare{U;x,\xi} + {\bm A}_{\bra{0}}\pare{U;x, \xi} } U +{\bm R}\pare{U}U.
\end{equation}
We now further refine the expression deduced in \eqref{eq:complexH} by expressing the leading term as a quasilinear perturbation of the unperturbed frequencies $ \omega_{\kap, \upgamma}\pare{ \xi } $, cf. \eqref{eq:disp_relation}. Using \eqref{eq:expansion_omega} we have that
\begin{equation*}
{\bm A}^{\textnormal{H}}_{\frac{3}{2}}\pare{U;x, \xi} = 
\ii \  \pare{ \frac{\mathtt{f}\pare{\eta;x}}{2} \omega_{\kap, \upgamma}\pare{ \xi }
+
\frac{1}{\sqrt{2\kap}} \frac{\mathtt{f}\pare{\eta;x}}{2} \pare{\frac{\upgamma}{2}}^2 \av{\xi}^{\frac{1}{2}} + {\bm A}_{\bra{-\frac{1}{2}}}\pare{U;x, \xi} }
\begin{bmatrix}
1 & 1 \\
- 1 & - 1
\end{bmatrix} , 
\end{equation*}
so that, defining
\begin{align*}
 {\bm J}_\mathbb{C} {\bm A}_{\frac{3}{2};\geqslant1}\pare{U, x}\defeq & \ \frac{\mathtt{f}\pare{\eta;x}}{2} \begin{bmatrix}
1 & 1 \\
- 1 & - 1
\end{bmatrix} \ \ii \  , 
\\
 {\bm J}_\mathbb{C} {\bm A}_{\frac{1}{2}} \pare{U, x} 
\defeq & \ 
\frac{1}{2}  \frac{1}{\sqrt{2\kap}} \pare{ \frac{  B^2_\upgamma \pare{\sM U ; x}}{2} + 2\pare{\frac{\upgamma}{2}}^2 \mathtt{f}\pare{\eta;x}  - w_\upgamma\pare{\sM U ; x } } \ \begin{bmatrix}
1 & 1 \\
- 1 & - 1
\end{bmatrix} \ \ii,\\
{\bm J}_\mathbb{C} {\bm A}_{1} \pare{U;x, \xi}\defeq & {\bm A}^{\textnormal{H}}_{1}\pare{U;x, \xi} ,
\end{align*}
we transform \eqref{eq:complexH} into
\begin{equation}\label{eq:KH6bis}
\begin{aligned}
    &U_t -{\bm D}_{\kap, \upgamma}\pare{D} U  \\
 &=   {\bm J}_\mathbb{C} \  \OpBW{{\bm A}_{\frac{3}{2};\geqslant1}\pare{U;x} \  \omega_{\kap, \upgamma}\pare{ \xi }  + {\bm A}_{1}\pare{U;x, \xi} + {\bm A}_{\frac{1}{2}}\pare{U;x } \  \av{\xi}^{\frac{1}{2}} + {\bm A}_{\bra{0}}\pare{U;x, \xi}} U +{\bm R}\pare{U} U .
\end{aligned} 
\end{equation}
Next, an explicit computation shows that
\begin{equation*}
{\bm D}_{\kap, \upgamma}\pare{\xi} = \omega_{\kap, \upgamma}\pare{ \xi } \  {\bm J}_\mathbb{C} 
\begin{bmatrix}
0&1\\1&0
\end{bmatrix}
 .
\end{equation*}
Thus, defining 
\begin{align*}
{\bm A}_{\frac{3}{2}}\pare{U;x}\defeq  \begin{bmatrix}
0&1\\1&0
\end{bmatrix} + {\bm A}_{\frac{3}{2};\geqslant1}\pare{U, x},
\end{align*}
{it remains only to show that \eqref{eq:KH6bis} is a complex Hamiltonian system and that ${\bm J}_\C{\bm A}_{\bra{0}}\pare{U;x, \xi} $ is linearly symplectic. Indeed complex variables transformation $ {\bm\sM}_{\kap, \upgamma}^{-1}$ in \eqref{eq:complex_var} is the composition of the complex transformation $ \cC^{-1}$ in \eqref{change complex} and the real symplectic map ${\bm S}_{\kap,\upgamma}^{-1}(D)$, see \eqref{def bmS}. Then the equation for $U$ has the complex Hamiltonian form in \eqref{new complexHAM}. Then we apply Lemma \ref{HS:repre} to each homogeneous component in \eqref{eq:KH6bis}. As the explicit positive order symbols are linearly symplectic by direct inspections, we eventually replace  ${\bm A}_{\bra{0}}\pare{U;x, \xi}$ as in \eqref{aunoa} to get a linearly symplectic matrix of symbols. This conclude the proof.}
\end{step}

\end{proof}

\section{Block diagonalization and reduction to constant coefficients}\label{sec:diag_cc}
In this section, we perform spectrally localized, linearly symplectic, and bounded transformations to conjugate the complex Kelvin-Helmholtz system \eqref{eq:KH6} into a diagonal constant-coefficient system, up to a smoothing remainder. Specifically, we prove the following.

\begin{prop}\label{prop:constantcoeff}
Let $ N\in \mathbb{N}$ and $ \vr > 3 \pare{N+1}$, there exists $ \underline{K'} > 0 $ such that for any $ K > \underline{K'} $ there are $ s_0 > 0 $, $ \epsilon_0 > 0 $ such that, for any solution $ U\in \Ball{K}{s_0} $ solution of \eqref{eq:KH6} there exists a real-to-real invertible matrix of spectrally localized maps $ {\bm B}\pare{U;t} $ such that 
\begin{enumerate}
\item $ {\bm B}\pare{U;t}, \ {\bm B}\pare{U;t}^{-1} \in \pare{\cS^0_{K, \underline{K'}-1, 0}\bra{\epsilon_0, N}}^{2\times 2} $ are linearly symplectic up to homogeneity $N$, according to \Cref{linsymphomoN}. Moreover, $ {\bm B}\pare{U;t} -\Id \in \pare{ \Sigma \cS^{\frac{3}{2}\pare{N+1}}_{K, \underline{K'}-1, 1}\bra{\epsilon_0, N} }^{2\times 2};$
\item The variable $ W\defeq {\bm B}\pare{U;t} U $ solves 
\begin{equation}\label{eq:constcoeff}
W_t = \Opvec{\ii \ \bra{\pare{1+\mathpzc{v}\pare{U;t}}\omega_{\kap, \upgamma}\pare{ \xi } + \mathpzc{V}_\upgamma\pare{U;t}\xi + \mathpzc{b}_{\frac{1}{2}}\pare{U;t}\av{\xi}^{\frac{1}{2}} + \mathpzc{b}_{0}\pare{U;t,\xi}} } W + \pmb{\mathsf{R}}\pare{U;t}W, 
\end{equation}
where
\begin{itemize}
\item $ \omega_{\kap, \upgamma}\pare{ \xi } \in \tilde \Gamma^{\frac{3}{2}}_0 $ is defined in \eqref{eq:disp_relation};

\item $ \mathpzc{v}\pare{U;t}\in \Sigma\cF^{\mathbb{R}}_{K, 0, 2}\bra{\epsilon_0, N} $ and is independent of $ x $;

\item  $ \mathpzc{V}_\upgamma \pare{U;t}\in \Sigma\cF^{\mathbb{R}}_{K, 1, 2}\bra{\epsilon_0, N} $ and is independent of $ x $;

\item  $ \mathpzc{b}_{\frac{1}{2}} \pare{U;t}\in \Sigma\cF^{\mathbb{R}}_{K, 2, 2}\bra{\epsilon_0, N} $ and is independent of $ x $;

\item  $ \mathpzc{b}_{0} \pare{U;t,\xi}\in \Sigma\Gamma^0 _{K, \underline{K'}, 2}\bra{\epsilon_0, N} $, is independent of $ x $ and such that $\Im \mathpzc{b}_{0} \pare{U;t,\xi}\in \Gamma^0 _{K, \underline{K'}, N+1}\bra{\epsilon_0} $;

\item $ \pmb{\mathsf{R}}\pare{U;t} \in \pare{\Sigma \cR^{-\vr + 3\pare{N+1}}_{K, \underline{K'}, 1}\bra{\epsilon_0, N}}^{2\times 2} $.

\end{itemize}
\end{enumerate}
\end{prop}
The rest of the section is devoted to the proof of \Cref{prop:constantcoeff}. The procedure is now rather classical so we state the main steps, for the missing details, we refer the interested reader to  \cite[Section 6]{BMM2022}.
\subsection{A complex Alinhac Good Unknown}

We first perform a change of variable, known as {\it Alinhac good unknown} that cancels the anti-diagonal terms of order one in \eqref{eq:KH6}.

\begin{lemma}[Alinhac Good Unknown for the Kelvin-Helmholtz equations in complex coordinates]\label{lem:GA}
Let $ N\in \mathbb{N} $, $\kap > 0$, $ \upgamma\in \mathbb{R} $ and $ \varrho \geqslant0 $, for any $ K\in \mathbb{N} $ there exists $ s_0 >0 $ and $ \epsilon_0 > 0 $ such that if $ U\in B^K_{s_0}\pare{I;\epsilon_0} $ solves \eqref{eq:KH6} there exists a real-to-real matrix of operators $\mathbf{G}(U)$ such that
\begin{enumerate}

\item \label{item:GU_B} 
$ {\bm G}\pare{U}, \ {\bm G}\pare{U}^{-1} \in \pare{\cS^0_{K,0, 0}\bra{\epsilon_0, N}}^{2\times 2} $ are linearly symplectic, according to \Cref{def:LS}. Moreover, $ {\bm G}\pare{U} -\Id \in \pare{ \Sigma \cS^{-\frac{1}{2}}_{K, 0, 1}\bra{\epsilon_0, N} }^{2\times 2}$;
\item \label{item:GU_C} The variable $ U_0\defeq {\bm G}\pare{U} U $ solves
\begin{equation} \label{eq:KH7}
\begin{aligned}
    \partial_t U_0
 &=   {\bm J}_\mathbb{C} \  \OpBW{{\bm A}_{\frac{3}{2}}^{(0)}\pare{U;x} \  \omega_{\kap, \upgamma}\pare{ \xi }   +  {\bm A}_{\frac{1}{2}}^{(0)} \pare{U;x } \  \av{\xi}^{\frac{1}{2}} + {\bm A}_{\bra{0;1}}\pare{U;x, \xi}} U_0
 \\
&\quad- \Opvec{\ii V_\upgamma\pare{\sM U;x}\xi} U_0 
 + {\bm R}\pare{U} U_0 ,
\end{aligned}
\end{equation}
where
\begin{itemize}
\item the symbols $ {\bm A}_{\frac{3}{2}}^{(0)} \defeq {\bm A}_{\frac{3}{2}}, \ \omega_{\kap, \upgamma} $ and $ V_\upgamma $ are the same as in the statement of \Cref{prop:KH_complex} and $ {\bm A}_{\bra{0;1}} $, ${\bm R}$ are as in \Cref{notation:A_refugium_peccatorum};
\item $ {\bm A}^{(0)}_{\frac{1}{2}}  \in \pare{\Sigma\cF^{\mathbb{R}}_{K, 0, 1}\bra{\epsilon_0, N}}^{2\times 2} $ is defined as
\begin{equation*}
{\bm A}_{\frac{1}{2}}^{\pare{0}} \pare{U; x} 
\defeq
A_{\frac{1}{2}}^{\pare{0}}\pare{U;x}  \ \begin{bmatrix}
1 & 1 \\
1 & 1
\end{bmatrix}
\defeq 
\frac{1}{2}  \frac{1}{\sqrt{2\kap}} \pare{   \frac{\upgamma^2}{2} \mathtt{f}\pare{\eta;x}  - w_\upgamma\pare{\sM U ; x } }  \ \begin{bmatrix}
1 & 1 \\
1 & 1
\end{bmatrix} , 
\end{equation*}
where $ w_\upgamma\pare{\eta, \psi; x} $ is defined in \cref{eq:wgamma}. 
\end{itemize}
\end{enumerate}

\end{lemma}

\begin{proof}
The procedure that we use here is the same as the one proved in \cite[Section 6.1]{BMM2022}, thus we omit to perform detailed computations, in particular the \Cref{item:GU_B} is true by direct computations and the fact that ${\bm G}\pare{U}-\Id \in \pare{\Sigma\cS^0_{K,0,1}\bra{\epsilon_0, N}}^{2\times 2}$. This latter fact is a consequence of \Cref{lemactionpara}. The only difference, compared to \cite{BMM2022}, is the symbol  $ {\bm A}_{\frac{1}{2}}\pare{U;x}\av{\xi}^{\frac{1}{2}} \in \Sigma \Gamma^{\frac{1}{2}}_{K, 0, 1}\bra{\epsilon_0, N} $ appearing in \eqref{eq:KH6} and defined explicitly in \eqref{eq:A1/2}, which shall be treated in detail. Let us define the transformation
\begin{equation}\label{eq:GU_map}
{\bm G}\pare{U}\defeq \Id - \frac{\ii}{2}  
\begin{bmatrix}
1&1\\-1&-1
\end{bmatrix}
\OpBW{B_\upgamma \pare{\sM U;x} \mathsf{m}_{\kap, \upgamma}^2 \pare{{\xi}}}.
\end{equation}
The fact that $ {\bm G}\pare{U} $ is a nilpotent perturbation of the identity allows us to compute its inverse as
\begin{equation*}
{\bm G}\pare{U}^{-1}\defeq \Id + \frac{\ii}{2} 
\begin{bmatrix}
1&1\\-1&-1
\end{bmatrix}
\OpBW{B_\upgamma \pare{\sM U;x} \mathsf{m}_{\kap, \upgamma}^2 \pare{{\xi}}}. 
\end{equation*}
Thus, defining 
$
U_0 \defeq {\bm G}\pare{U} U 
$,
and using the conjugation rules
\begin{footnotesize}
\begin{equation}\label{eq:conj_rules_GU}
\begin{aligned}
{\bm G}\pare{U}{\bm J}_\mathbb{C} \OpBW{{\bm A}_{\frac{3}{2}}\omega_{\kap,\upgamma}\pare{\xi}}{\bm G}\pare{U}^{-1} =
& \ 
 {\bm J}_\mathbb{C} \OpBW{{\bm A}_{\frac{3}{2}}\omega_{\kap,\upgamma}\pare{\xi} - \ii \frac{B_\upgamma}{2}\av{\xi}
\begin{bmatrix}
1&0\\0&-1
\end{bmatrix}
+\frac{B^2_\upgamma  }{4}\ \mathsf{m}^2_{\kap, \upgamma}\pare{\xi}
\begin{bmatrix}
1&1\\1&1
\end{bmatrix}
+{\bm A}_{\bra{0}}
 } +{\bm R}\pare{U},
 \\
 {\bm G}\pare{U} \Opvec{-\ii V_\upgamma \xi}\Id_{\mathbb{C}^2} {\bm G}\pare{U}^{-1}
 = & \ 
 \Opvec{-\ii V_\upgamma \xi}\Id_{\mathbb{C}^2},
 \\
 {\bm G}\pare{U} {\bm J}_\mathbb{C} \OpBW{ \ii \frac{B_\upgamma}{2}\av{\xi}
\begin{bmatrix}
1&0\\0&-1
\end{bmatrix}
 } 
 {\bm G}\pare{U}^{-1}
 = & \ 
  {\bm J}_\mathbb{C} \OpBW{ \ii \frac{B_\upgamma}{2}\av{\xi}
\begin{bmatrix}
1&0\\0&-1
\end{bmatrix}
-\frac{B^2_\upgamma  }{2}\ \mathsf{m}^2_{\kap, \upgamma}\pare{\xi}
\begin{bmatrix}
1&1\\1&1
\end{bmatrix}
+{\bm A}_{\bra{0}}
 } 
 + {\bm R}\pare{U},
\end{aligned}
\end{equation}
\end{footnotesize}
we get, using also \cref{eq:conj_rules_GU,eq:m_expansion_xi}, 
\begin{equation}\label{eq:conj_rules_GU_1}
\begin{aligned}
    {\bm G}\pare{U}{\bm J}_\mathbb{C}\OpBW{{\bm A}_{\frac{3}{2}}\omega_{\kap,\upgamma}\pare{\xi} + {\bm A}_1} {\bm G}\pare{U}^{-1} &= {\bm J}_\mathbb{C} \OpBW{{\bm A}_{\frac{3}{2}}\omega_{\kap,\upgamma}\pare{\xi}  -\frac{1}{\sqrt{2\kap}}\frac{B^2_\upgamma}{4} \ \av{\xi}^{\frac{1}{2}}
\begin{bmatrix}
1&1\\1&1
\end{bmatrix}
+
{\bm A}_{\bra{0}}
}\\
&\quad+ \Opvec{-\ii V_\upgamma \xi}+{\bm R}\pare{U}.
\end{aligned}
\end{equation}
Due to the particular nilpotent structure of ${\bm J}_\mathbb{C}\OpBW{{\bm A}_{\frac{1}{2}}\pare{U;x}\av{\xi}^{\frac{1}{2}}} $ and ${\bm G(U)}-\Id$,
the conjugation for the symbol of order $ 1/2 $ explicitly given by
\begin{equation}\label{eq:conj_rules_GU_2}
{\bm G}\pare{U}{\bm J}_\mathbb{C}\OpBW{{\bm A}_{\frac{1}{2}}\pare{U;x}\av{\xi}^{\frac{1}{2}}} {\bm G}\pare{U}^{-1} = {\bm J}_\mathbb{C}\OpBW{{\bm A}_{\frac{1}{2}}\pare{U;x}\av{\xi}^{\frac{1}{2}}} .
\end{equation}
Remark that there is a cancellation of the terms of the form $|\xi|^{\frac{1}{2}}$ between \eqref{eq:A1/2} and \eqref{eq:conj_rules_GU_1} leading the the expression of ${\bm A}^{(0)}_{\frac{1}{2}}$ in the statement. Moreover, we have
\begin{equation}\label{eq:conj_rules_GU_3}
    {\bm G}\pare{U} U_t = \partial_t U_0 - \pare{\partial_t {\bm G}\pare{U}}  {\bm G}\pare{U}^{-1} U_0.
\end{equation}
Hence, from \Cref{item:GU_B,prop compBW}, we have that
\begin{equation}\label{eq:conj_rules_GU_4}
   \pare{\partial_t {\bm G}\pare{U}}  {\bm G}\pare{U}^{-1} = \OpBW{{\bm A}_{\bra{-\frac{1}{2};1}}\pare{U;x, \xi}}, \quad {\bm A}_{\bra{-\frac{1}{2};1}}\triangleq  - \frac{\ii}{2}  
\begin{bmatrix}
1&1\\-1&-1
\end{bmatrix}
{\pa_t B_\upgamma \pare{\sM U;x} \mathsf{m}_{\kap, \upgamma}^2 \pare{{\xi}}}  .  
\end{equation}
The conjugations proved in \cref{eq:conj_rules_GU_1,eq:conj_rules_GU_2,eq:conj_rules_GU_3,eq:conj_rules_GU_4} and the computations performed in \cite[Section 6.1]{BMM2022} conclude thus the proof of \Cref{lem:GA}.  We remark that the zero-th order matrix ${\bm A}_{[0,1]}$ is the sum of the linearly Hamiltonian matrix ${\bm A}_{[-\frac12,1]}$ and the contributions coming from ${\bm A}_{[0]}$ in \cref{eq:conj_rules_GU,eq:conj_rules_GU_1}. To prove that ${\bm J}_\C{\bm A}_{[0]}$ is linearly Hamiltonian up to homogeneity $N$ we then apply Lemma \ref{spezzamento} to each homogeneous component of the operators in \cref{eq:conj_rules_GU,eq:conj_rules_GU_1} which are linearly Hamiltonian thanks to \Cref{lem conj linham}.
\end{proof}

\subsection{Diagonalization at the highest order}

In the present section, we diagonalize the quasi-linear contribution 
$$ {\bm J}_\mathbb{C} \OpBW{{\bm A}^{(0)}_{\frac{3}{2}}\pare{U;x}\omega_{\kap, \upgamma}\pare{ \xi }},$$ which reduces to the diagonalization of the matrix
\begin{equation}\label{eq:JbCA3/20}
{\bm J}_\mathbb{C} {\bm A}^{(0)}_{\frac{3}{2}}\pare{U;x} = \ii \begin{bmatrix}
-\pare{1+\mathtt{f}\pare{U;x}} &  - \mathtt{f}\pare{U;x}
\\
\mathtt{f}\pare{U;x} & 1+ \mathtt{f}\pare{U;x}
\end{bmatrix} , 
\end{equation}
where $ \mathtt{f}\pare{U;x} $ is defined in \eqref{eq:paralin_curvature}. 
The eigenvalues of \eqref{eq:JbCA3/20} are given by $ \pm \ii \ \lambda\pare{U;x} $ where 
\begin{equation}\label{def lbda}
    \lambda\pare{U;x}\defeq\sqrt{1+2\mathtt{f}\pare{U;x}} 
\end{equation} and we have that $ \lambda\pare{U;x}-1\in \Sigma\cF^{\mathbb{R}}_{K, 0, 1}\bra{\epsilon_0, N} $. Defining
\begin{align*}
\mathtt{h}\defeq \frac{1+\mathtt{f}+\lambda}{\sqrt{\pare{1+\mathtt{f}+\lambda}^2 - \mathtt{f}^2}} \ , 
&&
\mathtt{g}\defeq \frac{-\mathtt{f}}{\sqrt{\pare{1+\mathtt{f}+\lambda}^2 - \mathtt{f}^2}} , 
\end{align*}
the symmetric and symplectic matrix
\begin{align}\label{eq:ttF}
{\bm F} \defeq \begin{bmatrix}
\mathtt{h} & \mathtt{g} \\
\mathtt{g} & \mathtt{h}
\end{bmatrix} \in \pare{  \Sigma\cF^{\mathbb{R}}_{K, 0, 0}\bra{\epsilon_0, N}  }^{2\times 2}\ ,
&& 
{\bm F}^{-1} \defeq \begin{bmatrix}
\mathtt{h} & - \mathtt{g} \\
- \mathtt{g} & \mathtt{h}
\end{bmatrix} \in \pare{  \Sigma\cF^{\mathbb{R}}_{K, 0, 0}\bra{\epsilon_0, N}  }^{2\times 2} \ ,
\end{align}
which are such that
\begin{equation}\label{eq:ttF-Id}
{\bm F} - \Id, \ {\bm F}^{-1}-\Id \in \pare{  \Sigma\cF^{\mathbb{R}}_{K, 0, 1}\bra{\epsilon_0, N}  }^{2\times 2} \ , 
\end{equation}
diagonalize \eqref{eq:JbCA3/20} in the sense that
\begin{equation*}
{\bm F}^{-1} {\bm J}_\mathbb{C} {\bm A}^{(0)}_{\frac{3}{2}} {\bm F} =
\begin{bmatrix}
-\ii \lambda & 0 \\ 0 & \ii \lambda
\end{bmatrix} . 
\end{equation*}
We prove the following result.
\begin{lemma}\label{lem:block_diagonalization}
Let $N\in \mathbb{N}$ and $\vr> 0$,  for any $K \in \mathbb{N}^*$ there are $s_0 > 0$, $\epsilon_0 > 0$ such that, for any solution $U \in \Ball{K}{s_0}$ of \eqref{eq:KH7} there exists a real-to-real invertible matrix of spectrally localized maps ${\bm \Psi}_1(U)$ such that 
\begin{enumerate}
\item  ${\bm \Psi}_1(U), \ {\bm \Psi}_1(U)^{-1} \in \pare{ \cS^0_{K,0, 0}\bra{\epsilon_0} }^{2\times 2}$ are  linearly symplectic as per \Cref{def:LS}. Moreover, ${\bm \Psi}_1(U) - \Id \in \pare{ \Sigma \cS^{0}_{K, 0, 1}\bra{\epsilon_0, N} }^{2\times 2}.$
\item The variable $U_1 \triangleq {\bm \Psi}_1(U)U_0$ solves 
\begin{equation}\label{eq:block_diagonalized}
\begin{aligned}
    \partial_t U_1 &=  \Opvec{  \ii\ \lambda \pare{U;x} \omega_{\kap, \upgamma}\pare{ \xi }
  - \ii \ V_\upgamma ^{(1)}\pare{U;x}\xi  } U_1 \\
&\quad+ {\bm J}_\mathbb{C} \OpBW{{\bm A}_{\frac{1}{2}}^{(1)}\pare{U;x}\av{\xi}^{\frac{1}{2}} + {\bm A}_{\bra{0;1}}\pare{U;x,\xi}} U_1 + {\bm R}_{\bra{1}}\pare{U}U_1,
\end{aligned}
\end{equation}
where
\begin{itemize}
\item $\omega_{\kap, \upgamma}\pare{ \xi } \in \Gamma^{\frac{3}{2}}_0$ is defined in \eqref{eq:disp_relation};
\item  $ \lambda\pare{U;x} $ is defined in \eqref{def lbda};
\item $V_\upgamma ^{(1)} \pare{U;x}\in  \Sigma\cF^{\mathbb{R}}_{K, 0,1}\bra{\epsilon_0, N} $;
\item ${\bm A}_{\frac{1}{2}}^{(1)} \pare{U;x}\defeq \pare{\mathtt{h}\pare{U;x} + \mathtt{g}\pare{U;x}}^2 {\bm A}_{\frac{1}{2}}^{(0)}\pare{U;x} \in \pare{ \Sigma\cF^{\mathbb{R}}_{K, 0,1}\bra{\epsilon_0, N} }^{2\times 2}$.
\end{itemize}
\end{enumerate}
\end{lemma}

\begin{proof}
Arguing as in \cite[Lemma 6.4]{BMM2022} the 1-flow $ {\bm \Psi}_1\pare{U}\defeq\left. {\bm \Psi}^\tau \pare{U}\right|_{\tau=0} $ of 
\begin{equation}\label{log:flow}
\system{
\begin{aligned}
& \partial_\tau {\bm \Psi}^\tau\pare{U} = {\bm J}_\mathbb{C} \OpBW{
\begin{bmatrix}
\ii \ \log\pare{\mathtt{h}\pare{U;x} + \mathtt{g}\pare{U;x}} & 0 
\\
0 & -\ii \ \log\pare{\mathtt{h}\pare{U;x} + \mathtt{g}\pare{U;x}} 
\end{bmatrix}
}{\bm \Psi}^\tau,
\\
& {\bm \Psi}^0\pare{U} = \Id
\end{aligned}
} 
\end{equation}
is such that
\begin{align}\label{eq:Psi1symbol}
{\bm \Psi}_1\pare{U} = \OpBW{{\bm F}^{-1}\pare{U;x}} + {\bm R}\pare{U}, 
&&
{\bm \Psi}_1\pare{U}^{-1} = \OpBW{{\bm F} \pare{U;x}} + {\bm R}\pare{U}. 
\end{align}
Since the generator in \eqref{log:flow} is linearly Hamiltonian, \Cref{flow}  ensures that ${\bm \Psi}_1\pare{U}$ is a linearly symplectic, spectrally localized map in $ \pare{ \Sigma \cS^0_{K,K',1}[r,N] }^{2\times 2}$. 
Thus, the new variable $ U_1\defeq {\bm \Psi}_1\pare{U} U_0 $ solves (cf. \eqref{eq:KH7})
\begin{multline} \label{eq:DHO_0}
\partial_t U_1
 =   {\bm \Psi}_1\pare{U} {\bm J}_\mathbb{C} \  \OpBW{{\bm A}_{\frac{3}{2}}^{(0)}\pare{U;x} \  \omega_{\kap, \upgamma}\pare{ \xi }   +  {\bm A}_{\frac{1}{2}}^{(0)} \pare{U;x } \  \av{\xi}^{\frac{1}{2}} + {\bm A}_{\bra{0;1}}\pare{U;x, \xi}} {\bm \Psi}_1\pare{U}^{-1} U_1
 \\
 +\partial_t {\bm \Psi}_1\pare{U} {\bm \Psi}_1\pare{U}^{-1} U_1
- {\bm \Psi}_1\pare{U}\OpBW{\ii V_\upgamma\pare{\sM U;x}\xi} {\bm \Psi}_1\pare{U}^{-1} U_1
 + {\bm \Psi}_1\pare{U}{\bm R}\pare{U}{\bm \Psi}_1\pare{U}^{-1} U_1.  
\end{multline}
Following the same computations as in \cite[Eq. (6.3)]{BMM2022} we have that
\begin{equation}\label{eq:DHO_1}
\begin{aligned}
    &{\bm \Psi}_1\pare{U} {\bm J}_\mathbb{C} \  \OpBW{{\bm A}_{\frac{3}{2}}^{(0)}\pare{U;x} \  \omega_{\kap, \upgamma}\pare{ \xi } } {\bm \Psi}_1\pare{U}^{-1}
\\
&=
\OpBW{
 \ii \lambda\pare{U;x}\omega_{\kap, \upgamma}\pare{ \xi }
 \begin{bmatrix}
-1 & 0 \\ 0 & 1
\end{bmatrix}
+
A_{\bra{-\frac{1}{2};0}}\pare{U;x,\xi}
} 
+ {\bm R}\pare{U}. 
\end{aligned}
\end{equation}
Similar computations give us the conjugation 
\begin{equation}\label{eq:DHO_2}
\begin{aligned}
    &{\bm \Psi}_1\pare{U} {\bm J}_\mathbb{C} \  \OpBW{{\bm A}_{\frac{1}{2}}^{(0)}\pare{U;x} \  \av{\xi}^{\frac{1}{2}} } {\bm \Psi}_1\pare{U}^{-1}
\\
&=
{\bm J}_\mathbb{C}\OpBW{\pare{\mathtt{h}\pare{U;x} + \mathtt{g}\pare{U;x}}^2 {\bm A}^{(0)}_{\frac{1}{2}} \pare{U;x} \  \av{\xi}^{\frac{1}{2}}  + A_{\bra{-\frac{1}{2}; 0}}\pare{U;x, \xi}} + {\bm R}\pare{U}
\end{aligned} 
\end{equation} 
and 
\begin{equation}
\label{eq:DHO_3}
\begin{aligned}
{\bm \Psi}_1\pare{U}\OpBW{ {\bm A}_{\bra{0;1}}\pare{ U;x,\xi}} {\bm \Psi}_1\pare{U}^{-1} = & \ \OpBW{ {\bm A}_{\bra{0;1}}\pare{ U;x,\xi}} + {\bm R}\pare{U}, 
\end{aligned}
\end{equation}
while from \cref{eq:Psi1symbol,eq:ttF-Id} we obtain that
\begin{align}\label{eq:DHO_4}
\partial_t {\bm \Psi}_1\pare{U} {\bm \Psi}_1\pare{U}^{-1} = \OpBW{{\bm A}_{\bra{0;1}}\pare{U;x, \xi}} + {\bm R}_{\bra{1}}\pare{U}, 
\end{align}
and standard symbolic computations (cf. \Cref{prop compBW,eq:Psi1symbol}) give that
\begin{align}\label{eq:DHO_5}
{\bm \Psi}_1\pare{U}\Opvec{\ii V_\upgamma \pare{U;x}\xi} {\bm \Psi}_1\pare{U}^{-1} =  \Opvec{\ii V_\upgamma ^{(1)} \pare{U;x}\xi} + {\bm R}\pare{U}, 
&&
V_\upgamma ^{(1)} \pare{U;x}\in \Sigma \cF^{\mathbb{R}}_{K, 0, 1}\bra{\epsilon_0, N}.
\end{align}
We conclude plugging \cref{eq:DHO_1,eq:DHO_2,eq:DHO_3,eq:DHO_4,eq:DHO_5}  in \cref{eq:DHO_0}. {We remark that the zero-th order matrix ${\bm A}_{[0,1]}$ is the sum of  ${\bm A}_{[-\frac12;0]}$ form \cref{eq:DHO_1,eq:DHO_2} and  ${\bm A}_{[0;1]}$ from  \cref{eq:DHO_3,eq:DHO_4}. To prove that ${\bm J}_\C{\bm A}_{[0;1]}$ is linearly symplectic we then apply Lemma \ref{spezzamento} to each homogeneous components of the spectrally localized operators in \cref{eq:DHO_1,eq:DHO_2,eq:DHO_3,eq:DHO_4}, which are linearly Hamiltonian thanks to \Cref{lem conj linham}.}
\end{proof}

\subsection{Reduction to constant coefficients at the highest order}

\begin{lemma}[Reduction of the highest order]\label{lem:reduction_highest_order}
Let $N \in \mathbb{N}$ and $\vr > 2(N + 1)$. Then for any $K \in \mathbb{N}^*$ there are $s_0 > 0$, $\epsilon_0 > 0$ such that for any solution $U \in \Ball{K}{s_0}$ of \eqref{eq:block_diagonalized}, there exists a real-to-real invertible matrix of spectrally localized maps ${\bm \Psi}_2(U)$ satisfying 

\begin{enumerate}
\item   ${\bm \Psi}_2(U),\, {\bm \Psi}_2(U) \in\pare{ \cS^0_{K,0,0}[\epsilon_0] }^{2\times 2} $ are   linearly symplectic as per \Cref{def:LS}. Moreover,
${\bm \Psi}_2(U) - \Id \in \pare{\Sigma\cS^{N+1}_{K,0,2}[\epsilon_0,N] }^{2\times 2}$;

\item  The variable $U_2 \triangleq {\bm \Psi}_2(U)U_1$ solves the system
\begin{equation}\label{eq:KH8}
\begin{aligned}
    \partial_t U_2 &= \Opvec{ \ii \  \bra{\pare{ 1 + \mathpzc{v}(U) }\omega_{\kap, \upgamma}\pare{ \xi } - V_\upgamma ^{(2)}(U;x)\xi}}U_2 \\
&\quad+ {\bm J}_\mathbb{C}\OpBW{{\bm A}_{\frac{1}{2}}^{(2)}\pare{U;x}\av{\xi}^{\frac{1}{2}} +   {\bm A}_{\bra{0;1}}(U;x,\xi)}U_2 + {\bf R}(U)U_2,
\end{aligned}
\end{equation}
where
\begin{itemize}
\item $\mathpzc{v}(U)$ is a $x$-independent function in $\Sigma\cF^{\mathbb{R}}_{K,0,2}[\epsilon_0,N]$ and $\omega_{\kap,\upgamma}(\xi)$ is defined in \eqref{eq:disp_relation};
\item $V_\upgamma ^{(2)} (U;x)$ is a real valued function in $\pare{ \Sigma\cF^{\mathbb{R}}_{K,1,1}[\epsilon_0,N] }^2$;
\item $ {\bm A}_{\frac{1}{2}}^{(2)}\pare{U;x} \defeq   A_{\frac{1}{2}}^{(2)}\pare{U;x} \begin{bmatrix}
    1&1\\1&1
\end{bmatrix} $ where $ A_{\frac{1}{2}}^{(2)}\pare{U;x} \in \pare{\Sigma \cF^{\mathbb{R}}_{K, 0, 1}\bra{\epsilon_0, N}}^{2\times 2} $; 
\item  $\mathbf{R}\pare{U}\in \Sigma \cR^{-\vr +2\pare{N+1}}_{K, 1, 1}\bra{\epsilon_0,N} $.
\end{itemize}
\end{enumerate}
\end{lemma}
\begin{proof}
We refer the interested reader to \cite[Lemma 6.7]{BMM2022}, the only difference being the conjugation of the term of order $ 1/2 $, which is achieved by standard paracomposition theorems, such as \cite[Theorem 3.27]{BD2018}. Notice that the conjugation worsens the regularizing properties of the smoothing operator in \cref{eq:KH8}. 
\end{proof}

\subsection{Diagonalization up to smoothing remainders}
In this section we block-diagonalize the Kelvin-Helmholtz system up to a smoothing remainder.
\begin{lemma}[Diagonalization to arbitrary order]\label{lem:diagonalization_arbitrary_order}
Let $N \in \mathbb{N}$ and $\vr \gg N$. Then for any $n \in \mathbb{N}\cup\set{-1} $, there exists $K' \triangleq K'(n) \geqslant0$ such that for all $K \geqslant K' + 1$, there exist $s_0 > 0$ and $\epsilon_0 > 0$ such that for any solution $U \in \BallR{K}{s_0}$ of \cref{eq:KH8}, there exists a real-to-real invertible matrix of spectrally localized maps ${\bm \Phi}_n(U)$ satisfying
\begin{enumerate}
    \item   ${\bm \Phi}_n(U),\, {\bm \Phi}_n(U)^{-1} \in\pare{ \cS^0_{K,0,0}[\epsilon_0] }^{2\times 2} $ are   linearly symplectic up to homogeneity $N$ as per \Cref{linsymphomoN}. Moreover,
${\bm \Phi}_n(U) - \Id \in  \pare{ \Sigma \cS^0_{K,K',1}[r,N] }^{2\times 2} $;
    \item  The variable $U_{n+3} \defeq {\bm \Phi}_n(U) U_{2}$ solves
    \begin{equation}\label{eq:KH9}
    \begin{aligned}
        \partial_t U_{n+4} &= \Opvec{  \ii \  \bra{ \pare{1 + \mathpzc{v}(U)} \omega_{\kap, \upgamma}\pare{ \xi } - V_\upgamma ^{(2)} (U;x) \xi + a^{\pare{n}}_{\xfrac{1}{2}}\pare{U;x}\av{\xi}^{\frac{1}{2}} +  a^{(n)}_0(U;x,\xi)} } U_{n+3}
    \\
    &\quad+ {\bm J}_{\mathbb{C}} \OpBW{{\bm A}_{\bra{-n;K'+1}}\pare{U;x,\xi}} U_{n+3} + {\bf R} \pare{U} U_{n+3},
    \end{aligned}
    \end{equation}
    where
    \begin{itemize}
        \item $a^{\pare{n}}_{\frac{1}{2}}\pare{U;x}\in \cF^{\mathbb{R}}_{K,0,1}\bra{\epsilon_0,N}$,
        \item $a^{(n)}_0(U;x,\xi)$ is a symbol in $\Sigma \Gamma^0_{K,K',1}[\epsilon_0 ,N]$ with $\Im a^{(n)}_0(U;x,\xi) \in \Gamma^0_{K,K',N+1}[\epsilon_0]$, 
        \item  $\mathbf{R}\pare{U}\in \Sigma \cR^{-\vr +2\pare{N+1}}_{K, 1, 1}\bra{\epsilon_0,N} $.
    \end{itemize}
\end{enumerate}
\end{lemma}

\begin{proof}
\begin{step}[Reduction of the term of order $ 1/2 $]\label{step:reduction_diagonal_1/2}
Notice that the term of order $ 1/2 $ can be written as
\begin{align*}
\JC{\bm A}_{\frac{1}{2}}^{(2)}\pare{U;x}\av{\xi}^{\xfrac{1}{2}} = - \ii a_{\xfrac{1}{2}}^{(-1)}\pare{U;x}
\begin{bmatrix}
1 & 1 \\ -1 & -1
\end{bmatrix}
 \av{\xi}^{\frac{1}{2}} , 
 &&
 a_{\xfrac{1}{2}}^{(-1)}\pare{U;x}\in \Sigma\cF^{\mathbb{R}}_{K, 0, 1}\bra{\epsilon_0, N}. 
\end{align*}
Let now
\begin{align*}
{\bm F}^{\pare{-1}}\pare{U;x, \xi}
\defeq
\begin{bmatrix}
0 & F^{\pare{-1}}\pare{U;x, \xi}
\\
F^{\pare{-1}}\pare{U;x, \xi}
& 0
\end{bmatrix}, 
&&
F^{\pare{-1}}\pare{U;x, \xi} \defeq 
\frac{ a_{\xfrac{1}{2}}^{(-1)}\pare{U;x}\av{\xi}^{\xfrac{1}{2}}}{2\pare{1+\mathpzc{v}\pare{U}}\omega_{\kap, \upgamma}\pare{ \xi }}
\in \Sigma\Gamma^{-1}_{K, 0, 1}\bra{\epsilon_0, N}
\end{align*}
and let $ \pare{ {\bm \Phi}_{{\bm F}^{\pare{-1}}}^\tau \pare{U} }_{\av{\tau}\leqslant 1} $ be the flow generated by $  \OpBW{{\bm F}^{\pare{-1}}\pare{U;x, \xi}} $ with initial condition $  {\bm \Phi}_{{\bm F}^{\pare{-1}}}^0 \pare{U} = \Id  $ as per \Cref{flow}. We denote with $ {\bm \Phi}\pare{U}\defeq  {\bm \Phi}_{{\bm F}^{\pare{-1}}}^1 \pare{U} $. Since $\bm F^{(-1)}$ is linearly Hamiltonian, \Cref{flow}  ensures that ${\bm \Phi}\pare{U}$ is a linearly symplectic, spectrally localized map in $ \pare{ \Sigma \cS^0_{K,K',1}[r,N] }^{2\times 2}$. 
Let us define 
\begin{equation*}
U_3\defeq {\bm \Phi}_{{\bm F}^{\pare{-1}}}^1 \pare{U} U_2 = {\bm \Phi} \pare{U} U_2 , 
\end{equation*}
which is the solution of 
\begin{equation}\label{eq:KH9.1}
\begin{aligned}
    \partial_t U_3 &= {\bm \Phi} \pare{U}  \Opvec{   \ii \  \bra{\pare{ 1 + \mathpzc{v}(U) }\omega_{\kap, \upgamma}\pare{ \xi } - V_\upgamma ^{(2)}(U;x)\xi}}{\bm \Phi} \pare{U} ^{-1 } U_3 \\
&\quad+\pare{\partial_t {\bm \Phi} \pare{U} } {\bm \Phi} \pare{U} ^{-1} U_3
+ {\bm \Phi} \pare{U} {\bm J}_\mathbb{C}\OpBW{{\bm A}_{\frac{1}{2}}^{(2)}\pare{U;x}\av{\xi}^{\frac{1}{2}} +   {\bm A}_{\bra{0;1}}(U;x,\xi)}{\bm \Phi} \pare{U} ^{-1} U_3\\
&\quad+ {\bm \Phi} \pare{U} {\bf R} (U){\bm \Phi} \pare{U} ^{-1} U_3.
\end{aligned} 
\end{equation}
The following conjugation rule apply (cf. \cite[Lemma A.1]{BFP2018}) setting  $ {\bf F}\defeq \OpBW{{\bm F}^{\pare{-1}}\pare{U;x, \xi}}  $
\begin{align}\label{eq:Lie_conjugation}
 {\bm \Phi}\pare{U} {\bm M}\pare{U} {\bm \Phi}\pare{U}^{-1} = & \  {\bm M} + \sum _{q=1}^L\frac{1}{q!} \textnormal{Ad}^q_{\bf F}\bra{\bm M}
 +
 \frac{1}{L!}\int_0^1 \pare{1-\tau}^L  {\bm \Phi}_{{\bm F}^{\pare{-1}}}^\tau \pare{U}  \textnormal{Ad}^{L+1}_{\bf F}\bra{\bm M} {\bm \Phi}_{{\bm F}^{\pare{-1}}}^\tau \pare{U}^{-1} \dd \tau, 
 \\
 \label{eq:Lie_conjugation_pat}
 \pare{\partial_t {\bm \Phi}\pare{U}}{\bm \Phi}\pare{U}^{-1} = & \ \JC\OpBW{{\bm A}_{\bra{-1;1}}\pare{U;x, \xi}} + {\bm R}_{\bra{1}}\pare{U},  
\end{align}
An application of \cref{eq:Lie_conjugation} for $  L> \frac{3+2\vr}{4} -1  $ combined with symbolic calculus considerations give us that
\begin{equation}\label{eq:diagonalization_term_1/2}
\begin{aligned}
   & {\bm \Phi}\pare{U}\pare{ \Opvec{  \ii \  \pare{ 1 + \mathpzc{v}(U) }\omega_{\kap, \upgamma}\pare{ \xi } }
+ {\bm J}_\mathbb{C}\OpBW{{\bm A}_{\frac{1}{2}}^{(2)}\pare{U;x}\av{\xi}^{\frac{1}{2}}+{\bm A}_{\bra{0;1}}(U;x,\xi) }}  {\bm \Phi}\pare{U}^{-1}
\\
&=
\Opvec{  \ii\bra{\pare{ 1 + \mathpzc{v}(U) } \omega_{\kap, \upgamma}\pare{ \xi }
 +  a_{\xfrac{1}{2}}^{(-1)}\pare{U;x}
 \av{\xi}^{\frac{1}{2}} 
}}
+
\JC \OpBW{{\bm A}_{\bra{0;1}}\pare{U;x,\xi}} + {\bm R}\pare{U}, 
\end{aligned}
\end{equation}
in which the off-diagonal terms of order 1/2 have been canceled. 
Applying  \cref{eq:diagonalization_term_1/2,eq:Lie_conjugation,eq:Lie_conjugation_pat} to  \eqref{eq:KH9.1} (and renaming ${\bm \Phi} \pare{U} {\bf R} (U){\bm \Phi} \pare{U} ^{-1}\leadsto {\bf R} (U) $ in light of \Cref{prop compo mop}, \cref{item:MM_ext} we obtain that
\begin{equation}\label{eq:KH10}
\begin{aligned}
    \partial_t U_3 &= \Opvec{  \ii \  \bra{\pare{ 1 + \mathpzc{v}(U) }\omega_{\kap, \upgamma}\pare{ \xi } - V_\upgamma ^{(2)}(U;x)\xi + a_{\xfrac{1}{2}}^{(2)}\pare{U;x}\av{\xi}^{\frac{1}{2}} }}U_3 \\
&\quad+ {\bm J}_\mathbb{C}\OpBW{ {\bm A}_{\bra{0;1}}(U;x,\xi)}U_3 + {\bm R}(U)U_3.
\end{aligned}
\end{equation}
Notice that the conjugation rule in \cref{eq:Lie_conjugation} does not modify the principal symbol of the transport term. {Moreover the zero-th order matrix ${\bm A}_{[0,1]}$ is the sum of  ${\bm A}_{[-1;1]}$ form \cref{eq:Lie_conjugation_pat} and  ${\bm A}_{[0;1]}$ from  \cref{eq:diagonalization_term_1/2}. To prove that ${\bm J}_\C{\bm A}_{[0;1]}$ is linearly symplectic we then apply Lemma \ref{spezzamento} to each homogeneous components of the spectrally localized operators in \cref{eq:Lie_conjugation_pat,eq:diagonalization_term_1/2}, which are linearly Hamiltonian thanks to \Cref{lem conj linham}.}
\end{step}

\begin{step}[Reduction at non-positive order]
We sketch the proof which is very similar to the one outlined in \Cref{step:reduction_diagonal_1/2}, we refer the interested reader to \cite[Lemma 6.8]{BMM2022}.  The proof is performed by induction on $ n\in \mathbb{N}$, the case $ n=0 $ is proven in \Cref{step:reduction_diagonal_1/2}, so that we can consider the case $ n\leadsto n+1 $. Suppose \cref{eq:KH9} holds, we have to find a bounded, lineary-symplectic  transformation that pushes the off diagonal terms of $ \JC \OpBW{{\bm A}_{\bra{-n;K'+1}}} $ to lower order, similarly as in was done in \Cref{step:reduction_diagonal_1/2}. Since the matrix $ {\bm A}_{\bra{-n; K'+1}} $ is of the form
\begin{align}
\JC {\bm A}_{\bra{-n; K'+1}} = \JC \begin{bmatrix}
-\ii \bar{b}^\vee_{\bra{-n}} & - \bar{a}^\vee_{\bra{-n}}
\\
- {a} _{\bra{-n}} & \ii \bar{b} _{\bra{-n}}
\end{bmatrix}, 
&&
{a} _{\bra{-n}}, \ \bar{b} _{\bra{-n}}\in \Sigma \Gamma^{-n}_{K, K'+1, 1}\bra{\epsilon_0, N}  , \label{hp:symp}
\end{align}
such cancellation is achieved via conjugation by the flow
\begin{align*}
\system{\begin{aligned}
&\partial_\tau {\bm \Phi}_{{\bm F}^{(n)}}^\tau(U) = \OpBW{{\bm F}^{(n)}(U)} {\bm \Phi}_{{\bm F}^{(n)}}^\tau(U), \\
&{\bm \Phi}_{{\bm F}^{(n)}}^0(U) = \Id,
\end{aligned}}
&&
{\bm F}^{(n)}(U) \defeq \bra{
\begin{array}{cc}
0 & f_{-n-\frac{3}{2}} \\
\overline{f_{-n-\frac{3}{2}}^\vee} & 0
\end{array}
},
\end{align*}

\begin{equation*}
f_{-n-\frac{3}{2}}(U;t,x,\xi) \defeq -\frac{b_{-n}(U;x,\xi)}{2\ii\omega(\xi)(1+\mathpzc{v}(U))} \in \Sigma\Gamma_{K,K'+1, 1}^{-n-\frac{3}{2}}[\epsilon_0,N],
\end{equation*}
and defining the variable $ U_{n+4}\defeq {\bm \Phi}_{{\bm F}^{\pare{n}}}^1\pare{U} U_{n+3} $
and we refer the reader to \cite[Lemma 6.8]{BMM2022} for further details. As the matrix in \eqref{hp:symp} is linearly Hamiltonian up to homogeneity $N$ by inductive hypothesis, the explicitly defined generator $\bm F^{(n)}$ is linearly Hamiltonian up to homogeneity $N$ as well. Thus \Cref{flow} ensures that  ${\bm \Phi}_{{\bm F}^{(n)}}^\tau(U)$ is a linearly symplectic, spectrally localized map in $ \pare{ \Sigma \cS^0_{K,K',1}[r,N] }^{2\times 2}$. The bounded, linearly symplectic transformation is thus defined as
\begin{equation*}
{\bm \Phi}_n \pare{U} \defeq \prod_{\mathsf{j}=-1}^n {\bm \Phi}_{{\bm F}^{\pare{\mathsf{j}}}}^1\pare{U}. 
\end{equation*}
\end{step}
\end{proof}

We can thus apply \Cref{lem:diagonalization_arbitrary_order} setting $n\defeq n_1\pare{\vr}\geqslant-\vr +2\pare{N+1}$ so that $\JC \OpBW{{\bm A}_{\bra{-n;K'+1}}}$ can be incorporated in the smoothing reminder ${\bf R}\pare{U}$, thus setting $Z\defeq U_{n_1+4}$ we obtain that $Z$ solves the evolution equation
\begin{equation}\label{eq:KH11}
    \partial_t Z = \Opvec{\ii \  \bra{ \pare{1 + \mathpzc{v}(U)} \omega_{\kap, \upgamma}\pare{ \xi } - V_\upgamma ^{(2)} (U;x) \xi + a^{\pare{n_1}}_{\xfrac{1}{2}}\pare{U;x}\av{\xi}^{\frac{1}{2}} +  a^{(n_1)}_0(U;x,\xi)} } Z
     + {\bf R} \pare{U} Z.
\end{equation}

\subsection{Reduction to constant coefficients up to smoothing remainders}

\begin{notation}
    From now on we shall denote functions and symbols that are $x$-independent with calligraphic lower- and upper-case letters.
\end{notation}

\begin{lemma}\label{lem:conjcost}
Let $N \in \mathbb{N}$ and  $\vr>3(N+1)$. 
Then for any $ n\in \mathbb{N}$ there is $K''\triangleq K''\pare{\vr,n} >0$  such that for all $ K\geqslant K''+1$ there are $s_0 >0$, $\epsilon_0>0$ such that for any  solution 
$U \in B_{s_0,\mathbb{R}}^K(I;\epsilon_0)$ of  \eqref{eq:KH6},   there exists a real-to-real invertible matrix of spectrally localized maps  $ {\bm \Theta}_n(U)$ such that

\begin{enumerate}
\item $ {\bm \Theta}_n(U),{\bm \Theta}_n(U)^{-1}\in\pare{\cS_{K,  K'',0}^0 [\epsilon_0]}^{2\times 2} $ are linearly symplectic up to homogeneity $ N$ according to Definition \ref{linsymphomoN}. Moreover, $  {\bm \Theta}_n(U)-\Id \in \pare{\Sigma \cS^{\frac{N+1}
{2}}_{K,K'',1}[\epsilon_0,N] }^{2\times 2}.$
\item If $ Z $ solves \eqref{eq:KH11} then the variable
$ Z_{n}\triangleq {\bm \Theta}_n(U) Z $ solves
\begin{equation}
\label{pheq00n2}
\partial_t Z_{n}  = \Opvec{ \ii \, \mathpzc{d}_{\xfrac32}^{(n)}(U; t, \xi)+\ii a_{-\frac{n}{2}}(U;t,x,\xi)}Z_{n} + \pmb{ \mathsf{R} }(U;t) Z_{n} ,
\end{equation}
with the $x$--independent symbol 
\begin{equation} \label{emmenne}
\mathpzc{d}_{3/2} ^{(n)}(U;t,\xi)\triangleq  (1+\mathpzc{v}(U;t))\omega(\xi) + \mathpzc{V}_\upgamma (U;t) \xi + \mathpzc{b} _{\frac12}(U;t)|\xi|^{\frac12} +  \mathpzc{b} _0^{(n)}(U;t, \xi), 
\end{equation}
where
\begin{itemize}
\item  $\mathpzc{v}(U) \in \Sigma \cF^{\mathbb{R}}_{K,0,2}[\epsilon_0, N]$; 
\item  the function $\mathpzc{V}(U;t) \in \Sigma \cF^{\mathbb{R}}_{K,1,2}[\epsilon_0, N]$ is 
$ x $-independent;  
\item  the function $\mathpzc{b}_{\frac12}(U;t) \in \Sigma\cF^{\mathbb{R}}_{K,2,2}[\epsilon_0, N]$ is 
$ x $-independent;  
\item the symbol $\mathpzc{b}_{0}^{(n)}(U;t, \xi) \in \Sigma \Gamma^{0}_{K,K'',2}[\epsilon_0, N]$
 is $x$--independent and its imaginary part $\Im \mathpzc{b}_{0}^{(n)}(U;t, \xi) $ is in $ \Gamma^{0}_{K,K'',N+1}[\epsilon_0]$; 
\item the symbol  $a_{-\frac{n}{2}}(U;t, x,\xi)$ belongs to $\Sigma \Gamma^{-\frac{n}{2}}_{K,K''+1,1}[\epsilon_0, N]$ and its imaginary part $ \Im a_{-\frac{n}{2}}(U;t, x,\xi) $ is in $ \Gamma^{-\frac{n}{2}}_{K,K''+1,N+1}[\epsilon_0]$;
\item  $\pmb{ \mathsf{R} }(U;t) $ is a real-to-real matrix of smoothing operators in 
$ \pare{\Sigma \cR^{-\vr+3(N+1)}_{K,K''+1, 1}[\epsilon_0, N]}^{2\times 2}$.
\end{itemize}
\end{enumerate}
\end{lemma}

\begin{proof}
    The proof consist of a minor modification of the proof of \cite[lemma 6.9]{BMM2022}, so we refer the interested reader to \cite[lemma 6.9]{BMM2022} for details.
\end{proof}
Finally we prove \Cref{prop:constantcoeff}.

\paragraph{Proof of \Cref{prop:constantcoeff}}

In \Cref{lem:conjcost} we set 
$$n_2\defeq n_2\pare{\vr,N}\defeq n \geqslant2\pare{\vr + 3\pare{N+1}},\qquad \underline{K'}\defeq K'' + 1$$
and we define 
$$W\defeq Z_{n_2} = {\bm B}\pare{U;t} U,$$
where
\begin{equation*}
    {\bm B}\pare{U;t}\defeq {\bm \Theta}_{n_2}\pare{U;t} \circ {\bm \Phi}_{n_1}\pare{U;t} \circ {\bm \Psi}_2\pare{U;t} \circ {\bm \Psi}_1\pare{U;t} \circ {\bm G}\pare{U;t}. 
\end{equation*}
At last we set $\mathpzc{b}_0^{\pare{n}}\eqdef \mathpzc{b}_0$ and we conclude. \qed

\section{Hamiltonian Birkhoff normal form and  energy estimate}\label{riduzione_e_stima}
In this section we finally prove the almost global existence Theorem \ref{thm:main}. The previous reduction broke the Hamiltonian structure of the system. Therefore, a first step is to recover the complex Hamiltonian formulation through a symplectic correction. The corresponding result is stated in \Cref{darboux:0} and serves as the foundation for the subsequent normal form analysis. Then, in the technical subsection \ref{sec SAP}, we introduce the class of super-action preserving Hamiltonians that play a central role in the forthcoming analysis, as they do not contribute to the energy estimate. Next, in \Cref{sec:BNF}, we perform a Birkhoff normal form procedure, which extracts the Hamiltonian super-action preserving property up to homogeneity $N+1$, for a given fixed integer $N$. Finally, we implement in \Cref{sec NRJ} the energy estimate allowing to conclude the desired Theorem.\\


 The transformation $\bB(U)$ in Proposition \ref{prop:constantcoeff} destroyed the Hamiltonian property of the system. However $\bB(U)$ was still linearly symplectic and therefore, we can recover the complex Hamiltonian structure by applying the abstract Darboux Theorem of Berti-Maspero-Murgante \cite[Theorem 7.1]{BMM2022}, see also \Cref{conjham}. The corresponding result states as follows.


\begin{proposition}[{\bf Hamiltonian reduction up to smoothing operators}]\label{darboux:0}
Let $N \in \N$ and  {$\vr>{\vr(N)} \triangleq 3(N+1)+ \frac32 (N+1)^3$.}
Then, for any $ K\geqslant \underline K'$ (fixed in Proposition \ref{prop:constantcoeff}) 
there is $s_0 >0, \epsilon_0 > 0 $, such that
for any solution $U \in B_{s_0,\R}^K(I;\epsilon_0) $
of  \eqref{eq:KH6},  there exists a real-to-real matrix of pluri--homogeneous smoothing operators $\pmb{R}(U) $ in $ \pare{\Sigma_1^N \tilde \cR^{-\vr'}_q }^{2\times 2}$ 
for any $\varrho'\geqslant 0$, such that defining 
\be \label{zetone}
Z_0\triangleq  \big( \uno+ \bm{R}( \Phi(U)) \big) \Phi(U),\qquad \Phi(U)\triangleq\bm{B}(U;t)U,
\ee
where $\bm{B}(U;t)$ is defined in Proposition \ref{prop:constantcoeff},
  the following holds true:
  \begin{enumerate}[\bf i]
      \item {\bf Symplecticity:} The non-linear map $\big( \uno+ \bm{R}( \cdot ) \big)\circ\Phi$
in \eqref{zetone}
  is symplectic up to homogeneity $N$ according to Definition \ref{def:LSMN}.

  \item {\bf Conjugation:}
 The variable $Z_0$ solves the {\em Hamiltonian system up to homogeneity $N$}  (cfr. Definition \ref{def:ham.N}).

\be \label{teo62}
\begin{aligned}
\pa_tZ_0 & =  \ii \vOmega(D)Z_0 +\Opvec{\ii (\mathpzc{d}_{\frac32})_{\leqslant N}(Z_0;\xi)+\ii (\mathpzc{d}_{\frac32})_{>N}(U;t,\xi) }Z_0  + \pmb{R}_{\leqslant N}(Z_0)Z_0+ \pmb{R}_{>N}(U;t)U,
\end{aligned}
\ee
where 
\begin{itemize}
\item $\vOmega(D) \defeq \vOpbw{\omega_{\kap,\upgamma}(\xi)}$ is a matrix in the space $\pare{\Gamma^{3/2}_0}^{2\times 2}$;
\item $( \mathpzc{d}_{\frac32})_{\leqslant N}$  is a pluri-homogeneous, real valued symbol, independent of $x $,  in  
 $\Sigma_{2}^N \wt\Gamma^{\frac32}_q$; 

\item $( \mathpzc{d}_{\frac32})_{>N} $ is a non--homogeneous symbol, independent of $x$,  in $ \Gamma^{\frac32}_{K,\underline K',N+1}[\epsilon_0]$ with imaginary part\\ $\Im (\mathpzc{d}_{\frac32})_{>N} $ in $  \Gamma^{0}_{K, \underline K',N+1}[\epsilon_0]$;
\item  $\bm{R}_{\leqslant N}(Z_0)$  is a real-to-real matrix of smoothing operators in
$\pare{\Sigma_{1}^N \wt\cR^{-\varrho+\vr(N)}_q }^{2\times 2}$;
\item $\bm{R}_{>N}(U;t) $ is a real-to-real matrix of non--homogeneous smoothing operators in $ \pare{\mR^{-\vr+\vr(N)}_{K,\underline K',N+1}[\epsilon_0]}^{2\times 2}$.
\end{itemize}

\item {\bf Boundedness:}
The variable $Z_0 = {\bf M}_0(U;t)U$ with $\bM_{0}(U;t) \in \pare{\mM^0_{K,\underline{K}'-1,0}[\epsilon_0]
}^{2 \times 2}  $ and
 for any $s \geqslant s_0$, there is  $0 < \epsilon_0(s)<\epsilon_0$, such that  
for  any $U \in B_{s_0}^K(I;\epsilon_0)\cap C_{*\R}^{K}(I;\dot{H}^{s}(\T, \C^2))$,
there is a  constant $C\triangleq C_{s,K}>0$ such that, for all $k=0,\dots, K-\underline K'$,  
\be \label{equivalenzaZU}
C^{-1}\| U\|_{k,s}\leqslant \|Z_0 \|_{k,s}\leqslant C \| U\|_{k,s} \, .
\ee
  \end{enumerate}
 \end{proposition}

\subsection{Super-action preserving symbols and Hamiltonians}\label{sec SAP}

In this section we define the special class of ``super--action preserving" 
$\tSAP$
homogeneous symbols and Hamiltonians 
which will appear in the Birkhoff normal form reduction of the
Section \ref{sec:BNF}.

\begin{definition} {\bf ($\tSAP$  monomial)} \label{passaggioalpha}
Let  $p\in \N^*$. Given $ (\vec \jmath , \vec \sigma)= (j_a,\sigma_a)_{a=1,\dots, p} \in 
(\Z^*)^p \times \{ \pm \}^p$ we define the multi-index $ (\alpha,\beta) \in \N^{\Z^*} 	\times \N^{\Z^*}  $ with components, for any $k \in \Z^*$, 
\be\label{defalbe}
\begin{aligned}
& \alpha_k(\vec \jmath, \vec \sigma)\triangleq \# \big\{ a = 1, \ldots ,p \, : \, (j_a , \sigma_a) = (k,+) \big\} \, , \\
&  \beta_k(\vec \jmath, \vec \sigma)\triangleq \# \big\{ a = 1, \ldots ,p \, : \, (j_a , \sigma_a) = (k,-) \big\} \, .
\end{aligned}
\ee
We say that a monomial of the form 
$ z_{\vec \jmath}^{\vec \sigma} = z_{j_1}^{\sigma_1}\dots z_{j_p}^{\sigma_p} $ is super-action preserving if the associated multi-index $ (\alpha,\beta)= (\alpha(\vec \jmath, \vec \sigma),\beta(\vec \jmath, \vec \sigma))$ is super-action preserving according to Definition \ref{def:SAPindex}.
\end{definition}

We now introduce the subset $\mathfrak{S}_p  $ of the indexes of $\mathfrak{T}_p$
 composed by super-action preserving indexes
\be\label{fSp}
\mathfrak S_p  \triangleq  \Big\{(\vec \jmath , \vec \sigma) \in \fT_p \quad\textnormal{s.t.}\quad (\alpha(\vec \jmath, \vec \sigma), \beta(\vec \jmath, \vec \sigma)) \in \N^{\Z^*} 
\times \N^{\Z^*} \ \text{in} \ \eqref{defalbe} 
\ \text{are  super  action   preserving} \Big\} \, . 
\ee
We remark that the multi-index $(\alpha, \beta)$ associated to 
$ (\vec \jmath, \vec \sigma) \in (\Z^*\times \{\pm\})^p$ as in \eqref{defalbe} satisfies  $|\alpha+\beta|=p$ and 
\be \label{zjalpha}
z_{\vec \jmath}^{\vec \sigma}= z^\alpha \bar z^\beta\triangleq \prod_{j \in \Z \setminus \{0\} } z_j^{\alpha_j}{\ov{z_j}}^{\beta_j} = \prod_{n \in \N } z_n^{\alpha_n}z_{-n}^{\alpha_{-n}}
{\ov{z_n}}^{\beta_n}\ov{z_{-n}}^{\beta_{-n}}\,  \, .
\ee
It turns out
\be \label{omeginovec}
\vec \sigma \cdot  \vec{\omega}_{\kap,\upgamma}(\vec{\jmath})
= 
\sigma_1 {\omega}_{\kap,\upgamma}({j_1})+\dots+ \sigma_p{\omega}_{\kap,\upgamma}({j_p})= 
 (\alpha- \beta) \cdot \vec {\omega}_{\kap,\upgamma} = \sum_{k \in \Z^*} (\alpha_k- \beta_k)  {\omega}_{\kap,\upgamma}(k) \, ,
\ee
where we denote 
\be\label{defOmegakappa} 
\vec {\omega}_{\kap,\upgamma}({\vec \jmath})\triangleq ( {\omega}_{\kap,\upgamma}({j_1}), \dots, {\omega}_{\kap,\upgamma}({j_p})),\qquad \vec{{\omega}}_{\kap,\upgamma}\triangleq \{ {\omega}_{\kap,\upgamma}(j)\}_{j\in \Z^*}.
\ee

\begin{remark}\label{remino}
If the monomial $ z_{\vec \jmath}^{\vec \sigma}$ is super--action preserving then, for any $j \in \Z^*$,  the monomial $ z_{\vec \jmath}^{\vec \sigma}z_j \bar z_j$ is super-action preserving as well.
\end{remark}

\smallskip
For any $n \in \N^*$, we define the {\em super-action}
\begin{equation}
\label{sa}
J_n \triangleq |z_n|^2 + |z_{-n}|^2.
\end{equation}
The following is Lemma 7.7 in \cite{BMM2022}.
\begin{lemma}\label{lemma:Poissonbra}
The Poisson bracket between a  monomial $z_{\vec \jmath}^{\vec \sigma} $ 
and a super-action  $ J_n $, $ n \in \N $,  defined  
 in \eqref{sa},  is 
 \be\label{poisuperaction}
\set{ z_{\vec \jmath}^{\vec \sigma}, J_n } = 
\ii \big(\beta_n + \beta_{-n} - \alpha_n - \alpha_{-n}  \big) \, z_{\vec \jmath}^{\vec \sigma} \, ,
\ee
where $ (\alpha, \beta)= (\alpha(\vec \jmath, \vec \sigma),\beta(\vec \jmath, \vec \sigma))$ is the multi-index defined in \eqref{defalbe}. 
In particular a super-action preserving 
monomial $ z_{\vec \jmath}^{\vec \sigma} $
(according to Definition \ref{passaggioalpha})  Poisson commutes with any super-action $ J_n $, $ n \in \N $. 
\end{lemma}

We now define a super-action preserving  Hamiltonian.  

\begin{definition}\label{sapham} {\bf ($\tSAP$ Hamiltonian)}  Let $ p \in \N$. 
A $(p+2)$--homogeneous super-action preserving Hamiltonian $H^{(\tSAP)}_{p+2}(Z)$ is a real function of the form 
$$
H^{(\tSAP)}_{p+2}(Z)=\frac{1}{p+2} \sum_{\substack{(\vec \jmath_{p+2},\vec \sigma_{p+2})\in {\mathfrak{S}}_{p+2}} } H_{\vec \jmath_{p+2}}^{\vec \sigma_{p+2}} z_{\vec \jmath_{p+2}}^{\vec \sigma_{p+2}} 
$$
where $ \mathfrak{S}_{p+2}$ is defined as in \eqref{fSp}.
A pluri-homogeneous  super-action preserving Hamiltonian is a finite sum of  
homogeneous super-action preserving Hamiltonians. 
A Hamiltonian vector field is super-action preserving if it is generated by a
super-action preserving Hamiltonian.
\end{definition}

We now define a super-action preserving symbol. 

\begin{definition}\label{sapsym} {\bf ($\tSAP$ symbol)} 
 Let $ p \in \N $ and $m \in \R$. 
For $ p \geqslant 1 $ a  real valued,  $p$--homogeneous {\em super-action preserving symbol} of order $ m $ is a symbol $\tm_p^{({\tSAP })}(Z;\xi) $ in $ \wt\Gamma^m_p$, independent of $x$, of the form 
$$
\tm_p^{(\tSAP)}(Z;\xi)=
\sum_{({\vec{\jmath}_p},{\vec{\sigma}_p})\in {\mathfrak{S}}_p } M_{\vec{\jmath}_p}^{\vec{\sigma}_p}(\xi) z_{\vec{\jmath}_p}^{\vec{\sigma}_p }  \, .
$$
For $ p = 0 $ we say that any  symbol in  $ \wt\Gamma^m_0 $ is 
 super-action preserving.
A pluri-homogeneous  super-action preserving symbol is a finite sum of  
homogeneous super-action preserving symbols.
\end{definition}

\begin{remark}\label{nonsuper}
A super-action preserving symbol has even degree $p$ of homogeneity.
Indeed, if $z_{\vec{\jmath}_p}^{\vec{\sigma}_p} $ is super-
action preserving then $(\alpha,\beta)$ defined in \eqref{defalbe} satisfies 
$ |\alpha| = |\beta|$ and $ p= |\alpha+\beta| = 2| \alpha|$ is even.
\end{remark}

Given a super-action preserving symbol we associate a 
super-action preserving  Hamiltonian according to the following lemma (see Lemma 7.11 in \cite{BMM2022}).

\begin{lemma}\label{lem:ham.sap}
Let $ p \in \N $, $ m \in \R $.  
If $(\tm^{({\tSAP })})_p(Z;\xi)$ is a $p$--homogeneous super-action preserving symbol 
in $ \wt\Gamma^m_p $ according to Definition \ref{sapsym} then 
$$
H^{(\tSAP)}_{p+2} (Z)\triangleq\Re \psc{\OpBW{(\tm^{({\tSAP })})_p(Z;\xi)} z}{\bar z}_{\R}
$$
is a $(p+2)$--homogeneous super-action preserving Hamiltonian according to Definition \ref{sapham}.
\end{lemma}

\subsection{Birkhoff normal form reduction}\label{sec:BNF}

In this section we  finally  
transform the system \eqref{teo62} into its  
Hamiltonian Birkhoff normal form,  up to homogeneity $ N $.

\begin{proposition} {\bf (Hamiltonian Birkhoff normal form)} \label{birkfinalone}
Let $N \in \N$ and $0< \beta_1< \beta_2<4\pare{2+\sqrt{3}}$. 
Assume that
the parameter $\beta = \frac{\upgamma^2}{\kap} \in \bra{\beta_1,\beta_2}$  is outside the zero measure set
$ {\cal B} \subset \bra{\beta_1,\beta_2} $ defined in Proposition \ref{nres}. Then, there exists $ \underline \varrho=\underline \vr(N)>0$  such that, for any $ \varrho \geqslant  \underline \varrho $, for any  $ K \geqslant K'\triangleq \underline K' (\varrho) $ (defined in Proposition \ref{prop:constantcoeff}),
there exists $ \underline s_0 > 0 $ such that, for any $ s \geqslant \underline s_0 $ there is $ \underline \epsilon_0\triangleq  \underline \epsilon_0(s)>0$ such that  for all $ 0 < \epsilon_0< \underline \epsilon_0(s) $ small enough, and any solution $ U \in B_{\underline s_0}^K (I; \epsilon_0)\cap C^K_{*\R}(I; \dot H^{s}(\T;\C^2))$ of the complex Kelvin-Helmholtz system \eqref{eq:KH6},
there exists a 
non--linear map $ \mF_{\tnf}(Z_0)$  
such that:
\begin{enumerate}[\bf i]
    \item {\bf Symplecticity:}  $ \mF_{\tnf}(Z_0)$ is symplectic up to homogeneity $N$ (Definition \ref{def:LSMN});

    \item  {\bf Conjugation:} If $ Z_0 $ solves the system \eqref{teo62} then the variable $ Z\triangleq \mF_{\tnf}(Z_0)$ solves the {\em Hamiltonian system up to homogeneity $N$}  (cfr. Definition \ref{def:ham.N})
\be\label{final:eq}
\begin{aligned}
\pa_tZ & = {\ii \vOmega(D)} Z +  \bm J_\C \nabla H^{(\tSAP)}_{\frac32}(Z)+\bm J_\C \nabla H^{(\tSAP)}_{-\vr}(Z)+\Opvec{ \ii (\mathpzc{d}_{\frac32})_{>N}(U;t,\xi)}Z+ \pmb{R}_{>N}(U;t)U ,
\end{aligned}
\ee
where 
\begin{itemize}
\item $H^{(\tSAP)}_\frac32(Z) $ is the super-action preserving Hamiltonian 
$$
\Re \psc{ \OpBW{ (\mathpzc{d}_{\frac32}^{(\tSAP)})_{\leqslant N}(Z;\xi)}z}{ \bar z}_{\R},
$$
with a pluri-homogeneous super-action preserving  symbol 
$ ( \mathpzc{d}_{\frac32}^{(\tSAP)})_{\leqslant N}(Z;\xi)$  in $\Sigma_2^N \wt\Gamma_q^{\frac32}$,
according  to  Definition \ref{sapsym};
\item  ${\bm J}_\C \nabla H^{(\tSAP)}_{-\vr}(Z)$
 is a super-action preserving,  Hamiltonian, smoothing vector field in $ \Sigma_{3}^{N+1} \wt \X^{-\vr+ \underline \varrho}_q$ (see \ref{sapham});
\item  $(\mathpzc{d}_{\frac32})_{>N}(U;t,\xi)$ is a non--homogeneous symbol in $ \Gamma^{\frac32}_{K,\underline K',N+1}[\epsilon_0]$ with imaginary part $\Im (\mathpzc{d}_{\frac32})_{>N}(U;t,\xi) $ in $ \Gamma^{0}_{K,\underline K',N+1}[\epsilon_0]$;
\item  $\pmb{R}_{>N}(U;t)$ is a real-to-real matrix of non--homogeneous smoothing operators 
 in $\pare{\mR^{-\vr+ \underline \varrho}_{K,\underline K',N+1}[\epsilon_0]}^{2\times 2}$.
  \end{itemize}

  \item  {\bf Boundedness:} There exists a constant $C\triangleq C_{s,K}>0$ such that   for all  $ 0\leqslant k\leqslant  K$ and  any\\
  $Z_0\in B_{\underline s_0}^K(I;\epsilon_0)\cap C^{K}_{*\R}(I;\dot H^{s}(\T, \C^2))$  one has 
\be\label{mappaB}
C^{-1} \| Z_0 \|_{k,s}\leqslant \|\mF_{\tnf}(Z_0) \|_{k,s} \leqslant 
C\| Z_0 \|_{k,s}
\ee 
and 
 \be\label{equivalenzan}
 {C}^{-1} \| U(t) \|_{\dot H^s} \leqslant \| Z (t) \|_{\dot H^s} \leqslant C \| U(t) \|_{\dot H^s} \, , 
\quad \forall  t \in I.
 \ee
\end{enumerate}
\end{proposition}

\begin{notation}
   From now on we denote with $\bm R_p^\tH$ any smoothing remainder in $ \wt{ \cR}^{-\vr}_p$ such that $ \bm R_p^\tH(Z)Z$ is a $p+1$-homogeneous Hamiltonian vector field, namely $\bm R_p^\tH(Z)Z= \bm J_\C \nabla H_{p+2}(Z)$ for some $p+2$ homogeneous Hamiltonian as per \cref{defin phomHAM}. 
\end{notation}

\begin{proof}[Proof of \Cref{birkfinalone}]
The proof consists in $N$ steps and is analogous to the proof of Proposition 7.12 in \cite{BMM2022}. For completeness, we only sketch the main steps and we refer to \cite{BMM2022} for a detailed proof. 

We first reduce the quadratic terms of the vector field in \eqref{teo62}.
\\[1mm]
\noindent {\bf Step $1$: Elimination of the quadratic smoothing remainder in equation \eqref{teo62}.}
\\[1mm]
The $ x$-independent symbol  $(\mathpzc{d}_{\frac32})_{\leqslant N}(Z_0;\x)$ in \eqref{teo62}  belongs to $ \Sigma_2^N \wt \Gamma^{\frac32}_q$ and the only quadratic component of the vector field in \eqref{teo62} is
$\bm R_1^\tH(Z_0)Z_0$ where (recall the notation in \eqref{projectlessN})
\be \label{R1}
\bm R_1^\tH(Z_0)\triangleq \cP_1[{\bm R}_{\leqslant N}(Z_0)] 
 \in\pare{ \wt \cR_1^{-\varrho+\vr(N)}}^{2\times 2}. 
 \ee
 Since the system  \eqref{teo62} is  Hamiltonian  up to homogeneity $N$,  $\bm R_1^\tH(Z_0)Z_0$ is a Hamiltonian vector field 
that we expand in Fourier coordinates as (recall \eqref{mompresind})
\be\label{R1.exp}
\big(\bm R_1^\tH(Z_0)Z_0\big)_k^\sigma= \sum_{
\substack{ (j_1, j_2, k, \sigma_1, \sigma_2, - \sigma) \in \mathfrak T_3
}
} X_{j_1,j_2,k}^{\sigma_1,\sigma_2,\sigma} 
(z_0)_{j_1}^{\sigma_1} (z_0)_{j_2}^{\sigma_2}.
\ee
In order to remove $ \bm R_1^\tH (Z_0) Z_0 $ from \eqref{teo62}, 
we perform the change of variable 
 $ Z_1 = \mathtt{F}^{(1)}_{\leqslant N}(Z_0)$
 where $ \mathtt{F}^{(1)}_{\leqslant N}(Z_0)$ 
 is the map given by Lemma \ref{lem:app.flow.ham} relative to   
a Hamiltonian  smoothing vector field 
 \be\label{G1}
\big({\bm G}_1^\tH(Z_0)Z_0\big)_k^\sigma= \sum_{ \substack{ (j_1, j_2, k, \sigma_1, \sigma_2, - \sigma) \in \fT_3
}}G_{j_1,j_2,k}^{\sigma_1,\sigma_2,\sigma} (z_0)_{j_1}^{\sigma_1} (z_0)_{j_2}^{\sigma_2}, \quad \bm G_1^\tH(Z_0)\in \pare{\cR^{-\vr'}_1}^{2\times 2}
\ee
for some $\vr'>0$ to be determined.
Since $\bm G_1^\tH(Z)Z$ is a Hamiltonian vector field, 
Lemma \ref{lem:app.flow.ham} gives by construction that $\mathtt{F}^{(1)}_{\leqslant N}$ is symplectic up to homogeneity $N$ and it has the form  
\be\label{mFleqN}
Z_1 \triangleq\mathtt{F}^{(1)}_{\leqslant N}(Z_0)  = 
Z_0 + {\bm G}_1^\tH(Z_0)Z_0+ \bm F_{\geqslant 2}(Z_0) Z_0 , \qquad \bm F_{\geqslant 2}(Z_0) \in\pare{ \Sigma_{2}^N \wt \cR^{-\varrho'}_q}^{2\times 2}.
\ee
Applying  Lemma \ref{conj.ham.N},  we obtain that the variable $Z_1$ solves a  system which is 
 Hamiltonian up to homogeneity $N$. 
We compute it using  Lemma \ref{NormFormMap0}.
Its assumption {\bf (A)} at page \pageref{A} holds since $ Z_0  $ solves \eqref{teo62}.
  Then 
  Lemma \ref{NormFormMap0}  implies that the variable $Z_1$  solves 
\be\label{Z1}
\begin{aligned}
\pa_t Z_1
& = \ii \vOmega(D) Z_1+ \vOpbw{\ii (\mathpzc{d}_{\frac32})_{\leqslant N}^+(Z_1;\xi)+\ii (\mathpzc{d}_{\frac32})_{>N}^+(U;t,\xi)}Z_1 \\
& \quad +[ \bm R_1^\tH(Z_1) + \bm G^+_1(Z_1)]Z_1 
+ R_{\geqslant 2}^+(Z_1)Z_1 +  R^+_{>N}(U;t)U,
\end{aligned}
\ee
where \\
\noindent 
$\bullet$
$ (\mathpzc{d}_{\frac32})_{\leqslant N}^+(Z_1;\xi)$ is a real valued symbol, independent of $x$, in $\Sigma_{2}^N \wt \Gamma_q^{\frac32}$; \\
\noindent
$\bullet$
$(\mathpzc{d}_{\frac32})_{>N}^+(U;t,\xi)$ is a non-homogeneous real valued symbol, independent of $x$, in $\Gamma^\frac32_{K,\underline K',N+1}[\epsilon_0]$ with imaginary part ${\rm Im} (\mathpzc{d}_{\frac32})_{>N}(U; t, \xi) $ in $  \Gamma^0_{K,\underline K', N+1}[\epsilon_0]$;  \\
\noindent $\bullet$
$\bm R_1^\tH(Z_1)$ is defined in \eqref{R1}  and $\bm G_1^+(Z_1)Z_1\in \wt \X^{-\vr'+\frac32}_{2}$ has  Fourier expansion, by \eqref{ditigi40} and \eqref{G1},    
\be\label{G1+}
(\bm G^+_1(Z_1)Z_1)_{k}^\sigma=\!\!\! \!\!\! \!\!\! \sum_{
(j_1, j_2, k, \sigma_1, \sigma_2, - \sigma) \in \fT_3
} 
\!\!\! \!\!\! \!\!\! 
 \ii\big(\sigma_1\omega_{\gamma,\upgamma}(j_1)+\sigma_2 \omega_{\gamma,\upgamma}(j_2)- \sigma \omega_{\gamma,\upgamma}(k)\big)G_{j_1,j_2,k}^{\sigma_1,\sigma_2,\sigma} (z_1)_{j_1}^{\sigma_1} (z_1)_{j_2}^{\sigma_2};
\ee
\noindent
$\bullet$
$ \bm R_{\geqslant 2}^+(Z_1) $ is a matrix of pluri-homogeneous smoothing operators in  $ 
\pare{\Sigma_2^N \tilde \cR^{-\vr + \underline \varrho(2)}_q}^{2\times 2}$ where
\be\label{c2}
-\vr +\underline \varrho(2) \triangleq -\vr'+\frac32 \ ; 
\ee

\noindent
$\bullet$ 
$\bm R_{>N}^+(U;t) $ is a matrix of non--homogeneous smoothing operators in $\pare{\mR^{-\vr + \underline \varrho(2)}_{K,\underline K',N+1}[\epsilon_0]}^{2\times 2} $.

By \eqref{R1.exp}, \eqref{G1+} we solve 
\begin{equation}
\label{homologgica1}
 \bm R_1^\tH(Z_1)Z_1 + \bm G_1^+(Z_1)Z_1 = 0 
\end{equation}
by taking 
\be\label{G1+.coeff}
G_{j_1,j_2,k}^{\sigma_1,\sigma_2,\sigma}\triangleq 
\begin{cases} 0, &   \mbox{ if } \ 
(j_1,j_2,k,\sigma_1,\sigma_2,-\sigma) \notin \fT_3, \\
 \displaystyle 
- {\frac{ X_{j_1,j_2,k}^{\sigma_1,\sigma_2,\sigma}}{\ii\big(\sigma_1\omega_{\gamma,\upgamma}(j_1)+\sigma_2 \omega_{\gamma,\upgamma}({j_2})- \sigma \omega_{\gamma,\upgamma}({k})\big)}}, &  \mbox{ if } 
 (j_1,j_2,k,\sigma_1,\sigma_2,-\sigma) \in \fT_3 \, .
\end{cases} 
\ee
The previous expression and the fact that the vector field ${\bm R}_1^\tH(Z_0)Z_0$ is Hamiltonian justifies a posteriori that indeed ${\bm G}_1^\tH(Z_0)Z_0$ is Hamiltonian. Moreover, combining \eqref{R1} and Proposition \ref{nres}, taking $\beta \in \bra{\beta_1,\beta_2} \setminus \cB$, we get that ${\bm G}_1^\tH(Z_0)Z_0\in \wt \X^{-\varrho'}_2$ with $\vr' \triangleq \vr - \vr(N) - \tau$. Thus, by \cref{c2}, we get $\underline \varrho(2)=\vr(N)+\tau+\frac32$.

\noindent {\bf Step $p\geqslant 2$:} 
We now assume that the system is in normal form up to degree $p-1$ of homogeneity. Next, we illustrate how to perform the normal form transformation on the terms of degree $p$. Specifically, we assume that $Z_{p-1}$ solves 
\begin{align}
\pa_t Z_{p-1} & = \ii \vOmega(D)Z_{p-1}+{\bm J}_\C \nabla 
\big(H^{(\tSAP )}_{\frac32}\big)_{\leqslant p+1}(Z_{p-1}) +
{\bm J}_\C \nabla \big(H^{(\tSAP)}_{-\vr}\big)_{\leqslant p+1}(Z_{p-1}) \notag \\
& \quad + \vOpbw{\im (\mathpzc{d}_{\frac32})_{p}(Z_{p-1};\xi)+\im (\mathpzc{d}_{\frac32})_{\geqslant p+1}(Z_{p-1};\xi)}Z_{p-1}  +{\bm R}_{\geqslant p}^\tH(Z_{p-1})Z_{p-1} \notag \\
&\quad +\vOpbw{-\im (\mathpzc{d}_{\frac32})_{>N}(U;t,\xi)}Z_{p-1}+ \bm R_{>N}(U;t)U, \label{traspBNF}
\end{align}
where $\big(H^{(\tSAP)}_{\frac32}\big)_{\leqslant p+1}$, $\big(H^{(\tSAP)}_{-\vr}\big)_{\leqslant p+1}$ are pluri-homogeneous super-action preserving Hamiltonians, $(\mathpzc{d}_{\frac32})_{p}\in \wt \Gamma_p^\frac32$ and expands as 
\be \label{defm32p}
(\mathpzc{d}_{\frac32})_{p}(Z_{p-1};\xi)= \sum_{
(\vec{\jmath}_p,  \vec{\sigma}_p) \in \fT_p} \tD_{\vec{\jmath}_p}^{\vec{\sigma}_p}(\xi) \, (z_{p-1})_{\vec{\jmath}_p}^{\vec{\sigma}_p}, 
\qquad \overline{\tD_{\vec{\jmath}_p}^{-\vec{\sigma}_p}}(\xi)= \tD_{\vec{\jmath}_p}^{\vec{\sigma}_p}(\xi). 
\ee
Moreover $(\mathpzc{d}_{\frac32})_{\geqslant (p+1)}\in \Sigma_{p+1}^N \Gamma_{q}^\frac32$  and $ \bm R_{\geqslant p}^\tH\in \Sigma_{p}^N \cR_q^{-\vr+\underline{\vr}(p)}$. We reduce the homogeneous component\\
$ \vOpbw{-\im (\mathpzc{d}_{\frac32})_{p}}+ R_p$ into  its resonant normal form. First of all we define the intermediate variable 
\be\label{W.bir}
W\triangleq\Phi_p(Z_{p-1})\triangleq \pmb{\cG}_{g_p}^1(Z_{p-1})Z_{p-1},
\ee
where  $ \pmb{\cG}_{g_p}^1(Z_{p-1}) $ is 
the time $1$-linear flow generated by $\vOpbw{\ii g_p}$, 
 where $g_p$ is the Fourier multiplier 
 \be\label{bir.gp}
 g_p(Z_{p-1};\xi)\triangleq\sum\limits_{(\vec{\jmath}_p,  \vec{\sigma}_p) \in \fT_p} \tG_{\vec{\jmath}_p}^{\vec{\sigma}_p}(\xi)(z_{p-1})_{\vec{\jmath}_p}^{\vec{\sigma}_p} \in \wt \Gamma_p^{\frac32}.
 \ee
   Then 
 Lemma \ref{lem:conj.fou} implies that the 
variable $W $ defined in \eqref{W.bir} solves 
 \be \label{eq:W:Bir}
\begin{aligned}
\pa_t W & = \ii \vOmega(D)W+{\bm J}_\C\nabla \big(H^{(\tSAP )}_{\frac32}\big)_{\leqslant p+1}(W) +{\bm J}_\C \nabla \big(H^{(\tSAP)}_{-\vr}\big)_{\leqslant p+1}(W)\\
& \ + \vOpbw{  \im  [(\mathpzc{d}_{\frac32})_{p}(W;\xi) +g^+_p(W;\xi)] + \im (\mathpzc{d}_{\frac32})_{\geqslant p+1}^+(W;\xi)}W +\bm R_{\geqslant p}(W)W \\
&\ +\vOpbw{ \im (\mathpzc{d}_{\frac32})_{>N}^+ (U;t,\xi)}W+ \bm R_{>N}(U;t)U, 
\end{aligned}
\ee
where $g_p^+$ expands as 
\begin{align}
 g^+_p(W;\xi)= \sum_{
(\vec{\jmath}_p,  \vec{\sigma}_p) \in \fT_p}& \ii\vec \sigma_p \cdot \vec{\omega}_{\kap,\upgamma}(\vec{\jmath}_p)\tG_{\vec{\jmath}_p}^{\vec{\sigma}_p}(\xi)w_{\vec{\jmath}_p}^{\vec{\sigma}_p},
\end{align}
while $\bm R_{\geqslant p}\in \pare{\Sigma_1^N \tilde \cR^{-\vr +\underline{\vr}(p)+c(N,p)}_q }^{2\times 2}$, $\bm R_{>N}\in  \pare{\mR^{-\vr+\underline{\vr}(p)+c(N,p)}_{K,K',N+1}[\epsilon_0] }^{2\times 2}$ for some $c(N,p)>0$. In order to get a $p$-homogeneous super-action preserving normal form symbol, we solve the homological equation 
\begin{align*}
(\mathpzc{d}_{\frac32})_{p}(W;\xi)+ g^+_p(W;\xi)= (\mathpzc{d}^{(
    {\tSAP})}_{\frac32})_{p}(W;\xi) \triangleq  
\sum_{(\vec{\jmath}_p,  \vec{\sigma}_p) \in \mathfrak{S}_p}  \tD_{\vec{\jmath}_p}^{\vec{\sigma}_p}(\xi) \, w_{\vec{\jmath}_p}^{\vec{\sigma}_p},
\end{align*}
where $\mathfrak{S}_p$ has been introduced in \eqref{fSp}. One solution is obtained by choosing
\begin{align}\label{G.coeff.op}
 \tG_{\vec{\jmath}_p}^{\vec{\sigma}_p}(\xi) \triangleq\begin{cases} 0, & \mbox{ if } (\vec{\jmath}_p, \vec{\sigma}_p)\in \mathfrak{S}_p,\\ 
 \displaystyle{\frac{\tD_{\vec{\jmath}_p}^{\vec{\sigma}_p}(\xi)}{\ii\vec \sigma_p  \cdot \vec{\omega}_{\kap,\upgamma}(\vec{\jmath}_p)}}, &
 \mbox{ if }  (\vec{\jmath}_p, \vec{\sigma}_p)\not\in \mathfrak{S}_p .
  \end{cases} 
\end{align}
The previous expression \eqref{G.coeff.op} and the fact that $ (\mathpzc{d}_{\frac32})_p$ is real valued justifies a posteriori that indeed $g_p$ is real valued. Moreover, combining estimate  \eqref{estim symbol} for $(\mathpzc{d}_\frac32)_p $ and Proposition \ref{nres}, taking $\beta \in \bra{\beta_1,\beta_2} \setminus \cB$, we get that  $g_p$ satisfies the corresponding estimate \eqref{estim symbol} with $ \mu \leadsto \mu+\tau$. Now, we observe that, by Lemma \ref{lem:hamsym},
\be\label{msap.p.bir}
\vOpbw{  \im (\mathpzc{d}^{(\tSAP)}_{\frac32})_p(W;\xi) }W
={\bm J}_\C\nabla \big(H_{\frac32}^{(\tSAP)}\big)_{p+2}(W)+ \bm R_p(W)W ,
\ee 
with  Hamiltonian 
\be\label{H.SAP.p2}
\big(H_{\frac32}^{(\tSAP)}\big)_{p+2}(W) \triangleq\Re \Big \langle  \OpBW{(\mathpzc{d}_{\frac32}^{({\tSAP })})_p(W;\xi)} w, \bar w\Big \rangle_{\R} \ , 
\ee
which is super-action preserving by Lemma \ref{lem:ham.sap}, 
and  a matrix of smoothing operators  $\bm R_p(W) $ in  $ \pare{\wt\cR^{-\varrho'}_p}^{2\times 2}$ for any $\varrho' \geqslant 0$. 
Therefore \eqref{eq:W:Bir} becomes  
 \be \label{eq:W:Bir3}
\begin{aligned}
\pa_t W = &  \ii \vOmega(D)W+{\bm J}_\C \nabla \big(H^{(\tSAP )}_{\frac32}\big)_{\leqslant p+2}(W) +{\bm J}_\C\nabla  \big(H^{(\tSAP)}_{-\vr}\big)_{\leqslant p+1}(W)\\
& \ + \vOpbw{   \im (\mathpzc{d}_{\frac32})_{\geqslant p+1}^+(W;\xi)}W+ [\bm R_{\geqslant p}(W) ]W  \\
&\ +\vOpbw{ \im (\mathpzc{d}_{\frac32})_{>N}^+ (U;t,\xi)}W+ \bm R_{>N}(U;t)U, 
\end{aligned}
\ee
where (see \eqref{traspBNF},\eqref{H.SAP.p2})
\be\label{H.SAP.30}
\big(H^{(\tSAP )}_{\frac32}\big)_{\leqslant p+2}\triangleq \big(H^{(\tSAP )}_{\frac32}\big)_{\leqslant p+1} + \big(H_{\frac32}^{(\tSAP)}\big)_{p+2},
\ee
while $\bm R_{\geqslant p}$ is a new smoothing remainder in $ \pare{\Sigma_1^N \tilde \cR^{-\vr +\underline{\vr}(p)+c(N,p)}_q }^{2\times 2}$.
Note that the new system \eqref{eq:W:Bir3} is not Hamiltonian up to homogeneity $N$ (unlike system \eqref{traspBNF} for $Z_{p-1}$), since the map
$\Phi_p(Z_{p-1}) = \pmb{\cG}^1_{g_p}(Z_{p-1}) Z_{p-1} $  in 
\eqref{W.bir} is not symplectic up to homogeneity $N$. By Lemma \ref{flussoconst} we only know  
 that $\pmb{\cG}^1_{g_p}(Z_{p-1})$ is linearly symplectic.
We now apply Theorem \ref{conjham} to find a  correction of  $\Phi_p(Z_{p-1})$ which is symplectic up to homogeneity $N$. The  assumptions of Theorem \ref{conjham} are verified applying Lemma \ref{flussoconst}, therefore there exists $\bm R_{\leqslant N}^{(p)} \in   \left(   \Sigma_p^N \wt \cR^{-\vr-\frac32}_q\right)^{2\times 2}$ such that the variable  
\be\label{darbouxp}
V \triangleq \cC_N^{(p)}(W)\triangleq \big( \uno + \bm R_{\leqslant N}^{(p)}(W) \big)  W = 
\big( \uno + \bm R_{\leqslant N}^{(p)}(\Phi_p(Z_{p-1})) \big)  \Phi_p(Z_{p-1}) 
 \ee
 is symplectic up to homogeneity $N$, thus
solves a system which is  Hamiltonian  up to homogeneity $N$. Since $W$ solves \eqref{eq:W:Bir3} then Lemma \ref{NormFormMap0} implies that $V =W + {\bm R}_{\leqslant N}^{(p)}(W)W$ solves 
\be \label{eq:W:Bir4}
\begin{aligned}
\pa_t  V  = &  \ii \vOmega(D)  V+{\bm J}_\C\nabla \big(H^{(\tSAP )}_{\frac32}\big)_{\leqslant p+2}(V) +{\bm J}_\C\nabla \big(H^{(\tSAP)}_{-\vr}\big)_{\leqslant p+1}(V)\\
& \ + \vOpbw{   \im \wt{(\mathpzc{d}_{\frac32})}_{\geqslant p+1}(V;\xi)}V +
\bm J_\C\nabla \left(H_{-\vr}\right)_{p+2}(V)+ \bm  R_{\geqslant p+1}(V) V  \\
&\ +\vOpbw{ \im \wt{(\mathpzc{d}_{\frac32})}_{>N} (U;t,\xi)}V+  \bm R_{>N}(U;t)U,
\end{aligned}
\ee
where $ \wt{(\mathpzc{d}_{\frac32})}_{\geqslant p+1}\in \Sigma_{p+1}^N \widetilde{\Gamma}^\frac32_q$, $\bm J_\C\nabla \left(H_{-\vr}\right)_{p+2}\in \mathfrak{X}^{-\vr+\underline{\vr}(p)+ C(N,p)}_p$, $ \bm R_{\geqslant p+1}\in \Sigma_{p}^N \widetilde{\cR}^{-\vr+\underline{\vr}(p)+ C(N,p)}$. Notice that the Hamiltonian structure of $\bm J_\C\nabla \left(H_{-\vr}\right)_p(V)$ is justified a posteriori by the fact that  the $p$-homogeneous component of the vector field in \eqref{eq:W:Bir4} is Hamiltonian and the term in its first line is Hamiltonian, thus we deduce the Hamiltonianity by difference. Thus we can proceed as in Step $1$. First we Fourier expand $\bm J_\C\nabla \left(H_{-\vr}\right)_{p+2}$ as 
\be\label{wtR.p}
\big( \bm J_\C\nabla \left(H_{-\vr}\right)_{p+2}(V)\big)_k^\sigma=
\!\!\!\!\!\!\!\!\!\!\!\!\!\!\!
\sum_{
(\vec{\jmath}_{p+1}, k,  \vec{\sigma}_{p+1}, - \sigma) \in \fT_{p+2}}
\!\!\!\!\!\!\!\!\!\!\!\!\!\!
  X_{\vec{\jmath}_{p+1},k}^{\vec{\sigma}_{p+1},\sigma}  v_{\vec{\jmath}_{p+1}}^{\vec{\sigma}_{p+1}} \, . 
\ee
Then we use Lemma \ref{lem:app.flow.ham} to  generate a symplectic up to homogeneity $N$ map $\cF_{\leqslant N}^{(p)}$ associate to the Hamiltonian smoothing vector field 
\be\label{Gp}
\big(\bm G_p(V)V\big)_k^\sigma=
\!\!\!\!\!\!\!\!\!
\sum_{
(\vec{\jmath}_{p+1}, k,  \vec{\sigma}_{p+1}, - \sigma) \in \fT_{p+2}}
\!\!\!\!\!\!\!\!
 G_{\vec{\jmath}_{p+1},k}^{\vec\sigma_{p+1},\sigma} v_{\vec{\jmath}_{p+1}}^{\vec{\sigma}_{p+1}}. 
\ee
Applying Lemma \ref{NormFormMap0} the analogue of the homological equation \eqref{homologgica1} is 

\begin{equation}\label{homologgicap}
\bm J_\C\nabla \left(H_{-\vr}\right)_{p+2}(Z_p) + \bm G_p^+(Z_p)Z_p = {\bm J}_\C\nabla \left(H^{(\tSAP)}_{-\vr}\right)_p(Z_p)\triangleq \!\!\!\sum_{{(\vec\jmath_{p+1}, k,  \vec \sigma_{p+1}, - \sigma) \in  {\mathfrak{S}}_{p+2}}
} 
\!\!\!\!\!\!\!\!\!\!\!
 X_{\vec{\jmath}_{p+1},k}^{\vec{\sigma}_{p+1}, \sigma}  \, (z_p)_{\vec{\jmath}_{p+1}}^{\vec{\sigma}_{p+1}}  
\end{equation}
where the vector field $\bm G_p^+(Z_p)Z_p$ is explicitly given by 
$$
 (\bm G_p^+(Z_p)Z_p)_k^\sigma\triangleq 
\!\!\!\!\!\!\!\!\!\!\!\!
 \sum_{(\vec{\jmath}_{p+1}, k,  \vec{\sigma}_{p+1}, - \sigma)  \in \fT_{p+2}}
 \!\!\!\!\!\!\!\!\!\!\!\!\!\!\!\!\!\!\!\!\!\!\!
  \ii\big(\vec{\sigma}_{p+1}\cdot  \vec{\omega}_{\kap,\upgamma}(\vec{\jmath}_{p+1})- \sigma \omega_{\kap,\upgamma}(k)\big)G_{\vec{\jmath}_{p+1},k}^{\vec{\sigma}_{p+1},\sigma}  (z_p)_{\vec{\jmath}_{p+1}}^{\vec{\sigma}_{p+1}} 
  \ . 
$$
Then a solution of \eqref{homologgicap} is given by  
 \begin{align}\label{Gp.coeff}
 G_{\vec{\jmath}_{p+1},k}^{\vec{\sigma}_{p+1},\sigma}  \triangleq \begin{cases} 0 ,& \mbox{ if } (\vec\jmath_{p+1}, k,  \vec \sigma_{p+1}, - \sigma)\in \mathfrak{S}_{p+2},\\ 
 -\displaystyle{\frac{  X_{\vec{\jmath}_{p+1},k}^{\vec{\sigma}_{p+1}, \sigma} }{\ii (\vec{\sigma}_{p+1}\cdot \vec{\omega}_{\kap,\upgamma}(\vec{\jmath}_{p+1})- \sigma \omega_{\kap,\upgamma}(k))}} ,&
 \mbox{ if } (\vec\jmath_{p+1}, k,  \vec \sigma_{p+1}, - \sigma) \not\in \mathfrak{S}_{p+2} .
  \end{cases} 
\end{align}
The previous expression and the fact that the vector field $\bm J_\C\nabla \left(H_{-\vr}\right)_{p+2}$ is Hamiltonian justifies a posteriori that indeed $G_p(Z_p)Z_p$ is Hamiltonian. Moreover, since $\bm J_\C\nabla \left(H_{-\vr}\right)_{p+2}$ is $ \vr_p $-smoothing, we apply Proposition \ref{nres}, taking $\beta \in \bra{\beta_1,\beta_2} \setminus \cB$ and  we get that $\bm G_p(Z_p)Z_p\in \wt \X^{-\varrho_{p+1}}_2$ with $\vr_{p+1}\triangleq \vr_p - c(N) - \tau.$
\end{proof}

\subsection{Energy estimate and proof of the main Theorem}\label{sec NRJ}

\begin{rem}
We can derive local existence for \eqref{eq:KH6} following the computations in the main result of  \cite{BMM2021} and exploiting the fact that the local existence result of \cite{BMM2021} uses the very same paradifferential functional setting as in the present manuscript. Minor modifications in the works \cite{Ambrose2003,Lannes2013_ARMA} can also provide local existence for the setting of \eqref{eq:KH2}.
\end{rem}

The following result, analogous to \cite[Lemma 6.3]{BFP2018}, 
enables to bound the norms $ \| \pa_t^k U(t) \|_{s-\frac32 k} $ of  the time derivatives of a solution 
$  U(t) $ of \eqref{eq:KH6} by  $ \| U(t)\|_s  $. 

\begin{lemma}\label{lem:equivalence_norms}
Let $ K\in \N $. There exists $ s_0> 0 $ such that for any $ s\geqslant s_0 $, any $ \epsilon\in\pare{0, \overline{\epsilon_0}\pare{s}} $ small, if $ U $ belongs to
$ B^0_{s_0, \R}\pare{I;\epsilon}\cap \Cast{0}{s} $ 
and solves \eqref{prop:KH_complex} then $ U\in C^K_{* \R}\left(I; \dot H^s(\T;\C^2)\right) $ and there exists $ C_1 \defeq C_1\pare{s, K} \geqslant 1 $ such that
$$\norm{U\pare{t}}_s  \leqslant
 \norm{U\pare{t}}_{K, s}
 \leqslant C_1 \norm{U\pare{t}}_s,\qquad \forall t \in I.$$ 
\end{lemma}
\begin{proof}
    The proof is inductive and follows the same line of \cite[Lemma 6.3]{BFP2018}, so we sketch only the first step. As $\|U(t)\|_s\leqslant \| U\pare{t}\|_{K,s}$ is obvious in view of the definition \eqref{def norm Ks}, it remains to prove the second inequality. We start by estimating $ \| \pa_t U(t)\|_{s-\frac32}$. Using equation \eqref{prop:KH_complex} for $ U(t)$ and then  \eqref{para:dt} with $k=K'=p=0$, $ m=\frac32$  and estimate \eqref{stima:emmeop} with $m\leadsto -\vr$ and $k=0$, we get 

    \begin{align*}
        \| \pa_t U(t)\|_{s-\frac32} &\leqslant  \left\| \OpBW{{\bm A}_{\frac{3}{2}}\pare{U;x} \  \omega_{\kap, \upgamma}\pare{ \xi }  + {\bm A}_{1}\pare{U;x, \xi} + {\bm A}_{\frac{1}{2}}\pare{U;x } \  \av{\xi}^{\frac{1}{2}} + {\bm A}_{\bra{0}} \pare{U;x, \xi}} U\right\|_{s-\frac32} +\left\| {\bm R}\pare{U} U\right\|_{s-\frac32}
        \\
        &\lesssim  \| U(t)\|_{s}.
    \end{align*}
The estimates for $ k \geqslant 2 $ follow in a similar manner, additionally using estimates \eqref{para:dt} and \eqref{stima:emmeop} with $ k \geqslant  1 $ to handle higher-order derivatives.

\end{proof}
We fix now the parameters appearing in \Cref{birkfinalone}.  
In its statement, we set $\vr \defeq \underline{\vr}(N)$ and $K \defeq \underline{K'}(\vr)$,  
which implies the existence of the constant $\underline{s_0} > 0$ which we increase, if necessary, to fit the range of \Cref{lem:equivalence_norms}, such that for any $s \geqslant \underline{s_0}$, and any fixed  
$0 < \epsilon_0 \leqslant \min\{ \overline{\epsilon_0}(s), \underline{\epsilon_0}(s) \}$,  
where $\underline{\epsilon_0}(s)$ is defined in \Cref{birkfinalone},  
and $\overline{\epsilon_0}(s)$ in \Cref{lem:equivalence_norms}, the conclusions of \Cref{birkfinalone} and \Cref{lem:equivalence_norms} hold. Therefore, one can obtain the following energy estimate. Notice that the time-reversibility of the Kelvin-Helmholtz system allows us to restrict the discussion to positive times $t>0.$

\begin{lemma}[Energy estimate]\label{lem:SE}
Let $ U(t) $ be a solution of 
   equation \eqref{eq:KH6} in 
   $  \Ball{K}{\underline{s_0}} \cap \Cast{K}{s} $.
Then  there exists  $ \bar C_2 \pare{s} > 1 $ such that 
\begin{equation}\label{stimaboot}
\norm{U\pare{t}}_s^2 \leqslant \bar{C}_2 \pare{s}  \pare{\norm{ U\pare{0}}_s^2 
+ \int_0^t\norm{U\pare{\tau}}_{s_0}^{N+1}\norm{ U\pare{\tau}}_s^2 \dd \tau }, 
\quad \forall 0 < t < T. 
\end{equation}
\end{lemma}

\begin{proof}
The variable $Z$ defined in \Cref{birkfinalone} solves the normal form \Cref{final:eq}. Then, denoting by $H^{(\tSAP)}\triangleq H^{(\tSAP)}_{\frac32}+H^{(\tSAP)}_{-\vr}$ we have that
\begin{align*}
\frac{\di }{\di t } \norm{Z(t)}_s^2= & \  \di_Z \norm{Z}^2_s\left[ \ii \bm \omega_{\kap, \upgamma}(D)Z+ \bm J_\C \nabla H^{(\tSAP)} +\OpBW{\ii (\mathpzc{d}_{\frac32})_{>N}(U;t,\xi)}Z +\bm R_{>N}(U;t)U\right] \\
= & \sum _{n\in \N}\av{n}^{2s} \pbra{J_n}{H^{(\tSAP)}} 
\\
& 
+ \psc{\av{D}^{2s}\OpBW{ - \Im \pare{\pare{\mathpzc{d}_{\frac{3}{2}}}_{>N}\pare{U;t, \xi}}} Z  }{Z}
+ 2 \Re \psc{\av{D}^{2s} {\bm R}_{>N}\pare{U;t}U  }{Z}
\\
\lesssim & \ \norm{U}_{K', s_0}^{N+1} \pare{\norm{Z}_s + \norm{U}_s} \norm{Z}_s,
\end{align*}
where we used \Cref{lemma:Poissonbra} and the fact that  $- \Im \pare{\pare{\mathpzc{d}_{\frac{3}{2}}}_{>N}\pare{U;t, \xi}}$ is in $\Gamma^0_{K, K', N+1}[\epsilon_0] $ and ${\bm R}_{>N}\pare{U;t}$ is in  $\mathcal{M}^{0}_{K, K', N+1}[\epsilon_0] $. Integrating in time the above inequality, then using \cref{equivalenzan,lem:equivalence_norms}, we obtain \eqref{stimaboot}.
\end{proof}

The energy estimate \eqref{stimaboot} 
and the equivalence
\begin{equation}
    \norm{\eta\pare{t}}_{H^{s+\frac{1}{4}}_0} + \norm{\psi\pare{t}}_{\dot H ^{s-\frac{1}{4}}}
    \sim \norm{U\pare{t}}_s, 
\end{equation}
which is a consequence of \eqref{eq:complex_var}, 
imply, by a standard 
bootstrap argument,  \Cref{thm:main}. For detailed computations we refer the interested reader to \cite{BD2018,BMM2022,BFP2018}.

\appendix
\section{Hamiltonian formalism}
In this appendix we set the Hamiltonian formalism for the problem at hand, following \cite[Section 3]{BMM2022}. 
\subsection{Linearly Hamiltonian systems}
In this section we provide the linear Hamiltonian framework needed in our analysis. The linear Hamiltonian structure is the natural algebraic property that emerges from the paralinearization of an Hamiltonian system.
\begin{definition}[Linearly Hamiltonian paradifferential operator]\label{defin linhampara}
A real-to-real matrix of paradifferential complex operators is {\it linearly Hamiltonian} if it is of the form 
\begin{align}\label{eq:LinHamParadiff}
\JC \OpBW{\bB},\qquad\bB\triangleq\bra{
\begin{array}{cc}
b_1\pare{U;t,x,\xi} & b_2\pare{U;t,x,\xi} \\
\overline{b_2\pare{U;t,x,-\xi}} & \overline{b_1\pare{U;t,x,-\xi}}
\end{array}
},
&&
\system{
\begin{aligned}
& b_1\pare{U;t,x,-\xi} = b_1\pare{U;t,x,\xi}, \\
& b_2\pare{U;t,x,\xi} \in \mathbb{R}.
\end{aligned}
}
\end{align}
In other words, $b_1$ is even in $\xi$ and $b_2$ is real valued. This is equivalent to require that the matrix valued paradifferential operator $\OpBW{\bB}$ is symmetric. 
\end{definition}

\begin{definition}[Linearly Hamiltonian operator up to homogeneity  $N$]
\label{def:LinHamup}
A real-to-real  matrix of spectrally localized maps 
${\bm J}_\mathbb{C} {\bf B} (U;t) $ in $ \pare{ \Sigma {\cal S}_{K,K',p}[\epsilon_0,N]  }^{2\times 2} $
 is {\em linearly Hamiltonian up to homogeneity $N$} if 
 its  pluri-homogeneous component 
 $ \cP_{\leqslant N} ({\bf B} (U;t))$ (defined through \eqref{projectlessN}) is symmetric,  namely 
$$
\cP_{\leqslant N} ({\bf B}(U;t)) =   \cP_{\leqslant N} \pare{ {\bf B}(U;t)^\intercal }. 
$$
\end{definition}
In particular, a real-to-real matrix of paradifferential   operators is 
 linearly Hamiltonian up to homogeneity $N$ if it has the form    (cfr. \eqref{eq:LinHamParadiff})
\begin{align}
\label{LHS-cN}
{\bm J}_\mathbb{C} 
\OpBW{\begin{matrix}
b_1(U;t, x,\xi)  & b_2(U; t, x, \xi) \\
 \overline{b_2(U;t, x, -\xi)} &  \overline{b_1(U; t, x, -\xi)}
\end{matrix}}, 
&&
  \ \begin{cases} b_1(U;t, x, - \xi) - b_1(U;t, x, \xi) \in  \Gamma^m_{K,K',N+1}[\epsilon_0],  \\ \Im  
b_2(U;t, x,\x) \in \Gamma^{m'}_{K,K',N+1}[\epsilon_0] 
 \end{cases} 
\end{align}
for some $m, m' $ in $ \R $.

\begin{definition}[Linearly symplectic map]
\label{def:LS}
A real-to-real linear transformation $\cA$ is linearly symplectic if $\cA^* \Omega_\mathbb{C} = \Omega_\mathbb{C}$,
where $\Omega_\mathbb{C}$ is defined as
\begin{equation}\label{Symp_Omegone}
\Omega_\mathbb{C} \left(\vect{u_1}{ \bar u_1}, \vect{u_2}{ \bar u_2}\right) \defeq\psc{\EC \vect{u_1}{ \bar u_1}}{\vect{u_2}{ \bar u_2}}_{\mathbb{R}}, \qquad \EC\defeq \JC^{-1},
\end{equation}
 namely 
$$\cA^\top \EC \cA = \EC.$$ 
\end{definition}
\begin{definition} {\bf (Linearly symplectic map up to homogeneity $N$)} 
\label{linsymphomoN}
A real-to-real matrix of spectrally localized maps $\bS(U;t) $ in $ \pare{\Sigma \cS_{K,K',0}[r,N]}^{2\times 2} $ 
 is {\em  linearly symplectic up to homogeneity $N$} if   
\be \label{A:sym23}
\bS(U;t)^\top  \, {\bm E}_\C \, \bS(U;t)= {\bm E}_\C+ S_{>N}(U;t) , 
\ee
where $ {\bm E}_\bC  $ is the symplectic operator defined in \eqref{Symp_Omegone} and 
$S_{>N}(U;t) $ is a  
matrix of spectrally localized maps 
in $ \pare{\cS_{K,K',N+1}[\epsilon_0]}^{2\times 2}$. 
\end{definition}
Linearly symplectic maps up to homogeneity $N$ preserve the linear Hamiltonian structure up to homogeneity $N$. The following result is borrowed from \cite[Lemma 3.9]{BMM2022}.
\begin{lemma}\label{lem conj linham}
Let $ {\bm J}_\C \bB(U;t) $ be   a linearly Hamiltonian operator up to homogeneity $N $ 
(Definition \ref{def:LinHamup})
and $ {\bm S}(U;t) $ be an invertible map,   linearly symplectic to homogeneity $N$ (Definition \ref{linsymphomoN}). 
Then  the operators 
$ {\bm S}(U;t) {\bm J}_\C  {\bm B}(U;t) \bm S(U;t)^{-1} $ and 
$ (\pa_t{\bm S}(U;t)) {\bm S}^{-1}(U;t) $
 are linearly Hamiltonian up to homogeneity $N$. 
\end{lemma}

We consider  the flow of a linearly 
Hamiltonian up to homogeneity $ N $ paradifferential operator.  The following is Lemma 3.16 of \cite{BMM2022}.

\begin{lemma} \label{flow}
{\bf (Linear symplectic flow)}
Let  $p \in \N$,  $N,  K,K'  \in \N $ with $  K'\leqslant K $, $ m \leqslant 1 $, $r>0$.  Let  

$${\bm J}_\C \OpBW{{\bm B}}=\OpBW{\begin{bmatrix} \overline{ \ii b_2^\vee} &  \overline{\ii b_1^\vee}\\
	 \ii b_1 &   \ii b_2  \end{bmatrix}}$$
 be a linearly Hamiltonian operator up to homogeneity $N$ 
(Definition \ref{def:LinHamup})
where  ${\bm B}$  is a matrix of symbols
$$
{\bm B}\triangleq {\bm B}(\tau,U; t, x,  \xi )\triangleq 
\begin{pmatrix}
 b_1(\tau, U; t, x, \xi) &  b_2(\tau,U; t, x, \xi) \\  \overline{b_2(\tau,U; t, x, -\xi)} &  \overline{b_1(\tau,U; t , x, -\xi)}\end{pmatrix}, 
 \qquad
 \begin{cases}
  b_1  \in \Sigma \Gamma^{0}_{K,K',p}[r,N] ,\\
  b_2\in  \Sigma \Gamma^{m}_{K,K',p}[r,N]  ,  
  \end{cases}
$$
with 
$b_1^\vee - b_1 $ in $  \Gamma^0_{K,K',N+1}[\epsilon_0]$ and the imaginary part 
$\textup{ Im } b_2 $ in $ \Gamma^0_{K,K', N+1}[\epsilon_0]$ (cfr. \eqref{LHS-cN}) uniformly in $ |\tau| \leqslant 1  $.
Then there exists $ s_0 >0$ such that, for any $ U\in B_{s_0,\R}^{K}(I;r) $, 
the system 
 \begin{equation*}
\pa_\tau{\pmb \cG}^\tau_{{\bm B}}(U;t)=
{\bm J}_\C  \, \OpBW{{\bm B}(\tau,U; t, x, \xi)}
  \pmb{\cG}^\tau_{{\bm B}}(U;t),\quad \quad 
\pmb{\cG}_{{\bm B}}^0(U;t)=\Id 
\end{equation*}
has a unique solution $\pmb{\cG}^{\tau}_{\bm B}(U)$ defined for all $ | \tau | \leqslant 1  $, 
satisfying  the following properties:
\begin{itemize}
\item[(i)] {\bf Boundedness:} For any $s \in \R$ the linear map $\pmb{\cG}^{\tau}_{{\bm B}}(U;t)$ is invertible and  $\pmb{\cG}_{{\bm B}}^\tau(U;t)$ and $\pmb{\cG}_{{\bm B}}^\tau(U;t)^{-1}$ 
 are  non--homogeneous spectrally localized maps in $ \pare{ \cS_{K,K',0}^0 [\epsilon_0]}^{2\times2}$ according to Definition \ref{defin specloc}. 
\item[(ii)] {\bf Linear symplecticity:} The map  $ \pmb{\cG}^\tau_{{\bm B}}(U;t) $ is linearly symplectic up to homogeneity $N$ (Definition \ref{linsymphomoN}).
If ${\bm J}_\C \OpBW{\mathbf{B}}$ is  linearly Hamiltonian (Definition \ref{defin linhampara}), then $\pmb{\cG}^{\tau}_{{\bm B}}(U;t)$ is linearly symplectic (Definition \ref{def:LS}).
\item[(iii)] {\bf Homogeneous expansion:}  $\pmb{\cG}_{{\bm B}}^\tau(U;t)$ and its inverse are spectrally localized maps and  $ \pmb{\cG}_{{\pmb B}}^\tau(U;t)^{\pm}- \uno $ belong to 
$ \pare{\Sigma  \cS_{K,K',p}^{(N+1) m_0}[r, N]}^{2\times 2}$ with $m_0 \triangleq  \max(m,0)$, uniformly in $ |\tau| \leqslant 1  $.
\end{itemize}
\end{lemma}

The flow generated by a Fourier multiplier satisfies similar properties.  The following is Lemma 3.17 of \cite{BMM2022}. 
 
\begin{lemma} {\bf (Flow of a Fourier multiplier)} \label{flussoconst}
Let $ p \in \N$ and $ g_p(Z;\xi) $ be a $p$--homogeneous, 
 $ x $-independent,  real symbol in $ \wt \Gamma^{\frac32}_p $. Then,  the flow 
 $\pmb{\cG}_{g_p}^\tau (Z)$ defined by 
\be \label{FourierFlow}
\pa_\tau \pmb{\cG}_{g_p}^\tau(Z)= \vOpbw{\ii g_p(Z;\xi)}\pmb{\cG}_{g_p}^\tau(Z) \, , 
\qquad \pmb{\cG}_{g_p}^0(Z)=\uno \, ,
\ee
is well defined for any $ |\tau| \leqslant 1 $ and satisfies the following properties:
\begin{itemize}
\item[$(i)$] {\bf Boundedness:} 
For any $K\in \N$ and $r>0$  the flow  $\, \pmb{\cG}^{\tau}_{g_p}(Z)
 $ and its inverse  $\,\pmb{\cG}^{-\tau}_{g_p}(Z)
 $ are real-to-real diagonal matrix of spectrally localized 
 maps in $ \pare{\cS_{K,0,0}^0[\epsilon_0]}^{2\times 2}  $.
\item[$(ii)$]{\bf Linear symplecticity:} 
The flow map $\pmb{\cG}^\tau_{g_p}(Z)$    is linearly symplectic  (Definition \ref{def:LS}). 
\item[$(iii)$] {\bf Homogeneous expansion:} The flow map  $ \pmb{\mathcal{G}}^{\tau}_{g_p}(Z) $ and its inverse $ \pmb{\mathcal{G}}^{-\tau}_{g_p}(Z) $ are matrices of  spectrally localized maps 
such that $ \pmb{\cG}^{\pm \tau}_{g_p}(Z) -\uno 
$ belong to $ \pare{\Sigma \cS^{\frac32 (N+1)}_{K,0,p}[r,N]}^{2\times 2} $,
uniformly in $ |\tau| \leqslant 1  $.
\end{itemize}
\end{lemma}
\subsection{Non-linear Hamiltonian systems}
We first give the definition of $p$-homogeneous Hamiltonian functions.
\begin{definition}[Homogeneous Hamiltonian]\label{defin phomHAM}
    Let $p \in \N$, a $p+2$-homogeneous Hamiltonian is a function defined on $\dot H^\infty(\T;\C^2)$ with values in $\C$ which is real valued for any $ U\in \dot H^{\infty}_\R(\T;\C^2)$ of the form, c.f. \eqref{mompresind}
    \begin{align} 
H_{p+2}(U)= \langle M_p(U)U,U\rangle= \sum_{
(\vec \jmath , \vec \sigma) \in {\mathfrak{T}}_{p+2}}
H_{\vec \jmath}^{\vec \sigma} \, u_{\vec \jmath}^{\vec \sigma}, 
&&
M_p(U)\in \pare{\wt {\mathcal{M}}_p^m}^{2\times 2}\label{Ham:definina}
    \end{align}
for some $m\in \R$ and complex valued coefficients $H_{\vec \jmath}^{\vec \sigma}= H_{j_1,\ldots, j_{p+2}}^{\sigma_1,\ldots, \sigma_{p+2}}$. A pluri-homogeneous Hamiltonian is a polynomial of the form 
\be \label{plurihomHAM}
H(U)= \sum_{p=0}^N H_{p+2}(U),
\ee
where, for $p=0, \ldots, N$, $H_{p+2}$ is a $p+2$-homogeneous Hamiltonian. 
\end{definition}
From the definition the coefficients in \eqref{Ham:definina} satisfy the following:
\begin{enumerate}
    \item {\bf Symmetry restrictions:} for any $ \vec\sigma= (\sigma_1,\ldots ,\sigma_{p+2})\in \left\{ \pm\right\}^{p+2}$ and $\vec \jmath = (j_1,\ldots ,j_{p+2})\in (\Z^*)^{p+2},$ one has 
    \be \label{real+mom cond}
   \text{{\bf Reality condition:}} \ \  \overline{H_{\vec \jmath}^{\vec \sigma}}= H_{\vec \jmath}^{-\vec \sigma},  \quad \quad \textbf{Momentum condition:} \ \  H_{\vec \jmath}^{\vec \sigma} \neq 0 \implies\vec\sigma\cdot \vec \jmath =0.
    \ee
    \item {\bf Polynomial bound:} There are $ \mu, C>0$ and $m\in \R$ such that 
    $$
    \left| H_{\vec \jmath}^{\vec \sigma}\right|\leqslant C {\rm max}_3( |j_1|, \ldots, |j_{p+2}|)^\mu  {\rm max}_2( |j_1|, \ldots, |j_{p+2}|)^m.
    $$
\end{enumerate}

\begin{remark}
    In view of the momentum condition $ \vec \sigma \cdot \vec \jmath=0$ in \eqref{real+mom cond}, one has the equivalence
    $$
    {\rm max}( |j_1|, \ldots, |j_{p+2}|)\sim {\rm max}_2( |j_1|, \ldots, |j_{p+2}|).
    $$
 \end{remark}

 We now define the class of Hamiltonian systems up to homogeneity $N$ that we shall use in Section \ref{riduzione_e_stima}.

Let $K, K' \in \N$ with $ K'\leqslant K$, $r>0$  and $U \in B_{s_0}^K(I;r)$.
Let 
\be\label{Z.M0}
Z \triangleq  \bM_0(U;t)U \quad \text{ with }\quad  \bM_0(U;t) \in \pare{ \cM^0_{K,K',0}[\epsilon_0]}^{2\times 2}\, . 
\ee
\begin{definition}{\bf (Hamiltonian system up to homogeneity  $N$)} 
\label{def:ham.N}
Let $N, K, K' \in \N$  with $K \geqslant K'+1$ and
assume \eqref{Z.M0}.
A $U$--dependent system 
\be \label{U.Ham}
\pa_t Z = {\bm J}_\C \nabla H(Z) + {\bm M}_{> N}(U;t)[U] 
\ee
is {\em Hamiltonian up to homogeneity  $N$} if 

\noindent
$\bullet $ $H(Z) $ is a  pluri-homogeneous Hamiltonian as in \eqref{plurihomHAM};

\noindent
$\bullet $  ${\bm M}_{> N}(U;t)$ is a matrix of 
non-homogeneous operators in $ \pare{\cM_{K,K'+1,N+1}[\epsilon_0]}^{2\times 2}$. 
\end{definition}

We shall perform nonlinear changes of variables which are symplectic up to homogeneity $N$ according to the following definition.

\begin{definition}\label{def:LSMN}
{\bf (Symplectic map up to homogeneity $N$)} 
Let $ p, N \in \N $ with $p \leqslant N$. We say that  
\be\label{simplupN}
\cD (Z;t) = {\bm M}(Z;t)Z  \qquad \text{with} \qquad {\bm M}(Z;t) - \uno \in\pare{\Sigma \mM_{K,K',p}[r,N]}^{2\times 2} \, ,  
\ee
   is {\em symplectic up to  homogeneity  $N$}, if its pluri-homogeneous component $\cD_{\leqslant N}(Z)\triangleq  \big( \cP_{\leqslant N}  M(Z;t) \big) Z  $ satisfies 
 \be\label{milan}
\left[\di_Z \cD_{\leqslant N}(Z)\right]^\top \, {\bm E}_\C \, \di_Z \cD_{\leqslant N}(Z)
= {\bm E}_\C+ {\bm E}_{>N}(Z) \qquad \text{with} \qquad 
{\bm E}_{>N}(Z) \in   \pare{\Sigma_{N+1}\wt\cM_q}^{2\times2}  \, .
\ee
\end{definition}

\noindent 

A symplectic map up to homogeneity $ N $
transforms a Hamiltonian system up to homogeneity $N$ into 
another Hamiltonian system up to homogeneity $N$.

\begin{lemma}\label{conj.ham.N}
Let $p, N \in \N$ with $p \leqslant N$,  $K, K' \in \N$ with $K \geqslant K'+1$.
Let $ Z \triangleq  \bM_0(U;t)U  $ as in \eqref{Z.M0}.
Assume $\cD(Z;t)  = M(Z;t)Z   $ is a symplectic map up to homogeneity $N$ (Definition \ref{def:LSMN}) such that
\be \label{simdico}
M(Z;t) - \Id  \in
\system{
\begin{aligned}
 &\pare{\Sigma\cM_{K,K',p}[r,N]}^{2\times 2}, &
 \text{if} \quad  \bM_0(U;t)  = \uno \, , \\
  &\pare{\Sigma\cM_{K,0,p}[\breve r,N]}^{2\times 2} \, , & \, \forall \breve r > 0 \quad \text{otherwise} \, . 
\end{aligned}
}
\ee 
If  $Z(t)$ 
solves a $U$-dependent Hamiltonian system up to homogeneity $N$
(Definition \ref{def:ham.N}), 
then the variable
$ W \triangleq  \cD(Z;t)  $
solves another $U$-dependent Hamiltonian system up to homogeneity $N$ 
(generated by the transformed Hamiltonian). 
\end{lemma}
The following is Lemma 3.19 in \cite{BMM2022}.
\begin{lemma}\label{lem:hamsym} 
Let  $ p \in \N $, $ m \in \R $ and 
$  a (U;x, \xi) $  a real valued homogeneous 
symbol  in $ \wt\Gamma^m_p $.
Then the Hamiltonian vector field generated 
by the Hamiltonian 
$$
H (U)\triangleq \Re \psc{{\bm A}(U)u}{\bar u}_{\R}, \qquad {\bm A}(U)\triangleq \OpBW{a(U;x, \xi)}, 
$$
is 
$$
{\bm J}_\C \grad H (U) = \vOpbw{ \im a (U;x, \xi)} U + {\bm R}(U) U,
$$
where $ {\bm R}(U) $ is a real-to-real matrix of homogeneous smoothing operators in 
$ \pare{\wt\cR^{-\varrho}_p}^{2\times2}$ for any $\varrho \geqslant 0$.
\end{lemma}
The following is Lemma 3.20 in \cite{BMM2022}.
\begin{lemma}\label{HS:repre}
Let $p \in \N$, $m \in \R$ and  $\varrho \geqslant 0$. Let 
\be\label{XHS}
X(U) = {\bm J}_\C \OpBW{{\bm A}(U;x,\xi)}U + {\bm R}(U)U  = {\bm J}_\C \nabla H(U)
\ee 
be a $(p+1)$-homogeneous Hamiltonian vector field, where 
\be\label{auno}
{\bm A}(U;x,\xi) = 
\begin{pmatrix} 
a(U;x,\xi) & b(U;x,\xi) \\  
\overline {b(U;x,-\xi)}&  \overline{a(U;x,-\xi)}
\end{pmatrix}  
\ee
is a matrix of symbols  in $ \mats{\widetilde \Gamma_{p}^m}$
and $ {\bm R}(U)  $ is a real-to-real
 matrix of smoothing operators in 
$ \mats{\widetilde {\cal R}_{p}^{-\varrho}} $. 
Then, we may write 
 \be\label{XHSA1}
X(U) = {\bm J}_\C \OpBW{ {\bm A}_1(U;x,\xi) }U + {\bm R}_1(U)U, 
\ee 
where the matrix of paradifferential operators $\OpBW{{\bm A}_1(U;x,\xi)} $ is symmetric,   
with   matrix of symbols
\be\label{aunoa} 
{\bm A}_1(U;x,\xi) =  \frac12 
 \begin{pmatrix} 
a + a^\vee & b + \ov{b}\\ 
\ov{b}^\vee + b^\vee & \ov{ a+ a^\vee} 
\end{pmatrix} 
\ee
and  $ {\bm R}_1(U)$ is another real-to-real  matrix of smoothing operators  in 
 $\mats{\widetilde {\cal R}_{p}^{-\varrho} }$.
\end{lemma}
The following is Lemma 3.21 in \cite{BMM2022}.
\begin{lemma}\label{spezzamento}
Let $p \in \N$, $m \in \R$ and  $\varrho \geqslant 0$.
Let ${\bm S}(U)$ be a  matrix of spectrally localized homogeneous maps in $\pare{\wt \cS_{p}}^{2\times 2} $ which is linearly Hamiltonian of the form 
\be \label{SUAR}
 {\bm S}(U) =  {\bm J}_\C \OpBW{ {\bm A}(U;x,\xi )} +  {\bm R}(U) \, , 
\ee
where $ {\bm A}(U;x,\xi) $ is a real-to-real 
matrix of symbols in $ \pare{\wt\Gamma^m_p}^{2\times 2}$
as in  \eqref{auno}, 
 and $ {\bm R}(U) $ is a real-to-real matrix of  smoothing operators in $ \mats{\wt \mR^{-\vr}_{p}}$. 
 Then, 
 we may write 
 $$
 {\bm S}(U) = {\bm J}_\C\OpBW{ {\bm A}_1(U;x,\xi)} +  {\bm R}_1(U), 
$$
where  
the matrix of symbols $ {\bm A}_1(U;x,\xi) $ in $  \mats{ \wt \Gamma^m_{p}}$
has the form \eqref{aunoa}
 and  ${\bm R}_1(U)$ is another matrix of real-to-real  smoothing operators in 
 $\mats{\wt \mR^{-\vr}_{p}}$.
 In particular the homogeneous operator 
$ {\bm J}_\C\OpBW{ {\bm A}_1(U)}$ is linearly Hamiltonian.
\end{lemma}
\subsection{Symplectic corrections}
In this section we provide a symplectic correction to two different class of maps: linearly symplectic spectrally localized maps and smoothing perturbations of the identity. 
The main result is Theorem 7.1 in \cite{BMM2022} that we state below.
\begin{theorem}\label{conjham} 
Let $p, N \in \N $ with   $p \leqslant N$, $K, K' \in \N $ with 
$K'+1 \leqslant K$, $r >0$. 
Let 
$ Z = \bM_0(U;t)U $ with $ \bM_0(U;t) \in \left(\cM^0_{K,K',0}[\epsilon_0]\right)^{2x2} $. Assume 
that $ Z(t) $ solves a Hamiltonian system up to homogeneity $N$. Consider 
\be\label{defPhiUBU}
\Phi(Z)\triangleq \bB(Z;t)Z,
\ee
where
\begin{itemize}
\item 
$\bB(Z;t)-\uno$ is a   matrix of spectrally localized maps in
\be \label{c1c2Mteo71}
\bB(Z;t) - \Id   \in \begin{cases}
 \left(\Sigma\cS_{K,K',p}[r,N]\right)^{2x2}, \qquad \qquad 
 \text{if} \quad  \bM_0(U;t)  = \uno \, , \\
  \left(\Sigma\cS_{K,0,p}[\breve r,N]\right)^{2x2} \, , \, \forall \breve r > 0 \quad \text{otherwise} \, . 
\end{cases}
\ee 
\item $\bB(Z;t)$ is linearly symplectic up to homogeneity $N$, 
according to Definition \ref{linsymphomoN}.
\end{itemize}
Then, there exists a real-to-real matrix of pluri--homogeneous smoothing operators $ {\bm R}_{\leqslant N}( \cdot ) $ in $ \left(   \Sigma_p^N \wt \cR^{-\vr}_q\right)^{2x2} $, 
for any $\varrho >0 $,  such that the non-linear map  
$$ 
Z_+ \triangleq \big( \uno + {\bm R}_{\leqslant N}(\Phi(Z) )\big) \Phi(Z)
$$
is symplectic up to homogeneity $ N $  and thus $Z_+ $
solves a system which is Hamiltonian up to homogeneity $N$.
\end{theorem}
 Given a map of the form $ U \mapsto U+ {\bm J}_\C \nabla H_{p+2}(U)$ where  ${\bm J}_\C \nabla H_{p+2}(U)$ is a 
Hamiltonian, smoothing vector field,   we find a correction which  
is  symplectic  up to homogeneity $ N $. 

\begin{lemma}\label{lem:app.flow.ham}
Let $ p, N \in \N $ with $  p \leqslant N $. 
Let $ Y_{p+1}(U)= {\bm J}_\C \nabla H_{p+2}(U) $ be a  homogeneous 
Hamiltonian  smoothing vector field in 
$ \wt \X_{p+1}^{-\varrho} $ for some $ \varrho \geqslant 0 $. Then there is a map of the form 
\be \label{app:nlflo}  \mathfrak{F}_{\leqslant N}(U)= U + Y_{p+1}(U) + F_{\geqslant (p+2)}(U), \qquad F_{\geqslant (p+2)}(U)\in \Sigma_{p+2}^N \wt {\mathfrak X}^{-\vr}_q ,\ee
which is symplectic up to homogeneity $ N $ (Definition \ref{def:LSMN}).  
\end{lemma}
\begin{proof}
    It is a direct consequence of Lemmata 2.27 and 3.14 in \cite{BMM2}. A careful examination of the proofs reveals that  $F_{\geqslant (p+2)}(U)$  actually belongs to  $\Sigma_{2p+1}^N \wt {\mathfrak X}^{-\vr}_q$. However,  since this stronger result is not required for our purposes, we prefer to state the weaker conclusion $F_{\geqslant (p+2)}(U) \in \Sigma_{p+2}^N \wt {\mathfrak X}^{-\vr}_q $.
\end{proof}
\begin{remark}
    The map $\mathfrak{F}_{\leqslant N}(U)$ in \eqref{app:nlflo} is indeed the truncation up to homogeneity $N$ of the time-one flow generated by the Hamiltonian vector field $ {\bm J}_\C\nabla H_{p+2}(U)$. 
\end{remark}

\section{Auxiliary flows and conjugations}

The following 
conjugation Lemmata \ref{lem:conj.fou} and \ref{NormFormMap0} 
are used in the nonlinear Hamiltonian Birkhoff normal form reduction performed in Section \ref{riduzione_e_stima}. Their proof can be found in Appendix A of \cite{BMM2022}.

\smallskip

The following hypothesis shall be assumed in both Lemmata \ref{lem:conj.fou} and \ref{NormFormMap0}:
\\[1mm]
{\em {\bf Assumption (A)}: \label{A} \,
Assume  $Z \triangleq  \bM_0(U;t)U $ where $\bM_0(U;t) \in  \pare{\cM^0_{K,K',0}[\epsilon_0]}^{2\times 2}$, $U \in B^K_{s_0, \R}(I;\epsilon_0)$ for some $\epsilon_0,s_0>0$ and $0\leqslant K' \leqslant K$. 
 Let  $ N \in \N$  and assume that $Z$ solves the system
\begin{align}\label{Z.conj1}
\pa_t Z = \Opvec{ \ii \omega_{\kap,\upgamma}(\xi)+\ii a_{\leqslant N}(Z;\x)+ \ii a_{>N}(U;t,\xi)}Z + \bm R_{\leqslant N}(Z)Z+ \bm R_{>N}(U;t)U , 
\end{align}
where:
\begin{itemize}
\item $a_{\leqslant N} (Z;\x) $ is a real valued   pluri-homogeneous 
symbol, independent of $x$, in $ \Sigma_2^N \wt \Gamma^{\frac32}_q$; 
\item 
$a_{>N} (U;t,\xi) $ is a non-homogeneous symbol, independent of $x$, in $ \Gamma^{\frac32}_{K,K',N+1}[\epsilon_0]$ with imaginary part\\ 
${\rm Im} \, a_{>N} (U;t,\xi) $ in $  \Gamma^{0}_{K,K',N+1}[\epsilon_0]$;
\item 
 $\bm R_{\leqslant N} (Z)  $ is a real-to-real matrix of pluri-homogeneous  smoothing operators in $  \pare{\Sigma_1^N \tilde \cR^{-\vr}_q
}^{2\times 2} $; 
\item 
$\bm R_{>N} (U;t) $ is a real-to-real  matrix of
non-homogeneous smoothing operators 
in $ \pare{\mR^{-\vr}_{K,K',N+1}[\epsilon_0] }^{2\times 2} $.
\end{itemize}
}

\vspace{.5em}

\begin{lemma}[{\bf Conjugation under the flow of a Fourier multiplier}]\label{lem:conj.fou}
Assume {\bf (A)} at page \pageref{A}.
 Let $ g_p(Z;\xi)$ be a $p$--homogeneous real symbol independent of $x$ in $ \wt \Gamma^{\frac32}_p$,  $p \geqslant 2$, 
that we expand as 
\be \label{symtens}
g_p(Z;\x) = 
\sum_{
(\vec{\jmath}_p,  \vec{\sigma}_p) \in \fT_p
} G_{\vec{\jmath}_p}^{\vec{\sigma}_p}(\xi)  z_{\vec{\jmath}_p}^{\vec{\sigma}_p} \, , \qquad  \overline{G_{\vec{\jmath}_p}^{-\vec{\sigma}_p}}(\xi)= G_{\vec{\jmath}_p}^{\vec{\sigma}_p}(\xi)\in \C  
\ee
and denote by $ \pmb{\cG}_{g_p}(Z)\triangleq \pmb{\cG}_{g_p}^1(Z) $  the time $1$-flow defined in  \eqref{FourierFlow} generated by $\vOpbw{\im g_p(Z;\xi)}$.
If $Z(t)$ solves system \eqref{Z.conj1}, then the variable 
\be \label{new:foumu}
W\triangleq  \pmb \cG_{g_p} (Z) Z
\ee 
solves the system 
\be\label{Z.conj10}
\pa_t W =   \ii \vOmega(D) W +\vOpbw{ \ii a_{\leqslant N}^+(W;\x)+ \ii a_{>N}^+(U;t,\xi)}W  + 
\bm  R_{\leqslant N}^+(W)W+ \bm R^+_{>N}(U;t)U , 
\ee
where 
\begin{itemize}
\item  $ a^+_{\leqslant N} (W;\x) $ is a real valued pluri-homogeneous symbol, independent of $x$,  in  $\Sigma_2^N \wt \Gamma^{\frac32}_q $, with components
\be\label{ditigi0}
\begin{aligned}
& \cP_{\leqslant p-1} [a^+_{\leqslant N}(W;\xi)] =  \cP_{\leqslant p-1}[a_{\leqslant N}(W;\xi)] \, , \\
& \cP_{ p} \left[a^+_{\leqslant N}(W;\xi)\right]  = \cP_{ p} \left[a_{\leqslant N}(W;\xi) \right]
  + g^+_p(W;\xi),
  \end{aligned}
\ee
where $g^+_p(W;\xi) \in  \wt \Gamma^{\frac32}_p $ is the real, $x$-independent symbol 
\be \label{ditigi}
g^+_p(W;\xi)\triangleq \sum_{(\vec{\jmath}_p,  \vec{\sigma}_p) \in \fT_p}\ii \big(\vec{\sigma}_{p}\cdot \vec{\omega}_{\kap,\upgamma}(\vec{\jmath}_p) \big) G_{\vec{\jmath}_p}^{\vec{\sigma}_p}(\xi)  w_{\vec{\jmath}_p}^{\vec{\sigma}_p} ; \ee 
\item $a^+_{>N} (U;t,\xi) 
$ is a non-homogeneous symbol, independent of $x$, in $ \Gamma^{\frac32}_{K,K',N+1}[\epsilon_0]$ with imaginary part  ${\rm Im} \, a_{>N}^+ (U;t,\xi) $ 
belonging to $ \Gamma^{0}_{K,K',N+1}[\epsilon_0]$;
\item
$\bm R^+_{\leqslant N}(W) $ is a real-to-real matrix of pluri--homogeneous  smoothing operators in $ \pare{\Sigma_1^N \tilde \cR^{-\vr + c(N,p)}_q }^{2\times 2} $ for some $c(N,p) >0$ (depending only on $N,p$) and fulfilling 
\be\label{ditigi2}
\cP_{\leqslant p} [\bm R^+_{\leqslant N}(W)] =  \cP_{\leqslant p}[\bm R_{\leqslant N}(W)] \, ;
\ee
\item 
$\bm R^+_{>N}(U;t)$ is a real-to-real matrix of non--homogeneous smoothing operators in $ \pare{\mR^{-\vr+c(N,p)}_{K,K',N+1}[\epsilon_0] }^{2\times 2}$.
\end{itemize}
\end{lemma}

The following lemma describes how a system is conjugated under 
 a smoothing perturbation of the identity.

\begin{lemma}[{\bf Conjugation under a smoothing perturbation of the identity}]\label{NormFormMap0}
Assume {\bf (A)} at page \pageref{A}. 
Let $F_{\leqslant N}(Z) $ be a real-to-real matrix of pluri-homogeneous smoothing operators 
in $  \pare{\Sigma_p^N \wt \cR_q^{-\varrho'}}^{2\times 2}$ for some $ \varrho' \geqslant 0 $.
If $Z(t)$ solves \eqref{Z.conj1} 
 then the variable 
\be\label{W.Z1} 
W\triangleq  \cF_{\leqslant N}(Z)\triangleq  Z + F_{\leqslant N}(Z)Z
\ee
solves 
\begin{align}\label{normalformsmoo}
\pa_t W  = \ii & \vOmega(D) W +\vOpbw{ \ii a^+_{\leqslant N}(W;\x)+ \ii a^+_{>N}(U;t,\xi)}W + \bm R^+_{\leqslant N}(W)W+ \bm R^+_{>N}(U;t)U,
\end{align}
where 
\begin{itemize}
\item  
$ a^+_{\leqslant N} (W;\x) $ is a real valued pluri-homogeneous symbol, independent of $x$,  in  $\Sigma_2^N \wt \Gamma^{\frac32}_q $, with components 
\be\label{ditigi1}
\cP_{\leqslant p+1} [a^+_{\leqslant N}(W;\xi)] =  \cP_{\leqslant p+1}[a_{\leqslant N}(W;\xi)] \, ;
\ee
\item  $a^+_{>N} (U;t,\xi) 
$ is a non-homogeneous symbol, independent of $x$, in $ \Gamma^{\frac32}_{K,K',N+1}[\epsilon_0]$ with imaginary part 
${\rm Im} \, a_{>N}^+ (U;t,\xi) $ belonging to $ \Gamma^{0}_{K,K',N+1}[\epsilon_0]$;
\item 
$\bm R^+_{\leqslant N}(W) $ is a real-to-real matrix of 
pluri--homogeneous  smoothing operators in 
$ \pare{\Sigma_1^N \tilde \cR^{-\vr_*}_q }^{2\times 2}  $,  
$\varrho_* \triangleq  \min{(\varrho, \varrho'-\frac32)}$ ($\vr\geq0$ is the smoothing order in Assumption {\bf (A)} at page \pageref{A}), with components 
\be\label{ditigi20}
\cP_{\leqslant p-1} [\bm R^+_{\leqslant N}(W)] =  \cP_{\leqslant p-1}[\bm R_{\leqslant N}(W)] \ , 
\ee
and, denoting $\bm F_p(W) \triangleq  \cP_p(\bm F_{\leqslant N}(W))$ in $ \pare{\wt \cR^{-\varrho'}_p }^{2\times 2}$, one has 
\begin{align}\label{ditigi22}
\cP_{ p} [\bm R^+_{\leqslant N}(W)]  = \cP_p[\bm R_{\leqslant N}(W)]& +d_W \big(\bm F_{p}(W)W\big) \left[i\vOmega (D)\right]-i\vOmega (D)\bm F_{p}(W);
\end{align}
\item 
$\bm R^+_{>N}(U;t)$ is a real-to-real matrix of non--homogeneous smoothing operators in $ \pare{\mR^{-\vr_*}_{K,K',N+1}[\epsilon_0]}^{2\times 2} $.
\end{itemize}
In addition, if  $\cF_{\leqslant N}(Z)$ in \eqref{W.Z1} is the symplectic up to homogeneity $N$ map associated to a Hamiltonian vector field 
$\bm G_p(Z)Z= \bm J_\C \nabla H_{p+2}(Z)$ as per \Cref{lem:app.flow.ham}, where $\pare{\bm G_p(Z)\in \wt \mR^{-\vr'}_p}^{2\times 2}$ has Fourier expansion 
\be \label{gipi}
\big(\bm G_p(Z)Z\big)^\sigma_k=
\!\!\!\!\!\!\!\!\!\!\!\!\!\!\!\!\!\!
\sum_{
(\vec {\jmath}_{p+1}, k,  \vec{\sigma}_{p+1}, - \sigma) \in \fT_{p+2}
} 
\!\!\!\!\!\!\!\!\!\!\!\!\!\!\!\!\!\!
G_{\vec \jmath_{p+1},k}^{\vec{\sigma}_{p+1},\sigma}  z_{\vec{\jmath}_{p+1}}^{\vec{\sigma}_{p+1}} \, ,
\ee
then \eqref{ditigi22} reduces to 
\be\label{ditigi30}
\cP_{ p} [\bm R^+_{\leqslant N}(W)] =  \cP_{p}[\bm R_{\leqslant N}(W)]  + \bm G_p^+(W),
\ee
where $\bm G_p^+(W)\in  \pare{\wt \cR^{-\varrho'+\frac32}_p }^{2\times 2}$  is the smoothing operator with Fourier expansion 
\be\label{ditigi40}
\begin{aligned}
&(\bm G_p^+(W)W)^\sigma_k= 
\!\!\!\!\!\!\!\!\!\!\!\!\!\!\!\!\!\!
\sum_{
(\vec {\jmath}_{p+1}, k,  \vec{\sigma}_{p+1}, - \sigma) \in \fT_{p+2}
} 
\!\!\!\!\!\!\!\!\!\!\!\!\!\!\!\!\!\!
   \ii \big(  \vec{\sigma}_{p+1}\cdot  \vec{\omega}_{\kap,\upgamma}(\vec{\jmath}_{p+1})- \sigma {\omega}_{\kap,\upgamma}(k)\big)G_{ \vec{\jmath}_{p+1},k}^{\vec{\sigma}_{p+1},\sigma} \, w_{\vec{\jmath}_{p+1}}^{\vec{\sigma}_{p+1}} \,.
\end{aligned}
\ee
\end{lemma}


	\begin{footnotesize}
		\bibliography{references}
		\bibliographystyle{plain}
	\end{footnotesize}

\end{document}